\documentclass{amsart}

\usepackage[utf8]{inputenc}
\usepackage{lmodern}
\usepackage{microtype}
\usepackage[T1]{fontenc}
\usepackage[english]{babel}
\usepackage{amssymb,amsfonts,amsmath,latexsym,amsthm}
\usepackage{xcolor}
\definecolor{darkblue}{rgb}{0,0,0.6}
\usepackage{graphicx}
\usepackage[all,cmtip]{xy}
\usepackage[ocgcolorlinks,colorlinks=true, citecolor=darkblue, filecolor=darkblue, linkcolor=darkblue, urlcolor=darkblue]{hyperref}
\usepackage[nameinlink, capitalise]{cleveref}

\hypersetup{bookmarksopen=true}

\usepackage{array}
\usepackage{mathrsfs}
\usepackage{bm} 

\newtheorem{theorem}[equation]{Theorem} 
\newtheorem{prop}[equation]{Proposition}
\newtheorem{lem}[equation]{Lemma}
\newtheorem{kor}[equation]{Corollary}
\newtheorem{con-alt}[equation]{Conjecture}
\newtheorem{constr-alt}[equation]{Construction}
\theoremstyle{definition}
\newtheorem{ddd-alt}[equation]{Definition}
\newtheorem{ass-alt}[equation]{Assumption}
\newtheorem{prob-alt}[equation]{Problem}

\newenvironment{ddd}    
{%
	\pushQED{\qed}\begin{ddd-alt}}
	{\popQED\end{ddd-alt}}

{%
	\pushQED{\qed}\begin{ass-alt}}
	{\popQED\end{ass-alt}}

{%
	\pushQED{\qed}\begin{con-alt}}
	{\popQED\end{con-alt}}

{%
	\pushQED{\qed}\begin{constr-alt}}
	{\popQED\end{constr-alt}}

{%
	\pushQED{\qed}\begin{prob-alt}}
	{\popQED\end{prob-alt}}

\theoremstyle{remark}

\theoremstyle{definition}
\newtheorem{ex-alt}[equation]{Example}
\newtheorem{rem-alt}[equation]{Remark}

\newenvironment{ex}    
{%
	\pushQED{\qed}\begin{ex-alt}}
	{\popQED\end{ex-alt}}

\newenvironment{rem}    
{%
	\pushQED{\qed}\begin{rem-alt}}
	{\popQED\end{rem-alt}}

\crefname{theorem}{Theorem}{Theorems}
\crefname{lem}{Lemma}{Lemmas}
\crefname{prop}{Proposition}{Propositions}
\crefname{section}{Section}{Sections}
\crefname{ex-alt}{Example}{Examples}
\crefname{ddd-alt}{Definition}{Definitions}
\crefname{kor}{Corollary}{Corollaries}

\numberwithin{equation}{subsection}
\newcommand{\bA}{\mathbf{A}}
\newcommand{\bB}{\mathbf{B}}
\newcommand{\bC}{\mathbf{C}}

\newcommand{\bD}{\mathbf{D}}
\newcommand{\bE}{\mathbf{E}}
\newcommand{\bF}{\mathbf{F}}

\newcommand{\bI}{\mathbf{I}}
\newcommand{\bJ}{\mathbf{J}}

\newcommand{\bM}{\mathbf{M}}

\newcommand{\bP}{\mathbf{P}}
\newcommand{\bQ}{\mathbf{Q}}
\newcommand{\bR}{\mathbf{R}}

\newcommand{\bT}{\mathbf{T}}

\newcommand{\bV}{\mathbf{V}}

\newcommand{\A}{\mathbb{A}}

\newcommand{\G}{\mathbb{G}}

\newcommand{\nat}{\mathbb{N}}

\newcommand{\Q}{\mathbb{Q}}
\newcommand{\R}{\mathbb{R}}

\newcommand{\Z}{\mathbb{Z}}

\newcommand{\cA}{\mathcal{A}}
\newcommand{\cB}{\mathcal{B}}
\newcommand{\cC}{\mathcal{C}}
\newcommand{\cD}{\mathcal{D}}

\newcommand{\cF}{\mathcal{F}}

\newcommand{\cM}{\mathcal{M}}

\newcommand{\cP}{\mathcal{P}}

\newcommand{\cS}{\mathcal{S}}

\newcommand{\cU}{\mathcal{U}}

\newcommand{\cX}{\mathcal{X}}
\newcommand{\cY}{\mathcal{Y}}

\DeclareMathOperator{\burn}{\mathbf{A}^{\mathrm{eff}}}

\DeclareMathOperator{\Cofib}{Cofib}
\DeclareMathOperator{\unit}{unit}
\DeclareMathOperator{\Coind}{Coind}
\DeclareMathOperator{\coind}{coind}

\DeclareMathOperator*{\colim}{colim}

\DeclareMathOperator{\diag}{diag}

\DeclareMathOperator{\End}{End}
\DeclareMathOperator{\ev}{ev}

\DeclareMathOperator{\Fib}{Fib}

\DeclareMathOperator{\Fun}{{\mathbf{Fun}}}

\DeclareMathOperator{\ho}{ho}

\DeclareMathOperator{\Hom}{Hom}
\DeclareMathOperator{\Homol}{Hg}
\DeclareMathOperator{\id}{id}
\DeclareMathOperator{\Idem}{Idem}

\DeclareMathOperator{\incl}{incl}
\DeclareMathOperator{\Ind}{Ind}
\DeclareMathOperator{\indd}{ind}
\DeclareMathOperator{\Li}{\Li}
\DeclareMathOperator{\Map}{Map}
\DeclareMathOperator{\map}{map}
\DeclareMathOperator{\Nerve}{N}
\DeclareMathOperator{\pr}{pr}
\DeclareMathOperator{\Pro}{Pro}

\DeclareMathOperator{\Res}{Res}
\DeclareMathOperator{\res}{res}

\DeclareMathOperator{\Sing}{Sing}

\DeclareMathOperator{\tr}{tr}
\DeclareMathOperator{\tw}{Tw}

\DeclareMathOperator{\Yo}{Yo}
\DeclareMathOperator{\yo}{yo}

\renewcommand{\emptyset}{\varnothing}

\newcommand{\Ab}{{\mathbf{Ab}}}
\newcommand{\Add}{\mathbf{Add}}
\newcommand{\Alg}{{\mathbf{Alg}}}
\newcommand{\Born}{\mathbf{Born}}
\newcommand{\BC}{\mathbf{BC}}
\newcommand{\Cat}{\mathbf{Cat}}
\newcommand{\CAT}{\mathbf{CAT}}
\newcommand{\Coarse}{\mathbf{Coarse}}

\newcommand{\Ch}{{\mathbf{Ch}}}

\newcommand{\Mod}{{\mathbf{Mod}}}

\newcommand{\PSh}{{\mathbf{PSh}}}

\newcommand{\Set}{{\mathbf{Set}}}
\newcommand{\Sh}{{\mathbf{Sh}}}
\newcommand{\Sp}{\mathbf{Sp}}
\newcommand{\Spc}{\mathbf{Spc}}
\newcommand{\sSet}{{\mathbf{sSet}}}
\newcommand{\Top}{{\mathbf{Top}}}

\newcommand{\gbct}[1]{#1\mathbf{BC}^{\mathrm{born}}}

\newcommand{\homot}[1]{\Ch^b(#1)_\infty}
\newcommand{\SW}{\mathbf{SW}}

\newcommand*\cocolon{%
	\nobreak
	\mskip6mu plus1mu
	\mathpunct{}%
	\nonscript
	\mkern-\thinmuskip
	{:}%
	\mskip2mu
	\relax
}

\newcommand{\bEnd}{\mathbf{End}}
\newcommand{\cp}{\mathbf{CP}}
\newcommand{\FDC}{\mathbf{FDC}}
\newcommand{\Fl}{\mathbf{Fl}}
\newcommand{\preFl}{\mathbf{Fl}^{\mathrm{pre}}}
\newcommand{\Tw}{\mathbf{Tw}}
\newcommand{\Vcyc}{\mathbf{Vcyc}}

\DeclareMathOperator{\Aut}{Aut}

\newcommand{\As}{\mathrm{Ass}}
\newcommand{\cFDC}{\mathbf{FDC^{cp}}}

\newcommand{\ad}{\mathrm{ad}}
\newcommand{\free}{\mathrm{free}}

\newcommand{\UK}{\mathrm{UK}}
\newcommand{\Rel}{\mathbf{Rel}}

\newcommand{\All}{\mathbf{All}}

\newcommand{\Orb}{\mathbf{Orb}}

\newcommand{\Fin}{\mathbf{Fin}}

\newcommand{\Loc}{{\mathrm{ Loc}}}

\renewcommand{\hat}{\widehat}
\renewcommand{\tilde}{\widetilde}

\newcommand{\Clp}{\Cat^{\mathrm{Lex,perf}}_{\infty,*}}
\newcommand{\Cle}{\Cat^{\mathrm{Lex}}_{\infty,*}}
\newcommand{\Cre}{\Cat^{\mathrm{Rex}}_{\infty,*}}
\newcommand{\Crp}{\Cat^{\mathrm{Rex,perf}}_{\infty,*}}
\newcommand{\CL}{\mathbf{coPr}_{\omega,*}^{\mathrm{R}}}
\newcommand{\CLL}{\mathbf{CAT}^{\mathrm{Lex}}_{\infty,*}}
\newcommand{\CLLL}{\mathbf{CAT}^{\mathrm{cplt}}_{\infty,*}}
\newcommand{\bd}{\mathrm{bd}}
\newcommand{\op}{\mathrm{op}}
\newcommand{\CATi}{\mathbf{CAT_{\infty}}}
\newcommand{\Cati}{\mathbf{Cat_{\infty}}}
\newcommand{\eqsm}{\mathrm{eqsm}}
\newcommand{\Bc}{\mathbf{BC}}
\newcommand{\fl}{\mathrm{fl}}
\newcommand{\Flrm}{\mathrm{Fl}}
\newcommand{\pre}{\mathrm{pre}}

\newcommand{\mb}{\mathrm{mb}}
\newcommand{\stCat}{\mathbf{Cat}^{\mathrm{ex}}_{\infty}}

\newcommand{\perf}{\mathrm{perf}}
\newcommand{\sub}{\mathrm{sub}}
\newcommand{\Clep}{\Cat^{\mathrm{Lex,perf}}_{\infty,*}}
\newcommand{\Catlex}{\mathbf{Cat}^{\mathrm{Lex}}_{\infty}}
\newcommand{\Prl}{\mathbf{Pr}^{\mathrm{L}}_{\omega}}
\newcommand{\Prr}{\mathbf{Pr}^{\mathrm{R}}_{\omega}}
\newcommand{\Prlp}{\mathbf{Pr}^{\mathrm{L}}_{\omega,*}}

\newcommand{\cop}{\mathrm{cp}}

\newcommand{\Cex}{\mathbf{Cat}_{\infty}^{\mathrm{ex}}}

\newcommand{\add}{\mathrm{add}}
\newcommand{\Cadd}{\Cat^{\add}_\infty}

\begin{document}


\title[Controlled objects in $\infty$-categories and the Novikov conjecture]{Controlled objects in left-exact $\infty$-categories and the Novikov conjecture}

\author[U.~Bunke]{Ulrich Bunke}
\address{Fakult{\"a}t f{\"u}r Mathematik,
	Universit{\"a}t Regensburg,
	93040 Regensburg,
	Germany}
\email{ulrich.bunke@mathematik.uni-regensburg.de}

\author[D.-C.~Cisinski]{Denis-Charles Cisinski}
\address{Fakult{\"a}t f{\"u}r Mathematik,
	Universit{\"a}t Regensburg,
	93040 Regensburg,
	Germany}
\email{denis-charles.cisinski@mathematik.uni-regensburg.de}

\author[D.~Kasprowski]{Daniel Kasprowski}
\address{School of Mathematical Sciences, University of Southampton, \newline\indent Southampton SO17 1BJ, United Kingdom}
\email{d.kasprowski@soton.ac.uk}

\author[C.~Winges]{Christoph Winges}
\address{Fakult{\"a}t f{\"u}r Mathematik,
	Universit{\"a}t Regensburg,
	93040 Regensburg,
	Germany}
\email{christoph.winges@ur.de}

\date{\today}

\begin{abstract}
	We associate to every $G$-bornological coarse space $X$ and every left-exact $\infty$-category with $G$-action a left-exact infinity-category of equivariant $X$-controlled objects.
	Postcomposing
	with algebraic K-theory leads to new equivariant coarse homology theories.
	This allows us to apply the injectivity results for assembly maps by Bunke, Engel, Kasprowski and Winges to the algebraic K-theory of left-exact $\infty$-categories.
  \end{abstract}
\maketitle
\tableofcontents

\section{Introduction}\label{georgergregregregg}

This paper concerns
the construction of $G$-equivariant coarse homology theories  {in the sense of \cite[Def.~3.10]{equicoarse}}.
Given a left-exact $\infty$-category with $G$-action, we first construct a functor which associates to every $G$-bornological coarse space a new left-exact $\infty$-category  
of equivariant controlled objects.   
The coarse homology theory is then obtained by composing this functor with a localising invariant from left-exact $\infty$-categories to some  target  stable $\infty$-category.
We employ these equivariant coarse homology theories in order to study properties of assembly maps.

Any equivariant coarse homology theory can be restricted (see \cref{wrtohpwrtgrgrgregwr})
to a functor, denoted by $M \colon G\Orb\to \bM$ for the moment, on the orbit category $G\Orb$.
We then consider the assembly map 
\begin{equation}\label{rgoijgoiergregregwgregwgregrefwwrfwer1fff}
 \As_{\Fin,M}\colon \colim_{G_{\Fin}\Orb} M\to M(*)\ ,
\end{equation}
which
 {approximates} the value $M(*)$ of $M$ on the final object of $G\Orb$ by its values on 
 the subcategory $G_{\Fin}\Orb$ of orbits with finite stabilisers.
The word {\em Novikov conjecture} from the title refers to the assertion that this assembly map is split injective under certain conditions. We describe the history of this term in greater detail in \cref{wtgijowgerfwerfewrf}.

The relevance of coarse homology theories for the verification of split injectivity of the assembly map  \eqref{rgoijgoiergregregwgregwgregrefwwrfwer1fff} 
stems from the axiomatic approach to this question developed in  \cite{desc}, which builds on a long tradition of proofs using similar methods \cite{calped,br:asymptoticdim,gty:novikov,kasprowski:fdc}
The essential assumption on the functor $M$ is the  CP-condition
 {which we recall in} \cref{bioregrvdfb}.
It requires that $M$ arises from an equivariant coarse homology theory as  in \cref{wrtohpwrtgrgrgregwr} and that this coarse homology theory has various additional properties.

 We verify that the equivariant coarse homology theory constructed from 
a left-exact $\infty$-category  with $G$-action $\bD$ and algebraic $K$-theory in place of the localising invariant has the required properties to ensure that the resulting functor, denoted
by $K\bD_{G} \colon G\Orb\to \bM$ in \eqref{fwqewwedeqdwedqwdwdedqwdwedwd}, is a CP-functor. 
This approach
subsumes various previously known cases, but also adds new examples of functors on the orbit category which are
 therefore known to satisfy the CP-condition.

In \cref{sgiojeriosgegsesfefe}, we start with a more detailed discussion of the construction of functors on the orbit category from left-exact $\infty$-categories with $G$-action and localising invariants.
In particular, we
explain how some of the classical examples  {of functors on the orbit category} can be considered as special cases of our general construction.
In \cref{rgigjwoiegerggwregwerg},  we
provide a sample split injectivity result for the assembly map derived by combining \cite{desc} with the results of the present paper.

In \cref{contr} we
give a detailed overview on the construction of coarse homology theories from 
left-exact $\infty$-categories with $G$-action. The technical details of this construction account for the main body of this paper.

\subsection{Functors on the orbit category and split injectivity of assembly maps} \label{sgiojeriosgegsesfefe}
 {Let $G$ be a group.}
The orbit category $G\Orb$ is the category of transitive $G$-sets and equivariant maps.
For a family $\cF$ of subgroups of $G$ (\cref{etghoiwrththrheh}), let $G_{\cF}\Orb$ denote the full subcategory of the orbit category  {consisting} of $G$-sets with stabilisers in $\cF$ (\cref{ithowthwergreggwgergwrgwreg}).

We consider a functor
  $M\colon G\Orb\to \bM$    with a cocomplete target $\infty$-category.
 For any pair 
 of  families $\cF^{\prime}$ and $\cF$ of subgroups of $G$ such that $\cF^{\prime}\subseteq \cF$    we then  have a relative  assembly map (see \cref{ergoiegererg})
\begin{equation}\label{rgoijgoiergregregwgregwgregrefwwrfwerf}
 \As_{\cF^{\prime},M}^{\cF}\colon \colim_{G_{\cF^{\prime}}\Orb} M\to \colim_{G_{\cF}\Orb}M\ .
\end{equation}
It is a morphism between objects of $\bM$ and
induced by the inclusion of the index categories of the colimits in \eqref{rgoijgoiergregregwgregwgregrefwwrfwerf}.

 A natural question about the assembly map is whether it is an equivalence or at least split injective. The split injectivity question has been studied axiomatically in \cite{desc}
In this approach, the main assumption on the functor $M$ is that it is a CP-functor.

 As said above, being a CP-functor requires that $ {M}$ 
 extends to an equivariant coarse homology theory in a particular way, and that this equivariant coarse homology theory has various additional properties.
 Our main contribution in this direction is \cref{giejrgoergwergwergwergwergwrg} below stating that $K {\bD}_{G}$ is a CP-functor. We start with  {a} precise description of this functor.

A left-exact 
$\infty$-category  is an $\infty$-category which  contains a zero object and admits all finite limits. 
A functor between left exact $\infty$-categories is called left-exact if it preserves finite limits.
 We let $\Cle$ denote the large  {$\infty$-category}
 of small left-exact $\infty$-categories and left-exact functors, see \cref{whgiowgergrewgrgwgergwerg}.

Small left-exact $\infty$-categories with $G$-actions are objects of the functor category $\Fun(BG,\Cle)$.
For the following, we fix a small left-exact $\infty$-category with $G$-action $\bD$.

  The group $G$ considered as a $G$-set with the $G$-action by left translations is an object of $G\Orb$. 
  Its group of automorphisms is
  $G$ acting by right translations. By $BG$ we denote the groupoid
  consisting of a single object with group of automorphisms $G$.  
 {Sending the unique object in $BG$ to the free orbit $G$}
  provides an embedding
\begin{equation}\label{rewflkjmo34gergwegrge}
 j\colon BG\to G\Orb \ .
\end{equation}
  Since $\Cle$ admits all small 
  colimits (\cref{ioerjgoiegergwegergwrgwegrwerg}),
  we can form the left Kan extension
 \begin{equation}\label{egwkenkjergreggwergwergerwgeg}
 {j_!}\bD \colon G\Orb \to \Cle
\end{equation}
of $\bD$ along $j$.
We further compose $j_{!}\bD$ with the algebraic K-theory functor $K \colon \Cle \to \Sp$ in order to define the functor
\begin{equation}\label{fwqewwedeqdwedqwdwdedqwdwedwd}
K\bD_G := K \circ j_!\bD\colon G\Orb\to \Sp\ ,
\end{equation}
see \cref{ergoiergergergegerwgerggrerg43252} and \cref{regiowergerregeggregw} for details.

 {We refer to \cref{regiojergewergegergweggegregrwegreg} for a precise definition of the term ``hereditary CP-functor''}.
\begin{theorem}[\cref{rgiuhreiguhgwergergrwegwergreg}]\label{giejrgoergwergwergwergwergwrg}
The functor $K\bD_{G}$ is a hereditary CP-functor.
\end{theorem}

As said above, the CP-condition on $K\bD_{G}$ allows us to apply the axiomatic
approach to injectivity results for assembly maps developed in \cite{desc}.
The following theorem describes a typical example of such an application.

\begin{theorem}\label{rgigjwoiegerggwregwerg}
 Assume:
 \begin{enumerate}
 \item $G$ admits a finite-dimensional $CW$-model for the classifying space $E_{\Fin}G$.
 \item  \label{rgiogerervweeververvwevrvrevgerreferferefrgwerg} $G$ is a finitely generated subgroup of a linear group over a commutative ring with unit or of a virtually connected Lie group.
  \end{enumerate}
 Then the assembly map
 \[ \As_{\Fin,K\bD_{G}}^{\All} {\colon \colim_{G_\Fin\Orb} K\bD_G \to  \colim_{G \Orb} K\bD_G  } \]
 admits a left inverse. 
\end{theorem}Using that $*$ is the final object of $G\Orb$, we can identify the target of this assembly map with $K\bD_{G}(*)$ in order to get the version \eqref{rgoijgoiergregregwgregwgregrefwwrfwer1fff}.
\cref{giejrgoergwergwergwergwergwrg}
 exhibits \cref{rgigjwoiegerggwregwerg} as a consequence of \cref{oijfoifjoewfwejfoijoiewfqwefqwefqwefwefewfffqef}. 
 For a detailed review of the general results of \cite{desc} involving more complicated assumptions on $G$ (e.g.~the condition of finite decomposition complexity),
 we refer to \cref{wreuigheiugerergegwggreg}.

The assumptions
required in the theorems listed in   \cref{wreuigheiugerergegwggreg} can be separated into assumptions on the group  $G$ and the families $\cF^{\prime},\cF$ on the one hand, and the assumption on the functor $M$ being a CP-functor, see \cref{bioregrvdfb}, on the other hand. The present paper contributes to  the latter.  In particular, we make  no attempt to  enlarge the class of groups for which injectivity results are known.

 The Farrell--Jones conjecture predicts that the assembly map $\As_{\Vcyc,K\bD_G}^\All$ for the family of virtually cyclic subgroups $\Vcyc$ is an equivalence. 
Generalising work of Bartels \cite{bartels:domain}, we show  in  \cref{kor:vcyc-bartels} that the relative assembly map $\As_{\Fin,K\bD_G}^\Vcyc$ is always split injective. Therefore,
\cref{rgigjwoiegerggwregwerg} can also be read as providing evidence towards the Farrell--Jones conjecture.
In fact, the coarse homology theories constructed in the present paper are
used crucially in \cite{fvvsdfvsdfvfvsdfv} in order to extend
proofs of the Farrell--Jones conjecture from the linear case (see \cref{reoifjuoiewccwec}) to the version stated above.

 We now explain the relation between the functor \eqref{fwqewwedeqdwedqwdwdedqwdwedwd}
 and examples of functors whose assembly maps have been classically considered. We start with recalling their constructions.

\begin{ex} \label{reoifjuoiewccwec} 
The motivating and guiding example for our approach is  the equivariant algebraic $K$-theory functor
\[ K\bA_{G}\colon G\Orb\to \Sp \]
associated to an additive category $\bA$ 
with  a strict $G$-action.
We refer to this case as the linear case as opposed to the derived case.
 The functor $K\bA_{G}$ has first been constructed in \cite{davis_lueck}.
 
  In the following, we give a quick alternative construction of $ K\bA_{G}$   
which is analogous to the construction of the functor in \eqref{fwqewwedeqdwedqwdwdedqwdwedwd} above. We consider the   large $\infty$-category $\Add_{\infty}$   of small additive categories 
obtained from the category of small additive categories and additive functors by inverting equivalences. We then interpret $\bA$  as an object
 $\bA_{\infty}$  of $\Fun(BG,\Add_{\infty})$. We
  denote the left Kan extension of $\bA_{\infty}$ along $j$ by
\[ {j_!}\bA_{\infty}\colon G\Orb\to \Add_{\infty} \ .\]
 Finally, we let
 \begin{equation}\label{g54245g45g8uijo5gergergerg}
K^{\mathrm{Add}}\colon \Add_{\infty}\to \Sp
\end{equation}
be the non-connective $K$-theory functor for additive categories (constructed by  Pedersen--Weibel \cite{MR802790} and Schlichting \cite{MR2079996}). 
We then define the composed functor
\begin{equation}\label{sdavjoiwqejffvvsdvasdvadsvadsvasdvdsv}
 K\bA_{G}:=K^{\mathrm{Add}}\circ  {j_!}\bA_{\infty}\colon G\Orb\to \Sp
\end{equation} 
(compare with \cref{ergoiergergergegerwgerggrerg43252}).

\begin{theorem}[{\cite[Ex.~1.10]{desc},\cite[Ex.~2.6]{desc}}]\label{uiihfqwefwefffqfewf}
The functor $K\bA_{G}$ is a  hereditary CP-functor.
\end{theorem}
The argument for this result given in \cite{desc} is short but heavily uses results from \cite{equicoarse} and the quite  technical paper \cite{coarsetrans}.
 \cref{uiihfqwefwefffqfewf} allows us to apply the split injectivity results  for assembly maps from \cite{desc} to $K\bA_G$.
\end{ex}

\begin{ex}\label{ex:Atheory}
In \cite{Bunke:aa}, we studied the {nonconnective} equivariant Waldhausen $A$-theory functor
\[ \bA_{P}\colon G\Orb\to \Sp\ , \quad S\mapsto A(P\times_{G}S) \]
associated to a $G$-principal bundle with total space $P$,  {which is defined in terms of the algebraic K-theory of certain Waldhausen categories of retractive spaces, and proved the following.}
\begin{theorem}[{\cite[Thm.~5.17]{Bunke:aa}}] \label{regioregewgwergergew}
	$\bA_{P}$ is a hereditary $CP$-functor.
\end{theorem}
 {Again, this allows us to apply the split injectivity theorems about assembly maps in \cite{desc} to $\bA_P$.}
\end{ex}

In the following two examples we  show that  \cref{uiihfqwefwefffqfewf} and
\cref{regioregewgwergergew} can be viewed as special cases of \cref{giejrgoergwergwergwergwergwrg}.
 
\begin{ex}\label{gwegljrgkrewgergreg}
 Let $\bA$ be  {an additive category with strict $G$-action.}
 Then we can form the category of bounded chain complexes $\Ch^{b}(\bA)$ in $\bA$.
 By \cite[Prop.~2.7]{Bunke:2017aa}, the localisation $\Ch^{b}(\bA)_{\infty}$ of $\Ch^{b}(\bA)$ at the chain  homotopy equivalences is a stable $\infty$-category.
 Taking the $G$-action into account, we obtain 
 {a left-exact $\infty$-category with $G$-action} $\Ch^{b}(\bA)_{\infty}$  which can serve as input
 for our theory. 
  
 As we will show in \cref{eiohjgwergerqgergergdfbsdd},  there is 
 an equivalence 
 \[ K\Ch^{b}(\bA)_{\infty,G} \simeq K\bA_{G} \]
  of functors $G\Orb\to \Sp$, so \cref{giejrgoergwergwergwergwergwrg} strictly generalises \cref{uiihfqwefwefffqfewf}.
\end{ex}

\begin{ex}\label{ex:spacesop}
Consider the presentable $\infty$-category of pointed spaces $\Spc_{*}$. The full subcategory $\Spc_{*}^{\op,\omega}$ of cocompact objects in its opposite belongs to $\Cle$. 
This is a non-stable example which can serve as input for our theory.

Let $\ell \colon \Top \to \Spc$ be the canonical localisation functor.
Every topological $G$-space $P$ (i.e.~an object of $\Fun(BG,\Top)$) gives rise to an object $\ell(P)$ in $\Fun(BG,\Spc)$.
Since $\Cle$ is cocomplete, it is tensored over spaces
and we can form  {the left-exact $\infty$-category with $G$-action}
$\ell(P) \otimes \Spc_*^{\op,\omega}$. 
 By \cref{cor:A-coeffs} and \eqref{trbebertoijboireberb},
  there is an equivalence of functors
\[ K(\ell(P) \otimes \Spc_*^{\op,\omega})_G \simeq \bA_P \]
for every principal $G$-bundle $P$.
  Hence \cref{regioregewgwergergew} is also a special case of \cref{giejrgoergwergwergwergwergwrg}.
\end{ex}

 {The setting of left-exact $\infty$-categories, which includes all stable $\infty$-categories, allows us to consider a number of further examples.}

\begin{ex}  \label{werogijergegwegergrewref}
 Let $R$ in $\Alg(\Sp)$ be an associative ring spectrum, and let $\Mod(R)$ be its stable $\infty$-category of right modules.
 The $\infty$-category $\Mod(R)$ is  {compactly generated} presentable, and its subcategory $\Mod(R)^{\perf}$ of compact objects is an essentially small, idempotent complete stable $\infty$-category.
 {For technical reasons, one should consider its opposite $\Mod(R)^{\perf,\op}$, which is equivalently the full subcategory $\Mod(R)^{\op,\omega}$ of cocompact objects in $\Mod(R)^\op$, as input for our machinery.}
 
 {We equip $\Mod(R)^{\op,\omega}$ with the trivial $G$-action.
 In this case, one can check that
 \begin{equation}\label{rgoih1io3joir3gregwregqrg}
  j_!(\Mod(R)^{\op,\omega})(G/H) \simeq  \Mod(R[H])^{\op,\omega}\ ,
 \end{equation}
 where $R[H]$ is the group ring of $H$ with coefficients in $R$.}
\end{ex}
 
We conclude this introduction with an indication how \cref{giejrgoergwergwergwergwergwrg} is shown.
The first condition for a CP-functor (\cref{bioregrvdfb}) is that its target $\infty$-category is stable, complete, cocomplete, and compactly generated.
Note that the $\infty$-category of spectra $\Sp$ has all these properties. 

The remaining conditions for $K\bD_{G}$ being a CP-functor require that it extends (in a particular way, see below) to an equivariant coarse homology theory $E$, and that this equivariant coarse homology theory is continuous, strongly additive, and admits transfers.
It is the realisation of this condition which connects the study of CP-functors with the construction of equivariant coarse homology theories using controlled object functors.  
In the following, we need the category $G\BC$ of $G$-bornological coarse spaces (see \cref{regioergrgfewgwgregrwg}) and the notion of an equivariant coarse homology theory $E\colon G\BC\to \Sp$ (\cref{erguiheriwgregregwgrwegwg1}).
These notions were introduced in \cite{equicoarse} in the precise form needed.

The phrase ``extends in a particular way''  means that  for every $S$ in $G\Orb$
there is a natural equivalence
\begin{equation}\label{dvkjhienjckdwcadcdcdsacasdc}
K\bD_{G}(S)\simeq E(G_{can,min}\otimes S_{min,max})\ ,
\end{equation}
where we refer to \cref{qrgioqjrgoqrqfewfeqfqewfe} and \cref{etwgokergpoergegregegwergrg} for the notation appearing in the argument of $E$.

We show, as a consequence of \cref{efioqjwefoweweqfewfqwefqwefe}, that the functor $K\bD_{G}$ is the restriction of the equivariant coarse homology theory
\[ K\bC\cX_{G}\colon G\BC\to \Sp \]
defined in \cref{gfoiqerjgoergrgegwregrgregwre} with $\bC:=\Pro_{\omega}(\bD)$  {being the pro-completion of $\bD$.} 
More precisely, for every  {transitive $G$-set} $S$ 
there is a natural equivalence
\begin{equation}\label{ztnpokropgrtgwgwregwe}
K\bD_{G}(S)\simeq K\bC\cX_{G}(S_{min,max})\ .
\end{equation}
But this is not yet the correct ``particular way'' of extending as indicated in \eqref{dvkjhienjckdwcadcdcdsacasdc}.  
The correct equivariant coarse homology theory which has to be taken for $E$ is the coarse algebraic $K$-homology
\[ K\bC\cX^{G}\colon G\BC\to \Sp \]
with coefficients in $\bC$ which is defined in \cref{ergoierwjgowregregrwegwregwregw}.
By \cref{cor:coeffs-orbits}, we get a natural equivalence
\[ K\bC\cX_{G}(S_{min,max})\simeq K\bC\cX^{G}(G_{can,min}\otimes S_{min,max})\ ,\]
which together with \eqref{ztnpokropgrtgwgwregwe} gives the required natural equivalence
\[ K\bD_{G}(S)\simeq K\bC\cX^{G}(G_{can,min}\otimes S_{min,max})\ .\]
We now have to see that the equivariant coarse homology theory $K\bC\cX^{G}$ is strongly additive, continuous,  and admits transfers.

Continuity is built into the definition of $K\bC\cX^{G}$, see \cref{gioegwgwrgrgreggwregwrg}. 

The condition of strong additivity (\cref{wergiugewrgergergerwgwregw}) heavily depends on the fact that the $K$-theory functor  preserves products (\cref{rwqioqewrfeeffqefe}).
At the moment, we do not know any other non-trivial finitary localising invariant  (\cref{qrevoiqrjoirqfcwqecq}) with this property. 

Finally, the existence of transfers   depends on the construction of $K\bC\cX^{G}$ via categories of controlled objects.
We refer to \cref{tiowtwrtbtwtbewbtw} for the details.
This section is the derived analogue of the paper \cite{coarsetrans} covering the linear case.

 \begin{rem}\label{wtgijowgerfwerfewrf} 
 In its original formulation \cite[Section~11]{novikov:hermK2}, the Novikov conjecture asserts the homotopy invariance of higher signatures: given an oriented, closed and connected $n$-manifold $M$ with fundamental group $G$, every cohomology class $x$ in $H^{n-4k}(BG;\Q)$ gives rise to a higher signature
\[ \mathrm{sign}_x(M) := \langle \mathcal{L}_k(M) \cup c^*x, [M] \rangle \in \Q\ ,\]
where $\mathcal{L}_k(M) $ in $H^{4k}(M)$ is the $k$-th Hirzebruch polynomial and $c \colon M \to BG$ denotes the classifying map of the universal cover of $M$.
Novikov conjectured that these higher signatures are invariant under orientation-preserving homotopy equivalences.
See \cite{frr:novikov, kl:novikov} for detailed surveys on this question.
Surgery theory translated this conjecture into entirely homotopy-theoretic terms by showing that it is equivalent to the rational injectivity of the L-theoretic assembly map $BG \otimes \mathbb{L}(\Z) \to \mathbb{L}(\Z[G])$, see \cite[Ch.~17H]{wall} and \cite[Proposition~24.5]{ranicki:algL}.
For this reason, the analogous statement in algebraic K-theory also acquired the name {\em Novikov conjecture}. The latter has been proved for all groups with a classifying space of finite type \cite{bhm:trace}.
Over time, it   {has} become custom to use this name for a variety of injectivity results concerning assembly maps in K- and L-theory.
 
 Most prominently, this concerns the split injectivity of the integral assembly map for torsionfree groups considered by Carlsson and Pedersen \cite{calped}, Bartels and Rosenthal \cite{br:asymptoticdim} and Guentner, Tessera and Yu \cite{gty:novikov}.
These theorems were generalised to groups with torsion by Kasprowski \cite{kasprowski:fdc}.
To accomodate groups with torsion, one replaces the domain of the assembly map by a more general object which is best explained in terms of functors on the orbit category of a group as in \eqref{rgoijgoiergregregwgregwgregrefwwrfwer1fff}.
\end{rem}

\subsection{Controlled objects and coarse homology theories} \label{contr} 
 
Coarse geometry was invented by J.~Roe \cite{roe_lectures_coarse_geometry} \cite{roe_index_1}, \cite{MR1147350}.
Partially motivated by the study of assembly maps, controlled topology has been developed e.g.~in  
\cite{calped}, \cite{MR2030590},\cite{blr}, \cite{Bartels:2011fk} as a parallel branch.
Eventually, it has been observed in \cite{higson_pedersen_roe}, \cite{MR1834777}, \cite{ass} and other places  that one can interpret controlled topology as a part of coarse geometry via the cone construction.   
   
 In \cite{buen} (the non-equivariant case) and  \cite{equicoarse} (the equivariant case),  we provided a formal framework for coarse geometry and axiomatised the notion of a (equivariant) coarse homology theory
 (see \cref{erguiheriwgregregwgrwegwg1}).
 This framework subsumes the proper metric spaces studied in classical coarse geometry and the cones considered in controlled topology, but also allows more general constructions.
   
 More specifically, in \cite{equicoarse} we introduced the category of $G$-bornological coarse spaces $G\BC$ (\cref{gjwerogijwoergwergwergweg}).
 These are $G$-sets equipped with a compatible  $G$-bornology and $G$-coarse structure, see \cref{wthoiwhthwgreggwregwgr,trbertheheht,rgejqieogjrgoij1o4trqq}.
 The bornology is used to encode local finiteness conditions, while the coarse structure captures the large-scale geometry.
 While in the classical definition by Roe the bounded sets are determined by the coarse structure, in the case of $G$-bornological spaces there is much more freedom for the choice of the  bornology.
  
 Recall that in homotopy theory one studies topological spaces (or simplicial sets) up to weak equivalence.
 Analogously, the homotopy theory of bornological coarse spaces studies bornological coarse spaces up to
coarse equivalence (\cref{rewkgowegrerfrewfwr}) and flasques (\cref{rgiojgogregrgregre}).
 Homotopical invariants of $G$-bornological coarse spaces {which in addition satisfy an appropriate version of excision} are given by the evaluation of equivariant coarse homology theories
 \begin{equation}\label{gwergewg25tergregwg}
E\colon G\BC\to \bM\ ,
\end{equation}
where  $\bM$ is a cocomplete stable $\infty$-category, e.g., the category of spectra.
In \cite{equicoarse},  we  constructed the universal equivariant coarse homology theory
\[ \Yo^{s}\colon G\BC\to G\Sp\cX \]
which takes values in the stable $\infty$-category $G\Sp\cX$ of equivariant coarse motivic spectra. 

For another attempt to axiomatise coarse homology theories we refer to \cite{MR1834777}.
The examples of coarse homology theories prior to \cite{buen} (the non-equivariant case) and  \cite{equicoarse} (the equivariant case) were constructed under more restrictive assumptions on the spaces, and often only as group-valued functors satisfying a weaker set of axioms.
The most relevant properties  were coarse invariance and versions of excision.
Some versions of the vanishing on flasques property was considered as a particular property of the example.
This applies for example to the coarse ordinary homology and coarse topological $K$-homology which were defined as group-valued functors on the category of proper metric spaces and proper controlled maps \cite{MR1147350}, \cite{roe_index_coarse}.
The algebraic $K$-theory functors in controlled topology were usually defined on spaces which are cones over topological spaces \cite{MR1880196},\cite{blr}, but sometimes also for general metric space as in \cite{MR802790}. 

It turned out that the construction of all these examples could be modified in order to fit our notion of coarse homology theory.
We refer to \cite{buen}, \cite{equicoarse}, \cite{Bunke:2017aa} for the cases of ordinary coarse homology, topological coarse $K$-homology, and coarse algebraic $K$-homology with coefficients in an additive category.

In the present paper, we construct functors
\[ \bV\colon G\BC\to \Cle \]
which associate to  {every $G$-bornological coarse space} $X$ a left-exact $\infty$-category of $X$-controlled objects in a (previously chosen)  {compactly generated presentable $\infty$-category.}
These constructions are designed such that if
\[ \Homol\colon \Cle\to \bM \]
is a homological functor (\cref{qrevoiqrjoirqfcwqecq}), e.g., the composition of a finitary localising invariant on stable $\infty$-categories with the stabilisation functor, then the composition
\begin{equation}\label{qwefoijfoifjioejfioewfweqdwdwedqwed}
\Homol\circ \bV\colon G\BC\to \bM
\end{equation}
is an equivariant coarse homology theory.
 
 The idea to use controlled objects to produce coarse homology theories is natural and has been used in previous examples. The  first case is probably the use of controlled  Alexander chains in Roe's construction of ordinary coarse homology. Controlled objects in an additive category were used to construct the controlled or coarse versions of algebraic $K$-theory of additive categories, see e.g.~\cite{MR802790}, \cite{MR1880196}, \cite{blr}, \cite{MR2030590}, \cite{equicoarse}.  In an analogous fashion, coarse topological $K$-homology has been constructed using controlled objects in $C^{*}$-categories, see \cite{buen}, \cite{Bunke:ae},  \cite{Bunke:ad}.
Non-linear versions of categories of controlled objects, namely $X$-controlled retractive spaces over some auxiliary space, have been used to construct controlled $A$-theory 
\cite{MR1880196}, \cite{Ullmann:2015aa}, and   an equivariant  coarse homology theory extending equivariant $A$-theory in  \cite{Bunke:aa}. 

In all these examples, the coefficient category $\bC$ is an ordinary category.
In the present paper, we start with  {the opposite $\bC$ of a compactly generated presentable $\infty$-category with $G$-action.}
In the following, we explain how we associate  {to a $G$-bornological coarse space} $X$ a category of $X$-controlled objects in $\bC$.

Let $X$ be a $G$-set.
In the first step, consider the $\infty$-category
\[ \PSh_{\bC}(X):=\Fun(\cP_{X}^{\op},\bC)\]
(see \eqref{qewfoi1jo4irfrefqfef}) of contravariant functors from the poset $\cP_{X}$ of subsets of $X$ to $\bC$.
The group $G$ acts on $X$ (and hence on $\cP_{X}$) as well as on $\bC$,  so $\PSh_{\bC}(X)$ carries an induced $G$-action by conjugation.
 The $\infty$-category $\PSh_{\bC}(X)$ is again the opposite of a compactly generated presentable $\infty$-category.
We then set
\[ \PSh_{\bC}^{G}(X):=\lim_{BG} \PSh_{\bC}(X)\ .\]
The construction depends functorially on  {the $G$-set $X$.}
Using the forgetful functor  {which sends a $G$-coarse space to its underlying $G$-set,} we can view $\PSh_{\bC}^{G}$ as a functor defined on the category $G\Coarse$ of $G$-coarse spaces (\cref{thiowhwfgwrgwergwreg}). If $X$ is  {a $G$-coarse space},
we then use the coarse structure $\cC_{X}$ of $X$ in order to define a subcategory $\Sh^G_\bC(X)$ of sheaves in
$\PSh_{\bC}^{G}(X)$.
For every invariant entourage $U$ in $\cC^{G}_{X}$,
we consider the Grothendieck topology $\tau^{U}$ generated by $U$-covering families (\cref{rgiqjrgioqfweewfqewfqewf}) and let
\[ \Sh^{U}_{\bC}(X)\subseteq  \PSh_{\bC}(X) \]
be the full subcategory of $\tau^{U}$-sheaves.
 It is a complete, pointed, large $\infty$-category with $G$-action (\cref{erioghweirghwrejigrwefreferfw}).
Taking the union
\[ \Sh_{\bC}(X) :=\bigcup_{U\in \cC_{X}^{G}} \Sh^{U}_{\bC}(X) \]
 over all invariant entourages produces a finitely complete, pointed, large $\infty$-category with $G$-action (\cref{wfqwoifjqiofeqfe32rqewfqewf}).
By applying $\lim_{BG}$, we get the objects
\[ \Sh^{U,G}_{\bC}(X) :=\lim_{BG} \Sh_{\bC}^{U}(X) \quad \text{and} \quad \Sh^{G}_{\bC}(X) :=\lim_{BG}\Sh_{\bC}(X)\ .\]
 The  excision property of sheaves 
 leads to
  the Glueing Lemma (\cref{wgkwkgrewrgrg}).
 The construction of $\Sh_{\bC}^{G}(X)$ depends functorially on the $G$-coarse space $X$ and thus produces a functor
 \[ \Sh_{\bC}^{G} \colon G\Coarse\to \CLL\ .\]
 Morphisms between sheaves are natural transformations, so
 the functor $\Sh_{\bC}^{G}$ on $G\Coarse$ is far from being coarsely invariant.
 
We will introduce morphisms which propagate in the $X$-direction by performing  {an appropriate localisation} in \cref{sec:localization}.
If $V$ is an invariant entourage of $X$ containing the diagonal, then we can define a $G$-equivariant functor of posets (see \eqref{qwefew14rqefeqwf} for details) and a natural transformation  
\[ V(-)\colon \cP_{X}\to \cP_{X}  \ , \quad  V(-)\to \id\ .\]
The induced functor $V_{*}$ on presheaves 
preserves sheaves and descends to an endofunctor $V_{*}^{G}$ on $\Sh^{G}_{\bC}(X)$.
Denoting by $W_X$ the collection of all comparison morphisms $M \to V_*^GM$ with $M$ in $\Sh^{G}_{\bC}(X)$, we consider the Dwyer--Kan localisation
\[ \hat \bV_{\bC}^{G}(X):=\Sh^{G}_{\bC}(X)[W_{X}^{-1}]\ . \]
Some effort is needed to show that this construction {produces left-exact $\infty$-categories and} is covariantly functorial {in $X$ and $\bC$}.
It is important to observe that the construction also has a contravariant functoriality for a restricted class of morphisms {of $G$-coarse spaces} called coarse coverings (\cref{wefgihjwiegwergrwrg} and \cref{iqerfjrfqwuef98weufeqwfqfe}). 
{The functor $\hat \bV^{G}_{\bC}$} is excisive (in an appropriate sense) and coarsely invariant, but its values are still large. 

Using the forgetful functor $G\BC\to G\Coarse$, we can view $\hat \bV^{G}_{\bC}$ as a functor on $G\BC$ (\cref{gjwerogijwoergwergwergweg}).
For  {a $G$-bornological coarse space} $X$, we now use the bornology $\cB_{X}$ on $X$ in order to define a full subcategory
\[ \bV_{\bC}^{G}(X)\subseteq \hat \bV_{\bC}^{G}(X) \]
of objects represented by equivariantly small sheaves.  
In the non-equivariant case, the natural condition on a sheaf to be  small is that it sends  the bounded subsets of $X$ (i.e., the elements of the bornology  $\cB_{X}$) to cocompact objects in $\bC$. 
This assumption is not sufficient in the equivariant setting.
For example, we would like $\bV^G_\bC(G_{can,min})$ to be equivalent, at least up to idempotent completion, to the category  {$\bC^{G,\omega}$ of cocompact objects in the fixed points of $\bC$} via the global sections functor.
Unless we explicitly require that the evaluation of a sheaf on a $G$-bounded subset, i.e., the $G$-orbit of a bounded subset, is cocompact in $\bC^G$, the image of the global sections functor will not even be contained in $\bC^{G,\omega}$.
Since the condition must also be compatible with the contravariant functoriality for coverings, we are forced to require that an equivariant small sheaf evaluates to cocompact objects in $\bC^H$ on $H$-bounded subsets for all subgroups $H$ of $G$ (see \cref{prop:transfer-coind} in particular).
This construction finally leads to the functor
\[ \bV^{G}_{\bC}\colon G\BC\to \Cle\ ,\] 
see \eqref{bojoijgoi3jg3g34f}.  

We now obtain our first version of a functor of equivariant $X$-controlled objects in $\bC$
\[ \bV^{G,c}_{\bC}\colon G\BC\to \Cle \]
by forcing continuity (\cref{gerklgjerlgergergergergerg}) on $\bV^{G}_{\bC}$, see \cref{iowergergwegr}. Essentially, this means that we force the value of the functor on a $G$-bornological coarse space $X$ to be determined by its values on locally finite subsets of $X$.

In order to get the second version, we first apply the construction above to the trivial group leading to $\bV^{c}_{\bC}$. 
If we apply  {this functor to a $G$-bornological coarse space $X$}, then by functoriality we get  {left-exact $\infty$-category with $G$-action} $\bV^{c}_{\bC}(X)$ and set
\[ \bV_{\bC,G}^{c}(X) := \colim_{BG} \bV_{\bC}^{c}(X)\ .\]
This yields a functor
\[ \bV^{c}_{\bC,G}\colon G\BC\to \Cle\ .\]
The properties of both functors $\bV^{G,c}_{\bC}$ and $\bV_{\bC,G}^{c}$
  are stated in \cref{regwergergergrgwgregwregwregwergwreg,wthgiowergergwergwerglabel}.
Upon composition with the algebraic $K$-theory functor, they yield equivariant coarse homology theories
\[ K\bC\cX^{G}:=K\circ \bV_{\bC}^{G,c}\colon G\BC\to \Sp\ ,\]
see \cref{rgowekgprgrwgwrgwrgwregwerg}, and
\[ K\bC\cX_{G}:=K\circ \bV^{c}_{\bC,G}\colon G\BC\to \Sp\ ,\]
see \cref{ergijeogergergwergergwregegw}. Each of these theories features some additional properties which for example enable us to prove \cref{giejrgoergwergwergwergwergwrg}.

\subsection*{Acknowledgements}
	U.B.~thanks Thomas Nikolaus for a discussion about an early stage of this paper. The authors furthermore thank Arthur Bartels for suggesting the possibility of the linear analogue of the construction of $K\bC\cX_{G}^{c}$.
The authors thank the referees for their valuable feedback which led to a considerable improvement of the presentation.

 {U.B., D.-C.C. and C.W.~were supported by the SFB 1085 (Higher Invariants) funded by the Deutsche Forschungsgemeinschaft (DFG).}

	C.W.~acknowledges support by the Max Planck Society and Wolfgang L\"uck's ERC Advanced Grant ``KL2MG-interactions'' (no.~662400). D.K.~and C.W. were funded by the Deutsche Forschungsgemeinschaft (DFG, German Research Foundation) under Germany's Excellence Strategy - GZ 2047/1, Projekt-ID 390685813.

 \section{\texorpdfstring{$\infty$-category}{infinity-category} background}
 
 This aim of this preliminary section is to introduce notation and to collect a number of technical statements which we include for ease of reference.
\cref{egihweogergwreggwrgr} introduces the categories in which our constructions will take place, records some of their basic properties, and fixes notation we will use throughout the rest of the article.
\cref{sec:fractions} discusses mapping spaces in Dwyer--Kan localisations. Together with \cref{rgiojreogiergregregreg}, this material will be used in \cref{sec:localization} to perform the final step in our construction of controlled objects over a bornological coarse space.
\cref{sec:stab} elaborates on the relation between stable and left-exact $\infty$-categories, and how the notion of a localising invariant generalises from the former to the latter.
The contents of this section will only be used in \cref{roijqreoeffqewfewfqewfefq}.

\subsection{Left-exact \texorpdfstring{$\infty$-categories}{infinity-categories}}\label{egihweogergwreggwrgr}
We let $\Cati$ denote the large $\infty$-category of small $\infty$-categories.
 {We begin by giving names to those subcategories of $\Cati$ that will be important for us.
For ease of reference, we describe the members of the chain of inclusions}
\begin{equation}\label{weriughwergfgrefwrfwerfwref}
\Clep\subseteq \Cle \subseteq \Catlex
\subseteq \Cati\ .\end{equation} 
in the following collection of examples.
\begin{ex}
  $\Catlex$ is the subcategory of small $\infty$-categories admitting finite limits and finite limit preserving functors. A typical object of $\Catlex$ is the opposite of the $\infty$-category $\Spc^\cop$ of compact spaces.
\end{ex}

\begin{ex}\label{whgiowgergrewgrgwgergwerg}
 {For us, a left-exact $\infty$-category is a pointed $\infty$-category admitting finite limits.}
We denote by $\Cle$ the full subcategory of $\Catlex$ of left-exact $\infty$-categories.
 A typical object of $\Cle$ is the opposite
 of the $\infty$-category $\Spc_*^{\cop}$ of compact pointed spaces.  
 \end{ex}

 \begin{ex}\label{ex:stCat}
 The $\infty$-category $\stCat$ of small stable $\infty$-categories is a full subcategory of $\Cle$.  Both the $\infty$-category $\Sp^{\cop}$ of compact spectra and its opposite  $\Sp^{\cop,\op}$ are stable.
 See \cref{werogijergegwegergrewref} for more examples.  
\end{ex}

\begin{ex}\label{thioertherhthet}
 $\Clep$ is the full subcategory of $\Cle$ of  idempotent complete {left-exact $\infty$-categories.}
 The  $\infty$-categories  $\Spc_*^{\cop,\op}$ and  $\Sp^{\cop,\op}$  are idempotent complete.
\end{ex}

 {We let $\CATi$ denote the very large $\infty$-category of large $\infty$-categories.}
We will also consider the chain
\[ \CL\subseteq \CLLL \subseteq \CLL\subseteq \CATi \]
of subcategories described  in the following list of examples.

\begin{ex}\label{wfqwoifjqiofeqfe32rqewfqewf}
 $\CLL$ is the subcategory of large pointed $\infty$-categories admitting finite limits and finite limit preserving functors.
 It is the large analogue of $\Cle$   in \cref{whgiowgergrewgrgwgergwerg}.
\end{ex}

\begin{ex}\label{erioghweirghwrejigrwefreferfw}
 $\CLLL$ is the subcategory of $\CLL$ of $\infty$-categories which admit all small limits and limit preserving functors. 
\end{ex}

If $\bC$ is in $\CLLL$, then we can  consider the notion of a cocompact object $C$  in $\bC$, {which is equivalent to $C$ being compact in $\bC^\op$. Explicitly, this means the following:}

\begin{ddd}\label{rgiuehgieugegewgergwergwerg}
$C$ is cocompact if  the natural morphism \begin{equation}\label{qoirfjwiofjqoiwe} \colim_{\bI^{\op}}\Map_{\bC}(T,C)\xrightarrow{\simeq} \Map_{\bC}(\lim_{\bI}T,C)
\end{equation}
is an equivalence for every cofiltered diagram $T\colon \bI\to \bC$.
\end{ddd}

The $\infty$-category $\bC^{\omega}$ of cocompact objects of $\bC$ is an object of $\CLL$.

\begin{ex}\label{rgqeroqjkopqwefewfefqef}
 $\CL$ is the subcategory of $\CLLL$ whose opposites are pointed,  {compactly generated presentable $\infty$-categories.} 
 The morphisms in $\CL$ are right adjoint functors which preserve cocompact objects (equivalently, their left adjoints preserve cofiltered limits \cite[Prop.~5.5.7.2]{htt}).
 In other words, by definition we have an equivalence 
 \begin{equation}\label{wrfqewdewdqfewfqfwefqfewf}
  (-)^{\op}\colon \Prlp\xrightarrow{\simeq} \CL\ . 
 \end{equation} 
 If $\bC$ is in $\CL$, then its full subcategory  $\bC^{\omega}$ of cocompact objects belongs to  {$\Clep$}.
 Note that $\bC^{\omega}\simeq ((\bC^{\op})^{\cop})^{\op}$.
\end{ex}

 {Next, we consider the (co)completeness of the categories we have introduced, and describe which limits and colimits can be computed on the level of underlying $\infty$-categories.
Most importantly, we show that $\Cle$ is complete (\cref{prop:catex finitely complete}) and cocomplete (\cref{ioerjgoiegergwegergwrgwegrwerg}), and that the forgetful functor $\Cle \to \Cati$ preserves limits and filtered colimits.
Along the way, we discuss various constructions which preserve fully faithful functors.}

 \begin{prop}\label{prop:catex finitely complete} 
 	The $\infty$-category $\Cle$ has small limits as well
 	as small filtered colimits. Furthermore, the  inclusion  
 	$\Cle\to \Cati$ preserves
 	small limits and  small filtered colimits.
 \end{prop}
 \begin{proof} 
If we fix a small category $K$,
the forgetful functor from the $\infty$-category of small $\infty$-categories with $K$-shaped colimits (and functors preserving them) to the $\infty$-category of $\infty$-categories commutes with small limits: commutation with products is an easy exercise,
and the case of pullbacks follows from \cite[Lem.~5.4.5.5]{htt}.
Since the assignment $\bC \mapsto \bC^{\op}$ is an equivalence,
the same is true replacing $K$-shaped colimits by $K$-shaped limits:
the forgetful functor from the $\infty$-category of small $\infty$-categories with $K$-shaped limits (and functors preserving them) to the $\infty$-category of $\infty$-categories commutes with small limits. For filtered colimits, we have analogous compatibility properties, except that we need $K$ to be finite,
by \cite[Prop.~5.5.7.11]{htt} and its dual version.

In particular, if we let $K$ range over all finite indexing categories, it follows that $\Catlex$ has small limits as well as small filtered colimits,
and the  functor $\Catlex\to \Cati$ preserves them.
Moreover, the property of having an initial or terminal object is preserved
under small limits as long as the transition maps in the limit diagram
preserve initial or terminal objects, respectively (take $K=\varnothing$
in the discussion above). Similarly, since $\varnothing$ is finite,
the same holds for filtered colimits. One concludes that
the property of having a zero object is closed under small limits
as well as under small filtered colimits in $\Catlex$.
 \end{proof}
 
 \begin{lem}\label{lem:CL-complete}
 The $\infty$-category $\CL$ has small limits. Furthermore, the inclusion $\CL\to \CATi$ preserves small limits.
 \end{lem}
 \begin{proof}
  The $\infty$-category $\Prlp$ has small limits by a pointed version (see the proof of \cref{prop:catex finitely complete} for the derivation of the pointed from the unpointed version) of \cite[Prop.~5.5.7.6]{htt}. Hence  $\CL$   also has small limits. Since $(-)^{\op}$ and
the inclusion $\Prlp\to \Cati$   preserves small limits, the  inclusion $\CL\to \CATi$ preserves limits.
 \end{proof}

 In the following we discuss limits in $\Cati$ and $\Cle$.  There are analogous results for the large cases $\CATi$ and $\CLL$.
 
  We will use underlines to indicate constant diagrams.
 Let $\cC$ be any of the above categories, and let $\bI$ be a small category.
  Then we consider a diagram $\bC \colon \bI\to \cC$
  and assume that
 \begin{equation}\label{gtterpojgpergerg}
  \eta\colon \underline{\lim_{\bI} \bC}\to \bC
 \end{equation}
 presents an object $\lim_{\bI}\bC$  as the limit of the diagram in $\cC$.   
 
 For every object $i$ in $\bI$ we have an evaluation functor
 \begin{equation}\label{werwregfwreggregwrweggwergwerg}
\ev_{i}\colon \Fun(\bI,\cC)\to \cC\ .
\end{equation}
Applied to $\bC$ we get the underlying  object $\ev_{i}(\bC)\simeq \bC(i)$ in $\cC$. If we apply $\ev_{i}$ to $\eta$ from \eqref{gtterpojgpergerg} and
use the canonical equivalence $\ev_{i}(\underline{(-)})\simeq \id$,  then we get the evaluation morphism \begin{equation}\label{geroij4oi33rgg3g34}
	e_{i}\colon \lim_{\bI} \bC \to  \bC(i)
\end{equation}
in $\cC$.

If $\phi\colon \bC\to \bD$ is a {natural transformation of such diagrams},
then the following diagram commutes for every $i$ in $\bI$:
\begin{equation}\label{ergregrg}
\xymatrix{
 \lim_{\bI} \bC\ar[r]^{e_{i}}\ar[d]_{\lim_{\bI}\phi} & \bC(i)\ar[d]^{\phi(i)} \\
 \lim_{\bI} \bD\ar[r]^{e_{i}} & \bD(i)}
\end{equation}

\begin{lem}\label{greoigergreeg}
The collection of functors $(e_{i})_{i\in \bI}$ detects equivalences.
\end{lem} 
\begin{proof}
If $\bD$ is  in $\Cati$   with objects $D,D^{\prime}$, then we can present the mapping space by the pullback in $\Cati$
\[\xymatrix{\Map_{\bD}(D,D^{\prime})\ar[r]\ar[d]&\Fun(\Delta^{1},\bD)\ar[d]\\\Delta^{0}\times \Delta^{0}\ar[r]^{(D,D^{\prime})}&\bD\times \bD}\ ,\]
where the right vertical map is given by the evaluations at the two boundaries of $\Delta^{1}$.
Thus for  a pair of objects  $X,Y$  of $\lim_{\bI} \bC$ we have a pullback in $\Cati$
\begin{equation}\label{2eoifh2ioufhd1}
\xymatrix{\Map_{\lim_{\bI} \bC}(X,Y)\ar[r]\ar[d]&\Fun(\Delta^{1},\lim_{\bI} \bC) \ar[d]\\ \Delta^{0}\times \Delta^{0} \ar[r]^-{(X,Y)}&\lim_{\bI} \bC\times \lim_{\bI} \bC}\ .
\end{equation}
We  now define $M(X,Y)$ in $\Fun(\bI,\Cat_{\infty})$ by the pullback square
\[\xymatrix{M(X,Y)\ar[r]\ar[d]&\Fun(\Delta^{1},\bC)\ar[d]\\\underline{\Delta^{0}}\times \underline{\Delta^{0}}\ar[r]&\bC\times \bC}\ ,\]
where the lower horizontal map corresponds to $(X,Y)$ under the
  $(\underline{-},\lim_{\bI})$-adjunction.   

We apply $\lim_{\bI}$ to this diagram.
Since the functor $\lim_{\bI}$ preserves pullbacks, $\lim_{\bI}\underline{\Delta^{0}}\simeq \Delta^{0}$, and $\lim_{\bI} \Fun(\Delta^{1},\bC)\simeq \Fun(\Delta^{1},\lim_{\bI} \bC)$ we get the pullback diagram \eqref{2eoifh2ioufhd1}. In other words, we have an equivalence
\begin{equation}\label{gkjrgnkjgnjk23d2d2d}
\Map_{\lim_{\bI} \bC}(X,Y)\simeq \lim_{\bI} M(X,Y)\ .
\end{equation}
On the other hand, for $i$ in $\bI$  the functor
$\ev_{i}\colon \Fun(\bI,\Cati)\to \Cati$  preserves pullback diagrams. In view of the definition of the functor $e_{i}$ und using the equivalences $\ev_{i}(\underline{\Delta^{0}})\simeq \Delta^{0}$ and $\Fun(\Delta^{1},\bC(i))\simeq \ev_{i}(\Fun(\Delta^{1},\bC))$   we get the pullback diagram
\[\xymatrix{\ev_{i}(M(X,Y))\ar[r]\ar[d]& \Fun(\Delta^{1},\bC(i))\ar[d] \\ \Delta^{0}\times \Delta^{0} \ar[r]^{ (e_{i}(X),e_{i}(Y))}&\bC(i)}\ ,\]
i.e., the equivalence
\[ \ev_{i}(M(X,Y))\simeq \Map_{\bC(i)}(e_{i}(X),e_{i}(Y))\ .\]
Let now $f\colon Y\to Y^{\prime}$ be a morphism in $\lim_{\bI} \bC$ such that $e_{i}(f)$ is an equivalence for every $i$ in $\bI$. Then for every $X$ in $\lim_{\bI} \bC$ the induced morphism
\[ \Map_{\bC(i)}(e_{i}(X),e_{i}(Y))\to \Map_{\bC(i)}(e_{i}(X),e_{i}(Y^{\prime})) \]
is an equivalence.
Hence the induced map  
$\ev_{i}(M(X,Y))\to \ev_{i}(M(X,Y^{\prime}))$ is an equivalence for all $i$ in $\bI$ . 
We conclude that the morphism $M(X,Y)\to M(X,Y^{\prime})$  induces an equivalence after applying $\lim_{\bI}$. In view of the equivalence \eqref{gkjrgnkjgnjk23d2d2d}, we thus have shown that
the induced morphism $\Map_{\lim_{\bI} \bC}(X,Y )\to \Map_{\lim_{\bI} \bC}(X,Y^{\prime})$ is an equivalence.
Since $X$ is arbitrary, we conclude that $f\colon Y\to Y^{\prime}$ is an equivalence.
\end{proof}

 {Let $\phi \colon \bC\to \bD$ be a natural transformation of functors $\bI \to \Cati$.}
\begin{lem}\label{tieqorgfgregwegergw}
 Assume that for every $i$ in $\bI$ the morphism ${\phi(i) \colon} \bC(i)\to \bD(i)$ is fully faithful.
Then we have the following assertions:
\begin{enumerate}
\item \label{reigowergwrgrgwgrg} The morphism ${\lim_{\bI} \phi \colon} \lim_{\bI}\bC\to \lim_{\bI}\bD$ is fully faithful.
\item\label{jfqiwoefojeqwfqewfefqewf} The essential image of the morphism in \eqref{reigowergwrgrgwgrg} consists of those objects  $D$ of $\lim_{\bI}\bD$ whose evaluation $e_{i}(D)$ belongs to the essential image of ${\phi(i) \colon} \bC(i)\to \bD(i)$ for all $i$ in $\bI$.
\end{enumerate}
\end{lem}
 \begin{proof}
  \eqref{reigowergwrgrgwgrg} is well-known, but also follows from the discussion of mapping spaces in the proof of \cref{greoigergreeg}. In order to see \eqref{jfqiwoefojeqwfqewfefqewf}, one easily checks that
  the indicated subcategory of $\lim_{\bI} \bD$ has the required universal property.
 \end{proof}
   
We consider a diagram {$\bC \colon \bI \to \Cle$.}
 \begin{lem}\label{rgeoiergegegregvsdvsdvvdfverrf3erfer}
 The collection of functors $(e_{i})_{i\in \bI}$ detects finite limits.
 \end{lem}
\begin{proof}
Let $\bJ$ be a finite category and  {let $X \colon \bJ \to \lim_\bI \bC$ be a diagram.}
Let furthermore $Y$ be an object of $\lim_{\bI} \bC$ and $\iota\colon Y\to \lim_{\bJ}X$ be a morphism.
We want to show that $\iota$ is an equivalence provided the induced functor
\[ e_{i}(Y)\xrightarrow{e_{i}(\iota)} e_{i}(\lim_{\bJ}X)\xrightarrow{!}\lim_{\bJ}e_{i}(X) \]
is an equivalence for all $i$ in $\bI$. Since $e_{i}$ is left-exact, the marked morphism is an equivalence. 
Hence
{the lemma} follows from  the fact that the collection $(e_{i})_{i\in \bI}$ detects equivalences (\cref{greoigergreeg}).
\end{proof}

 We consider  a diagram  {$\bC \colon \bI \to \CLLL$.}
 \begin{lem}\label{rgeoiergegegregvsdvsdvvdfverrf3erfer1}
 The collection of functors $(e_{i})_{i\in \bI}$ detects  limits.
 \end{lem}
 \begin{proof} The argument is the same as for  \cref{rgeoiergegegregvsdvsdvvdfverrf3erfer}. We just drop all finiteness assumptions on $\bJ$ and use that in this case the evaluations (being morphisms in $\CLLL$) preserve small limits.   \end{proof}
 
In the following, we consider colimits of diagrams of left-exact $\infty$-categories. 
We will show that $\Clp$ and $\Cle$ are  cocomplete (the cases of limits and filtered colimits have already been settled in \cref{prop:catex finitely complete}).  
We further study instances where colimits
preserve fully faithfulness of transformations. 

We have a chain of inclusions
\[ \Crp \subseteq \Cre\subseteq \Cati \]
with the following description:
\begin{enumerate}
\item $\Cre$ is the subcategory of small pointed $\infty$-categories admitting finite colimits, and finite colimit-preserving functors  \cite[{Not.~5.5.7.7}]{htt}.
\item $\Crp$ is the full  subcategory of $\Cre$ of idempotent complete  {right-exact $\infty$-categories} \cite[Sec.~4.4.5]{htt}.
\end{enumerate}

As above,  we let $\Prlp$ denote the very large $\infty$-category of  {pointed, compactly generated presentable $\infty$-categories}
and left adjoint functors which preserve compact objects.
We have the $\Ind$-completion functor
\[ \Ind_{\omega}\colon \Cre \to \Prlp \]
 {which freely adjoins filtered colimits.}
Its restriction to idempotent complete right-exact $\infty$-categories
  is an equivalence $\Ind_{\omega}\colon \Crp\xrightarrow{\simeq} \Prlp$ whose inverse is the functor
  \[ (-)^{\cop}\colon  \Prlp\to \Crp \] sending 
a presentable $\infty$-category to its full subcategory of compact objects. 
We have the idempotent completion functor (the pointed version of \cite[Prop.~5.5.7.10]{htt}) fitting into an adjunction
\[ \Idem:=(-)^{\cop}\circ \Ind_{\omega}\colon  \Cre \leftrightarrows  \Crp \cocolon \incl\ .\]

 Left-exact and right-exact $\infty$-categories are connected by the equivalence
  \begin{equation}\label{wefkjbwejkfnkwefdewfwf1eeeded}
\Cre\xrightarrow{\simeq} \Cle\ , \quad \bC\mapsto \bC^{\op}\ .
\end{equation}
 We then have a $\Pro$-completion $\Pro_{\omega}\colon \Cle\to \CL$ functor defined such that
 \[\xymatrix{\Cle \ar[d]^{\op}_{\simeq}\ar[r]^{\Pro_{\omega}}& \CL\ar[d]_{\simeq}^{\op}\\ \Cre \ar[r]^{\Ind_{\omega}}&\Prlp}\]
 commutes. It induces an equivalence
 \begin{equation}\label{adewfqfewfqewfq}
\Pro_{\omega}\colon \Clp\xrightarrow{\simeq} \CL\ .
\end{equation}
 The inverse of the functor  \eqref{adewfqfewfqewfq} is the functor \begin{equation}\label{134g98hj134iof13frfqerfrf}
(-)^{\omega}\colon \CL\to \Clep
\end{equation} taking the full subcategory of cocompact objects (\cref{rgiuehgieugegewgergwergwerg}). 
Finally, we have an adjunction \begin{equation}\label{ewfbjhbfjhqerfqfewfqefe}
\Idem:=(-)^{\omega}\circ \Pro_{\omega}\colon \Cle\leftrightarrows  \Clp\cocolon \incl\ .
\end{equation}
 
\begin{lem}\label{iqjoigrgwegerwgwg}
The functor $\Idem$ preserves fully faithfulness.
\end{lem}
\begin{proof}
The operations $\Ind_{\omega}$, $\op$ and $(-)^{\omega}$ going into the definition of $\Pro_{\omega}$ preserve fully faithfulness  {\cite[Prop.~5.3.5.11]{htt}}.
\end{proof}

\begin{lem}\label{ergiowergerwgergwergwreg}
The $\infty$-category $\Clp$ admits small colimits.
\end{lem}
\begin{proof}
We let $\Prr$ denote the $\infty$-category of presentable $\infty$-categories and right adjoint functors  {which preserve filtered colimits}.
By \cite[Prop.~5.5.7.6]{htt}, this $\infty$-category admits small limits which are preserved by the inclusion $\Prr\to  {\CATi}$.
We have an equivalence
\[\ad \colon \Prl \xrightarrow{\simeq} {\Prr}^{\op}\ ,\]
which is the identity on objects and replaces morphisms by their right adjoints.
Consequently, $\Prl$ admits all small colimits. This implies that $\Prlp$ also admits small colimits. In view of equivalence 
\eqref{wrfqewdewdqfewfqfwefqfewf}, $\CL$ admits small colimits. Finally,
$\Clep$ also  admits small colimits by the equivalence \eqref{adewfqfewfqewfq}.
\end{proof}

If $\bI$ is a groupoid, then we have an  equivalence $\iota\colon \bI^{\op}\to \bI$.

 Let $\bC\colon \bI\to \CL$ be a diagram.   The colimit in the following lemma is interpreted  in $\Clep$.
 \begin{lem}\label{rwigowefqewffqwfwfqwef}
Assume that $\bI$ is a groupoid.
\begin{enumerate}
\item There is an equivalence \begin{equation}\label{wqefqoijfqoiwffqefefewfqefqe}
\colim_{\bI} \bC^{\omega}\simeq \left(\lim_{\bI^{\op} } \iota^{*} \bC\right)^{\omega}\ .
\end{equation}
\item \label{tjborbjeorpbrberbertbertbrtb} For every object $i$ in $\bI$ the diagram
\[\xymatrix{ {\bC(i)^\omega}\ar[d]_{\mathrm{can}(i)}\ar[rr]& &\bC(i)\ar[d]^{e_{i,*}}\\  \colim_\bI  \bC^\omega   \ar[r]^-{\eqref{wqefqoijfqoiwffqefefewfqefqe}}_-{\simeq} & (\lim_{\bI^{\op}}\iota^{*}\bC)^{\omega} \ar[r]& \lim_{\bI^{\op}}\iota^{*}\bC}  \]
commutes, where  the arrow   $e_{i,*}$ is the right adjoint of the canonical morphism $e_{i} \colon \lim_{\bI^{\op}}\iota^{*}\bC \to \bC(i)$, and the two unmarked horizontal arrows are the inclusions of the full subcategories of cocompact objects.
 \item  \label{tjborbjeorpbrberbertbertbrtb1} 
 The  essential images of the functors $\mathrm{can}(i)$ for all $i$ in $\bI$ generate $\colim_\bI \bC^\omega$ under finite limits and retracts.
 \end{enumerate}
\end{lem}
\begin{proof}
By assumption, we have an equivalence $\iota\colon \bI^{\op}\to \bI$ and an equivalence of diagrams $\ad(\bC^{\op})\simeq \iota^{*}\bC^{\op}$ in $\Fun(\bI^{\op},\CATi)$. This gives, using the description of colimits in $\Clep$ provided by the proof of \cref{ergiowergerwgergwergwreg} in the first step,
\begin{align*}
\colim_{\bI} \bC^{\omega}
&\simeq \left((\ad^{-1}\lim_{\bI^{\op}} \ad(\bC^{\op}))^{\op}\right)^{\omega} \\
&\simeq \left(( \lim_{\bI^{\op}} \iota^{*}\bC^{\op})^{\op}\right)^{\omega} \\
&\simeq \left( \lim_{\bI^{\op}} \iota^{*}\bC\right)^{\omega}\ .
 \end{align*}
Since $\iota^*\bC^\op$ defines diagrams in both $\Prl$ and $\Prr$, and since both inclusions $\Prl \to \Cati$ and $\Prr \to \Cati$ preserve limits, for every $i$ in $\bI$ the canonical transformation $e_{i} \colon  \lim_{\bI^\op} \bC   \to \bC(i)$ admits both adjoints. By the proof of \cref{ergiowergerwgergwergwreg}, its right adjoint corresponds to the canonical transformation
\[ e_{i,*} \colon \bC(i) \to  \colim_\bI \bC\ ,\]
which evidently restricts to the subcategories of cocompact objects yielding $\mathrm{can}(i)$. 
 
 We are left with showing that the essential images    of the transformations  $\mathrm{can}(i)$  for all $i$ in $\bI$ generate  the target under finite limits and retracts. 
 Since $\colim_\bI \bC$ has a set of cocompact generators,
  it suffices to check that a morphism $f \colon x \to y$ in $\colim_\bI \bC$ is an equivalence whenever
\[ f^* \colon \Map_{\colim_\bI \bC}(x,e_{i,*}(c)) \to \Map_{\colim_\bI \bC}(y,e_{i,*}(c)) \]
is an equivalence for all $i$ in $\bI$ and $c$ in $\bC(i)$ (The proof of \cite[Thm.~1.11]{AR} carries over to the setting of $\infty$-categories.). Since
\[ \Map_{\colim_\bI \bC}(x,e_{i,*}(c)) \simeq \Map_{\bC(i)}(e_i(x),c)\ ,\]
such a morphism has the property that $e_i(f)$ is an equivalence for all $i$ in $\bI$. So it is indeed an equivalence by \cref{greoigergreeg}.
\end{proof}

\begin{prop}\label{ioerjgoiegergwegergwrgwegrwerg}
	The $\infty$-category $\Cle$ admits small colimits.
\end{prop}
\begin{proof}
Let $\bC\colon \bI\to \Cle$ be a diagram.
Then for every $i$ in $\bI$ we have a canonical morphism \[ \iota_{i}\colon  \bC(i)\to \Idem (\bC(i))  \xrightarrow{e_{i,*}}  \colim_{\bI} \Idem(\bC)\ ,\]
where the colimit is interpreted in $\Clep$. 
We let $\bD$ be the full left-exact subcategory of  $\colim_{\bI} \Idem(\bC)$ generated by the images of the functors $\iota_{i}$ for all $i$ in $\bI$.
For every $\bT$   in $\Cle$  we then have the following commutative diagram:
\[\xymatrix{ \Map_{\Fun(\bI,\Cle)}(\Idem(\bC),\underline{\Idem(\bT)})\ar[d]_{!}^{\simeq} &\ar[l]_{\simeq}\Map_{\Cle}(\colim_{\bI}\Idem(\bC) ,\Idem(\bT))\ar[d]^{!!}_{\simeq}\\ \Map_{\Fun(\bI,\Cle)}(\bC,\underline{\Idem(\bT)}) &\Map_{\Cle}(\bD ,\Idem(\bT))\\\Map_{\Fun(\bI,\Cle)}(\bC,\underline{\bT})\ar[u]^{!!!} &\ar@{..>}[l]_{\simeq}\Map_{\Cle}(\bD ,\bT)\ar[u]_{!!!!}}\ .\]
The morphism $!$ is induced by the  morphism $\bC\to \Idem(\bC)$ and is an  equivalence by the universal property of the latter.
The morphism  $!!$ is induced by the inclusion $\bD\to \colim_{\bI} \Idem(\bC)$ and is an equivalence by a similar reason since the induced morphism $\Idem(\bD)\to 
 \colim_{\bI} \Idem(\bC)$ is an equivalence by construction of $\bD$ and \cref{rwigowefqewffqwfwfqwef}.\ref{tjborbjeorpbrberbertbertbrtb1}.
 The morphisms marked by $!!!$ and $!!!!$ are  induced by the fully faithful functor $\bT\to \Idem(\bT)$. They  are inclusions of  collections of components.  The $\infty$-category $\bD$ is constructed exactly such that the  dotted arrow exists and is a bijection on $\pi_{0}$. Since it is natural in $\bT$ we can conclude that the colimit of the diagram $\bC$ in $\Cle$ exists and is represented by $\bD$. \end{proof}

Let $\bI$ be a small $\infty$-category, and let {$\phi \colon \bC \to \bD$ be a natural transformation of functors $\bI \to \Cle$.}
\begin{lem} \label{ewfweijor23r23r23r32r23r23}
Assume:
\begin{enumerate}
\item $\bI$ is a groupoid.
\item The functor $ {\phi(i) \colon} \bC(i)\to \bD(i)$ is fully faithful for all $i$ in $\bI$.
\end{enumerate}
Then the functor $ {\colim_\bI \phi \colon} \colim_{\bI}\bC\to \colim_{\bI}\bD$ is fully faithful.
\end{lem}
\begin{proof}
We first consider the analogous assertion for diagrams in $\Clep$. 
In this case, we can apply the idea of the proof of \cref{rwigowefqewffqwfwfqwef}.
We use the formula 
\[ \colim_{\bI} \bC\simeq \left((\ad^{-1}\lim_{\bI^{\op}} \ad( \Pro_{\omega}(\bC)^{\op}))^{\op}\right)^{\omega}\ .\]
for the colimit given in the proof of \cref{ergiowergerwgergwergwreg}.
 As the operations $\Pro_{\omega}$,  $(-)^{\omega}$ and $(-)^{\op}$ preserve fully faithfulness  {\cite[Prop.~5.3.5.11]{htt}},
we must show the following assertion. 

 Assume that $f\colon \bP\to \bQ$ is a morphism in $\Fun(\bI,\Prl)$ such that $f(i)$ is fully faithful for every $i$ in $\bI$. Then  the functor $(\ad^{-1}  \lim_{\bI^{\op}}\ad)(f)$ is fully faithful. 
 
  We have a morphism $\ad(f)\colon \ad(\bQ)\to \ad(\bP)$ in $\Fun(\bI^{\op}, {\CATi})$. 
The functor $\ad$ (which replaces functors by their right adjoints) does not preserve fully faithfulness. 
To overcome this problem we use that the fully faithful left adjoints $f(i)$ of the functors $\ad(f(i))$ for all $i$ in $\bI$ assemble to a natural transformation in the opposite direction. 
Since  $\bI$ is a groupoid, we have an equivalence $\iota\colon \bI^{\op}\to \bI $.  We get the diagram in $\Fun(\bI^{\op},{\CATi})$ \[\xymatrix{\ad(\bP)\ar@/_{0.5cm}/[r]_{\tilde f}\ar[d]^{\simeq}&\ar@/_{0.5cm}/[l]_{\ad(f)}\ad(\bQ)\ar[d]^{\simeq}\\ \iota^{*}\bP\ar[r]^{\iota^{*}f}&\iota^{*} \bQ}\ ,\]
where $\tilde f$  is defined by commutativity of the lower square. By construction, $\tilde f(i)$ is the left adjoint $f(i)$ of $\ad(f(i))$.
 We now apply $\lim_{\bI^{\op}}$ to the upper line and get an adjunction
\[\xymatrix{\lim_{\bI^{\op}} \ad(\bP)\ar@/_{0.5cm}/[r]_{\lim_{\bI^{\op}}\tilde f} &\ar@/_{0.5cm}/[l]_{\lim_{\bI^{\op}}\ad(f)}\lim_{\bI^{\op}}\ad(\bQ) }\ .\]
In particular, we have an equivalence
\[\lim_{\bI^{\op}}\tilde f\simeq (\ad^{-1}\lim_{\bI^{\op}}\ad)(f)\ .\]
Since  a limit of a diagram of fully faithful functors in  {$\CATi$} is again fully faithful  {by \cref{tieqorgfgregwegergw}},
we can conclude that
$\lim_{\bI^{op}}\tilde f$  and hence $(\ad^{-1}\lim_{\bI^{op}}\ad)(f)$ are fully faithful.   
This finishes the case of diagrams with values in $\Clp$.

Assume now that $f\colon \bC\to \bD$ is a morphism of $\bI$-indexed diagrams with values in $\Cle$ which is objectwise  fully faithful.
Then $\Idem(f)\colon \Idem(\bC)\to \Idem(\bD)$ is such a diagram in $\Clp$ with the same property by \cref{iqjoigrgwegerwgwg}.
 We just have shown that the lower line in the square in $\Cati$
 \[\xymatrix@C=4.5em{\colim_{\bI}\bC\ar[r]^-{\colim_{\bI}f}\ar[d]&\colim_{\bI}\bD\ar[d]\\\colim_{\bI}\Idem(\bC)\ar[r]^-{\colim_{\bI}\Idem(f)}&\colim_{\bI} \Idem(\bD)}\]
 is fully faithful. In the proof of \cref{ioerjgoiegergwegergwrgwegrwerg}, we have presented the colimits (in $\Cle$) in the upper line as full subcategories of the colimits (in $\Clp$) in the lower line. Consequently, the vertical arrows are fully faithful.
 This implies that also the upper horizontal arrow is fully faithful.
\end{proof}

In general, if $\cC$ is a pointed $\infty$-category admitting finite products and coproducts, then for every finite  set $F$ and family $(\bC_{f})_{f\in F}$ in $\cC$ we have a natural morphism
\begin{equation}\label{ecwcknwlcwecwec}
q\colon \coprod_{f\in F}\bC_{f}\to \prod_{f^{\prime}\in F} \bC_{f^{\prime}}\ .
\end{equation}
This morphism is classified by the collection of morphisms
\begin{equation}\label{wvwvknekwlwecwecewcwcwecewc}
(q_{f^{\prime}}\colon \coprod_{f \in F}\bC_{f }\to \bC_{f^{\prime}})_{f^{\prime}\in F}\ ,
\end{equation}
where $q_{f^{\prime}}$ itself is classified by the collection of morphisms
\begin{equation}\label{rvwrvwevwevewcewcwecwec}
(q_{f ,f^{\prime}}\colon \bC_{f }\to \bC_{f^{\prime}})_{f \in F}
\end{equation}
such that $q_{f ,f^{\prime}}$ is zero for $f \not= f^{\prime}$ and $\id_{\bC_{f}}$ for $f=f^{\prime}$.

\begin{ddd}[{\cite[Def.~6.1.6.13]{HA}}]\label{ghiowjgoergwergwergregwg}
	The $\infty$-category $\cC$ is called semi-additive, if it is pointed, admits finite products and coproducts, and
	if the morphism \eqref{ecwcknwlcwecwec} is an equivalence for every finite set $F$ and family $(\bC_{f})_{f\in F}$ of objects in $\cC$.
\end{ddd}

\begin{lem}\label{girgjowegfewfw9ef}
	The $\infty$-category $\Cle$ is semi-additive.
\end{lem}
\begin{proof}
	The $\infty$-category $\Cle$ is complete by  \cref{prop:catex finitely complete} and therefore admits products.
	 $\Cle$ also admits coproducts by \cref{ergiowergerwgergwergwreg}. Finally,
	  $\Cle$ is pointed by the one-point category $*$.
	
	We now fix a finite set $F$ and a family $(\bC_{f})_{f\in F}$ in $\Cle$.  For $f$ in $F$
	let $\iota_{f}\colon \bC_{f^{\prime}}\to \coprod_{f\in F}\bC_{f}$ be the canonical inclusion and $\pi_{f}\colon \prod_{f^{\prime}\in F}\bC_{f^{\prime}}\to \bC_{f}$ be the canonical projection.
	For every $f''$ in $F$ we  have a morphism
	\[ p_{f''}:=\iota_{f''}\circ \pi_{f''}\colon \prod_{f^{\prime}\in F}\bC_{f^{\prime}}\to \coprod_{f\in F}\bC_{f}\ .\]
	Their product\footnote{This product of functors  can formally be  understood as a  right Kan extension along the functor of discrete categories $F\to *$. It exists since $F$ is finite and $ \coprod_{f\in F}\bC_{f}$ being left-exact admits finite products.}
	in the left-exact $\infty$-category $\coprod_{f\in F}\bC_{f}$ is a morphism
	\[ p:=\times_{ f''\in F} p_{f''}\colon \prod_{f^{\prime}\in F}\bC_{f^{\prime}}\to \coprod_{f\in F}\bC_{f}\ .\]
	We claim that $p$ and the morphism $q$ from \eqref{ecwcknwlcwecwec} are mutually inverse equivalences. 
	In order to show that $p\circ q\simeq \id_{ \coprod_{f\in F}\bC_{f}}$ it suffices to provide equivalences  $p\circ q\circ \iota_{f} \simeq \iota_{f}$ for all $f$ in $F$. They are  given by the following chains of equivalences:
	\begin{eqnarray*}
		p\circ q\circ \iota_{f}&\simeq&(\times_{  f''\in F}\iota_{f''}\circ \pi_{f''})\circ q\circ \iota_{f}\\
		&\simeq&\times_{  f''\in F}(\iota_{f''}\circ  \pi_{f''}\circ q\circ \iota_{f})\\
		&\simeq&\times_{  f''\in F}\iota_{f''}\circ  q_{f''}\circ \iota_{f}\\&\simeq&\times_{  f''\in F}(\iota_{f''}\circ q_{f'',f}) \\&\simeq&
		\iota_{f }  \ , 
	\end{eqnarray*}
	where we use the notation from \eqref{wvwvknekwlwecwecewcwcwecewc} and \eqref{rvwrvwevwevewcewcwecwec}.
	Similarly, in order to show that $q\circ p\simeq \id_{\prod_{f^{\prime}\in F}\bC_{f^{\prime}}}$ it suffices to provide an equivalence
	$\pi_{f^{\prime}}\circ q\circ p\simeq \pi_{f^{\prime}}$ for every $f^{\prime}$ in $F$.
	They are given by the following chains of equivalences
	\begin{eqnarray*}
		\pi_{f^{\prime}}\circ q\circ p&\simeq&\pi_{f^{\prime}}\circ  q\circ (\times_{  f''\in F}\iota_{f''}\circ \pi_{f''})\\& \simeq& q_{f'}\circ (\times_{  f''\in F}\iota_{f''}\circ \pi_{f''})
		\\& \stackrel{!}{\simeq}& \times_{  f''\in F}
		q_{f'}\circ    \iota_{f''}\circ \pi_{f''})\\&\simeq&
		\times_{  f''\in F}q_{f',f''}\circ   \pi_{f''}\\&\simeq&
		\pi_{f'}\ ,
	\end{eqnarray*}
	where at the marked equivalence we use that $q_{f^{\prime}}$ preserves finite products. 
\end{proof}

 In the remainder of this subsection, we consider a situation where a limit and a colimit can be interchanged.     
  
   \begin{rem}\label{lem:factoring-transformations}
We consider a  small category $\bI$, functors {$\bT , \bC,\bD \colon \bI \to \Cati$},
and a natural transformation $\phi \colon \bC \to \bD$.
We consider  the subspace $\Map'_{\Fun(\bI,\Cati)}(\bT,\bD)$ consisting of the components of  $\Map_{\Fun(\bI,\Cati)}(\bT,\bD)$  of those natural transformations $\psi \colon \bT\to \bD$ such that for every $i$ in $\bI$ the functor $\psi(i) \colon \bT(i)\to \bD(i)$ takes values in the essential image of $\phi(i) \colon \bC(i)\to \bD(i)$.
 If $\phi$ is objectwise fully faithful, then the canonical map induces an equivalence
 \[\Map_{\Fun(\bJ,\Cat_\infty)}(\bT,\bC)\xrightarrow{\simeq} \Map'_{\Fun(\bJ,\Cati)}(\bT,\bD)\ .\qedhere\]
\end{rem}

Let $\bI$ and $\bJ$ be small categories and let $\bC \colon \bI \times \bJ \to \Cat_\infty$ be a functor.

\begin{lem}\label{lem:lim-colim-commute}
	Assume:
	\begin{enumerate}
		\item $\bI$ is filtered.
		\item $\bJ$ has only finitely many objects.
		\item For every morphism $i \to i'$ in $\bI$ and every  object $j$ in $\bJ$ the functor $\bC(i,j) \to \bC(i',j)$ is fully faithful.
	\end{enumerate}
	Then the natural functor \begin{equation}\label{ververkj34f34f34f3f}
 \colim_\bI \lim_\bJ \bC \to \lim_\bJ \colim_\bI \bC
\end{equation}	 
	is an equivalence.
\end{lem}
\begin{proof}
	By assumption, the transformation $\bC(i,-) \to \bC(i',-)$ of diagrams $\bJ \to \Cat_\infty$ is objectwise fully faithful for every morphism $i \to i'$ in $\bI$.
	 The induced functor $\lim_\bJ \bC(i,-) \to \lim_\bJ \bC(i',-)$ is fully faithful  {by \cref{tieqorgfgregwegergw}}.
	Since $\bI$ is filtered, it follows that the functor
	\[ \lim_\bJ \bC(i,-) \to \colim_\bI \lim_\bJ \bC \]
	is fully faithful for every $i$ in $\bI$.
	
	Similarly, the canonical functor $\bC(i,j) \to \colim_\bI \bC(-,j)$ is fully faithful for all $i$ in $\bI$ and $j$ in $\bJ$.
	We conclude that the induced map
	\[ \lim_\bJ \bC(i,-) \to \lim_\bJ \colim_\bI \bC \]
	is fully faithful for every $i$ in $\bI$.
	
	Since $\bI$ is filtered and since we have a commutative diagram
	\[\xymatrix{
		\lim_\bJ \bC(i,-)\ar[r]\ar[d] & \lim_\bJ \colim_\bI \bC \\
		\colim_\bI \lim_\bJ \bC\ar[ru] &
	}\]
	for every $i$ in $\bI$, it follows that $\colim_\bI \lim_\bJ \bC \to \lim_\bJ \colim_\bI \bC$ is fully faithful.
	
	We are left with showing essential surjectivity  {of the functor \eqref{ververkj34f34f34f3f}}. Since $\lim_\bJ$ is right adjoint to the functor $\underline{-}$ taking constant {$\bJ$-}diagrams, an object {$A$} in $\lim_\bJ \colim_\bI \bC$ corresponds to a natural {transformation}
	\[ \underline{\Delta^0} \to \colim_\bI \bC \]
	of diagrams $\bJ \to \Cat_\infty$.
	Since $\Delta^0$ is compact and since $\bJ$ has only finitely many objects, there exists some $i_0$ in $\bI$ such that $\underline{\Delta^0}(j) = \Delta^0 \to \colim_\bI \bC(-,j)$ factors for every $j$ in $\bJ$ through the canonical map $\bC(i_0,j) \to \colim_\bI \bC(-,j)$.
	Applying \cref{lem:factoring-transformations} yields a transformation $\underline{\Delta^0} \to \bC(i_0,-)$ which fits into a commutative triangle
	\[\xymatrix{
		\underline{\Delta^0}\ar[r]\ar[d] & \colim_\bI \bC \\
		\bC(i_0,-)\ar[ru] &
	}\ .\]
	Hence, we obtain an object in $\lim_\bJ \bC(i_0,-)$ whose image under the canonical map $\lim_\bJ \bC(i_0,-) \to \colim_\bI \lim_\bJ \bC$ provides the required preimage of $A$.
\end{proof}

 {For future reference, let us also recall what it means for filtered colimits to distribute over products.}
Let $\bM$ be an  $\infty$-category admitting small filtered colimits and small products.
\begin{ddd}\label{wergewgregwergw}
 We say that filtered colimits distribute over products in $\bM$ if for any family of small filtered categories  $(\bF_{i})_{i\in I}$ and family of functors
 $(E_{i} \colon \bF_{i}\to \bM)_{i\in I}$ the canonical morphism
 \[ \colim_{(F_{i})_{i}\in \prod_{i\in I}\bF_{i}} \prod_{i\in I} E_{i}(F_{i}) \to \prod_{i\in I}\colim_{F_{i}\in \bF_{i}} E_{i}(F_{i}) \]
 is an equivalence.
\end{ddd}

\begin{ex} \label{rewgoiu0erogwegwergwrgw}
 Examples of $\infty$-categories  {in which filtered colimits distribute over products}
 are $\Spc$, $\Cati$, and hence also $\Cle$ by \cref{prop:catex finitely complete}. 
\end{ex}

\subsection{The calculus of fractions formula}\label{sec:fractions}

 {Our goal in this section is to establish conditions under which there exists an explicit formula to compute mapping spaces in Dwyer--Kan localisations.}

\begin{ddd}\label{oijoiwegfwefgewfwe}
 A  {relative $\infty$-category} $(\bC,W)$ is an object $\bC$ of $\Cati$ together with a subcategory $W$ containing  {all identity morphisms}.
\end{ddd}

If $(\bC,W)$ is a  relative $\infty$-category,  its Dwyer-Kan localisation
\begin{equation}\label{qfweqkjfnkjewfewfqefqefqf}
\ell\colon \bC\to \bC[W^{-1}]
\end{equation}
 satisfies the universal property that for every $\bD$ in $\Cati$ the functor $\ell$ induces an equivalence  
 \[\Fun (\bC[W^{-1}],\bD)\to  \Fun^W(\bC,\bD)\ ,\]
 where the right-hand side is the subcategory of  functors which send  {morphisms in} $W$ to equivalences.
 
 As explained in \cite[Sec.~1]{Barwick:Ktheory}, one can define a large $\infty$-category  $\Rel_\infty$ of relative $\infty$-categories and functors preserving the chosen subcategories as a subcategory of the arrow category $ {\Fun(\Delta^1,\Cati)}$.
 The universal property of the Dwyer--Kan localisation implies that 
  {there exists a functor}
 \begin{equation}\label{rghvfigvsfgfdsg}
  \Loc \colon \Rel_{\infty}\to \Cati
 \end{equation}
 sending a relative $\infty$-category $(\bC,W)$ to $\bC[W^{-1}]$  {which is left adjoint} to the functor $\Cati \to \Rel_\infty$ which sends an $\infty$-category $\bC$ to the relative $\infty$-category $(\bC,\bC^{\simeq})$, where $\bC^{\simeq}$ denotes the groupoid core of $\bC$.

 {In the following, we summarise the main results of \cite[Sec.~7.2]{Cisinski:2017}.}
We consider a  {relative $\infty$-category} $(\bC,W)$, and we let $A$ be an object of $\bC$.

\begin{ddd}[{\cite[Def.~7.2.2]{Cisinski:2017}}]
 {A putative calculus of fractions at $A$ is a functor $\pi \colon W(A) \to \bC$ with the following properties:
 \begin{enumerate}
  \item $W(A)$ has a final object $A_0$ with $\pi(A_0) \simeq A$;
  \item the image of every morphism $B \to A_0$ in $W(A)$ lies in $W$.\qedhere
 \end{enumerate}}
\end{ddd}

 {\begin{ex}
The canonical example of a putative calculus of fractions is given by the full subcategory $W(A)$ of $\bC_{/A}$ spanned by the morphisms in $W$, together with the projection
\[ \pi \colon W(A) \to \bC_{/A} \to \bC\ .\qedhere\]
\end{ex}}

 Let $\pi \colon W(A) \to \bC$ be a putative calculus of fractions.
  {The colimit of the diagram
  \[ W(A)^\op \xrightarrow{\pi^\op} \bC^\op \xrightarrow{\yo}\Fun(\bC,\Spc) \]
  defines a functor
  yields a functor
 \begin{equation}\label{eq:fraction-formula}
   \cM_A \colon \bC \to \Spc,\quad B \mapsto \colim_{ A'    \in W(A)^\op} \Map_\bC(\pi(A'),B)\ .
 \end{equation}

\begin{prop}\label{prop:charactrightcalculus}
 {The following are equivalent:
 \begin{enumerate}
  \item The functor $\cM_A$ inverts all morphisms in $W$.
   \item The canonical map
   \begin{equation}\label{54joriesgergw3r}
   \colim_{A^{\prime} \in W(A)^{\op}} \Map_{\bC}(\pi(A^{\prime}),B) \to \Map_{\bC[W^{-1}]}(\ell(A),\ell(B)) \end{equation}
  is an equivalence of spaces for every object $B$ of $\bC$.
  \end{enumerate}}
 \end{prop}
\begin{proof}
 {In the language of \cite{Cisinski:2017}, $\cM_A$ inverting all morphisms in $W$ means that $(W(A),\pi)$ is a right calculus of fractions.
 Hence \cite[Thm.~7.2.7]{Cisinski:2017} implies that the morphism \eqref{54joriesgergw3r} is an equivalence for every object $B$ of $\bC$.
Conversely, if \eqref{54joriesgergw3r} is an equivalence for every object $B$ of $\bC$, then $\cM_A$ sends all morphisms in $W$ to equivalences.}
\end{proof}

\subsection{Localisations of left-exact \texorpdfstring{$\infty$}{infinity}-categories} \label{rgiojreogiergregregreg}
 
 We will be primarily interested in localisations of left-exact $\infty$-categories.
 
   \begin{ddd}\label{gjierogjeoigerg3554656}
  	A relative left-exact $\infty$-category $(\bC,W)$ is a  {relative $\infty$-category} such that $\bC$ is {left-exact}.
    \end{ddd}
 
 In general, there is no reason to expect that the Dwyer--Kan localisation of a relative left-exact $\infty$-category is also left-exact.
 The following records sufficient conditions that guarantee this does in fact happen.
 
 Let $(\bC,W)$ be a  {relative $\infty$-category}.

\begin{ddd}\label{ioejregergrgrereg}
We say that  $W$ is preserved by pullbacks if every diagram
\[\xymatrix{&B^{\prime}\ar[d]^{f}\\A\ar[r]&B} \]
in $\bC$  with $f$ in $W$ can be extended to a  {pullback} diagram 
\[\xymatrix{A'\ar[r]\ar[d]_{g}&B^{\prime}\ar[d]^{f}\\A\ar[r]&B}\]
in $\bC$ with $g$ in $W$.
\end{ddd}

Let $(\bC,W)$ be a  {relative $\infty$-category}.  
\begin{lem}\label{efiweofwefewfewf} Assume:
 \begin{enumerate}
 \item $W$ is preserved by pullbacks.
 \item $W$  has the two-out-of-three property.
 \item  $\bC$ admits finite limits.
 \end{enumerate} 
 {Then $\ell \colon \bC\to \bC[W^{-1}]$ preserves finite limits.}
 \end{lem}
\begin{proof}
  Since $\bC$ has finite limits, it may be considered
  as a category with weak equivalences and fibrations in the sense of \cite[Def.~7.4.12]{Cisinski:2017}, where the weak equivalences are the
  elements of $W$, and the fibrations are all maps in $\bC$. Hence we can apply \cite[Prop.~7.5.6]{Cisinski:2017}.
  
 Alternatively, one can combine \cite[Thm.~7.2.16]{Cisinski:2017} with \cref{prop:charactrightcalculus} to see that the mapping spaces in the localisation are given by filtered colimits,
  and use this directly to show the lemma.
\end{proof}
  
 {Let $(\bC,W)$ be a relative left-exact $\infty$-category, and let $\ell:\bC\to \bC[W^{-1}]$ be the Dwyer--Kan localisation.
 
\begin{ddd}\label{wegioegfvsfggdfgs}
We say that  $\ell$ is a localisation among left-exact $\infty$-categories if
$\bC[W^{-1}]$ and $\ell$ are left-exact and if for every left-exact $\infty$-category $\bD$
 {the restriction functor}
\[ \ell^*\colon \Fun_{\Cle}(\bC[W^{-1}],\bD)\to\Fun^{W}_{\Cle}(\bC,\bD) \]
 {is an equivalence}, where $\Fun^{W}_{\Cle}(\bC,\bD)$  is the full subcategory of $\Fun_{\Cle}(\bC,\bD)$ on functors which send the morphisms in $W$ to equivalences.
\end{ddd}

\begin{prop}\label{ih43iugu34g34g3}
If $\ell\colon \bC\to \bC[W^{-1}]$ preserves finite limits, then
 $\ell$ is a  localisation among left-exact $\infty$-categories.
Furthermore, the induced functor 
  $\bC^{\Delta^{1}}\to \bC[W^{-1}]^{\Delta^{1}}$ is essentially surjective.
\end{prop}
\begin{proof}
Let $\tilde W$ be the subcategory of $\bC$ given by the morphisms which are sent to equivalence by $\ell$. 
Then $\tilde W$ satisfies the two-out-of-three property, $W\subseteq \tilde W$, \ {and $\tilde W$ is closed under pullbacks since $\ell$ preserves finite limits.}

 {Moreover, let $\overline W$ denote the smallest subcategory of $\bC$ which contains $W$, satisfies the two-out-of-three property, and is preserved by pullbacks.
Then $W \subseteq \overline W \subseteq \tilde W$, so the Dwyer--Kan localisations at each of these three subcategories agree.
\cite[Prop.~7.5.11]{Cisinski:2017} applies to the localisation at $\overline W$ to show that $\bC[W^{-1}]$ is left-exact and has the correct universal property.}

 {By \cite[Thm.~7.2.16]{Cisinski:2017}, \cref{prop:charactrightcalculus} provides a formula for the mapping spaces in $\bC[\overline{W}^{-1}]$ which shows in particular that}
 any map $A\to B$ in $\bC[\overline W^{-1}]$ may be written up to  equivalence as a composition $f s^{-1}$, where $f \colon A'\to B$ is a map in $\bC$, while $s \colon A'\to A$ is a map in $\overline W$. 
\end{proof}

  We let $\Rel^\mathrm{Lex}_{\infty,*}$  denote the subcategory of $\Rel_{\infty}$ of pairs $(\bC,W)$ where $\bC$ is left-exact
and $\bC\to \bC[W^{-1}]$ is left-exact, and left-exact functors. Then the localisation functor from \eqref{rghvfigvsfgfdsg} restricts to a functor \begin{equation}\label{fqwepofjkqweopfdqewdq}
\Loc\colon\Rel^\mathrm{Lex}_{\infty,*}\to \Cle\ .
\end{equation}

\subsection{Stabilisation and cofibres}\label{sec:stab}
 In this section, we describe the process of stabilisation of left-exact $\infty$-categories
 and exhibit a class of sequences in $\Cle$ which give rise to cofibre sequences in $\stCat$ upon stabilisation.

We have a functor
\begin{equation}\label{vervoiehjeoirverveve}
\hat \Sp \colon \Cle\to \Cle\ , \quad  \hat \Sp(\bC):= \colim (\bC\xrightarrow{\Omega} \bC\xrightarrow{\Omega} \bC\xrightarrow{\Omega}\dots)
\end{equation}   
 and a natural transformation
$\hat \Omega^{\infty}\colon \id\to \hat \Sp$.
 {Note that this construction corresponds to Spanier--Whitehead stabilisation under the identification $\Cle \simeq \Cre$  {given by  {$\bC\mapsto \bC^{\op}$}.}
Recall that $\stCat$ denotes the full subcategory of $\Cle$ of stable $\infty$-categories.
\begin{lem}\label{griooergergergreg}
 The functor $\hat \Sp$ has an essentially unique factorisation
\[ \xymatrix{&\stCat \ar[d]\\\Cle\ar@{..>}[ur]^{\tilde \Sp}\ar[r]^{\hat \Sp}& \Cle} \]
which fits into an adjunction
\[ \tilde \Sp \colon\Cle\leftrightarrows \stCat \cocolon \incl\ .\]
 \end{lem}
\begin{proof}
 Under conjugation with the equivalence $(-)^\op \colon \Cle \xrightarrow{\simeq} \Cre$, the functor $\hat\Sp$ corresponds to the Spanier--Whitehead stabilisation, so the lemma follows from \cite[Prop.~C.1.1.7]{SAG}.
\end{proof}

\begin{ddd}\label{foiwregergegregreg}
We call $\tilde \Sp \colon \Cle\to \stCat $ the stabilisation functor.
\end{ddd}

\begin{lem}\label{geroigergergre}
The stabilisation functor $\tilde \Sp$ preserves fully faithfulness.
 \end{lem}
\begin{proof} 
The functor $\hat \Sp$ from \eqref{vervoiehjeoirverveve} is given by a filtered colimit in $\Cle$. Note that filtered colimits in $\Cle$ can be calculated in $\Cati$ by \cref{prop:catex finitely complete}. Furthermore a filtered colimit of fully faithful functors  in $\Cati$  is fully faithful. This implies the assertion.
\end{proof}

Let $\phi \colon \bC \to \bD$ be a morphism in $\Cle$.
\begin{ddd}\label{rgioergieorfrgergergergergerge}
We define the stable cofibre of $\phi$ to be the stable $\infty$-category
\[ \Cofib^{s}(\phi) := \Cofib(\tilde \Sp(\phi))\ .\qedhere \]
\end{ddd}

 {If $\phi$ is fully faithful, \cref{geroigergergre} implies by \cite[Prop.~A.{3.7.iii)}]{thenine2} that
\[ \tilde\Sp(\bD) \xrightarrow{\tilde\Sp(\phi)} \tilde\Sp(\bC) \to \Cofib^{s}(\phi) \]
is a  Karoubi 
sequence in the sense of \cite[Def.~A.{3.5}]{thenine2}.
In the following we will
show that $\Cofib^s(\phi)$ admits an explicit model in terms of a Dwyer--Kan localisation of $\bC$.}

Let $\bC$ be in $\Cle$, and let $f \colon C\to C^{\prime}$ be a morphism  in $\bC$.
Using the existence of finite limits and a zero object in $\bC$,
the fibre of $f$ can be defined by
\[ \Fib(f):= 0\times_{C^{\prime}}C\ .\]
If $\phi \colon \bD \to \bC$ is a morphism in $\Cle$, then we define the {relative} left-exact $\infty$-category $(\bC,W_{\phi})$
such that $W_{\phi}$ is  {the smallest subcategory containing those morphisms whose fibre lies in the essential image of $\phi$, which satisfies the two-out-of-three-property and which is closed under pullbacks.
Then the Dwyer--Kan localisation is left-exact by \cref{efiweofwefewfewf} and \cref{ih43iugu34g34g3}.}

\begin{lem}\label{rgioowfwefwefw}
 {The sequence of functors
\[ \tilde\Sp(\bD) \to \tilde\Sp(\bC) \to \tilde\Sp(\bC[W_\phi^{-1}]) \]
 is a cofibre sequence
  in $\stCat$.}
\end{lem}
\begin{proof}
 {For every stable $\infty$-category $\bE$, we have a commutative square
 \[\xymatrix{
  \Fun_{\stCat}(\tilde\Sp(\bC[W_\phi^{-1}]),\bE)\ar[r]\ar[d]_{\simeq} & \Fun_{\stCat}(\tilde\Sp(\bC),\bE)\ar[d]^{\simeq} \\
  \Fun_{\Cle}(\bC[W_\phi^{-1}],\bE)\ar[r] & \Fun_{\Cle}(\bC,\bE)
  }\]
  in which all morphisms are given by restriction functors.
  As indicated, the vertical morphisms are equivalences by \cref{griooergergergreg}.}
  
 {Since $\bE$ is stable, a left-exact functor $\bC \to \bE$ vanishes on $\bD$ if and only if it inverts all morphisms whose fibre lies in $\bD$.
  Moreover, the subcategory of morphisms that get inverted by a left-exact functor $\bC \to \bE$ is closed under pullbacks and satisfies the two-out-of-three-property.
  Hence such a functor vanishes on $\bD$ if and only if it inverts all morphisms in $W_\phi$.
  It follows from \cref{ih43iugu34g34g3} that the essential image of the bottom horizontal functor given precisely by the functors vanishing on $\bD$.}
  
 {Moreover, the restriction of a functor $\tilde\Sp(\bC) \to \bE$ vanishing on $\tilde\Sp(\bD)$ along the unit map $\bC \to \tilde\Sp(\bC)$ also vanishes on $\bD$.
  Conversely, the functor $\tilde\Sp(\bC) \to \bE$ induced by a left-exact functor $\bC \to \bE$ vanishing on $\bD$ is trivial on $\tilde\Sp(\bD)$ since any object in $\tilde\Sp(\bD)$ lies in $\bD$ after finitely many applications of $\Omega$.
  Hence the right vertical functor restricts to an equivalence between the full subcategories given by functors which vanish on $\tilde\Sp(\bD)$ and $\bD$, respectively, which proves that $\tilde\Sp(\bC[W_\phi^{-1}]$ has the desired universal property.}
\end{proof}

\begin{rem}\label{gieorgegergregerg}
 {Let $\phi \colon \bD \to \bC$ be an exact functor between stable $\infty$-categories.
 Since the essential image of $\phi$ is a full stable subcategory of $\bC$, the subcategory $W_\phi$ is given precisely by the collection of morphisms whose fibre lies in the essential image of $\phi$.
 Again by stability, $W_\phi$ is equivalently given by the collection of morphisms whose cofibre lies in the essential image of $\phi$.
 Hence \cref{efiweofwefewfewf} and its dual imply that the localisation functor $\ell \colon \bC \to \bC[W_\phi^{-1}]$ preserves both finite limits and colimits, and it follows that $\bC[W_\phi^{-1}]$ is stable.
 So the unit map $\bC[W_\phi^{-1}] \to \tilde\Sp(\bC[W_\phi^{-1}])$ is an equivalence in this case, and \cref{rgioowfwefwefw} recovers the well-known fact that cofibres of stable $\infty$-categories are given by Dwyer--Kan localisations.}
 
  {In particular, we have for every morphism $\phi \colon \bD \to \bC$ in $\Cle$ an equivalence
 \[ \tilde \Sp(\bC)[W_{\tilde \Sp(\phi)}^{-1}] \xrightarrow{\simeq} \tilde \Sp(\bC[W_{\phi}^{-1}])\ .\qedhere\]}
\end{rem}

\subsection{Excisive squares in \texorpdfstring{$\Cle$}{CatLex}  {and homological functors}} \label{sec:excisivesquares} 
 Any localising invariant on stable $\infty$-categories induces an invariant on left-exact $\infty$-categories by precomposition with the stabilisation functor $\tilde\Sp \colon \Cle \to \stCat$.
\cref{geroigergergre} implies that a localising invariant sends
a commutative square
\begin{equation}\label{gioegegergreg}
\xymatrix{\bD \ar[r]^{\phi }\ar[d]_{\psi^{\prime} }&\bC \ar[d]^{\psi }\\\bD^{\prime} \ar[r]^{\phi^{\prime} }&\bC^{\prime} }
\end{equation}
 {of left-exact $\infty$-categories}
to a pushout
if it is excisive in the sense of
the following definition.
 
\begin{ddd}\label{ugioerguoerug}
 A square \eqref{gioegegergreg} in $\Cle$ is called excisive if:
 \begin{enumerate}
 \item The functors $\phi \colon \bD\to \bC $ and $\phi^{\prime} \colon \bD^{\prime}\to \bC^{\prime}$
are fully faithful.
 \item The induced functor {on stable cofibres}
 $\bar \psi \colon \Cofib^{s}(\phi) \to \Cofib^{s}(\phi^{\prime} )$ is an equivalence. \qedhere
\end{enumerate}
\end{ddd}  

Under some conditions a colimit of a diagram of excisive squares is again excisive.
 Let $\bI$ be a small $\infty$-category, and consider an $\bI$-indexed diagram of squares of the shape \eqref{gioegegergreg}.
\begin{lem}\label{rfiorjgfoqfwewfefewfqfe} Assume:
	\begin{enumerate}
		\item One of the following holds:
		\begin{enumerate}
			\item $\bI$ is filtered.
			\item $\bI$ is a groupoid.
		\end{enumerate}
\item  The evaluation of  the diagram at every object of $\bI$ is an excisive square in $\Cle$.
\end{enumerate}
 Then
 \begin{equation}\label{gioegegergreg111}
\xymatrix@C=3em{\colim_{\bI}\bD \ar[r]^{\colim_{\bI}\phi }\ar[d]_{\colim_{\bI}\psi^{\prime} }&\colim_{\bI}\bC  \ar[d]^{\colim_{\bI}\psi }\\\colim_{\bI}\bD^{\prime} \ar[r]^{\colim_{\bI}\phi^{\prime} }&\colim_{\bI}\bC^{\prime}  } 
\end{equation}
is an excisive square in $\Cle$.
 \end{lem}
\begin{proof}
 If $\bI$ is filtered, then the functors $\colim_{\bI}\phi$ and $\colim_{\bI}\phi^{\prime}$ are fully faithful
 since a filtered colimit of fully faithful functors in $\Cle$ (which can be calculated in $\Cat_{\infty}$, see  \cref{prop:catex finitely complete}) is again fully faithful.
In the other case, i.e., when $\bI$ is a groupoid, we use \cref{ewfweijor23r23r23r32r23r23} to conclude fully faithfulness.

 Since $\tilde\Sp$ and taking cofibres preserve colimits,
  it follows that $\Cofib^s(\colim_\bI \phi) \to \Cofib^s(\colim_\bI \phi')$ is an equivalence.
\end{proof}

As indicated above, we are mostly interested in localising invariants defined on left-exact $\infty$-categories. To conclude the discussion in this section, we introduce some related terminoloy and explain the relation between localising invariants in our sense and localising invariants on stable $\infty$-categories.
 
 We consider a functor $\Homol \colon \Cle\to \bM$.
 We say that $\Homol$ inverts Morita equivalences if it sends the morphism $\bC\to \Idem(\bC)$ (the unit of adjunction \eqref{ewfbjhbfjhqerfqfewfqefe}) to an equivalence for every $\bC$ in $\Cle$.
 
 \begin{ddd}\label{qrevoiqrjoirqfcwqecq}
  The functor $\Homol$ is called homological if it has the following properties:
 \begin{enumerate}
 \item\label{ergqfqewffedqdeqdqwed} $\bM$ is stable and cocomplete.
 \item $\Homol$ preserves filtered colimits.
 \item $\Homol$ sends excisive squares  in $\Cle$  (\cref{ugioerguoerug})  to pushout squares.
 \end{enumerate}
  The functor $\Homol$ will be called a finitary localising invariant if it in addition inverts Morita equivalences.
\end{ddd}

Following \cite[Def.~8.1]{MR3070515}, a functor $\stCat\to \bM$ is called a stable finitary localising invariant\footnote{We added the adjective {\em stable} in order to distinguish this notion from the one introduced in \cref{qrevoiqrjoirqfcwqecq}. We further added the word {\em finitary} in order to highlight
that the functor preserves filtered colimits, as one might want to drop this assumption in certain applications.} if $\bM$ is stable and cocomplete, and the functor  inverts Morita equivalences, preserves filtered colimits, and sends  {Karoubi} sequences \cite[Def.~A.{3.5}]{thenine2}  to  {cofibre} sequences. 

 {By \cref{griooergergergreg}}, we have an adjunction
\begin{equation}\label{qwefqwefewdqdqedqed}
\tilde \Sp\colon\Cle\leftrightarrows   \stCat\cocolon \incl\ .
\end{equation} 
 If $\Homol$ is a finitary localising invariant, then $\Homol\circ \incl$ is clearly a stable    finitary localising invariant.
 The following lemma justifies our terminology and shows that  finitary localising invariants correspond to stable finitary localising invariants
 by precomposition with the stabilisation functor $\tilde \Sp$.
 \begin{lem}\label{lem:loc-vs-stloc}
 	\ \begin{enumerate}
 		\item If $\Homol$ is a
		 homological functor, then the natural transformation
		\[  {\Homol \to \Homol\circ \incl \circ \tilde\Sp} \] (induced by the unit of adjunction \eqref{qwefqwefewdqdqedqed}) is an equivalence.
		\item \label{eri9gqergrqgrgreqerggqrg} If $L$ is a stable finitary localising invariant, then $L \circ \tilde\Sp$ is a finitary localising invariant and the transformation
		\[ L \circ \tilde\Sp \circ \incl \to L \]
		(induced by the counit of adjunction \eqref{qwefqwefewdqdqedqed}) is an equivalence.
	\end{enumerate}
 \end{lem}
 \begin{proof}
 Let $\bC$ be in $\Cle$. Then we have the following excisive square in $\Cle$:
 \begin{equation*}
\xymatrix{0\ar[d] \ar[r] &\bC\ar[d]\\0\ar[r]&\tilde \Sp(\bC)}\ .
\end{equation*} 
Indeed, the horizontal morphisms are fully faithful, and the induced morphism on stable cofibres is the identity of $\tilde \Sp(\bC)$.
The functor $\Homol$ sends this square to a pushout square in $\bM$. Since $\Homol(0) \simeq 0$, we conclude that $\Homol(\bC)\to \Homol(\tilde \Sp(\bC))$ is an equivalence.

 For the second assertion, we observe that $L \circ \tilde\Sp$ is a finitary localising invariant since $\tilde\Sp$ commutes with filtered colimits and idempotent completion {(since filtered colimits of idempotent complete $\infty$-categories are idempotent complete \cite[Cor.~4.4.5.21]{htt})},
  and preserves fully faithful functors by \cref{geroigergergre}.
 Moreover, the counit $\tilde\Sp \circ \incl \to \id$ is an equivalence.
 \end{proof}

\begin{ex}
 We let
 \begin{equation}\label{v4toi3hfio3f3f34f3f}
\cU_{loc} \colon \stCat \to \cM_{loc}
\end{equation}
denote the universal (stable finitary) localising invariant of  Blumberg--Gepner--Tabuada \cite[Thm.~8.7]{MR3070515}. 
 The target $\cM_{loc}$ is a presentable stable $\infty$-category.
      The composition
      \begin{equation}\label{ihgoihjerigojgerggregre}
\UK \colon \Cle\xrightarrow{\tilde \Sp} \stCat \xrightarrow{\cU_{loc}} \cM_{loc}
\end{equation}
is a  finitary localising invariant by \cref{lem:loc-vs-stloc}.
\end{ex}

Let $\Homol \colon \Cle\to \bM$ be a functor.
\begin{lem}\label{roijewrogwergregwregwregw}
	If  $\Homol$ is homological, then it preserves coproducts.
 \end{lem}
\begin{proof}
For $\bC,\bD$  in $ \Cle$ 
 {the commutative square
\[\xymatrix{0\ar[r]\ar[d]&\bC\ar[d]\\\bD\ar[r]&\bC {\oplus} \bD} \]
is excisive, so it becomes a pushout upon application of $\Homol$.
Since $\Homol(0) \simeq 0$, the induced square exhibits $\Homol(\bC \oplus \bD)$ as a pushout of $\Homol(\bC)$ and $\Homol(\bD)$.}
\end{proof}

\begin{rem}
As a consequence of  \cref{roijewrogwergregwregwregw}, a homological functor is additive. More precisely, let 
$F,G \colon \bC\to \bD$ be two morphisms in $\Cle$ between the same objects. Then we have an equivalence \begin{equation}\label{rfoijqwfoiejfojoiewfqwefqfqewfqewf}
\Homol(F+G)\simeq \Homol(F)+\Homol(G)\end{equation}
of morphisms from $\Homol(\bC)$ to $\Homol(\bD)$.
\end{rem}

\section{Sheaves on bornological coarse spaces}
 {As explained in \cref{contr}, our main goal is the construction of a functor
\[ \bV\colon G\BC\to \Cle \]
which associates to every $G$-bornological coarse space $X$ a left-exact $\infty$-category of $X$-controlled objects.
This construction proceeds in several steps, each of which makes use of some more structure encoded in the notion of a $G$-bornological coarse space.
Our presentation does not assume previous knowledge of $G$-bornological coarse spaces.
Instead, we will introduce $G$-bornological coarse spaces step by step as our constructions require.}

 {Throughout this and the following sections, we will freely use the notation introduced in \cref{egihweogergwreggwrgr}. In particular}, $\CL$ denotes the very large $\infty$-category of opposites of pointed, compactly generated presentable $\infty$-categories and right adjoint functors preserving cocompact objects, and $\Cle$ denotes the large $\infty$-category of small, pointed, left-exact $\infty$-categories and left-exact functors.

\subsection{Presheaves}\label{wiwoerthgwrtgwregergwrg}
 {In the first step of our construction, we simply assign to every set $X$ and $\infty$-category $\bC$ the $\infty$-category of functors from the power set of $X$ to $\bC$, and record the functoriality of this construction.
More importantly, we also introduce the notion of an entourage in \cref{def:ent} and define the associated thickening and thinning functors in \eqref{V-thick} and \eqref{V-thin}.}

Let $X$ be a set. By $\cP_{X}$ we denote the poset of subsets of $X$ with the inclusion relation.  For $\bC$ in $\CL$ 
we consider the functor category
\begin{equation}\label{qewfoi1jo4irfrefqfef}
\PSh_{\bC}(X):=\Fun(\cP_{X}^{\op},\bC)
\end{equation}
called the $\infty$-category of $\bC$-valued presheaves on $X$. It is
  again an object of $\CL$. 

A map of sets $f\colon X\to X^{\prime}$  gives rise to the inverse image map $f^{-1}(-)\colon \cP_{X^{\prime}}\to \cP_{X}$ of posets. By precomposition it induces a morphism  
\[ \hat f_{*}\colon\PSh_{\bC}(X)\to \PSh_{\bC}(X^{\prime}) \]
in $\CL$ between the presheaf categories.

A morphism $\phi \colon \bC\to \bC^{\prime}$  in $\CL$ gives rise to a morphism
\begin{equation}\label{wterbwkobwrevwerv}
\hat \phi_{*}\colon\PSh_{\bC}(X)\to \PSh_{\bC^{\prime}}(X )
\end{equation}  in $\CL$ by postcomposition with $\phi$.
These constructions can be turned into a functor \begin{equation}\label{qewfoijoi23rfgervfwevwvre}
\PSh\colon\Set\times \CL\to \CL\ .
\end{equation}
By taking images, the map $f$ also induces the morphism of posets $f(-) \colon \cP_{X}\to \cP_{X^{\prime}}$. The relations
$f(f^{-1}(Y^{\prime}))\subseteq Y^{\prime}$ for all $Y^{\prime}$ in $\cP_{X^{\prime}}$ and $Y\subseteq f^{-1}(f(Y))$  for all $Y$ in $\cP_{X}$  {provide} the counit and the unit of an adjunction \begin{equation}\label{qweflkqnmwfklmqwfwefqefqfewf}
f(-)\colon\cP_{X}\leftrightarrows \cP_{X^{\prime}}\cocolon f^{-1}(-)
\end{equation}  between poset morphisms.
We get an induced adjunction
\begin{equation}\label{ad-hat-f}\hat f^{*}\colon  \PSh_{\bC}(X^{\prime})\leftrightarrows  \PSh_{\bC}(X)\cocolon \hat f_{*} \end{equation}  
 of functors between the  presheaf categories,  where $\hat f^{*}$ is given by precomposition with $f(-)$. It is also   a morphism in $\CL$.

If $\phi\colon \bC\to \bC^{\prime}$ is a morphism in $\CL$, then $\hat \phi_{*}$ is a morphism in $\CL$ and therefore fits into an adjunction
\begin{equation}\label{ererggrefqwfewfeqwfqef} \hat \phi^{*}\colon \PSh_{\bC^{\prime}}(X)\leftrightarrows \PSh_{\bC}(X)\cocolon  \hat \phi_{*}\ ,\end{equation}
where $\hat \phi^{*}$ preserves cofiltered limits (\cref{rgqeroqjkopqwefewfefqef}).

\begin{ddd}\label{def:ent}
 {An entourage on $X$ is a subset of $X \times X$, i.e., a relation on $X$.}
\end{ddd}

 {Consider an entourage $V$ on $X$ such that $\diag(X) \subseteq V$.}
 {For every subset $Y$ of $X$} we define the $V$-thickening
\begin{equation}\label{V-thick} V[Y]:=\{x \in X\:|\: (\exists y\in Y\:|\: (x,y)\in V)\}\end{equation}
and the $V$-thinning
\begin{equation}\label{V-thin} V(Y):=\{x\in X\:|\: V[\{x\}]\subseteq Y\}\ .\end{equation}
 {The reflexivity of $V$ guarantees that $Y \subseteq V[Y]$ and $V(Y) \subseteq Y$ for every subset $Y$ of $X$.}

\begin{ex}\label{ex:ent}
 {\ \begin{enumerate}
  \item On every set $X$, we can consider the entourage $\diag(X)$. Both $\diag(X)$-thickening and $\diag(X)$-thinning are given by the identity functor on $\cP_X$.
  \item At the other extreme, $X \times X$ is also an entourage containing $\diag(X)$.
   The $(X \times X)$-thickening of a non-empty subset is the entirety of $X$, while $(X \times X)[\emptyset] = \emptyset$.
   Analogously, the $(X \times X)$-thinning of any proper subset of $X$ is empty, while $(X \times X)(X) = X$.
  \item Suppose that $d$ is a metric on $X$. Then for $r \geq 0$
 \begin{equation}\label{qwefwokjkqowdewdeqd}  V_r := \{ (x,y) \in X \times X \mid d(x,y) \leq r \}  
\end{equation} 
  is an entourage on $X$. For a subset $Y$ of $X$, the $V_r$-thickening $V_r[Y]$ is the union over all closed $r$-balls with center in $Y$, and the $V_r$-thinning $V_r(Y)$ is the set of those points $y$ in $Y$ such that the closed $r$-ball around $y$ is entirely contained in $Y$.\qedhere
 \end{enumerate}}
\end{ex}

Thickening and thinning induce  morphisms of posets
\begin{equation}\label{qwefew14rqefeqwf}
V[-], V(-)\colon\cP_{X}\to \cP_{X}\ . 
\end{equation}
We define the  morphisms
\[ V^{*},V_{*}\colon \PSh_{\bC}(X)\to \PSh_{\bC}(X) \]
in $\CL$ through precomposition with $V[-]$ and $V(-)$, respectively.
 The relations
$Y\subseteq V(V[Y])$ and $V[V(Y)]\subseteq Y$ for all $Y$ in $\cP_{X}$ provide the unit and counit of an adjunction \begin{equation}\label{etkjnkjenkwervwerbververvwv}
V[-]\colon \cP_{X}\leftrightarrows \cP_{X}\cocolon V(-)
\end{equation} between endofunctors of $\cP_{X}$.
We therefore get an induced adjunction
\begin{equation}\label{ad-hat-V}
 V^{*}\colon  \PSh_{\bC}(X)\leftrightarrows  \PSh_{\bC}(X)\cocolon  V_{*}
\end{equation}   
between the presheaf categories.

Let $G$ be a group and let $G\Set:=\Fun(BG,\Set)$ be the category of $G$-sets.
The functor $\PSh$ from \eqref{qewfoijoi23rfgervfwevwvre} induces a functor $\PSh^{G}$ {which sends a pair $(X,\bC)$ {in $G\Set \times \Fun(BG,\CL)$} to the fixed point category $\lim_{BG} \PSh_\bC(X)$ with respect to the conjugation action, and is formally} defined as the following composition:
\begin{align}\label{equiv-psh}\PSh^{G}\colon G\Set \times \Fun(BG,\CL)&\xrightarrow{\PSh} \Fun(BG\times BG,\CL)\\&\xrightarrow{\diag_{BG}^{*}} \Fun(BG,\CL)\nonumber\\&
\xrightarrow{\lim_{BG}}\CL\ .\nonumber
\end{align}
 The first morphism is  given by postcomposition with
$\PSh$, the functor $\diag_{BG} \colon BG\to BG\times BG$ is the diagonal embedding, and the limit over $BG$ exists since $\CL$ is complete by \cref{lem:CL-complete}. We write
\[ \PSh^{G}_{\bC}\colon G\Set\to \CL \]
for the specialisation  of the functor $\PSh^G$ from \eqref{equiv-psh} at $\bC$ in $ \Fun(BG,\CL)$.
 
In the following, we repeatedly use that a $G$-equivariant adjunction induces an adjunction after passing to the limit over $BG$. 
 
If $f\colon X\to X'$ is a morphism in $G\Set$,  and if $\bC$ is in $\Fun(BG,\CL)$, then by passing to the limit over $BG$
the adjunction \eqref{ad-hat-f} induces an adjunction 
\[ \hat f^{*,G}\colon  \PSh^{G}_{\bC}(X^{\prime})\leftrightarrows  \PSh^{G}_{\bC}(X)\cocolon \hat f^{G}_{*}\ .\]
 Similarly, if $X$ is a $G$-set
  and the  entourage $V$ is $G$-invariant, then the adjunction \eqref{ad-hat-V}
induces an adjunction
\begin{equation}\label{v-adj-ggg}
V^{*,G}  \colon  \PSh^{G}_{\bC}(X)\leftrightarrows  \PSh^{G}_{\bC}(X)\cocolon  V^{G}_{*}\ .
\end{equation} 
Finally, if $\phi \colon \bC\to \bC^{\prime}$ is a morphism in $\Fun(BG,\CL)$, then the adjunction  \eqref{ererggrefqwfewfeqwfqef} induces an adjunction 
\begin{equation}\label{phi-g-g}
\hat \phi^{*,G} \colon\PSh_{\bC^{\prime}}^{G}(X)\leftrightarrows \PSh_{\bC }^{G}(X)\cocolon \hat \phi^{G}_{*}\ .
\end{equation}
The right adjoints $\hat f^{G}_{*}$, $V^{G}_{*}$, and $\hat \phi_{*}^{G}$ in these adjunctions are morphisms in $\CL$, and 
their left adjoints $\hat f^{*,G}$, $V^{*,G}$ and $\hat \phi^{*,G}$ preserve  {co}filtered limits.

\subsection{Sheaves} \label{rhoijrthhgrewerggwerg}
 {In the second step, we restrict our attention to presheaves which are completely determined by their values on sufficiently small subsets, where smallness is encoded in the notion of a $U$-bounded subset (\cref{def:bdd}).
Saying that a presheaf is determined by its values on $U$-bounded subsets amounts to a rather simple sheaf condition, and we observe that the adjunctions recorded in \cref{wiwoerthgwrtgwregergwrg} descend to the level of sheaves.
Moreover, we dicuss the functoriality of the full subcategory of $U$-sheaves with respect to $U$, and prove a glueing formula for $U$-sheaves.}

Let $X$ be a set with an entourage $U$ which contains the diagonal of $X$.
\begin{ddd}\label{def:bdd}
A subset $B$ of $X$ is called $U$-bounded if $B\times B\subseteq U$.
\end{ddd}

\begin{ex}\label{ex:bdd}
 {We continue \cref{ex:ent}.
 \begin{enumerate}
  \item The $\diag(X)$-bounded subsets of $X$ are precisely the sets with at most one element.
  \item Every subset of $X$ is $(X \times X)$-bounded.
  \item If $X$ carries a metric, a subset is $V_r$-bounded if and only if its diameter is at most $r$.\qedhere
 \end{enumerate}}
\end{ex}

 {Let $Y$ be a subset of $X$},
and let $\cY$ be a family of subsets of $Y$.
\begin{ddd}  \label{rgiqjrgioqfweewfqewfqewf}
	The family $\cY$ is a $U$-covering family of $Y$ if for every {non-empty} $U$-bounded subset $B$ of  {$Y$} there exists a member  {$Y'$} of $\cY$ such that $B\subseteq {Y'}$.
\end{ddd}

\begin{ex}\label{ex-excisive-pair}
 {If $Y$ and $Z$ are subsets of $X$ satisfying $Y \cup Z = X$, then $U[Y]$ and $Z$ form a $U$-covering of $X$.}
\end{ex}

The collections of $U$-covering families of 
 {subsets of $X$} determine a Grothendieck topology $\tau^{U}$ on $\cP_{X}$.  
For  $\bC$  in $\CL$ 
we let $\Sh^{U}_{\bC}(X)$ denote the full subcategory of $\PSh_{\bC}(X)$ 
of $\tau^{U}$-sheaves,
 {which we also call $U$-sheaves.}

\begin{rem} \label{wrgwr2twrgwergw}
 {Let $Y$ be a subset of $X$ and let}
$\cY$ be a $U$-covering family.
Then we consider the associated sieve $S_{\cY}$.
It is the full subcategory of $(\cP_{X})_{/Y}$ consisting of objects $Z\to Y$ such that $Z$ is contained in a member of $\cY$.

An object  $M$ in $\PSh_{\bC}(X)$ is a $U$-sheaf if and only if the canonical morphism \begin{equation}\label{wergwergwegregwegwergergergwegrg}
M(Y)\to  \lim_{(Z\to Y)\in (S_{\cY})^{\op}}  M(Z)
\end{equation}
  is an equivalence
for all $Y$ in $\cP_{X}$ and all $U$-covering families $\cY$ of $Y$.
 {Note that the empty family is a $U$-covering of the empty set, so the sheaf condition forces $M(\emptyset) \simeq 0$.} 
\end{rem}

\begin{ex}
 {We continue \cref{ex:bdd}.
 \begin{enumerate}
  \item A $\diag(X)$-sheaf is the same as a function associating an object of $\bC$ to each point in $X$.
  \item An $(X \times X)$-sheaf is a pointed functor $\cP_X^\op \to \bC$, i.e., a functor sending the empty set to a zero object of $\bC$.
  \item If $X$ carries a metric, being a $V_r$-sheaf is a non-trivial notion. In general, we do not know of a better description than the one provided by \cref{qrgioergrewwergwgregrweg} below.\qedhere
 \end{enumerate}}
\end{ex}

The covering family of  {a non-empty subset} $Y$ consisting of all $U$-bounded subsets refines every other $U$-covering family.
As a consequence, it is easy to construct the $U$-sheafification functor. 

Let $i\colon \cP_{X}^{U\bd}\to \cP_{X}$ be the inclusion of the sub-poset of $\cP_{X}$ of  {non-empty} $U$-bounded elements. Then we have an adjunction
\begin{equation}\label{i-adj}
 i^{*}\colon \PSh_{\bC}(X) {\leftrightarrows} \Fun(\cP_{X}^{U\bd,\op},\bC)\colon i_{*}\ ,
\end{equation}
where $i^{*}$ is the restriction, and $i_{*}$ is the right Kan extension functor along $i$ (which  exists since $\bC$ is complete).
We set
\begin{equation}\label{ludef}
L^{U}:=i_{*}i^{*}\ .
\end{equation}

\begin{lem}\label{qrgioergrewwergwgregrweg}
We have an adjunction  \begin{equation}\label{qwefwihjoiwqejfoiqwefqwefqewfwfq}
L^{U}\colon \PSh_{\bC}(X)\leftrightarrows \Sh^{U}_{\bC}(X)\cocolon \incl\ .  \end{equation}
 Moreover, $L^{U}$ preserves small limits and $\Sh^{U}_{\bC}(X)\in  \CL$.
\end{lem}
\begin{proof}
 {Since $i$ is fully faithful, the counit of the adjunction \eqref{i-adj} is an equivalence $i^{*}i_{*}\simeq \id$. Therefore, it suffices to show that $i_*$ identifies $\Fun(\cP_{X}^{U\bd,\op},\bC)$ with the full subcategory of $U$-sheaves. This in particular implies that $\Sh^{U}_{\bC}(X)\in \CL$.}

 {Let $M \colon \cP_{X}^{U\bd,\op} \to \bC$ be a functor. We must show that $i_*M$ is a $U$-sheaf.
The pointwise formula for Kan extensions implies $i_*M(\emptyset) \simeq 0$, so we only need to consider the evaluation of $i_*M$ on non-empty subsets of $X$.}
Since the covering family by $U$-bounded subsets refines every other $U$-covering family, it suffices to
check the sheaf condition for the sieves associated to this family.
Thus let $Y$ be  {a non-empty subset of $X$.}
The sieve associated to the covering family of all $U$-bounded subsets
is the slice category $(\cP^{U\bd}_{X})_{/Y}$ (\cref{wrgwr2twrgwergw}).
We must show that
\begin{equation}\label{sheaf-desire}
( {i_*M})(Y)\to \lim_{B\in  ((\cP^{U\bd}_{X})_{/Y})^{\op}} ( {i_*M})(B)
 \end{equation}
  is an equivalence\footnote{In order to shorten the notation we  denote the objects of $(\cP^{U\bd}_{X})_{/Y}$ by $B$ instead of $B\to Y$.}.
 Since $\cP^{U\bd}_X$ is a full subcategory of $\cP_X$, we have $(i_*M)(B) \simeq M(B)$ for every non-empty $U$-bounded subset $B$ of $X$. So the pointwise formula for the right Kan extension shows that the morphism \eqref{sheaf-desire} is an equivalence.

We now consider a $U$-sheaf $M$ and show that $M\to L^{U}M$ is an equivalence.
Indeed, the evaluation of this morphism at  {a non-empty subset $Y$ of $X$}
is equivalent to
\[ M(Y)\to \lim_{B\in ((\cP^{U\bd}_{X})_{/Y})^{\op}}M(B) \]
which is an equivalence by the sheaf condition.

Since the functors $i^{*}$ and $i_{*}$ preserve small limits, so does $L^{U}$.
\end{proof}

  Let $\phi\colon \bC\to \bC^{\prime}$ be a morphism in $\CL$. 
  \begin{lem}
  The functor $\hat \phi_{*}$ from \eqref{ererggrefqwfewfeqwfqef} preserves $U$-sheaves, and we have an adjunction \begin{equation}\label{rqeoijqoerbqrgqrgqgqqqrg}
  L^{U}\hat \phi^{*}\colon \Sh^{U}_{\bC'}(X)\leftrightarrows \Sh^{U}_{\bC}(X)\cocolon\hat \phi_{*}\ .
  \end{equation}
  Moreover, $L^{U}\hat \phi^{*}$ preserves small cofiltered  limits.
  \end{lem}
\begin{proof}
Since $\phi$ preserves small limits, it is clear that $\hat \phi_{*}$ preserves $U$-sheaves.
It then follows from \cref{qrgioergrewwergwgregrweg} and the adjunction \eqref{ererggrefqwfewfeqwfqef} that the left adjoint is given by the claimed formula. The last assertion follows from this formula and the  fact that
  $L^{U}$ and $\hat \phi^{*}$ preserve small  {co}filtered limits.
\end{proof}

 Consider a map of sets $f\colon X\to X^{\prime}$.   Let $U^{\prime}$ be an entourage of $X^{\prime}$ such that $f(U)\subseteq U^{\prime}$ (where $f(U)$ abbreviates $(f\times f)(U)$).
\begin{lem}\label{fklrgbwgergwergwrgr}
The functor $\hat f_{*}$ from \eqref{ad-hat-f} sends $U$-sheaves to $U^{\prime}$-sheaves, and we have an adjunction
\begin{equation}\label{ad-hat-f-sh} 
L^{U}\hat f^{*}\colon  \Sh^{U^{\prime}}_{\bC}(X^{\prime})\leftrightarrows  \Sh^{U}_{\bC}(X)\cocolon \hat f_{*}\ .
\end{equation}
Moreover,  $L^{U}\hat f^{*}$ preserves small  limits.
\end{lem}
\begin{proof}
 By assumption, we have the following commutative diagram:
\[\xymatrix{
 \cP^{U\bd}_X\ar[r]^{f(-)}\ar[d] & \cP^{U'\bd}_{X'}\ar[d] \\
 \cP_X\ar[r]^{f(-)} & \cP_{X'}
}\]
The functor $\hat f_{*}$ in \eqref{ad-hat-f}  is given by right Kan extension along $f(-) \colon \cP_X \to \cP_{X'}$.
In the proof of  \cref{qrgioergrewwergwgregrweg}  we have seen that right Kan extension along the vertical maps in the square above produces sheaves. It then follows from the transitivity of right Kan extensions that $\hat f_{*}$ sends $U$-sheaves to $U'$-sheaves.
  
It is now clear that the  left adjoint is given by the claimed formula. 
   The last assertion follows from this formula and the fact that  
  $L^{U}$ and $\hat f^{*}$ preserve small limits.
  \end{proof}

The inverse of an entourage $V$ and the
 composition  of entourages $U,V$ of $X$ are defined by
 \begin{equation}\label{inv-ent-def}
  V^{-1}:=\{(y,x)\in X\times X\:|\: (x,y)\in V\}
  \end{equation}
  and 
 \begin{equation}\label{comp-def-ent}
  UV:=\{(x'',x)\in X\times X\:|\: (\exists x'\in X\:|\: (x'',x')\in U\wedge (x',x)\in V)\}\ .
 \end{equation}
Let $V$ be an entourage of $X$ with $\diag(X)\subseteq V$ and assume that $U^{\prime}$ is an entourage of $X$ such that $VUV^{-1}\subseteq U^{\prime}$.
 \begin{lem}\label{ioiregjoifgrgwergwregwew}
The functor $V_{*}$ from \eqref{ad-hat-V} sends $U $-sheaves to $U^{\prime}$-sheaves, and we have an adjunction \begin{equation}\label{adj-U-shrik}
 L^{U }  V^{*}\colon\Sh_{\bC}^{U^{\prime}}(X)\leftrightarrows\Sh^{U } _{\bC}(X)\cocolon  V_{*}\ .
\end{equation}
Moreover, $L^{U}  V^{*}$ preserves small limits.
 \end{lem}
\begin{proof}
As in the proof of \cref{fklrgbwgergwergwrgr}, the second assertion is a formal consequence of the first.
To prove the first assertion, we observe that the assumptions provide the following commutative diagram:
\[\xymatrix{
 \cP^{U\bd}_X\ar[r]^{V[-]}\ar[d] & \cP^{U'\bd}_{X}\ar[d] \\
 \cP_X\ar[r]^{V[-]} & \cP_{X}
}\]
Since $V_*$ is given by right Kan extension along the thickening functor $V[-]$,  it follows as in the proof of \cref{fklrgbwgergwergwrgr} 
that $V_*$ sends $U$-sheaves to $U'$-sheaves.
\end{proof}

Let $G$ be a group, let $X$ be a $G$-set, and let $\bC$ be in $\Fun( {BG},\CL)$.
Assume that $U$ is a $G$-invariant entourage of $X$,
 i.e., that for every $g$ in $G$ and $(x,y)$ in $U$, we also have $(gx,gy) \in U$.
 We further assume that $U$ contains the diagonal of $X$.
Under this condition, the action of $G$ on $X$ preserves $U$-sheaves by \cref{fklrgbwgergwergwrgr}.
We define
\begin{equation}\label{nwerjkgwegergwger}
\Sh^{U,G}_{\bC}(X):=\lim_{BG} \Sh^{U}_{\bC}(X) \ ,
\end{equation}
where the limit is interpreted in $\CATi$.
{By \cref{qrgioergrewwergwgregrweg}, $\Sh^U_\bC(X)$ is an object of $\CL$, and the same is true for $\Sh^{U,G}_{\bC}(X)$ by \cref{lem:CL-complete}.
As a consequence of \cref{tieqorgfgregwegergw}, $\Sh^{U,G}_{\bC}(X)$ is a full subcategory of $\PSh^{G}_{\bC}(X)$.}

\begin{kor}\label{qoiwejfqowefefqfewfq}
 The adjunction \eqref{qwefwihjoiwqejfoiqwefqwefqewfwfq} induces an adjunction
 \[ L^{U,G}\colon \PSh^{G}_{\bC}(X)\leftrightarrows \Sh^{U,G}_{\bC}(X)\cocolon \incl \ . \]
 Moreover, $L^{U,G}$ preserves small limits.
\end{kor}

Let $\phi\colon  \bC\to \bC^{\prime}$ be a morphism in $\Fun(BG,\CL)$, and let $U$ be an invariant entourage of $X$.
\begin{kor}\label{ergioegrergwergwregwergwegrwr}
The adjunction \eqref{rqeoijqoerbqrgqrgqgqqqrg} induces an adjunction 
\begin{equation}\label{rbhwtibuhwibwetbweb}
L^{U,G}\hat \phi^{*,G}\colon \Sh^{U,G}_{\bC^{\prime}}(X)\leftrightarrows \Sh^{U,G}_{\bC }(X)\cocolon \hat \phi^{G}_{*}\ .
\end{equation}
Moreover, $L^{U,G}\hat \phi^{*,G} $ preserves small cofiltered limits.
\end{kor}

Let $V$ be an invariant entourage of $X$ such that $\diag(X)\subseteq V$.
Then $V_{*}$ and $V^{*}$ in \eqref{adj-U-shrik} are equivariant.
Assume that $U^{\prime}$ is an invariant entourage of $X$ which in addition satisfies $VUV^{-1}\subseteq U ^{\prime}$.
\begin{kor} \label{equi-sheaffff}
The adjunction  \eqref{adj-U-shrik}
induces an adjunction  
\[ L^{ U ,G}  V^{*,G} \colon\Sh_{\bC}^{U^{\prime},G}(X) \leftrightarrows \Sh^{U,G} _{\bC}(X)\cocolon  V^{G}_{*}\ .\]
Moreover, $ L^{U^{\prime},G}V^{G}_{*}$ preserves small  limits.
\end{kor}

Assume that $f\colon X\to X^{\prime}$ is a map between $G$-sets, and that $U^{\prime}$ is an invariant entourage of $X^{\prime}$ such that $f(U)\subseteq U^{\prime}$.
\begin{kor} \label{iuqhfiuewvffewfqfewf}
The adjunction \eqref{ad-hat-f-sh} induces an adjunction 
\begin{equation}\label{ad-hat-f-sh-G} 
L^{U,G}\hat f^{*,G}\colon  \Sh^{G,U^{\prime}}_{\bC}(X^{\prime})\leftrightarrows  \Sh^{G,U}_{\bC}(X)\cocolon \hat f_{*}^{G}
\end{equation}
Moreover, $L^{U,G}\hat f^{*,G}$ preserves limits.
\end{kor}
  
Let $X$ be a $G$-set
and $i\colon Y\to X$ be the inclusion of an invariant subset. Let $U$ be an invariant entourage of $X$ containing the diagonal and set $U_{Y}:=(Y\times Y)\cap U$.

\begin{lem}\label{sub-adj}
We have an adjunction \begin{equation}\label{wefvwlijnkjnvernvwievewrvewrvewrv}
\hat i^{*,G}\colon \Sh_{\bC}^{U,G}(X)\leftrightarrows \Sh_{\bC}^{U_{Y},G}(Y) \cocolon \hat i^{G}_{*}
\end{equation}
and the relation $\hat i^{*,G}\hat i^{G}_{*}\simeq \id$.
\end{lem}
\begin{proof}
We have $i(U_{Y})\subseteq U$ and can therefore apply 
 \cref{iuqhfiuewvffewfqfewf} to $i$ in place of $f$ and $U_{Y}$ in place of $U'$. In order to remove the application of the sheafification functor $L^{U,G}$ on the left adjoint side
it then suffices to observe that $\hat i^{*,G}$ obviously sends $U$-sheaves to $U_{Y}$-sheaves.
Since $i^{-1}(-)\circ i(-)=\id$ on $ \cP_{Y}$   we get the relation $\hat i^{*}\hat i_{*}\simeq\id$ on $\PSh_{\bC}$ which induces the desired relation $\hat i^{*,G}\hat i^{G}_{*}\simeq \id$ by applying $\lim_{BG}$ and restricting to sheaves.
 \end{proof}

Let $X$ be a $G$-set
with invariant subsets $Y$ and $Z$ such that $Y\cup Z=X$. Then we have the following inclusions:
\[\xymatrix{
	Y\cap Z\ar[r]^-{l}\ar[d]_{m}\ar[dr]^{k} & Y\ar[d]^{j} \\
	Z\ar[r]^-{i} & X
}\]
Let $U$ be an invariant entourage of $X$ containing the diagonal and consider $M$ in $\Sh^{U,G}_{\bC}(X)$. In the square below, the morphisms are the units of the adjunctions $(\hat i^{*,G},\hat i_{*}^{G})$  etc.
 
\begin{lem}[Glueing Lemma]\label{wgkwkgrewrgrg}If $(Y,Z)$ is a $U$-covering family of $X$, then we have a cartesian 
square
\begin{equation}\label{glue-sq}\xymatrix{
	M\ar[r]\ar[d] & \hat j^{G}_{*}\hat j^{*,G}M\ar[d]\\\hat i^{G}_{*}\hat i^{*,G}M\ar[r] & \hat k^{G}_{*}\hat k^{*,G}M
}\end{equation} 
 in $\Sh^{U,G}_{\bC}(X)$.
 {Moreover, the base change transformation $\hat i^{*,G} \hat j_*^G \to \hat m_*^G \hat l^{*,G}$ is an equivalence, so there is an induced equivalence  }
 \begin{equation}\label{qergqffewfqef}
  \hat k^{G}_{*}\hat k^{*,G} \simeq \hat  i^{G}_{*}\hat i^{*,G} \hat j^{G}_{*}\hat j^{*,G}\ .\
 \end{equation}
\end{lem}
\begin{proof}
The filler of the square is obtained from the equality $im=k=jl$.
The evaluation $\Sh^{U,G}_{\bC}(X)\to \Sh^{U}_{\bC}(X)$ detects cartesian squares by \cref{rgeoiergegegregvsdvsdvvdfverrf3erfer}. So it suffices to check
the assertion in the non-equivariant case. 

Since  \eqref{glue-sq}  is a square of $U$-sheaves  (this follows from  \cref{sub-adj})  it suffices to check that the evaluation of the square  \eqref{glue-sq}  at every $U$-bounded  $B$ in $\cP_{X}$ is cartesian.  Since $(Y,Z)$ is a  $U$-covering, $B$ is contained in   one of $Y$ or $Z$. We consider the case that $B\subseteq Y$ (the case $B\subseteq Z$ is analoguous).
The evaluation of \eqref{glue-sq} at $B$ is the square in $\bC$
 \begin{equation*}\label{glue-sq1}\xymatrix{
 	M(B)\ar[r]^{\simeq}\ar[d] & M(B)\ar[d] \\
 	M(Z\cap B)\ar[r]^{\simeq} & M(Z\cap B)}
 \end{equation*} 
 which is obviously cartesian.
 
 {The base change transformation is the composite
 \[ \hat i^{*,G}\hat j_*^G \to \hat i^{*,G}\hat j_*^G \hat l_*^G \hat l^{*,G} \simeq \hat i^{*,G}\hat i_*^G \hat m_*^G \hat l^{*,G} \to  \hat  m_*^G \hat l^{*,G} \]
 arising from the respective unit and counit. The counit $\hat i^{*,G}\hat i_*^G \to \id$ is an equivalence by \cref{sub-adj}. The map induced by the unit morphism is itself induced by the canonical inclusion
 \[ j^{-1}(i(B)) \subseteq l(l^{-1}(j^{-1}(i(B)))) \]
 for every $B \subseteq Y$. It is straightforward to check that these sets are equal, so the base change transformation is an equivalence. In particular, we obtain the induced equivalence
 \[ \hat k^{G}_{*}\hat k^{*,G} \simeq \hat  i^{G}_{*} \hat m_*^{G} \hat l^{*,G} \hat j^{*,G} \xleftarrow{\simeq} \hat  i^{G}_{*}\hat i^{*,G} \hat j^{G}_{*}\hat j^{*,G}\ .\qedhere\]}
\end{proof}

\subsection{Sheaves on \texorpdfstring{$G\Coarse$}{GCoarse}}\label{qerogijqiowgrfqreg}
In \cref{rhoijrthhgrewerggwerg}, we considered $\bC$-valued sheaves as a functor on pairs of a $G$-set equipped with an invariant entourage. In the present section, we get rid of the explicit choice of an entourage by equipping the $G$-set with a whole collection of such entourages called a coarse structure and considering the union of the sheaf categories  for all these coarse entourages. In this way we eventually obtain a functor of $\bC$-valued sheaves on the category $G\Coarse$ of $G$-coarse spaces.
We also discuss how the various adjunctions from the preceding section descend to the categories of sheaves.
To provide a context in which restriction/transfer functors still make sense, we introduce the notion of a coarse covering in \cref{wefgihjwiegwergrwrg} and prove a base change formula for coarse coverings in \cref{rgkoqregqregqqef}.

Let $X$ be
 {a $G$-set}.
\begin{ddd}\label{trbertheheht}
A $G$-coarse structure on $X$ is a subset $\cC_{X}$ of $\cP_{X\times X}$ satisfying the following conditions:
\begin{enumerate}
\item $\cC_{X}$ is $G$-invariant.
\item $\diag(X)\in \cC_{X}$.
\item \label{qerighioergergwgergwergwergwerg}$\cC$ is closed under forming subsets, finite unions, inverses  {(see} \eqref{inv-ent-def}) and compositions  {(see} \eqref{comp-def-ent}).
\item \label{igwoegwergergwrgrg}The sub-poset of $G$-invariants $\cC_{X}^{G}$ is cofinal in $\cC_{X}$. \qedhere
\end{enumerate}
 \end{ddd}
 The elements of $\cC_{X}$ will be called coarse entourages of $X$.
 We will often use the notation $\cC_{X}^{G,\Delta}$ for the sub-poset of  $\cC_{X}$ of invariant coarse entourages containing the diagonal.
 
 \begin{ddd}\label{thiowhwfgwrgwergwreg}
A $G$-coarse space is a pair $(X,\cC_{X})$ (usually denoted just by $X$) of  a $G$-set with a $G$-coarse structure.
\end{ddd}
The idea of a  coarse space has been introduced by {John} Roe, see e.g.~\cite[Ch.~2]{roe_lectures_coarse_geometry}.
\begin{ex}\label{ergiowghergregwreg}
 {We continue \cref{ex:ent}, assuming in addition that $X$ is $G$-set.
\begin{enumerate}
 \item The minimal coarse structure on $X$ contains precisely the subsets of $\diag(X)$. We denote the associated $G$-coarse space by $X_{min}$.
 \item The maximal coarse structure contains all entourages on $X$. We denote the associated $G$-coarse space by $X_{max}$.
 \item If $X$ carries a metric $d$ and the group $G$ acts by isometries on $X$, then the metric coarse structure on $X$ is given by
 \[ \cC_d := \{ U \subseteq X \times X \mid \exists\,r\geq 0\colon U \subseteq V_r \}\]
 with $V_{r}$ as in \eqref{qwefwokjkqowdewdeqd}.\qedhere
\end{enumerate}}
\end{ex}

\begin{ex}\label{ex:Gcan}
 {Let $G$ be a group. The canonical coarse structure on $G$ is the smallest $G$-coarse structure on $G$ containing all sets of the form $B \times B$, where $B$ is a finite subset of $G$.
 If $G$ is countable, this coarse structure coincides with the metric coarse structure induced by any choice of $G$-invariant proper metric on $G$.
 We denote the associated $G$-coarse space by $G_{can}$.}
\end{ex}

Consider two $G$-coarse spaces
$(X,\cC_{X})$ and $(X^{\prime},\cC_{X^{\prime}} )$ and an equivariant map between the underlying $G$-sets
$f \colon X\to X^{\prime}$.

\begin{ddd}
	The map $f$ is controlled if $f(\cC_{X})\subseteq \cC_{X^{\prime}}$.
\end{ddd}

The category $G\Coarse$ of $G$-coarse spaces and controlled maps is complete and cocomplete \cite[Prop.~2.18 and 2.21]{equicoarse}, and the forgetful functor to $G\Set$ preserves limits and colimits since it has a left adjoint $X\mapsto X_{min}$ and a right adjoint $X\mapsto X_{max}$.

Using precomposition with the forgetful functor $G\Coarse\to G\Set$, the functor $\PSh^G$ from \eqref{equiv-psh} induces a functor (denoted by the same symbol) \begin{equation}\label{ergwerggregwergregwgwegw}
\PSh^{G}\colon G\Coarse\times \Fun(BG,\CL)\to \CL\ .
\end{equation}
Let $X$ be  in $G\Coarse$.
 Consider the invariant (not necessarily coarse) entourage 
\begin{equation}\label{pinull}
U(\pi_{0}(X)):=\bigcup_{U\in \cC_{X}} U\ .
\end{equation}
It is an invariant equivalence relation on $X$.
\begin{ddd}\label{wetghiojweorgergwgwgegrwrg423t2} The 
 $G$-set of equivalence classes $\pi_{0}(X)$ with respect to $U(\pi_{0}(X))$ is called the set of coarse components of $X$.
\end{ddd} 

\begin{ex}
  \ \begin{enumerate}
\item For a $G$-set $X$ we have $\pi_{0}(X_{min})\cong X$ and $\pi_{0}(X_{max})\cong *$.
\item  {If $X$ carries a metric $d$, we have $\pi_0(X_d) \cong *$. If we allow metrics to take the value $\infty$, two points in $X$ lie in the same coarse component if and only if they are at a finite distance from each other.
\item Similarly, the canonical coarse structure on a group also has only one coarse component.}\qedhere
 \end{enumerate}
\end{ex}

\begin{rem}\label{t4hgiorthgetrhtrhtheht}
 If $Y$ {is a subset of of a $G$-coarse space $X$}, then
 we can consider $Y$ with the induced   coarse structure, and then take $\pi_{0}(Y)$ in the sense of \cref{wetghiojweorgergwgwgegrwrg423t2} for the trivial group.
 Alternatively, we can consider the subset \begin{equation*}\label{rvqrfwewfqewfqfe}
  \{Z\in \pi_{0}(X) \mid Z\cap Y\not=\emptyset\}
 \end{equation*}
 of $\pi_{0}(X)$. Both constructions give canonically isomorphic sets.
 The latter description shows that $\pi_{0}(Y)$ is a $G$-invariant subset of $\pi_{0}(X)$ if $Y$ is a $G$-invariant subset.
\end{rem}

 For $\bC$ in $\Fun(BG,\CL)$
we will abbreviate $\Sh^{U(\pi_{0}(X)),G}_{\bC}(X) $ (see \eqref{nwerjkgwegergwger}) by $\Sh^{\pi_{0},G}_{\bC}(X)$ and $L^{U(\pi_{0}(X)),G}$ (see \cref{qoiwejfqowefefqfewfq}) by
$L^{\pi_{0},G}$. 

If $f\colon X\to X^{\prime}$ is a morphism  {of $G$-coarse spaces},
then $f(U(\pi_{0}(X))\subseteq U(\pi_{0}(X^{\prime}))$.
As a consequence of \cref{iuqhfiuewvffewfqfewf}, we get:

\begin{kor}\label{erguiwergrwgwrgwgwr}
We have an adjunction 
\begin{equation}\label{pi0adj}
L^{\pi_{0},G}\hat f^{*,G}\colon \Sh^{\pi_{0},G}_{\bC}(X^{\prime}) \leftrightarrows  \Sh_{\bC} ^{\pi_{0},G}(X) \cocolon\hat f_{*}^{G} \end{equation}
Moreover, $L^{\pi_{0},G}\hat f^{*,G}$ preserves small limits.
\end{kor}

Moreover,  using in addition the  existence of the right adjoints in   \eqref{rbhwtibuhwibwetbweb}, we obtain a
subfunctor
\begin{equation}\label{wergpok2pgegr5gwegwefv}
\Sh^{\pi_{0},G}\colon G\Coarse\times \Fun(BG,\CL)\to \CL
\end{equation}
of the functor $\PSh^G$ from \eqref{ergwerggregwergregwgwegw}.  

If $U$ and $U^{\prime}$  are invariant entourages of $X$
such that $U\subseteq U^{\prime}$, then applying \cref{equi-sheaffff} for $V=\diag(X)$ we get an inclusion
$\Sh^{U,G}_{\bC}(X)\to \Sh^{U^{\prime},G}_{\bC}(X)$.
 We define the $\infty$-category \begin{equation}\label{qewfqwefjhb1ji234rwefqewff}
\Sh_{\bC}^{G}(X):=\colim_{U\in \cC^{G}_{X}} \Sh^{U,G}_{\bC}(X)\ .
\end{equation}
The filtered colimit is interpreted in  $\CAT_{\infty}$.
 The objects of $\Sh_{\bC}^{G}(X)$ are called sheaves.
 Since  this $\infty$-category still admits 
 finite limits, it is actually
an object of $\CLL$ (\cref{wfqwoifjqiofeqfe32rqewfqewf}). 

As a consequence of \cref{iuqhfiuewvffewfqfewf} and  \cref{ergioegrergwergwregwergwegrwr}, 
we get a subfunctor \begin{equation}\label{regewgk2p5getgwreg}
\Sh^{G}\colon G\Coarse\times \Fun(BG,\CL)\to \CLL
\end{equation}
 of the functor $\Sh^{\pi_{0},G}$ from \eqref{wergpok2pgegr5gwegwefv}. 
 Note the difference in the targets in \eqref{wergpok2pgegr5gwegwefv} and \eqref{regewgk2p5getgwreg}.
  
 In general, we do not expect that for a morphism $f \colon X\to X^{\prime}$  {of $G$-coarse spaces}
 the morphism $\hat f^{G}_{*} \colon \Sh^{G}_{\bC}(X)\to  \Sh^{G}_{\bC}(X^{\prime})$ has a left adjoint. The reason is that the left adjoint in the adjunction \eqref{ad-hat-f-sh-G} explicitly depends on the entourage $U$. But we have such left adjoints  for a special sort of morphism in $G\Coarse$ called coarse coverings. 
 
Let  $f \colon X\to X^{\prime}$ be a  morphism  in $G\Coarse$.
\begin{ddd}\label{wefgihjwiegwergrwrg} The morphism 
$f$ is called a coarse covering if it satisfies the following conditions:
\begin{enumerate}
\item \label{wrkojgwergrefwrfrf} The restriction  $f_{|Y} \colon Y\to f(Y)$ to every coarse component $Y$ of $X$ is an isomorphism
 {of coarse spaces, and $f(Y)$ is a coarse component of $X'$.}
\item \label{tegwegergweg}The $G$-coarse  structure of $X$ is generated by the entourages $f^{-1}(U')\cap U(\pi_{0}(X))$ for all $U^{\prime}$ in $\cC_{X^{\prime}}$. \qedhere
\end{enumerate}
\end{ddd}

\begin{ex}
 {If $Y$ is a collection of coarse components of $X$, then the inclusion of $Y$ into $X$ is a coarse covering.}
\end{ex}

\begin{ex}\label{ergwegegegeggw}
If $W$ is a $G$-set, then the projection $W_{min}\otimes X\to X$ is a coarse covering. Here $\otimes$ is the cartesian product in $G\Coarse$ and $W_{min}$ is as in \cref{ergiowghergregwreg}.
\end{ex}

Since every coarse entourage  $U$ of $X$ is contained in  $ U(\pi_{0}(X))$ from \eqref{pinull}, we have an inclusion $\Sh^{U,G}_{\bC} (X)\subseteq \Sh_{\bC}^{\pi_{0},G}(X)$.
This explains the meaning of the word ``restricts'' in the following statement.
\begin{lem}\label{unex-left}
If $f \colon X\to X^{\prime}$ is a coarse covering, then the adjunction \eqref{pi0adj}
restricts to an adjunction
\begin{equation}\label{hfjrefewrfwrefwergf}
L^{\pi_{0},G}\hat f^{*,G}\colon \Sh_{\bC}^{G}(X')\leftrightarrows \Sh_{\bC}^{G}(X) \cocolon \hat f^{G}_{*}\ .
\end{equation}
Moreover, the left adjoint $L^{\pi_{0},G}\hat f^{*,G}$ preserves finite limits.
 \end{lem}
\begin{proof}
First we observe that the non-equivariant case implies the equivariant case by passing to the limit over $BG$. 
It then  suffices to show that $L^{\pi_{0}}\hat f^{*}$ preserves sheaves. 
Let $U'$ be in $\cC_{X'}$ and $M$ be in $\Sh_{\bC}^{U'}(X^{\prime})$. We will show that $ L^{\pi_{0}}\hat f^{*}M\in \Sh_{\bC}^{U}(X)$ for $U:=f^{-1}(U')\cap U(\pi_{0}(X))$ which is a coarse entourage of $X$ by   \cref{wefgihjwiegwergrwrg} \eqref{tegwegergweg}.  It suffices to show that $ (L^{\pi_{0}}\hat f^{*}M)_{|Y}\in \Sh_{\bC}^{U_{Y}}(Y)$ for every coarse component $Y$ of $X$,
where $U_{Y}:=U\cap (Y\times Y)$. We first note that  $(L^{\pi_{0}} \hat f^{*}M)_{|Y}\simeq (\hat f^{*}M)_{|Y}$  by the formula \eqref{ludef} for $L^{\pi_{0}}$. Since $f$ restricts to isomorphisms between coarse components (as coarse spaces), the pullback is an equivalence $\hat f^{*}_{|Y}\colon\Sh^{f(U_{Y})}_{\bC}(f(Y))\xrightarrow{\simeq} \Sh_{\bC}^{U_{Y}}(Y)$.  Since 
 $ (\hat f^{*}M)_{|Y} \simeq \hat f^{*}_{|Y} (M_{|f(Y)})$ and $f(U_{Y})=U'\cap (f(Y)\times f(Y))$, we have $M_{|f(Y)}\in \Sh^{f(U_{Y})}_{\bC}(f(Y))$. We conclude that $ (\hat f^{*}M)_{|Y}\in \Sh_{\bC}^{U_{Y}}(Y)$.
 
 The second assertion is an immediate consequence of the second assertion in \cref{erguiwergrwgwrgwgwr}.
\end{proof}

We now consider a pullback square \begin{equation}\label{coarse-pull-back}
\xymatrix{Y\ar[d]_{g}\ar[r]^{f^{\prime}}&Y^{\prime}\ar[d]^{g^{\prime}}\\X\ar[r]^{f}&X^{\prime}}
\end{equation}
in $G\Coarse$.

\begin{lem}\label{rgkoqregqregqqef}
The canonical morphism of functors
\[ L^{\pi_{0},G}\hat g^{\prime,*,G} \hat f^{G}_{*}\to \hat f^{\prime,G}_{*}L^{\pi_{0},G}\hat g^{*,G}\colon \Sh^{\pi_{0},G}_{\bC}(X)\to \Sh^{\pi_{0},G}_{\bC}(Y^{\prime}) \]
is an equivalence.
\end{lem}
\begin{proof}
The non-equivariant case implies the equivariant case by passing to the limit over $BG$.

Using that the underlying square of sets is cartesian, 
we have for  $Z^{\prime}$ in $\cP_{Y^{\prime}}$ the relation $f^{-1}(g^{\prime}(Z') )=g(f^{\prime,-1}(Z'))$. This immediately implies that the canonical morphism
 \[ \hat g^{\prime,*} \hat f_{*}\to \hat f^{\prime}_{*} \hat g^{*}\colon \PSh_{\bC}(X)\to \PSh_{\bC}(Y^{\prime}) \]
is an equivalence. The morphism in question is given by
\[ L^{\pi_{0}}\hat g^{\prime,*} \hat f_{*}\simeq L^{\pi_{0}} \hat f^{\prime}_{*} \hat g^{*}\to L^{\pi_{0}} \hat f^{\prime}_{*} L^{\pi_{0}}\hat g^{*}\xleftarrow{\simeq} \hat f^{\prime}_{*} L^{\pi_{0}}\hat g^{*}\ ,\]
where for the last equivalence we employ the fact that $\hat f^{\prime}_{*}$ preserves $\pi_{0}$-sheaves by \cref{fklrgbwgergwergwrgr}.
So it remains to show that $L^{\pi_{0}} \hat f^{\prime}_{*} \hat g^{*}M \to L^{\pi_{0}} \hat f^{\prime}_{*} L^{\pi_{0}}\hat g^{*}M$ is an equivalence for every  
$M$ in $\Sh^{\pi_{0}}_{\bC}(Y)$.
 Using formula \eqref{ludef} for $L^{\pi_{0}}$, its evaluation on a subset $Z'$ in $\cP_{Y^{\prime}}$ is given by the morphism
\begin{equation}\label{vergl-canmap}
\prod_{D' \in \pi_{0}(Z')} M(g(f^{\prime,-1}(D'))) \to \prod_{D'\in \pi_{0}(Z')} \prod_{D\in \pi_{0}(f^{\prime,-1}(D'))} M(g(D)))\ ,
\end{equation}
which in the factor $D'$ is induced by the restrictions along the embeddings $g(D)\to g(f^{\prime,-1}(D'))$ for all $D$ in $ \pi_{0}(f^{\prime,-1}(D'))$. Using that $M$ is a $\pi_{0}$-sheaf, we can rewrite the domain  of the map in the form  
\begin{equation}\label{vergl-canmap1}
\prod_{D' \in \pi_{0}(Z')} \prod_{C \in \pi_{0}(g(f^{\prime,-1}(D')))} M(C)\to  \prod_{D' \in \pi_{0}(Z')}  \prod_{D \in \pi_{0}(f^{\prime,-1}(D'))} M(g(D))\ .
\end{equation}
 We now argue that for fixed $D'$ in $\pi_0(Z')$  the map
\begin{equation}\label{weroijowegrgewregw}
 \pi_0(f^{\prime,-1}(D')) \to  \pi_{0}(g(f^{\prime,-1}(D'))),\quad D \mapsto g(D)
\end{equation} 
is a well-defined bijection. Let us first show that $g(D)$ is a coarse component of $g(f^{\prime,-1}(D'))$ whenever $D$ is a coarse component of $f^{\prime,-1}(D')$. Suppose $d$ is a point in $D$ and $x$ is a point in $g(f^{\prime,-1}(D'))$ such that $\{(g(d),x)\} \in \cC_X$. Then $x = g(y)$ for some point $y$ in $f^{\prime,-1}(D')$. Since $D'$ is coarsely connected  we have $\{(f^\prime(d),f^\prime(y))\} \in \cC_{Y'}$. We conclude that $\{(d,y)\} \in \cC_Y$ by the characterisation of the entourages in a pullback in $G\Coarse$. Hence $y\in D$ and thus $x\in g(D)$.

 Since the map \eqref{weroijowegrgewregw} is surjective by definition, we only have to check injectivity.
Let $D_0$ and $D_1$ be in $\pi_0(f^{\prime,-1}(D'))$ such that $g(D_0) = g(D_1)$. Then there are points $z_0$ in $D_0$ and $z_1$ in $D_1$ such that $\{(g(z_0),g(z_1))\} \in \cC_X$. Since we also have $\{(f^{\prime}(z_{0}),f^{\prime}(z_{1}))\}\in \cC_{Y^{\prime}}$, it follows that $\{(z_{0},z_{1})\}\in \cC_{Y}$, again by the characterisation of entourages in a pullback in $G\Coarse$.

 This implies that the morphism \eqref{vergl-canmap1} is an equivalence.
 \end{proof}
 
Let $i \colon Y\to X$ be the inclusion of a subspace  in $G\Coarse$. As a consequence of \cref{sub-adj} we get:

 \begin{kor}\label{rqkjgqregergwgwregweg}
 We have an adjunction
 \begin{equation}\label{fvsnejkvwevrevwv}
\hat i^{*,G}\colon \Sh^{G}_{\bC}(X)\leftrightarrows \Sh^{G}_{\bC}(Y)\cocolon \hat i^{G}_{*}
\end{equation}
 and the relation  $\hat i^{G,*}\hat i^{G}_{*}\simeq \id$.
 \end{kor}

We again consider the pullback square \eqref{coarse-pull-back}. If   
$g^{\prime}$ is a coarse covering, then  $g$ is a coarse covering, too  {(see \cite[Lem.~2.11]{coarsetrans})}. Similarly, if $g^{\prime}$ is the inclusion of a subspace, then so is $g$.

 \begin{kor} \label{ropgkpwoegrewgwregwgregw} If $g^{\prime}$ is a coarse covering or   an inclusion of a subspace, then
the canonical morphism of functors
\[ L^{\pi_{0},G}\hat g^{\prime,*,G} \hat f^{G}_{*}\to \hat f^{\prime,G}_{*}L^{\pi_{0},G}\hat g^{*,G} \colon \Sh^{ G}_{\bC}(X)\to \Sh^{ G}_{\bC}(Y^{\prime}) \]
is an equivalence.
\end{kor}
\begin{proof}
If $g$ is a coarse covering, then we use \cref{rgkoqregqregqqef} together with \cref{unex-left}.
If $g$ is an inclusion, then we use \cref{rqkjgqregergwgwregweg} and the fact that we can drop the application of $L^{\pi_{0}}$.
\end{proof}

Let $X$ be a $G$-coarse space,
and let $V$ be a $G$-invariant coarse entourage on $X$ containing the diagonal.
Note that {for every coarse entourage $U$ of $X$}, we also have $VUV^{-1}\in \cC_{X}$ by \cref{trbertheheht}~\eqref{qerighioergergwgergwergwergwerg}.
As a consequence of \cref{equi-sheaffff}, we get:

 \begin{kor}\label{egiweogwergrgwgwerg}
 The functor $ V_{*}^{G}$ from \eqref{v-adj-ggg} restricts to an endofunctor of
 $\Sh^{G}_{\bC}(X)$.
 \end{kor}
  
 Let $f \colon X\to X^{\prime}$ be a morphism in $G\Coarse$ and 
 let $V'$ be a $G$-invariant coarse entourage on $X'$ containing the diagonal.
Then we define the entourage $V :=f^{-1}(V^{\prime})\cap U(\pi_{0}(X))$ {on}  $X$  (see \eqref{pinull}).
If $f$ is a coarse covering, then $V$  {is also a $G$-invariant coarse entourage containing the diagonal}.

\begin{lem}\label{wergiowggwergwergreg}
If $f$ is a coarse covering, then we have an equivalence of functors
\[ L^{\pi_{0},G}\hat f^{*,G}   V^{\prime,G}_{*}\simeq  V^{G}_{*}  L^{\pi_{0},G}\hat f^{*,G}\colon \Sh^{G}_{\bC}(X^{\prime})\to  \Sh^{G}_{\bC}(X )\ .\]
\end{lem}
\begin{proof}
 Using that $f$ is a coarse covering, we see that
 for every {subset $Y$ of $X$}
 we have $f(V (Y))=V^{\prime}(f(Y))$. 
  Furthermore, using the explicit formulas, one checks that $L^{\pi_{0},G}$ and $V^{G}_{*}$ commute.
  This implies the chain of equivalences
\[L^{\pi_{0},G}\hat f^{*,G}  V^{\prime,G}_{*} \simeq L^{\pi_{0},G}  V^{ G}_{*}\hat f^{*,G} \simeq    V^{ G}_{*}L^{\pi_{0},G}  \hat f^{*,G}\ .\qedhere\]
\end{proof}

 Let $i \colon Y\to X$ be an inclusion of a subspace in $G\Coarse$, and let $V$ be  {a $G$-invariant coarse entourage on $X$ containing the diagonal}.
 Then we let $j \colon V(Y)\to X$ be the inclusion and set $V_{Y}:=V\cap(Y\times Y)$.
 
 \begin{lem}\label{egiwetthtrehtrhethetheth}
 We have the relation
\begin{equation}\label{egiwetthtrehtrhethethethrgwergwregwe}
\hat i^{*,G}V^{G}_{*}\simeq V_{Y,*}^{G}\hat i^{*,G} \hat j^{G}_{*}\hat j^{*,G} \colon \Sh^{G}_{\bC}(X)\to\Sh^{G}_{\bC}(Y)\ .
\end{equation}
 \end{lem}
\begin{proof}
 The non-equivariant case implies the equivariant case by passing to the limit over $BG$.
 For every {subset $Z$ of $Y$}
 we have the relation
 $V(i(Z))=i(V_{Y}(Z) )\cap V(Y)$. This implies the desired relation between operations on presheaves and therefore on sheaves. 
\end{proof}

\subsection{Equivariantly small sheaves on \texorpdfstring{$G\BC$}{$G$-bornological coarse spaces}}\label{regioergrgfewgwgregrwg}
 The category $\Sh^G_{\bC}(X)$ is too large to have interesting $K$-theoretic invariants since it always admits Eilenberg swindles. To remedy this defect, we introduce the concept of equivariantly small sheaves.
We equip the coarse spaces with an additional structure, a collection of subsets (called bounded subsets) satisfying the axioms of a bornology.  The  naive idea would then be to require that  a $\bC$-valued sheaf is equivariantly small if its values on bounded subsets belong to $\bC^{\omega}$, the full subcategory of cocompact objects in $\bC$. This condition {leads} to a well-behaved coarse homology theory, but this theory does not take the ``correct'' values; more specifically, \cref{prop:coeffs-sets} would fail for this theory. However, a sufficiently equivariant version of this condition turns out to behave exactly as desired.

Let $X$ be  {a $G$-set}.
\begin{ddd}\label{wthoiwhthwgreggwregwgr}
A $G$-bornology on $X$ is a subset $\cB_{X}$ of $\cP_{X}$ satisfying the following properties:
\begin{enumerate}
\item $\cB_{X}$ is $G$-invariant.
\item $\cB_{X}$ is closed under forming subsets and finite unions.
\item\label{qergoijiogrgqgreg} $\bigcup_{B\in \cB_{X}}B=X$. \qedhere
\end{enumerate}
\end{ddd}
 
A $G$-bornological space is a pair $(X,\cB_{X})$ of a $G$-set $X$ with a $G$-bornological structure $\cB_{X}$. 
The elements of $\cB_{X}$ will be called the bounded subsets of $X$.

\begin{ex}\label{ex:born}
 {\ \begin{enumerate}
  \item On every $G$-set $X$, there are both the minimal $G$-bornology, which contains only the finite subsets of $X$, and the maximal $G$-bornology, which contains all subsets of $X$.
  \item If $d$ is a metric on $X$ and $G$ acts by isometries, the collection of metrically bounded subsets defines a $G$-bornology on $X$. \qedhere
 \end{enumerate}}
\end{ex}
   
 Consider $G$-bornological spaces $(X,\cB_{X})$ and $(X^{\prime},\cB_{X^{\prime}})$ and a  {$G$-equivariant map between the underlying $G$-sets $f\colon X\to X^{\prime}$}.

\begin{ddd}\label{rgiojrgofdewqfe}\mbox{}
\begin{enumerate}
\item $f$ is proper if $f^{-1}(\cB_{X^{\prime}})\subseteq \cB_{X}$.
\item\label{rigojowergrewgwregwergwerwg} $f$ is bornological if $f(\cB_{X})\subseteq \cB_{X^{\prime}}$. \qedhere
\end{enumerate}
\end{ddd}

We denote by $G\Born$ the category of $G$-bornological spaces and proper maps.

\begin{ex}
 {We continue \cref{ex:born}. Let $f \colon X \to X'$ be a   morphism   of $G$-sets. 
\begin{enumerate}
 \item With respect to the minimal bornology on $X$ and $X'$, $f$ is bornological, but it is only proper if it is finite-to-one.
 \item At the other extreme, $f$ is proper and bornological with respect to the maximal bornologies.
 \item If both $X$ and $X'$ carry a metric, $f$ is proper with respect to the metric bornology on $X$ and $X'$ if and only if preimages of subsets with finite diameter also have finite diameter.
 Analogously, it is bornological precisely if images of subsets with finite diameter also have finite diameter.\qedhere 
\end{enumerate}}
\end{ex}

Let $X$ be a $G$-set with a $G$-coarse structure $\cC_{X}$ and a $G$-bornology $\cB_{X}$. 
\begin{ddd}\label{rgejqieogjrgoij1o4trqq}
$\cC_{X}$ and $\cB_{X}$ are compatible if for every $B$ in $\cB_{X}$ and $V$ in $\cC_{X}$ also $V[B]\in \cB_{X}$ (see \eqref{V-thick}).
\end{ddd}

 \begin{rem}
 {If $\cC_X$ and $\cB_X$ are compatible and $U$ is a coarse entourage of $X$, then every $U$-bounded subset $B$ of $X$ is bounded in the sense of the bornology: by definition, $B$ is contained in the $U$-thickening $U[\{x\}]$ for any $x$ in $B$.}
 \end{rem}
 
\begin{ddd}\label{gjwerogijwoergwergwergweg}A $G$-bornological coarse space is a triple $(X,\cC_{X},\cB_{X})$ (usually abbreviated by the symbol $X$) of a $G$-set $X$ with a $G$-coarse structure $\cC_{X}$ and a $G$-bornological structure $\cB_{X}$ such that $\cC_{X}$ and $\cB_{X}$ are compatible.
A morphism between $G$-bornological coarse spaces is a morphism of the underlying $G$-coarse spaces which is in addition proper.
\end{ddd}

In this way we get a category $G\BC$ of $G$-bornological coarse spaces. It comes with a forgetful functor \begin{equation}\label{gwegljgoiregwrgergregewg}
G\BC\to G\Coarse\ .
\end{equation}

\begin{ex}\label{qrgioqjrgoqrqfewfeqfqewfe} 
 For  {a $G$-set} $S$
 we can consider the objects $S_{max,max}$, $S_{min,max}$ and $S_{min,min}$ in $G\BC$, where
the first index $min$ or $max$ refers to the minimal or maximal $G$-coarse structure (\cref{ergiowghergregwreg}),
and the second $min$ or $max$ refers to the minimal {or maximal $G$-bornology (\cref{ex:born}).
Note that the object $S_{max,min}$ is ill-defined unless $S$ is finite}.
We actually get a functor 
\begin{equation}\label{qreoijqoiegjoqirffewfq}
i \colon  G\Set\to G\BC\ , \quad S\mapsto S_{min,max}\ .
\end{equation}
 {Sending $S$ to $S_{max,max}$ also defines a functor (which we are less interested in), whereas sending $S$ to $S_{min,min}$ is not functorial in maps of $G$-sets.}
\end{ex}

\begin{ex}
 {Let $(X,d)$ be a metric space with an isometric $G$-action.
 Then the metric coarse structure (\cref{ergiowghergregwreg}) and the metric bornology (\cref{ex:born}) combine to define a $G$-bornological coarse space $X_d$.}
\end{ex}

\begin{ex}\label{etwgokergpoergegregegwergrg}
 {On a discrete group $G$, the canonical coarse structure from \cref{ex:Gcan} and the minimal bornology combine to define a $G$-bornological coarse space $G_{can,min}$.}
\end{ex}

\begin{rem}
By dropping condition \cref{wthoiwhthwgreggwregwgr} \eqref{qergoijiogrgqgreg}, one gets an even better category of generalised $G$-bornological coarse spaces $G\widetilde{\BC}$ which (in contrast to $G\BC$, see the discussion in \cite[Sec. 2.2]{equicoarse}) is complete and cocomplete \cite{Heiss:aa}.
As shown in this reference, $G{\BC}$ and $G\widetilde{\BC}$ yield equivalent categories of equivariant coarse homology theories.
As the present paper builds on \cite{equicoarse}, \cite{Bunke:ac} and \cite{desc} working with $G\BC$, we keep using this category also in the present paper.
\end{rem}

\begin{rem}\label{rgiogjweoirgjwergwergwrg}
The category $G\BC$ has a symmetric monoidal structure $\otimes$. For $X,X^{\prime}$ in $G\BC$ the underlying $G$-set of $X\otimes X^{\prime}$ is $X\times X^{\prime}$ with the diagonal action. Its $G$-coarse structure is generated by the entourages $U\times U^{\prime}$ for  all $U$ in $\cC_{X}$ and $U^{\prime}$ in $\cC_{X^{\prime}}$, and its $G$-bornology is generated by the sets $B\times B^{\prime}$ for all $B$ in $\cB_{X}$ and $B^{\prime}$ in $\cB^{\prime}_{X}$. The forgetful functor \eqref{gwegljgoiregwrgergregewg}  is symmetric monoidal  {with respect to} $\otimes$  on $G\BC$ and  the cartesian product on $G\Coarse$ (which we also denote by $\otimes$).
See also \cite[Ex.~2.17]{equicoarse}.
\end{rem}

 By pulling back the functors from \eqref{ergwerggregwergregwgwegw}, \eqref{wergpok2pgegr5gwegwefv} and \eqref{regewgk2p5getgwreg} along the forgetful  functor \eqref{gwegljgoiregwrgergregewg}, we obtain functors
 \[ \PSh^{G},\Sh^{\pi_{0},G}  \colon G\BC\times \Fun(BG,\CL)\to \CL\]
  and
 {\begin{align}\label{2rfjknfkjnvwevevwerv}
\Sh^{G}\colon G\BC\times \Fun(BG,\CL)&\to \CLL \\
((X,\cC_X,\cB_X),\bC) &\mapsto \Sh^G_\bC((X,\cC_X))\ .\nonumber
\end{align}}
 {We proceed to introduce a second type of morphism between $G$-bornological coarse spaces which we call coverings.
Ultimately, we are going to show that sending a $G$-bornological coarse space to its category of controlled objects is covariantly functorial on $G\BC$ and contravariantly functorial with respect to coverings.}

Consider $G$-bornological coarse spaces $X,Y$
and a map $f \colon Y\to X$ between the underlying $G$-sets.
\begin{ddd}\label{ergwergggqrgqregqergq}
$f$ is a   covering if it has the following properties:
\begin{enumerate}
\item\label{wegkowegeregwegerg} $f$ is a coarse covering (\cref{wefgihjwiegwergrwrg}).
\item $f$ is bornological (\cref{rgiojrgofdewqfe} \eqref{rigojowergrewgwregwergwerwg}).
\item \label{qweiuqhfqfeffqeqf} For every $B$ in $\cB_{Y}$  there exists a finite bound
(depending on $B$)  on  the cardinality of the fibres of the map $\pi_{0}(B)\to \pi_{0}(X)$ (\cref{t4hgiorthgetrhtrhtheht}). \qedhere
\end{enumerate}
\end{ddd}

\begin{rem}\label{rgioweogerg2143fr}
\cref{ergwergggqrgqregqergq} is equivalent to \cite[Def.~2.14]{equicoarse}.
The condition \cref{ergwergggqrgqregqergq} \eqref{qweiuqhfqfeffqeqf} is preserved by forming finite unions or taking subsets. So it suffices to check this condition on a set of generators of the bornology $\cB_{Y}$.
\end{rem}

We obtain a category $G\Bc^{\dag}$ of $G$-bornological coarse spaces and  coverings.
  It  comes with a forgetful functor
\[ G\Bc^{\dag}\to  G\Coarse\ .\]

\begin{ex}\label{egowpegrrewfrewfwrefewfref}
Let $W$ be a $G$-set and $X$ be in $G\BC$. Then the projection
\[ W_{min,min}\otimes X\to X \]
is a covering (compare with \cref{ergwegegegeggw}).
 \cref{ergwergggqrgqregqergq} essentially says that a covering is locally a map of this type.
\end{ex}

 Using the existence of the left adjoints asserted
 in \cref{erguiwergrwgwrgwgwr} and \cref{unex-left},  we can consider sheaves  on $
G\Bc^{\dag}$ as   functors
\begin{align}
 \Sh^{\pi_{0},G,\dag} &\colon G\Bc^{\dag,\op}\times \Fun(BG,\CL)\to \CL \ , \nonumber\\
   \quad \Sh^{G,\dag}&\colon G\Bc^{\dag,\op}\times \Fun(BG,\CL)\to \CLL\ .\label{ergioegergwergerwgwerg}
\end{align}

Let $X$ be a $G$-bornological space and let $H$ be a subgroup of $G$.
\begin{ddd}\label{riugfhiqeurgfqwefqwefewf}
 A subset $Y$ of $X$ is called $H$-bounded if there exists a bounded subset $B$ of $X$
 such that $Y=HB$. 
 \end{ddd}

\begin{rem} 
Note that $H$-bounded subsets are $H$-invariant by definition.

The $G$-bounded subsets of $X$
generate a $G$-bornology denoted by $G\cB_{X}$. We have $\cB_{X}\subseteq G\cB_{X}$.

If $X$ is a $G$-bornological coarse space, then $G\cB_{X}$ is again compatible with the coarse structure.
\end{rem}

 Next, we introduce the notion of equivariant smallness and show that all relevant functors on sheaves preserve equivariantly small sheaves.
While this is relatively straightforward for most of them, the proof that the transfer functors along coverings preserve equivariant smallness (\cref{qergioewgergregwergergwergwreg}) requires some effort.

If $H$ is a subgroup of $G$ and $\cC$ is a complete $\infty$-category, then we can consider 
  the natural transformation
 \begin{equation}\label{rqefoijoirgjweoigerog1}
 r^{G}_{H}\colon \lim_{BG}\to  \lim_{BH} \circ \Res_H^G \colon \Fun(BG,\cC)\to \Fun(BH,\cC)\ .
\end{equation}
We write $\res^{G}_{H} \colon G\BC\to H\BC$ for the functor which restricts the action from $G$ to $H$ (we use the same notation also for the restriction of actions on other sorts of spaces). 
For a  $G$-bornological coarse space $X$ and $\bC$ in $\Fun(BG,\CL)$ we have an equivalence $\Res^{G}_{H}\Sh_{\bC}(X)\simeq
 \Sh_{\Res^{G}_{H}\bC}(\res^{G}_{H}X)$ in $\Fun(BH,\CLL)$ (compare with \eqref{qreoijqroifwefqfew}). We will usually drop the symbol $\Res^{G}_{H}$ in front of the $\infty$-category $\bC$, but we will always write $\res^{G}_{H}$ in front of space variables.
 The transformation \eqref{rqefoijoirgjweoigerog1} thus induces a left-exact functor
 \begin{equation}\label{rqefoijoirgjweoigerog}
r^{G}_{H} \colon  \Sh^{G}_{\bC}(X)\to \Sh^{H}_{\bC}(\res^{G}_{H}X)\ .\end{equation}
Varying $X$ these functors give rise to a natural transformation
\[ r^{G}_{H} \colon \Sh^{G}_{\bC} \to \Sh^{H}_{\bC}\circ \res^{G}_{H} \]
 of functors from $G\BC$ to $\CLL$.

 Let $X$ be  {a $G$-bornological coarse space},
 and let $\bC$ be in $\Fun(BG,\CL)$. For a subgroup $H$ of $G$ and an  $H$-invariant subset $Y$ of $X$ we have a left-exact evaluation functor
 \begin{equation}\label{evevevev}
\ev_Y \colon \Sh^G_\bC( X) \xrightarrow{r^{G}_{H}} \Sh^H_\bC(\res^{G}_{H}X) \xrightarrow{\hat i^{*,H}} \Sh^H_\bC(Y) \xrightarrow{\hat p^H_*} \Sh^H_\bC(*) \simeq \bC^H\ ,
\end{equation} 
 where $i \colon Y \to \res^{G}_{H}X$   and $p \colon Y \to *$ denote the  inclusion map and the projection, respectively. If $M$ is a sheaf on $X$,  then we write $M(Y):=\ev_{Y}(M)$ for the value of $M$ on $Y$ considered as an object of  {$\bC^H$}.
  
 Let $M$ be in $\Sh_{\bC}^{G}(X)$.
\begin{ddd}\label{qreigoqrgwergwrgre}
We call $M$ equivariantly small if  {$M(Y)$ is a cocompact object of $\bC^H$}
for all subgroups $H$ of $G$ and all $H$-bounded subsets $Y$ of $X$.   \end{ddd}
 
Let $X$ be  {a $G$-bornological coarse space},
and let $\bC$ be in $\Fun(BG,\CL)$.
\begin{ddd}	\label{def:eqsm}
	 We denote by $\Sh^{G,\eqsm}_{\bC}(X)$ the full subcategory of $\Sh_{\bC}^{G}(X)$ of equivariantly small objects.
\end{ddd}

\begin{lem}\label{rqwpokopfqwfewfewfqfqefqfqwef}
	We have
	$\Sh^{G,\eqsm}_{\bC}(X)\in \Cle$.
\end{lem}
\begin{proof}
	 Since $\ev_Y$ preserves finite limits and finite limits of cocompact objects are cocompact, $\Sh^{G,\eqsm}_\bC(X)$ is closed under taking finite limits.
This implies that $\Sh^{G,\eqsm}_{\bC}(X)\in \CLL$. It remains to show that $\Sh^{G,\eqsm}_{\bC}(X)$ is essentially small.
 
For every {subset $Y$ of $X$},
the family $(B)_{B \in (\cB_X)_{/Y}}$ is a $U$-covering family of $Y$ for every entourage $U$.
It follows that
  the restriction functor $\Sh^G_\bC(X) \to \Fun^G(\cB_X^\op,\bC)$ is fully faithful.
	Since the values of equivariantly small sheaves on bounded subsets are cocompact, this exhibits $\Sh^{G,\eqsm}_\bC(X)$ as a full subcategory of $\Fun^G(\cB_X^\op,\bC^\omega)$. The latter $\infty$-category is essentially small, so $\Sh^{G,\eqsm}_\bC(X)$ is  essentially small, too.
\end{proof}    

Let $f\colon X\to X^{\prime}$ be a morphism {of $G$-bornological coarse spaces}.

\begin{lem}\label{eriogwetgwegregwrgregwrg}
	The morphism $\hat f^{G}_{*}$ from \eqref{hfjrefewrfwrefwergf}  preserves equivariantly small objects.
\end{lem}
\begin{proof} Let $M$ be in $\Sh^{G,\eqsm}_\bC(X)$.
	 If $Y'$ is an $H$-bounded subset of $X'$, then $f^{-1}(Y')$ is an $H$-bounded subset of $ X$. Consequently, $\hat f^G_*M(Y') \simeq M(f^{-1}(Y')) \in \bC^{H,\omega}$.
\end{proof}

Let $X$ be {a $G$-bornological coarse space},
and let $V$ be {be a $G$-invariant coarse entourage of $X$ containing the diagonal}.
Recall the endofunctor $ V_{*}^{G}$ of $\Sh_{\bC}^{G}(X) $ from \cref{egiweogwergrgwgwerg}. 
\begin{lem}\label{rgiwoegwergrwegwregwrg}
	The endofunctor $V_{*}^{G}$  preserves equivariantly small objects.
\end{lem}
\begin{proof} Let $M$ be in $\Sh^{G,\eqsm}_\bC(X)$ and assume that $Y$ is an $H$-bounded subset of $X$. We must show that $(V_{*}^{G}M)(Y) \in \bC^{H,\omega}$.
 
Note that  $V(Y)$ is also $H$-bounded. Let $i \colon Y\to \res^{G}
_{H}X$, $j \colon V(Y)\to \res^{G}_{H}X$ and $k \colon V(Y)\to Y$ denote the inclusions, and let $p \colon Y\to *$ and $q \colon V(Y)\to *$ denote  the projections. We have the following chain of equivalences
\begin{eqnarray*}
	(V_{*}^{G}M)(Y)&\stackrel{\mathrm{def}}{\simeq}&\ev_{Y}(V^{G}_{*}M)\\&\stackrel{\eqref{evevevev}}{\simeq}&  \hat p^{H}_{*} \hat i^{*,H} r^{G}_{H}V^{G}_{*}M\\&  \simeq&
 \hat p^{H}_{*} \hat i^{*,H} V^{H}_{*} r^{G}_{H}M\\&\stackrel{\small \cref{egiwetthtrehtrhethetheth}}{\simeq}&
  \hat p^{H}_{*}  V^{H}_{Y,*}  i^{*,H}\hat j^{H}_{*} \hat j^{*,H}  r^{G}_{H}M\\
 &\stackrel{j=i\circ k}{\simeq}& \hat p^{H}_{*}  V^{H}_{Y,*}  i^{*,H}\hat i^{H}_{*}\hat k^{H}_{*} \hat j^{*,H}  r^{G}_{H}M\\ 
  &\stackrel{\small \cref{rqkjgqregergwgwregweg}}{\simeq}&
    \hat p^{H}_{*}  V^{H}_{Y,*}   \hat k_{*}^{H} \hat j^{*,H}   r^{G}_{H}M\\&\stackrel{V_{Y}(Y)\cap V(Y)=V(Y)}{\simeq}&
    \hat q^{H}_{*}    \hat j^{*,H}  r^{G}_{H}M\\&\stackrel{\eqref{evevevev}}{\simeq}&
    \ev_{V(Y)}(M)\\ &\stackrel{\mathrm{def}}{\simeq}&M(V(Y))\ .
  \end{eqnarray*}
  Consequently, we have $(V^G_*M)(Y) \simeq M(V(Y)) \in \bC^{H,\omega}$.
\end{proof}

Let $\phi\colon \bC\to \bC^{\prime}$ be a morphism in $\Fun(BG,\CL)$.
\begin{lem}\label{qrkgoqrgergwrgwegwrg}
	The morphism $\hat \phi^{G}_{*}\colon  \Sh_{\bC}^{G}(X)\to \Sh_{\bC^{\prime}}^{G}(X)$ preserves equivariantly small objects.
\end{lem}
\begin{proof}
 Note that  $\phi^H \colon \bC^H \to \bC^{\prime,H}$, being a  
  morphism in $\CL$,  preserves cocompact objects. 
  Let $M$ be  {an equivariantly small sheaf on $X$}.
  If $Y$ is an $H$-bounded subset of $X$, then $\hat \phi^{G}_{*}M(Y)\simeq \phi^{H}(M(Y))\in \bC^{\prime,H,\omega}$.
\end{proof}

As a consequence of \cref{rqwpokopfqwfewfewfqfqefqfqwef,eriogwetgwegregwrgregwrg,qrkgoqrgergwrgwegwrg} we get:
\begin{kor}
	We have a subfunctor
	\[ \Sh^{G,\eqsm}\colon G\BC\times \Fun(BG,\CL)\to \Cle \]
	of $\Sh^{G}$ from \eqref{2rfjknfkjnvwevevwerv}.
\end{kor} 

Let $X$ be  {a $G$-bornological coarse space},
and let $i \colon Y \to X$ be the inclusion of a $G$-invariant subset.

\begin{lem}\label{lem:restrict-eqsm}
	 The morphism $\hat i^{*,G}$ from \eqref{fvsnejkvwevrevwv} preserves equivariantly small objects.
\end{lem}
\begin{proof}
	Any $H$-bounded subset $Z$ of $Y$ is also an $H$-bounded subset of $X$, so $\hat i^{*,G}M(Z) \simeq M(i(Z)) \in \bC^{H,\omega}$.
\end{proof}

Let $p \colon Z \to X$ be a  {a covering of $G$-bornological coarse spaces}.

\begin{prop}\label{qergioewgergregwergergwergwreg}
	The morphism $L^{\pi_{0},G}\hat p^{*,G}\colon \Sh_{\bC}^{G}(X)\to \Sh_{\bC}^{G}(Z)$
	preserves equivariantly small objects.
\end{prop}

\begin{kor}
	We have a subfunctor
	\[ \Sh^{G,\dag,\eqsm}\colon G\Bc^{\dag,\op}\times \Fun(BG,\CL)\to \Cle \]
	of $\Sh^{G,\dag}$ from \eqref{ergioegergwergerwgwerg}.
\end{kor}

 The proof of \cref{qergioewgergregwergergwergwreg} requires some preparation. We will first describe the behaviour of the transfer in some concrete cases, and then observe that every  covering is modelled locally on these special cases. This allows us to show that transfers along coverings also preserve equivariantly small sheaves.

 Let $H$ be a subgroup of $G$.  
 For any complete   category $\cC$ 
 the restriction and right Kan extension functors along     the inclusion functor  $j:BH\to BG$ are parts of an adjunction
 \[ \Res^{G}_{H} \colon \Fun(BG,\cC) \leftrightarrows \Fun(BH,\cC) \cocolon \Coind^{G}_{H} \ .\]
This can be applied to  $\cC:=\CL$ since
 the latter $\infty$-category is complete by \cref{lem:CL-complete}.
In the following, we consider sets  as discrete categories. For $\bC$ in $\CL$ and a set $X$
we have a canonical identification
\begin{equation}\label{wfevowjvowevwevw}
 \Fun(X,\bC)\xrightarrow{\simeq} \prod_{X}\bC 
\ ,\quad C\mapsto (C(x))_{x\in X}
\end{equation}
in $\CL$. More generally, for a $G$-set $X$ 
and $\bC$ in $\Fun(BG,\CL)$ we get an object
$\Fun(X,\bC)$ in $\Fun(BG,\CL)$ with the $G$-action by conjugation. Furthermore,
\begin{equation}\label{qgfefefqefefqfewfqefqef}
\Fun^{G}(X,\bC):=\lim_{BG}\Fun(X,\bC)  
\end{equation}
is the $\infty$-category of $G$-equivariant functors. 
The diagram \begin{equation}\label{qreoijqroifwefqfew}
\xymatrix@C=6em{
	 G\Set^{{\op}}\ar[r]^-{\Fun(-,\bC)}\ar[d]_-{{\res}_H^G} & \Fun(BG,\CL)\ar[d]^-{\Res_H^G} \\
	H\Set^{{\op}}\ar[r]^-{\Fun(-,\Res_H^G \bC)} & \Fun(BH,\CL)
}\end{equation} {commutes.}

Given a $G$-set $X$ we have an inclusion of $H$-sets
\begin{equation}\label{wergoihiorgohjwegoiwre}
i \colon \res_{H}^{G}X \to \res_{H}^{G}(G/H \times X)\ , \quad i(x):=(eH,x) .
\end{equation}
It induces the restriction functor
\begin{equation}\label{vdoijvsodvdv}
 i^* \colon \Res_H^G \Fun(G/H \times X,\bC) \to \Res_H^G \Fun(X,\bC)\ ,
\end{equation}
whose definition implicitly uses the square \eqref{qreoijqroifwefqfew}.
\begin{lem}\label{lem:coindC}  
 The adjoint
 \[\Fun(G/H \times X,\bC)\to \Coind_{H}^{G}\Res_{H}^{G}\Fun(X,\bC)\]
 of $i^{*}$ in \eqref{vdoijvsodvdv}
 is an equivalence in $\Fun(BG,\CL)$.
\end{lem}
\begin{proof}
Note that $\Res^{G}_{H}$ and $\Coind_{H}^{G}$ are the restriction and the right Kan extension functors along $j\colon BH \to BG$.
It thus suffices to show that 
the morphism $i^*$ exhibits $\Fun(G/H \times X,\bC)$  as a right Kan extension of $\Res_H^G \Fun(X,\bC)$ along the inclusion $j$.
  
  We first consider the case $X = *$.  
  It suffices to check that the functor
	\[ \Fun(G/H,\bC) \to \lim_{BG_{/j}} \Res_H^G\bC \]
	induced by $i^*$ is an equivalence.
	The  functor  $BG_{/j} \to \ho(BG_{/j}) \simeq G/H$ is an equivalence. Any choice of section $s \colon G/H \to G$ to the projection map induces an inverse equivalence $e_s \colon G/H \xrightarrow{\simeq} BG_{/j}$. Under this identification, we must show that the map
	\[ (i^* \circ s(gH))_{gH \in G/H} \colon \Res^G_{\{1\}} \Fun(G/H,\bC) \to \prod_{G/H} \Res^G_{\{1\}}\bC \]
	is an equivalence of $\infty$-categories. 
	By \eqref{wfevowjvowevwevw} we can rewrite the domain as a product over $G/H$,
	and then the functor becomes a product of equivalences, hence is an equivalence itself.
	 	
 For a general $G$-set $X$  we use the canonical equivalence
\[ \Fun(G/H \times X,\bC) \simeq \Fun(G/H, \Fun(X,\bC)) \]
in $\Fun(BG,\CL)$ in order to reduce the problem to the special case above  with  $\Fun(X,\bC)$ in place of $\bC$.  This finishes the proof of the lemma.
\end{proof}

 Let $p \colon G/H \times X \to X$ be the $G$-equivariant projection map and recall the notation $r^{G}_{H}$ from \eqref{rqefoijoirgjweoigerog1}.
 Recall further the convention that we usually drop $\Res^{G}_{H}$ in front of $\bC$.
 Let
\[ \eta\colon \Fun(X,\bC)\to \Coind_{H}^{G}\Res^{G}_{H}\Fun(X,\bC) \]
be the component of the unit of the adjunction $(\Res^{G}_{H},\Coind_{H}^{G})$ at $\Fun(X,\bC)$. We then define the functor
 \[ \coind_{H}^{G} \colon \Fun^{H}(\res^{G}_{H}X,\bC)\to \Fun^{G}(X,\bC) \]
as the composition
\[ \hspace{-0.5cm}\Fun^{H}(\res^{G}_{H}X,\bC)\simeq \lim_{BG}\Coind_{H}^{G}\Res^{G}_{H}\Fun(X,\bC)\xrightarrow{\lim_{BG}\eta_{*}}
 \lim_{BG} \Fun(X,\bC)\simeq \Fun^{G}(X,\bC)\ ,\]
 where $\eta_{*}$ is the right adjoint of $\eta$ (its existence will be shown in the proof of \cref{cor:coind} below).

 The left vertical arrow in \eqref{fsdvlksmvklvsdvadvadvadsv}   implicitly uses the square \eqref{qreoijqroifwefqfew}. 
 \begin{lem} \label{cor:coind}
	  There exists a commutative diagram
	 \begin{equation}\label{fsdvlksmvklvsdvadvadvadsv}
\xymatrix@C=4em{
		\Fun^G(G/H \times X,\bC)\ar[r]^-{p_*}\ar[d]_{r^{G}_{H}}\ar[dr]^{\simeq} & \Fun^G(X,\bC) \\
		\Fun^H(\res^{G}_{H}(G/H \times X),\bC)\ar[r]_-{\lim_{BH}(i^*)} & 
		\Fun^H(\res^{G}_{H}X,\bC)\ar[u]_-{\coind_H^G}
	} \end{equation}
	in $\CL$ whose diagonal is an equivalence.
Furthermore, the functor
	$\coind_H^G$ restricts to a functor
	\[ \coind_H^{G,\omega}\colon \Fun^H(\res^{G}_{H} X,\bC)^{\omega}\to \Fun^G(X,\bC)^{\omega} \]
	whose essential image generates the target under finite limits and retracts.
\end{lem}
\begin{proof}
By \cref{lem:coindC} (for the lower triangle and the diagonal equivalence) and the identity $p\circ i=\id_{X}$ (for the upper triangle), we have the following commutative diagram
\begin{equation}\label{weqwefcweqfqew}\xymatrix@C=4em{
\Fun(G/H\times X,\bC)\ar[dr]^{\simeq}\ar[d]_{\unit}&\ar[l]_-{p^{*}} \Fun(X,\bC)\ar[d]^{\eta}\\
\Coind_{H}^{G}\Res_{H}^{G} \Fun(G/H\times X,\bC)\ar[r]_-{\Coind_{H}^{G}(i^{*})}&
\Coind_{H}^{G}\Res_{H}^{G} \Fun(X,\bC)}
\end{equation}
in $\Fun(BG,\CL)$.
Since $p^{*}$ has a right adjoint (given by the right Kan extension functor $p_{*}$ along $p$), also $\eta$ has one which will be denoted by $\eta_{*}$.  Passing to the right adjoints in the upper triangle in \eqref{weqwefcweqfqew} we get the diagram
\begin{equation*}\xymatrix@C=4em{
\Fun(G/H\times X,\bC)\ar[dr]^{\simeq}\ar[d]_{\unit}\ar[r]^-{p_{*}} &\Fun(X,\bC) \\
\Coind_{H}^{G}\Res_{H}^{G} \Fun(G/H\times X,\bC)\ar[r]_-{\Coind_{H}^{G}(i^{*})}&
\Coind_{H}^{G}\Res_{H}^{G} \Fun(X,\bC)\ar[u]_{\eta_{*}}}
\end{equation*}
in $\Fun(BG,\CL)$.
We now apply $\lim_{BG}$. Use that $r^{G}_{H}$ in \eqref{rqefoijoirgjweoigerog1} is given by
\begin{equation}\label{qfrufqrweufuqrifiwefwqefqewfqef}
\lim_{BG}\xrightarrow{\lim_{BG}(\unit)} \lim_{BG} \Coind_{H}^{G}\Res_{H}^{G}\simeq \lim_{BH}{\Res_H^G}\ ,
\end{equation}
and recall the square \eqref{qreoijqroifwefqfew} in order to rewrite the lower left corner. Lastly, use the definition 
$\coind_{H}^{G}:=\lim_{BG}\eta_{*}$ in order to get the square \eqref{fsdvlksmvklvsdvadvadvadsv}.

We now show the second assertion. Since  the functor $\coind_{H}^{G}$  is a morphism in $\CL$, it preserves cocompact objects and therefore restricts to  $\coind_{H}^{G,\omega}$ as asserted. 
In order to see that the  essential image of $\coind_{H}^{G,\omega}$ generates the target under finite limits and retracts, we consider the  diagram
\begin{equation}\label{wfeqoijoiqjfwefqwefqwefqwef}
\xymatrix{\Fun(X,\bC)(*_{BG})^{\omega}\ar@/^1cm/[rr]^{a_{G}}\ar@{=}[d]\ar[r]^{a_{H}}&\Fun^{H}(\res^{G}_{H}X,\bC)^{\omega}\ar[d]_{\eqref{wqefqoijfqoiwffqefefewfqefqe}}^{\simeq}\ar[r]^{\coind_{H}^{G,\omega}}&\Fun^{G}(X,\bC)^{\omega}\ar[d]_{\eqref{wqefqoijfqoiwffqefefewfqefqe}}^{\simeq}\\\Fun(X,\bC)(*_{BG})^{\omega}\ar@/_1cm/[rr]_-{\mathrm{can}_{G}}\ar[r]^-{\mathrm{can}_{H}}&\colim_{BH^{\op}}\Fun(\res^{G}_{H}X,\bC)^{\omega}\ar[r]^{!}&\colim_{BG^{\op}}\Fun(X,\bC)^{\omega}}\ .
\end{equation}
The functors $a_{H}$ and $a_{G}$ are the restrictions to cocompact objects of the right adjoints (which exist as seen in the proof of  \cref{rwigowefqewffqwfwfqwef}) of the canonical functors
\[ \hspace{-0.3cm}e_{*_{BH}} \colon \Fun^{H}(\res^{G}_{H}X,\bC)\to \Fun(X,\bC)(*_{BG})\ , \quad e_{*_{BG}} \colon \Fun^{G}(X,\bC)\to \Fun(X,\bC)(*_{BG}) \]
(instances of \eqref{geroij4oi33rgg3g34}).
  The left square   commutes  by \cref{rwigowefqewffqwfwfqwef} \eqref{tjborbjeorpbrberbertbertbrtb}, and the arrow marked by $!$ is defined such that the right square commutes. In order to show that the upper triangle commutes, we first observe\footnote{We consider the diagram
  \[\xymatrix{&BG\ar[dr]^{v_{G}}&\\{*}\ar[ur]^{u_{G}}\ar[dr]_{u_H}&&{*}\\&BH\ar[uu]^{j}\ar[ur]^{v_{H}}&}\ .\]
  Then we must observe that the morphism $v_{G,*} \to    v_{G,*} u_{G,*} u_{G}^{*} \simeq  u_{G}^{*}$ induced by the unit of the $(u_{G}^{*},u_{G,*})$-adjunction is equivalent to  the morphism $v_{G,*} \to v_{G,*} j_{*}j^{*}\simeq v_{H,*} j^{*}\to  v_{H,*} u_{H,*} u_{H}^{*}j^{*}\simeq u_{G}^{*}$ 
  induced by the units of the adjunctions $(u_{H}^{*},u_{H,*})$ and $(j^{*},j_{*})$.} 
  that \[\xymatrix{\Fun(X,\bC)(*_{BG})&\ar[l]^-{e_{*_{BH}}}\Fun^{H}(\res^{G}_{H}X,\bC)&\ar[l]^-{r^{G}_{H}}\Fun^{G}(X,\bC)\ar@/_1cm/[ll]_-{e_{*_{BG}}}}\]
  (implicitly we identify $\Fun(X,\bC)(*_{BG})$ with $\Res^{G}_{H}\Fun(X,\bC)(*_{BH})$)
 commutes in a canonical way. Then we  
  take right adjoints, use \eqref{qfrufqrweufuqrifiwefwqefqewfqef} in order to identify the right adjoint of $r^{G}_{H}$ with $\coind_{H}^{G}$,   and restrict to cocompact objects.
      
By  \cref{rwigowefqewffqwfwfqwef} \eqref{tjborbjeorpbrberbertbertbrtb1} the essential images of $\mathrm{can}_{H}$ (this is $\mathrm{can}(*_{BG})$ in the notation of  \cref{rwigowefqewffqwfwfqwef}) and $\mathrm{can}_{G}$ generate their respective targets under retracts and finite limits.  Hence the  essential image of the restriction  of $\coind_{H}^{G}$ to cocompact objects
   generates its target under finite limits and retracts, too.
\end{proof}

\begin{rem}
Let $\bC$ be in $\Fun(BG,\CL)$.
Applying \cref{cor:coind} to the case $X=*$, we get a morphism 
\begin{equation}
\coind_{H}^{G} \colon \bC^{H}\to \bC^{G}
\end{equation}
in $\CL$. This morphism restricts to a morphism \begin{equation}\label{ewrgkpowergewrgwergwre}
\coind_{H}^{G,\omega} \colon \bC^{H,\omega}\to \bC^{G,\omega}
\end{equation}
in $\Clep$
whose essential image generates the target under  retracts and finite limits.
\end{rem}

 The commutative diagram \eqref{fsdvlksmvklvsdvadvadvadsv} can be realised as a diagram of categories of sheaves:

\begin{kor}\label{cor:coindSh}
	 There exists a commutative diagram
	\begin{equation}\label{wregpokopwtglmrbletemblwtwb}
\xymatrix@C=4em{
		\Sh^G_\bC((G/H)_{min} \otimes X_{min})\ar[r]^-{\hat p^{G}_*}\ar[d]_{r^{G}_{H}} \ar[dr]^{\simeq}& \Sh^G_\bC(X_{min}) \\
		\Sh^H_\bC(\res^{G}_{H}((G/H)_{min} \otimes X_{min}))\ar[r]_-{\hat i^{*,H}} & 
		\Sh^H_\bC( \res^{G}_{H}X_{min})\ar[u]_-{\coind_H^G}
	}\ .  \end{equation}
Furthermore, the 
functor $\coind_H^G$ restricts to a functor
\[ \coind_H^{G,\omega} \colon \Sh^H_\bC(\res^{G}_{H}X_{min})^\omega \to \Sh^G_\bC(X_{min})^\omega \]
whose essential image generates  {its target} under finite limits and retracts.
\end{kor}
\begin{proof} If $Y$ is a $G$-set,
then $(\{y\})_{y \in Y}$ is a $\diag(Y)$-{covering family} of $Y$. Consequently, in view of \eqref{qgfefefqefefqfewfqefqef},
the functor $M\mapsto (M(\{y\}))_{y\in Y}$  induces an equivalence $\Sh^G_\bC(Y_{min}) \xrightarrow{\simeq} \Fun^G(Y,\bC)$. 
We thus can replace the  functor categories in the corners of 
 \eqref{fsdvlksmvklvsdvadvadvadsv} by sheaf categories. 
 One further checks that the horizontal morphisms are given by the sheaf-theoretic operations as indicated.
 This yields 
 the diagram  \eqref{wregpokopwtglmrbletemblwtwb}. 
 The second assertion of \cref{cor:coindSh}   follows from the second assertion of \cref{cor:coind}.
\end{proof}

 Recall that $\otimes$ denotes the cartesian product in $G\Coarse$.
Let $H$ be a subgroup of $G$, and let {$Y$} be an $H$-coarse space.
We form the  $G$-coarse space $G_{min}\otimes \res^{H}_{\{1\}}Y$ with
  the  $G$-action $(g',(g,y))\mapsto (g'g,y)$. The group   
 $H$  acts   by automorphisms on the $G$-coarse space  $G_{min}\otimes \res^{H}_{\{1\}}Y$  such that
$(h,(g,y))\mapsto (gh^{-1},hy) $. 
The colimit in the following definition is interpreted in $G\Coarse$. 
\begin{ddd}\label{fivaorvrvefvsfdvfvs}
We define the $G$-coarse space
\[ G \otimes_H Y:=\colim_{BH} (G_{min}\otimes \res^{H}_{\{1\}}Y)\ .\qedhere \]
\end{ddd}

 \begin{rem}
 The underlying set of $G\otimes_{H}Y$ can be identified with the set $G\times_{H}Y$ of equivalence classes $[g,y]$ with
 $[g,y]=[gh^{-1},hy]$ for $h$ in $H$. It is equipped with the smallest
 $G$-coarse structure such that the canonical map from $G_{min}\otimes \res^{H}_{\{1\}}Y$, $(g,y)\mapsto [g,y] $, is controlled.
 
We have an adjunction
\[ G \otimes_H - \colon H\Coarse \leftrightarrows G\Coarse \cocolon \res^{G}_{H}\ . \]
As in the corresponding adjunction between $G\Set$  and $H\Set$, the $H$-equivariant inclusions
\[  Y \to \res_{H}^{G} (G \otimes_H Y)\ , \quad y\mapsto [e,y] \]
for $Y$ in $H\Coarse$ define the unit, while the $G$-equivariant multiplication maps 
\[ \mu \colon G \otimes_H \res_H^G X \to X\ , \quad [g,x]\mapsto gx \]
for $X$ in $G\Coarse$ provide the counit of the adjunction.
\end{rem}

 Let $X$ be a $G$-coarse space, 
 and let $Y$ be an $H$-invariant subspace for some subgroup $H$ of $G$. Denote by $i \colon Y \to \res^{G}_{H}X$ the inclusion map. We   consider  the composition 
\begin{equation}\label{qrg0rig0qrgqrgqrgqrg}
\iota \colon G \otimes_H Y \xrightarrow{G \otimes_H i} G \otimes_H \res^{G}_{H}X \xrightarrow{\mu} X
\end{equation} 
in $G\Coarse$.
Since   {$G\otimes_{H}i$} is an inclusion and $\mu$ is a coarse covering, both   $\hat{G \otimes_H i}^{*,G}$ and $L^{\pi_0}\hat \mu^{*,G}$ define functors on sheaves {by \cref{unex-left} and \cref{rqkjgqregergwgwregweg}, respectively.} We  consider the composition\footnote{Note that this is a slight abuse of notation since $\hat \iota^{*,G}$ is not directly obtained from $\iota$ by one of our previous constructions.}
\[ \hat \iota^{*,G} := \hat{G \otimes_H i}^{*,G} \circ L^{\pi_0}\hat \mu^{*,G}\colon \Sh^{G}_{\bC}(X)\to \Sh^{G}_{\bC}(G\otimes_{H}Y)\ .\]
Recall the notation $\ev_{Y}$ from \eqref{evevevev}.
\begin{lem}\label{lem:transfer-coind}
	 There exists a commutative diagram
	 \begin{equation}\label{qewlknmwlkvfvsdfvsdfvsdvsdv}\xymatrix@C=3em{
		\Sh^G_\bC(X)\ar[d]_-{\ev_Y}\ar[r]^-{\hat \iota^{*,G}} & 	\Sh^G_\bC(G \otimes_H Y)\ar[d]^-{\ev_{G \times_H Y}} \\
		\bC^H\ar[r]^-{\coind_H^G} & \bC^G \\		
	}\end{equation}	
\end{lem}
\begin{proof}
	Let ${q} \colon Y \to {*}$ denote the projection map in $H\Coarse$. We have a canonical isomorphism   $G \otimes_H * \cong {(G/H)_{min}}$. By naturality of the transformation \eqref{rqefoijoirgjweoigerog},
  we have a commutative diagram
\begin{equation}\label{wergwrgrgewgwgrwegw243tr34t2t34t}
\hspace{-1.5cm}\xymatrix@C=3em{
		\Sh^G_\bC(X)\ar[d]^{{r^{G}_{H}}}\ar@/^1cm/[rr]^-{\hat \iota^{*,G}}\ar[r]^-{L^{\pi_0,G}\hat \mu^{*,G}} & \Sh^G_\bC(G \otimes_H \res^{G}_{H}X)\ar[d]^-{r^{G}_{H}}\ar[r]^-{\hat{G \otimes_H i}^{*,G}} & \Sh^G_\bC(G \otimes_H Y)\ar[d]^-{r^{G}_{H}}\ar[r]^-{\hat{G \otimes_H {q}}^G_*} & \Sh^G_\bC({(G/H)_{min}})\ar[d]^{{r^{G}_{H}}} \\
		\Sh^H_{\bC}(\res^{G}_{H}X)\ar[r]^-{L^{\pi_0,H}\hat \mu^{*,H}} \ar@/_1cm/[rr]_{\hat \iota^{*,H}}& \Sh^H_{\bC}(\res_{H}^{G}(G \otimes_H \res^{G}_{H}X))\ar[r]^-{\hat{G \otimes_H i}^{*,H}} & \Sh^H_{\bC}(\res^{G}_{H}(G \otimes_H Y))\ar[r]^-{\hat{G \times_H {q}}^H_*} & \Sh^H_{\bC}(\res_{H}^{G}(G/H)_{min})
	}
\end{equation} 	 
	Since the inclusion $i$  can be factorised as the    composition 
	\[ Y \xrightarrow{j} \res_{H}^{G}(G \otimes_H Y) \xrightarrow{\res_{G}^{H}(G \otimes_H i)} \res_{H}^{G}(G \otimes_H \res_{H}^{G}X) \xrightarrow{\res_{H}^{G}(\mu)} \res_{H}^{G}X\ ,\]
	where $j$ is the canonical inclusion $y\mapsto [e,y]$, we have an equivalence
	\begin{equation}\label{eq:transfer-coind1}
	 \hat i^{*,H} \simeq \hat j^{*,H}\hat \iota^{*,H}\ .
	\end{equation}
	Let $k \colon * \to  \res_{H}^{G}(G/H)_{min}$ denote the  inclusion of the point $eH$. Since
	\[\xymatrix@C=3em{
		Y\ar[d]_-{j}\ar[r]^-{{q}} & {*}\ar[d]^-{k} \\
		\res_{H}^{G}(G \otimes_H Y)\ar[r]^-{G \otimes_H {q}} & (G/H)_{min} \\
	}\]
	 is a pullback in $H\Coarse$,  by \cref{ropgkpwoegrewgwregwgregw} we obtain an equivalence
	\begin{equation}\label{eq:transfer-coind2}
	 \hat k^{*,H} \circ \hat{G \otimes_H {q}}_{*}^{H} \simeq \hat {{q}}^H_* \circ \hat j^{*,H}\ .
	\end{equation}
	Combining \eqref{eq:transfer-coind1} and \eqref{eq:transfer-coind2} with the commutative diagram \eqref{wergwrgrgewgwgrwegw243tr34t2t34t}, we have the solid part of the commutative diagram
	\[\hspace{-1.5cm}\xymatrix@R=2.5em@C=3em{
		\Sh^G_\bC(X) \ar@{..}@/_1.5cm/[dd]+0{}\ar[r]^-{\hat \iota^{G,*}}\ar[d]^{r^{G}_{H}} & \Sh^G_\bC(G \otimes_H Y)\ar@/^1cm/@{..>}[rr]^{\ev_{G\times_{H}Y}}\ar[r]^-{\hat{G \otimes_H q}^G_*}\ar[d]^{r^{G}_{H}} & \Sh^G_\bC((G/H)_{min})\ar[d]^{r^{G}_{H}}\ar@{-->}[r]^{\hat p^{G}_{*}}&\Sh^{G}_{\bC}(*) \simeq \bC^G \\
		\Sh^H_\bC(\res^{G}_{H}X)\ar[r]^-{\hat \iota^{*,H}}\ar[rd]_-{\hat i^{H,*}} & \Sh^H_\bC(\res^{G}_{H}(G \otimes_H Y))\ar[r]^-{\hat{G \otimes_H  {{q}}^H_*}}\ar[d]^-{\hat j^{*,H}} & \Sh^H_\bC(\res^{G}_{H}(G/H)_{min})\ar[d]^-{\hat k^{*,H}} && \\
		& \Sh^H_\bC(Y)\ar[r]^-{{\hat q}^H_*} & \Sh^H_\bC(*) \simeq \bC^H \ar@{-->}[uur]_-{\coind_H^G}\ar@{<..}@/^.5cm/[ll]_(.8){\ev_Y}{}
	}\]
	We can extend the diagram by the dashed part by applying \cref{cor:coindSh} (with $X=*$ and $i=k$).  The combination with the dotted part (reflecting the definition of the evaluation maps) then yields the asserted square.
\end{proof}

\begin{rem}\label{etgioowrgwegreg}
Note that the diagonal map in \eqref{qewlknmwlkvfvsdfvsdfvsdvsdv} sends $M$ in $\Sh^{G}_{\bC}(X)$ to
$L^{\pi_{0},G} \hat \mu^{*,G}M(G\times_{H}Y)$ in $\bC^{G}$.
\end{rem}

Let $X$ be a $G$-coarse space, and let $Y$ be  {a coarse component of $X$}.
Denote by
\[ G_{Y} := \{ g \in G \mid gY = Y \} \]
the stabiliser of $Y$. We consider $Y$ as a $G_{Y}$-coarse subspace of $\res_{G_{Y}}^{G}X$.

\begin{lem}\label{lem:orbits-components}
	  If $X = GY$ (as sets),
	  then the multiplication map $\mu \colon G \otimes_{G_{Y}} Y \to X$ is an isomorphism {of $G$-coarse spaces}.
\end{lem}
\begin{proof}
	The morphism
	\[ G_{min}\otimes \res^{G_{Y}}_{\{1\}}Y\to X\ , \quad (g,y)\mapsto gy \]
	in $G\Coarse$ is controlled and constant along $G_{Y}$-orbits for the $G_{Y}$-action
	\[ (h,(g,y))\mapsto (gh^{-1},hy) \]
	on its domain. By the universal property of the colimit  in \cref{fivaorvrvefvsfdvfvs}  it factorizes over  the multiplication map $\mu \colon G \otimes_{G_{Y}} Y \to X$ in $G\Coarse$. The latter is surjective
	  by our assumptions.  
	  
	We now show that $\mu$ is injective.
	Consider $[g,y], [g,y']$ in $G \otimes_{G_{Y}} Y$ and suppose that $\mu([g,y]) = \mu([g',y'])$. Then $y = g^{-1}g'y'$. Let $y''$ be any point in $Y$.
	Since $Y$ is coarsely connected,   we have $\{(y,y'')\}\in \cC_{X}$ and $\{(y',y'')\} \in \cC_{X}$.    Since $\cC_{X}$ is $G$-invariant, we also have
	$g^{-1}g'\{(y',y'')\} = \{(y,g^{-1}g'y'')\}\in \cC_{X}$ and therefore $g^{-1}g'y''\in Y$.  Since $y''$ is arbitrary, we can conclude that $g^{-1}g'\in G_{Y}$, and hence 
	$[g,y] = [g',y']$.
	
	 If $U$ is a $G$-invariant entourage of $X$, then
	 we have $U = G(U \cap (Y \times Y))$. This entourage is thus  the image of $\diag(G)\times (U\cap (Y\times Y))$
	 under the composition of the projection $G\times Y\to G\times_{G_{Y}}Y$ and the   multiplication map and therefore the image of a coarse 
	  entourage of $G \otimes_{G_{Y}} Y$ under the multiplication map $\mu$.  This shows that the map $\mu$ is an isomorphism of $G$-coarse spaces.
\end{proof}

 Let $X$ be  {be a $G$-coarse space},
 and let $A$ be a coarsely connected subspace of $X$.  By $[A]$ in $\pi_{0}(X)$ we denote the coarse closure of $A$. We consider $GA$ as a $G$-coarse subspace of $X$ and $G_{[A]}A$ as a $G_{[A]}$-coarse subspace of $\res_{G_{[A]}}^{G}X$.

\begin{kor}\label{cor:coarsely-connected-orbit}
	 The multiplication map induces an isomorphism
	 \[  G \otimes_{G_{[A]}} G_{[A]}A\cong GA \ . \]
\end{kor}
\begin{proof}
	 The subspace $G_{[A]}A$ of $GA$ is  a coarse component of $GA$.
		Consequently, the corollary follows from \cref{lem:orbits-components} applied to $GA$ in place of $X$ and $G_{[A]}A$ in place of $Y$.
\end{proof}

Let $p \colon Z \to X$
be a coarse covering (\cref{wefgihjwiegwergrwrg}) of $G$-coarse spaces, let $M$ be in $\Sh^G_\bC(X)$, and let $H$ be a subgroup of $G$.

\begin{prop}\label{prop:transfer-coind}
 If $B$ is a coarsely connected subset of $Z$, then
 \[ L^{\pi_0,G}\hat p^{*,G}M(HB) \simeq \coind_{H_{[B]}}^H(M(p(H_{[B]}B)))\ .\]
\end{prop}
\begin{proof}
	We consider the $H_{[B]}$-invariant subspace   $H_{[B]}B$ of $Z$ with the induced coarse structure as an object of $H_{[B]}\Coarse$.
 By \cref{cor:coarsely-connected-orbit}, we have an isomorphism of $H$-coarse spaces $HB \cong H \otimes_{H_{[B]}} H_{[B]}B$.
	Since $p$ is a coarse covering,  {it induces} in view of \cref{wefgihjwiegwergrwrg}.\eqref{wrkojgwergrefwrfrf}  an isomorphism $H_{[B]}B \xrightarrow{\cong} p(H_{[B]}B) = H_{[B]}p(B)$, and we can identify $p_{|HB}$ with the multiplication map
	\[ H \otimes_{H_{[B]}} p(H_{[B]}B) \to H(p(H_{[B]}B)) \ .\]
	The desired equivalence is now given by the composition
	\begin{eqnarray*}
L^{\pi_0,G}\hat p^{*,G}M(HB)&\simeq&L^{\pi_0,H}\hat p^{*,H} r^{G}_{H}M(HB)\\&\stackrel{!}{\simeq}
&  \coind_{H_{[B]}}^H(r^{G}_{H}M(p(H_{[B]}B)))\\&\simeq&  \coind_{H_{[B]}}^H( M(p(H_{[B]}B)))
\ ,
\end{eqnarray*}
where the marked equivalence is given by the filler of the square {\eqref{qewlknmwlkvfvsdfvsdfvsdvsdv}}.
Here we apply \cref{lem:transfer-coind} to $H$ in place of $G$, $H_{[B]}$ in place of $H$,  $p(H_{[B]}B)$ in place of $Y$, and
$\res^{G}_{H}X$ in place of $X$. We furthermore use   \cref{etgioowrgwegreg}.
\end{proof}

Let $\bD$ be  {a large $\infty$-category which admits small limits},
and let $(D_i)_{i \in I}$ be a family of objects in $\bD$.
\begin{lem}\label{lem:products-cpt}
	 {If $\prod_{i \in I} D_i$ is cocompact in $\bD$},
	 then $D_i \simeq 0$ for all but finitely many $i$ in $I$.
\end{lem}
\begin{proof}
 {Since $\prod_{i \in I} D_i \simeq \lim_F \prod_{i \in F} D_i$, with $F$ running over all finite subsets of $I$, is a cofiltered limit},
the identity morphism factors through some finite product:
\[ \id \colon \prod_{i \in I} D_i \to \prod_{i \in F} D_i \to \prod_{i \in I} D_i \ .\]
This implies $\id_{D_i} \simeq 0$ for all but finitely many $i$, and thus $D_i \simeq 0$ for all but finitely many $i$.
\end{proof}

Let $X$ be {a $G$-bornological coarse space, let $M$ be an equivariantly small sheaf on $X$},
and let $B$ be a bounded subset of $X$.

\begin{kor}\label{lem:eqsm-comps}
	 We have $M(B \cap C) \simeq 0$ for all but finitely many  {coarse components $C$ of $X$}.
\end{kor}
\begin{proof}
	 Since $(B \cap C)_{C \in \pi_0(X)}$ is a $ {U(\pi_0)}$-covering family of $B$ and $M$ is in particular a $ {U(\pi_{0})}$-sheaf, we have an equivalence
		\[ M(B) \simeq \prod_{C \in \pi_0(X)} M(B \cap C)\ .\]
		By the equivariant smallness condition  \cref{qreigoqrgwergwrgre} (applied to the trivial subgroup), we have  $M(B)\in \bC^\omega$. Now apply \cref{lem:products-cpt}.
\end{proof}

\begin{proof}[Proof of \cref{qergioewgergregwergergwergwreg}]
Recall that we consider a covering $p\colon Z\to X$ of $G$-bornological coarse spaces.
 {Let $M$ be an equivariantly small sheaf on $X$.}
Let $H$ be a subgroup of $G$ and let $Y$ be an $H$-bounded subset of $Z$. Then we want to show that $L^{\pi_0}\hat p^{*,G} M(Y)$ is a cocompact object of $\bC^{H}$.
 We can choose a bounded subset $B$ of $Z$ such that $Y = HB$. Since $p$ is a covering (\cref{ergwergggqrgqregqergq}) there exists a finite, coarsely disjoint family $(B_{j})_{j\in J}$
of subsets of $B$ such that $B = {\bigcup_{j\in J} B_j}$ and such that the induced map on coarse closures ${p_{|[B_{j}]}\colon} [B_j] \to [p(B_j)]$ is an isomorphism of coarse spaces for every $j$ in $J$.
 We fix $j$ in $J$ for the moment.
 Then for every coarse component $C$ of $Z$ we have $B_{j}\cap C=\emptyset$ or
 an equivalence
 \[ L^{\pi_0}\hat p^{*,G}M(B_j \cap C) \simeq M(p(B_j) \cap p(C))  \ . \]
 Note that in the latter case $p_{|B_{j}\cap C}:B_{j}\cap C\to p(B_j) \cap p(C)$ is an isomorphism by \cref{wefgihjwiegwergrwrg}.\eqref{wrkojgwergrefwrfrf}.
Since $M$ is equivariantly small, 
 {\cref{lem:eqsm-comps} and condition \eqref{qweiuqhfqfeffqeqf} of \cref{ergwergggqrgqregqergq}} imply that we have $L^{\pi_0}\hat p^{*,G}M(B_j \cap C)\simeq 0$  for all but finitely many   $C$  in $\pi_{0}(Z)$.  Hence, without changing the value $M(Y)$, we can replace $B$ (and hence $Y=HB$)
 by a smaller subset which is contained in the union of finitely many coarse components of $Z$.
 In addition, we can then assume that $B_j$ is coarsely connected for every $j$ in $J$.
	
 Let $\sim$ denote the equivalence relation on $J$ such that $j \sim j'$ if there exists $h$ in $H$ with $h[B_j] = [B_{j'}]$. {We choose a set $S$   of representatives of $J/\sim$. Note that $S$ is finite.
 For every $j$ in $J$ we choose $h_{j}$ in $H$ such that  $h_j[B_j] = [B_{s}]$ with $s$ in $S$ and $s\sim j$.  For every $s$ in $S$ we then
 define  the subset
 \[ B'_{s} := \bigcup_{j\sim s} h_{j}B_{j} \] of $Y$.} {We observe that $(B_s')_{s\in S}$ is a finite family of coarsely connected subsets,   $(HB'_{s})_{s\in S}$ is a coarsely disjoint family of subsets, and that $Y=\bigcup_{s\in S} HB_{s}'$.}
 We have the following chain of equivalences in $\bC^{H}$	
 \[ L^{\pi_0}\hat p^{*,G}M(Y) \simeq \prod_{s\in S}  L^{\pi_0}\hat p^{*,G}M(HB_s') \simeq \prod_{s\in S} \coind_{H_{[B'_s]}}^H M(p(H_{[B'_s]}B'_s))\ ,\]
 where the first follows from the sheaf condition on $M$, and the second is given by 	 \cref{prop:transfer-coind}. 		Since $\coind_{H_{[B'_s]}}^H$ preserves cocompact objects by the second assertion of \cref{cor:coind}, we see that  {$L^{\pi_0}\hat p^{*,G} M(Y)$ is cocompact in $\bC^{H}$}.
 This finishes the proof of \cref{qergioewgergregwergergwergwreg}.
\end{proof}
 
 To close this section, we record some consequences of the equivariant smallness condition which will only come to bear in \cref{sec:calculations}. The common theme among them is to translate the equivariant smallness condition into a more concrete cocompactness property for sheaves on some specific $G$-bornological coarse space.
 
 \begin{rem}\label{egiojwotgwegergwgrwregwrg}
 If $Y$ is a $G$-coarse space admitting a maximal entourage $U$, e.g, $Y=X_{min}$ for a $G$-set $X$, then we have
 $\Sh^{G}_{\bC}(Y)= \Sh^{U,G}_{\bC}(Y)\in {\CL}$ by \cref{qrgioergrewwergwgregrweg}.
 Consequently, we can consider cocompact objects in $\Sh^{G}_{\bC}(Y)$.
\end{rem}

 Let $X$ be  {a set}.
 Recall the bornological coarse space $X_{min,max}$ from \cref{qrgioqjrgoqrqfewfeqfqewfe}.
For $x$ in $X$, let $i_x \colon \{x\} \to X$ denote the inclusion map.
 
\begin{lem}\label{lem:eqsm-minmax}
	 There exists a commutative diagram \begin{equation}\label{eqrlwkenlweergegw}
\xymatrix{
	 \Sh^{\eqsm}_\bC(X_{min,max})\ar[r] & \Sh_\bC(X_{min,max})\ar[d]^-{M\mapsto (M(\{x\})_{x}} \\  
	 \coprod_X \bC^\omega\ar[u]^-{\sum_{x \in X} \hat i_{x,*}}\ar[r] & \prod_X \bC \ .
	}
\end{equation}
	in which both vertical maps are equivalences. In particular,
	 \begin{equation}\label{ergggrgewrg243g3rverv2rf3f}
 \Sh^{\eqsm}_\bC(X_{min,max}) \simeq \Sh_\bC(X_{min,max})^\omega\ .
\end{equation}
\end{lem}
\begin{proof}
	The upper horizontal morphism is the canonical inclusion. The lower horizontal morphism is the composition  
\[ \coprod_{X}\bC^{\omega}\stackrel{\eqref{wqefqoijfqoiwffqefefewfqefqe}}{\simeq} (\prod_{X}\bC)^{\omega}\hookrightarrow  \prod_{X}\bC\ . \]
	The right vertical morphism is an equivalence in view of the sheaf condition for $M$ since $(\{x\})_{x\in X}$ is a $\diag(X)$-covering family of $X$.
	
	Note that $X$ is a bounded subset of $X_{min,max}$.
	It follows from the definition of equivariant smallness and \cref{lem:eqsm-comps} that $M$ in $\Sh_\bC(X_{min,max})$ is equivariantly small if and only if $M(\{x\}) \in \bC^\omega$ for all $x$ in $X$ and $M(\{x\}) \simeq 0$ for all but finitely many $x$ in $X$.  This implies that the left vertical morphism is well-defined and an equivalence, too.
The up-right-down composition in 	\eqref{eqrlwkenlweergegw} is canonically equivalent to the lower horizontal map.
	
Since the lower horizontal morphism is equivalent to the inclusion of the cocompact objects of its target, the same also applies to the upper horizontal map. 
This shows the second assertion.	
\end{proof}

 Let $X$ be  {a $G$-set}
 and consider the embedding \eqref{wergoihiorgohjwegoiwre} for $H:=\{1\}$.
 It induces an embedding  of bornological coarse spaces 
\[ \res^{G}_{\{1\}}X_{min,max} \to \res^{G}_{\{1\}}(G_{min,min} \otimes X_{min,max})\ .\]
  The diagonal map in 
\eqref{wregpokopwtglmrbletemblwtwb}  provides an equivalence \begin{equation}\label{reoijiojoiervevwev}
\theta \colon \Sh^G_\bC(G_{min,min} \otimes X_{min,max}) \xrightarrow{\simeq} \Sh_\bC (\res^{G}_{\{1\}}X_{min,max})\ .
\end{equation} 
\begin{lem}\label{lem:eqsm-Gminmax}
	The equivalence $\theta$ restricts to an equivalence
	\begin{equation}\label{rewgpogoegergergwergewgergewrg}
\theta^{\eqsm} \colon \Sh^{G,\eqsm}_\bC(G_{min,min} \otimes X_{min,max}) \xrightarrow{\simeq} \Sh^\eqsm_\bC (\res^{G}_{\{1\}}X_{min,max})\ .
\end{equation} 
In particular, we have an equivalence \begin{equation}\label{3roihfeoirhjo3f123feffqewfq}
 \Sh^{G,\eqsm}_\bC(G_{min,min} \otimes X_{min,max})\simeq  \Sh^{G}_\bC(G_{min,min} \otimes X_{min,max})^{\omega}\ .
\end{equation}
\end{lem}
\begin{proof}
The second assertion of the lemma follows from the first combined with \eqref{reoijiojoiervevwev} and the second assertion of \cref{lem:eqsm-minmax}.

We now show the first assertion.
The equivalence $\theta$   has a decomposition
\begin{align*} \Sh^{G}_{\bC}(G_{min,min}\otimes X_{min,max})&\xrightarrow{r^{G}_{\{1\}}} \Sh_{\bC}(\res^{G}_{\{1\}} (G_{min,min}\otimes X_{min,max}))\\&\xrightarrow{\hat i^{*}}  \Sh_{\bC}(\res^{G}_{\{1\}} X_{min,max})\ .\end{align*}
The morphism $r^{G}_{\{1\}}$ obviously  preserves equivariantly small objects, and the morphism $\hat i^{*}$ preserves equivariantly small objects by 
\cref{lem:restrict-eqsm}. Hence the functor $\theta^{\eqsm}$ in \eqref{rewgpogoegergergwergewgergewrg} is well-defined.
 Since it is a restriction of an equivalence, it is clearly fully faithful.
 It remains to show that $\theta^{\eqsm}$ is essentially surjective. 
     
   Let $M$ be an object of  $\Sh^G_\bC(G_{min,min} \otimes X_{min,max})$   such that $\theta(M)$   is equivariantly small in $\Sh_\bC(\res^{G}_{\{1\}}X_{min,max})$. We claim that  then $M$ itself is equivariantly small.   
  The claim (for all $M$) immediately implies that $\theta^{\eqsm}$ is essentially surjective. 
	
	 Let $H$ be a subgroup of $G$ and $Y$ be an $H$-bounded subset of $G_{min,min} \otimes X_{min,max}$.
Then we must show that  {$M(Y)$ is cocompact in $\bC^{H}$}.

	 We can choose  a bounded subset  $B$    of $G_{min,min} \otimes X_{min,max}$ such that $Y = HB$. After replacing $B$ by a smaller subset if necessary, we can assume that there exists a finite subset $F$ of $G$  and a family $(X_{f})_{f\in F}$ of subsets of $X$  
such that the projection $F\to G/H$ is injective and 	$B = \bigcup_{f \in F} \{f\} \times X_f$.  Then $(H(\{f\} \times X_f))_{f\in F}$ is a coarsely disjoint family of  $H$-bounded subsets of $G_{min,min}\otimes X_{min,max}$   and  $Y = \bigcup_{f \in F} H(\{f\} \times X_f)$. Since $M(Y)\simeq \prod_{f\in F} 
M(H(\{f\} \times X_f))$ by the sheaf condition on $M$, it suffices to show that  {$M(H(\{f\} \times X_f))$ is cocompact in $\bC^{H}$} for every $f$ in $F$. 

 We now fix $f$ in $F$, set $A:=X_{f}$,  and consider  {the $H$-bornological coarse space} $Z:=H(\{f\} \times A)$  with the structures induced from the embedding $Z\to \res^{G}_{H}(G_{min,min} \otimes X_{min,max})$.

We have a diagram 
\[ \xymatrix{Z\ar[rr]^{\incl}&& \res^{G}_{H}(G_{min,min}\otimes X_{min,max})\\&H \otimes_{\{1\}} A_{min,max}\ar[ul]_{\cong}^{m}\ar[ur]^{\iota}&} \] in $H\BC$,
	where $\iota$ is  as in \eqref{qrg0rig0qrgqrgqrgqrg} (for $H$ in place of $G$, $\{1\}$ in place of $H$, $\res^{G}_{H}(G_{min,min}\otimes X_{min,max})$ in place of $X$, and $\{f\}\times A$ in place of $Y$),
	and $m$ is an isomorphism given by $(h,a)\mapsto (hf,ha)$.
	 By \cref{lem:transfer-coind}, we have the middle square of the following
	 commutative diagram
	 \[\xymatrix@C=3em{
	 	\Sh^H_\bC( G_{min,min} \otimes X_{min,max})\ar[d]_{r^{G}_{H}} \ar[r]^{\simeq}_{\theta}& \Sh_{\bC}(\res^{G}_{\{1\}}X_{min,max})\ar[d]^{\ev_{A}} \\
		\Sh^H_\bC(\res^{G}_{H}(G_{min,min} \otimes X_{min,max}))\ar[r]^-{\ev_{\{e\}\times A}}\ar[d]_-{\hat\iota^{*,H}} &\bC\ar[d]^-{\coind_{\{1\}}^H} \\
		\Sh^H_\bC(H \otimes_{\{1\}} A_{min,max})\ar[r]^-{\ev_{H \times A}}\ar[d]_-{\hat m^{H}_{*}}^-{\simeq} & \bC^H \\
		\Sh^H_\bC(Z) \ar[ur]_-{\ev_Z} & \\
	}\ .\]
	The upper square and the lower triangle obviously commute.
	The  {composition along the bottom left corner}
	sends $M$ in $\Sh^H_\bC( G_{min,min} \otimes X_{min,max})$ to $M(Z)$ in $\bC^{H}$.
	The filler of the diagram provides an equivalence $M(Z) \simeq \coind_{\{1\}}^{H} \theta(M)(A)$.
	 {Since $\theta(M)(A)$ is cocompact in $\bC$} by assumption and $\coind_{\{1\}}^H$ preserves cocompact objects by the second assertion of \cref{cor:coind},
	we conclude that  {$M(Z)$ is cocompact in $\bC^{H}$}.	  
	This finishes the proof of the claim and hence of the lemma.
\end{proof}

Let $X$  {be a $G$-set},
and let $\bC$ be in $\Fun(BG,\CL)$. 
The following proposition is an analogue of the second assertion of \cref{lem:eqsm-Gminmax} 
in the case where we replace the minimal  coarse structure of $G$ by the canonical one (\cref{etwgokergpoergegregegwergrg}). If $G$ is infinite, then $\Sh^G_\bC(G_{can,min} \otimes X_{min,max}) 
$ is not expected to admit all cofiltered limits, i.e.~that it only belongs to $\CLL$ instead of $\CLLL$ or even $\CL$. In particular,  it does not makes sense to consider cocompact objects therein. 
But $\PSh^G_\bC(G_{can,min} \otimes X_{min,max})$ belongs to $ \CL$ and therefore admits a good notion of cocompact objects.  
\begin{prop}\label{prop:Gcan-cpt}
	We have an inclusion
	\[ \Sh^{G,\eqsm}_{\bC}(G_{can,min} \otimes X_{min,max}) \subseteq \PSh^G_\bC(G_{can,min} \otimes X_{min,max})^{\omega}\ . \]
\end{prop}
For the   proof  we need the following generalisation of \cite[Prop.~5.3.4.13]{htt}.  
Let $K$ be a $G$-finite $G$-simplicial set,  {i.e., a $G$-simplicial set with only finitely many $G$-orbits of non-degenerate simplices}.
By $K^{0}$ we denote the $G$-set of vertices of $K$. 
For $k$ in $K^{0}$ we let $G_{k}$ denote the stabilizer of $k$ in $G$. Let $\bC$ be in $\Fun(BG,\CL)$.
Then we have  $\Fun^{G}(K,\bC)\in \CL$.
 Fix $M$ be in $\Fun^{G}(K,\bC)$.  
For $k$ in $K^{0}$ we consider $M(k)$ in $\bC^{G_{k}}$ in the natural way.
 \begin{lem}\label{lem:fun-cpt}
  {If $M(k)$ is cocompact in $\bC^{G_k}$} for all $k$ in $K^{0}$,
 then  {$M$ is cocompact in $\Fun^G(K,\bC)$}.
\end{lem}
\begin{proof}
	We first assume that $K$ is zero-dimensional and that {$K^0$ is a transitive $G$-set}.
	We fix a point $k$ in $K^{0}$. The diagonal equivalence in \eqref{fsdvlksmvklvsdvadvadvadsv} applied to
	$X:=*$, $H:=G_{k}$ (and using the identification $K=K^{0}\cong G/G_{k}$) provides an equivalence
	\[ \Fun^G(K,\bC) \simeq \Fun^{G_{k}}(*,\bC)\simeq \bC^{G_k}\ , \quad M\mapsto M(k)\ .\]
	 In this case, the lemma holds true for obvious reasons.
	
	 In the next step, we assume that $K=\Delta^{n}\times S$ for some transitive $G$-set $S$
	 and $n$ in $\nat$.
 We have an equivalence
	\[ \Fun^G(S \times \Delta^n,\bC) \simeq \Fun(\Delta^n, \Fun^G(S,\bC))\ .\]
	By \cite[Prop.~5.3.4.13]{htt} applied to the right-hand side of this equivalence, we have
	 {that $M$ is cocompact in $\Fun^G(S \times \Delta^n,\bC)$ if and only if
	the evaluation of $M$ at every vertex of $\Delta^n$ is cocompact in $\Fun^G(S,\bC)$}.
	By the case considered  in the first paragraph, the latter condition is implied by our assumption on $M$.
	
	Suppose now that we are given a pushout square
	\[\xymatrix{
		L'\ar[r]^-{f'}\ar[d]_-{i} & K'\ar[d]^-{j} \\
		L\ar[r]^-{f} & K
	}\]
	of $G$-simplicial sets where $i$ is a cofibration. Then the induced diagram of $\infty$-categories
	\[\xymatrix{
		\Fun(K,\bC)\ar[r]^-{j^*}\ar[d]_-{f^*} & \Fun(K',\bC)\ar[d]^-{f^{',*}} \\
		\Fun(L,\bC)\ar[r]^-{i^*} & \Fun(L',\bC)
	}\]
	is a pullback in $\CATi$, and hence also a pullback in $\Fun(BG, \CL)$.
	The restriction functors in this diagram all have right adjoints given by right Kan extension functors. These right adjoints are also  morphisms in $\Fun(BG,\CL)$.
	
	Applying $\lim_{BG}$ yields the pullback square
	\[\xymatrix{
		\Fun^G(K,\bC)\ar[r]^-{j^{*,G}}\ar[d]_-{f^{*,G}} & \Fun^G(K',\bC)\ar[d]^-{f^{',*,G}} \\
		\Fun^G(L,\bC)\ar[r]^-{i^{*,G}} & \Fun^G(L',\bC)
	}\]
	in $\CL$. Furthermore, all the functors in this square again have right adjoint morphisms in $\CL$. This implies that they all preserve cofiltered limits.
Therefore, we can apply \cite[Lem.~5.4.5.7]{htt}  to see that $M$ in $\Fun^G(K,\bC)$ is cocompact if $j^{*,G}M$,  $f^{*,G}M$, and $i^{*,G}f^{*,G}M$ are all cocompact.

 The general case follows from this observation by induction on the number of {non-degenerate} equivariant simplices in $K$.
\end{proof}

\begin{proof}[Proof of \cref{prop:Gcan-cpt}]
	Let $M$ be  {an equivariantly small sheaf on in $G_{can,min} \otimes X_{min,max}$}.
	Let $U$ be an entourage such that $M$ is a $U$-sheaf. After enlarging $U$ if necessary we may assume that $U = G(F \times F) \times \diag(X)$ for some finite subset $F$ of $G$.
	
	Let $F^{\prime}$ be a subset of $F$. Then $F'\times X$ is a bounded subset of $G_{can,min}\otimes X_{min,max}$.
	By \cref{lem:eqsm-comps} and since $X_{min,max}$ has the discrete coarse structure, there exists a finite subset $X_{F'}$ of $X$ such that
	$M(F'\times \{x\})=0$ for all $x$ in $X\setminus X_{F'}$. Since $F$ is finite, the subset $X_{0}:=\bigcup_{F'\in \cP(F)} X_{F'}$ of $X$ is again finite.
	Moreover, for every subset $F'$ of $F$ we have $M(F'\times \{x\})\simeq 0$ for all $x$ in $X\setminus X_{0}$.
	 
	We consider the $G$-invariant subset $G(F\times X_{0})$ of $G\times X$. Furthermore, by    $i \colon \cP^{U\bd}_{G(F \times X_0)} \to \cP_{G\times X}$  we denote the inclusion of the $G$-subposet  of $U$-bounded subsets of $G(F\times X_{0})$.  Then we have an adjunction
	\[ i^{*,G} \colon \PSh^G_{\bC}(G \times X)   \leftrightarrows \Fun^G(\cP^{U\bd,\op}_{G(F \times X_0)},\bC) \cocolon i_*^G\ .\]	
	We can consider $M$ as an object of $\PSh_{\bC}^{G}(G\times X)$ and claim that the unit $M\to i_{*}^{G}i^{*,G}M$ is an equivalence. By \cref{tieqorgfgregwegergw} and since $M$ is a $U$-sheaf, it suffices to show that
	$M(B)\to i_{*}^{G}i^{*,G}M(B)$ is an equivalence for every $U$-bounded subset $B$ of $G_{can,min}\otimes X_{min,max}$.
We thus must  show that the canonical morphism
	\begin{equation}\label{valknjvoiarvafvasdvadsv}
M(B)\to \lim_{B'\in ((\cP^{U\bd}_{G(F \times X_0)})_{/B})^{\op}} M(B')
\end{equation}
	is an equivalence.   
	
	Since $B$ is $U$-bounded, there exists $g$ in $G$ and a subset $F'$ of $F$ such that $B=g(F'\times \{x\})$. Any other subset of $B$ is then of the form $B'=g(F''\times \{x\})$ for some subset $F''$ of $F'$.
	 Since $M$ is a $G$-invariant sheaf,   we  have an equivalence    $M(B')\simeq   M(F''\times \{x\})$.	  	
We distinguish two cases:
\begin{enumerate}
	\item If $B\not\in   \cP^{U\bd}_{G(F \times X_0)}$, then  $x\not\in X_{0}$ and hence $M(B')\simeq 0$ for all $B'$ in $(\cP^{U\bd}_{G(F \times X_0)})_{/B}$. In this case  $M(B)\simeq 0$ and $\lim_{B'\in (\cP^{U\bd}_{G(F \times X_0)})_{/B}^{\op}} M(B')\simeq 0$.
	\item If $B\in   \cP^{U\bd}_{G(F \times X_0)}$, then $B$ is final in $(\cP^{U\bd}_{G(F \times X_0)})_{/B}$
\end{enumerate}
Thus in both cases \eqref{valknjvoiarvafvasdvadsv} is obviously an equivalence.

Above we have seen that 
  	 $M\simeq i^G_*i^{*,G}M$.
	Since the restriction functor $i^{*,G}$ preserves cofiltered limits, $i^G_*$ preserves cocompact objects.
	Hence in order to show that  {$M$ is cocompact in $\PSh_{\bC}^{G}(G\times X)$},
	 it suffices to check that  {$i^{*,G}M$ is cocompact in $\Fun^G(\cP^{U\bd,\op}_{G(F \times X_0)},\bC)$}. 
	Since $F$ and $X_{0}$ are finite, 
	  the poset $\cP^{U\bd}_{G(F \times X_0)}$ is $G$-finite.
 {Since $M(B)$ is cocompact in $\bC^{G_B}$} for every bounded subset $B$ by assumption, we indeed have   $i^{*,G}M\in  \Fun^G(\cP^{U\bd,\op}_{G(F \times X_0)},\bC)^{\omega}$  by \cref{lem:fun-cpt}.
\end{proof}

 \subsection{Localisation}\label{sec:localization}
 If $f,g \colon X \to X'$ are two morphisms between of $G$-bornological coarse spaces, then 
 one says that $f$ and $g$ are close to each other if there exists a coarse entourage $U'$ of $X'$ containing the  {pairs} $(f(x),g(x))$ for all $x$ in $X$.
 A morphism $f:X\to X'$ between $G$-bornological  coarse is a coarse equivalence  if it  admits an inverse   up to closeness. It is one of the basic principles of coarse geometry that coarsely equivalent bornological coarse spaces should be considered  as  equivalent.  So invariants of bornological coarse spaces
 should be  constructed from functors which send pairs of close morphisms to equivalent morphisms. Such functors are called coarsely invariant, see also \cref{rewkgowegrerfrewfwr}.
  The functor  $\Sh^{G,\eqsm}_\bC$  is not coarsely invariant since morphisms  in  $\Sh^{G,\eqsm}_\bC(X)$    are natural transformations between sheaves, which are local in $X$.

In the present section, we turn the functor $\Sh^{G,\eqsm}_\bC$ into a coarsely invariant functor (see \cref{sec:coarseinv} below for proofs) by performing 
Dwyer-Kan localisations  {of} its values. 
The localisation of $ \Sh^{G,\eqsm}_\bC(X)$ adds morphisms which may propagate on $X$ in a way which is controlled by the coarse entourages of $X$. In order to calculate the mapping spaces of the localisation and verify its left-exactness, we use the
 calculus of fractions formula for mapping spaces introduced in \cref{rgiojreogiergregregreg}. 
    
 Consider a $G$-coarse space $X$.
 If $U$ is {an invariant coarse entourage of $X$ containing the diagonal, and $Y$ is a subset of $X$},
 then we have $U(Y)\subseteq Y$ (see \eqref{V-thin} {for the definition of the $U$-thinning $U(-)$}).
 We consider the family of these inclusions for all $Y$ as a transformation $U(-)\to \id$ of  endofunctors of {the power set} $\cP_{X}$.
 It induces a transformation $\id \to U^{G}_{*}$ of endofunctors of $\PSh^{G}_{\bC}(X)$ which by
 \cref{egiweogwergrgwgwerg} restricts to $\Sh^{G}_{\bC}(X)$.

We form the relative   $\infty$-category  $(\Sh^{G}_{\bC}(X),W_{X})$, where the subcategory $W_{X}$ {is  {generated
 by the collection of morphisms}
\begin{equation}\label{rqwfoijqoiwefdewdqdewdqd}\{M\to  U^{G}_{*}M \mid U\in \cC_{X}^{G,\Delta}\ , M\in \Sh_{\bC}^{G}(X)\} \ ,
\end{equation} 
where $\cC_{X}^{G,\Delta}$ denotes the poset of $G$-invariant coarse entourages of $X$ which contain the diagonal.

 \begin{ddd}\label{trhwrthwgrgrewgwreg}
We define $\hat \bV^{G}_{\bC}(X)$ as the Dwyer--Kan localisation
\[ \hat \bV^{G}_{\bC}(X) := \Sh^{G}_{\bC}(X)[W_{X}^{-1}]\ .\qedhere\]
\end{ddd}

Recall \cref{wegioegfvsfggdfgs} of the notion of a localisation among left-exact $\infty$-categories.
\begin{prop}\label{qriooqergegwgrer}\mbox{}
\begin{enumerate}
 \item For $M,N$ in $\Sh^{G}_{\bC}(X)$ there is a natural equivalence of mapping spaces:
\begin{equation}\label{qwefqwfoij1111}
 \colim_{U \in \cC^{G,\Delta}_X} \Map_{\Sh^G_\bC(X)}(M, U^G_*N) \xrightarrow{\simeq} \Map_{\hat \bV_{\bC}^{G}(X)}(\ell M,\ell N)\ .
 \end{equation}
 \item\label{qriooqergegwgrer:1} {The localisation functor $\Sh^{G}_{\bC}(X)\to \hat \bV^{G}_{\bC}(X)$ presents $\hat \bV^{G}_{\bC}(X)$ as the localisation among (large) left-exact $\infty$-categories.}
\end{enumerate}
\end{prop}
\begin{proof}
 {We first establish the formula for mapping spaces in $\Sh^{G}_{\bC}(X)[W_{X}^{-1}]$ using the    version of the material in \cref{sec:fractions} for the opposite categories.
 Let $N$ be in $\Sh^{G}_{\bC}(X)$.
 The functor $\pi \colon \cC^{G,\Delta}_{X} \to \Sh^{G}_{\bC}(X)$ given by $U \mapsto U^{G}_{*}N$ is a putative calculus of fractions at $N$ since $\diag(X)$ is initial in $\cC^{G,\Delta}_{X}$ and each morphism $N \to U^G_*N$ lies in $W_X$ by definition.
 By  \cref{prop:charactrightcalculus}, it suffices to show that the functor
 \[ \cM_N \colon \Sh^{G}_{\bC}(X) \to \Spc,\quad M \mapsto \colim_{U \in \cC^{G,\Delta}_{X}} \Map_{\Sh^G_\bC(X)}(M,U^G_*N) \]
 inverts all morphisms in $W_X$.}
 
 {Consider an arbitrary invariant entourage $V$ which contains the diagonal and an object $M$ in $\Sh^G_\bC(X)$.
 Using the adjunction \eqref{v-adj-ggg}, we can rewrite the induced map
\[ \colim_{U\in \cC^{G,\Delta}_{X}} \Map_{\Sh^{G}_{\bC}(X)}(V_{*}^{G}M, U^{G}_{*}N)\to 
\colim_{U\in \cC^{G,\Delta}_{X}} \Map_{\Sh^{G}_{\bC}(X)}( M , U^{G}_{*}N)\ .\]
 in the form 
\begin{equation}\label{vasvdvsdvsdvavd}
 \colim_{U\in \cC^{G,\Delta}_{X}} \Map_{\PSh^{G}_{\bC}(X)}( U^{*,G}V_{*}^{G}M, N) \to 
 \colim_{U\in \cC^{G,\Delta}_{X}} \Map_{\PSh^{G}_{\bC}(X)}(   U^{*,G}M, N)\ .
\end{equation}}
We will now see by a cofinality argument that this morphism is an equivalence.
For every  {subset $Y$ of $X$}
we have inclusions
\[ V(U[Y])\subseteq U[Y]\subseteq  V(VU[Y])\subseteq VU[Y]\ .\]
We thus get transformations 
 \begin{equation}\label{wqdqkerjeorijvoerqfefwef}
(VU)^{*,G}  M \to   (VU)^{*,G}V_{*}^{G}M\to  U^{*,G}  M\to   U^{*,G} V_{*}^{G} M\ .
\end{equation}
Since the functor $ \cC_{X}^{G,\Delta}\to\cC_{X}^{G,\Delta}$ given by $U \mapsto VU$ is cofinal,
 and in view of \eqref{wqdqkerjeorijvoerqfefwef},  {the two-out-of-six-property for equivalences implies that \eqref{vasvdvsdvsdvavd} is an equivalence}.
 
 {Since $ \cC_{X}^{G,\Delta}$ is filtered, we can conclude from formula \eqref{qwefqwfoij1111} that the functor $\ell$ preserves finite limits.
So \cref{ih43iugu34g34g3} implies the second assertion of the proposition.}
\end{proof}

Let now $X$ be {a $G$-bornological coarse space}.
In view of \cref{rgiwoegwergrwegwregwrg}, we can consider the  {relative} left-exact $\infty$-category
$(\Sh_{\bC}^{G,\eqsm}(X),W_{X}^{\eqsm})$, where $W_{X}^{\eqsm}$ is {the subcategory {generated} by the collection of morphisms}
\begin{equation}\label{bvsfvdvdfvrgherthrtgfew}
 \{M\to  U^{G}_{*}M \mid U\in \cC_{X}^{G,\Delta}\ , M\in \Sh^{G,\eqsm}_{\bC}(X)\}\ .
\end{equation}

\begin{ddd}\label{rtheorthertherthetrhe}
 {We define $\bV^{G}_{\bC}(X)$ as the Dwyer--Kan localisation
\[ \bV^{G}_{\bC}(X) := \Sh^{G,\eqsm}_{\bC}(X)[W_{X}^{\eqsm,-1}]\ .\qedhere\]}
\end{ddd}

\begin{prop}\label{wetiwoegreergwrger}\ 
\begin{enumerate}
\item For $M,N$ in $\Sh^{G,\eqsm}_{\bC}(X)$ there is a natural equivalence of mapping spaces:
 \begin{equation}  \label{qwefqwfoij11111}
  \colim_{U \in \cC^{G,\Delta}_X} \Map_{\Sh^G_\bC(X)}(M, U^G_*N) \xrightarrow{\simeq} \Map_{\bV_{\bC}^{G}(X)}(\ell M,\ell N) \ .
 \end{equation}
\item  {The localisation functor $\Sh^{G,\eqsm}_{\bC}(X)\to \bV^{G}_{\bC}(X)$ presents $\bV^{G}_{\bC}(X)$ as the localisation among left-exact $\infty$-categories.}
\end{enumerate}
\end{prop}
\begin{proof}
 This is shown by the same argument as for \cref{qriooqergegwgrer}.
\end{proof}
The universal property of the localisation provides the marked arrow in \begin{equation}\label{arfeqrfffweqq} 
\xymatrix{\Sh^{G,\eqsm}_{\bC}(X)\ar[r]\ar[d]_{\ell}&\Sh^{G}_{\bC}(X)\ar[d]^{\ell}\\\bV_{\bC}^{G}(X)\ar[r]^{!}&\hat \bV_{\bC}^{G}(X)}\ .\end{equation}
\begin{kor}\label{regklwgrregwrgegr}
The marked arrow in \eqref{arfeqrfffweqq} is a fully faithful inclusion. 
\end{kor}
\begin{proof}
We use that the upper horizontal arrow in \eqref{arfeqrfffweqq} is fully faithful, and that the mapping spaces in
$ \bV_{\bC}^{G}(X)$  and $\hat \bV_{\bC}^{G}(X)$  are given by the coinciding formulas \eqref{qwefqwfoij11111} and
\eqref{qwefqwfoij1111}, respectively.
\end{proof}

\begin{ex}\label{uihwieurghiwuerhgiurwegwregwregg}
 {Consider a set $X$ with the minimal coarse structure $X_{min}$.}
Then $\cC_{X}^{G,\Delta}=\{\diag_{X}\}$. Since for $M$ in $\Sh_{\bC}^{G}(X)$ the morphism
$M\to \diag_{X,*}^{G}M$ is equivalent to the   identity of $M$, the subcategory $W_{X}$ consists of equivalences. Consequently, the canonical morphism $\ell \colon \Sh^{G}_{\bC}(X)\to \hat \bV_{\bC}^{G}(X)$ is an equivalence.

Similarly, if $X$  {is a $G$-bornological coarse space which carries}
the minimal coarse structure, then $W_{X}^{\eqsm}$ consists of equivalences and
the canonical morphism $\ell \colon \Sh^{G,\eqsm}_{\bC}(X)\to  \bV_{\bC}^{G}(X)$ is an equivalence.
\end{ex}

\subsection{Functoriality of \texorpdfstring{$\bV^{G}$}{VG}}\label{vwerlivjwovwvwevewv}
In this section, we show that  {the left-exact $\infty$-category $\bV_{\bC}^{G}(X)$ introduced in \cref{sec:localization}} depends functorially on $X$ and $\bC$. 
 
If $f\colon X\to X^{\prime}$ is a morphism {of $G$-coarse spaces}
and $\bC$ is in $\Fun(BG,\CL)$, then we have the bold part of the following diagram
\begin{equation}\label{ergoijergerg}\xymatrix{\Sh^{G}_{\bC}(X) \ar[d]_{\ell_{X} }\ar[r]^{\hat f^{G}_{*} }&\Sh^{G}_{\bC}(X^{\prime}) \ar[d]^{\ell_{X^{\prime}} }\\
 \hat \bV^{G}_{\bC}(X) \ar@{..>}[r]^{f_{*} }&\hat \bV^{G}_{\bC}(X^{\prime}) }\ .
\end{equation}
 
\begin{lem}\label{iofjowefwfewfewfef}
The  morphism  $\hat f^{G}_{*} $   descends essentially uniquely to a morphism $f_{*} \colon \hat \bV^{G}_{\bC}(X)  \to \hat \bV^{G}_{\bC}(X^{\prime})$ in $\CLL$ completing the square \eqref{ergoijergerg}.
\end{lem}
\begin{proof}
 In view of the universal property of   Dwyer--Kan localisation{s},
 it suffices to show that $\ell_{X^{\prime}} \hat f^{G}_{*}$ sends
 {morphisms in} $W_{X}$ to equivalences in $\hat \bV^{G}_{\bC}(X^{\prime})$.
 
Let $V$ be {an invariant coarse entourage of $X$ containing the diagonal}
and set $f(V)_{\Delta}:= f(V)\cup\diag(X^{\prime})$.
 {Then $f(V)_\Delta$ is an invariant coarse entourage of $X'$ containing the diagonal}.
 {For a subset $Y^{\prime}$ of $X^{\prime}$}
we have an inclusion 
\[ f^{-1}(f(V)_{\Delta}(Y^{\prime}))\subseteq V(f^{-1}(Y^{\prime})) \]
of subsets of $X$.
The family of these inclusions for all  {subsets $Y^{\prime}$ of $X^{\prime}$}
induces a transformation 
\[ t\colon \hat f^{G}_{*}\circ   V^{G}_{*}\to  f(V)_{\Delta,*}^{G} \circ \hat f^{G}_{*}\colon  \Sh^{G}_{\bC}(X)\to  \Sh^{G}_{\bC}(X^{\prime})\ .\]
If   $M$ is  in $\Sh^{G}_{\bC}(X)$, then
we consider $\iota_{V}(M)\colon M\to  V^{G}_{*}M$ in $W_{X}$.
The morphism $\hat  f^{G}_{*}\iota_{V}(M) $
fits into the following sequence of morphisms in $\Sh^{G}_{\bC}(X^{\prime})$:
\[ \xymatrix@C=1.75em{\hat f^{G}_{*}M\ar@/^1cm/[rrr]^{!}_{\iota_{f(V)_{\Delta}}( \hat  f^{G}_{*} M)}\ar[rr]_-{\hat  f^{G}_{*}\iota_{V}(M) }&& \hat  f^{G}_{*}V^{G}_{*}M  \ar[r]^-{t}\ar@/_1cm/[rrrr]^{!}_{\iota_{f(V)_{\Delta}}(\hat  f^{G}_{*}V^{G}_{*} M)}& f(V)^{G}_{\Delta,*}\hat  f^{G}_{*}M \ar[rrr]^-{f(V)^{G}_{\Delta,*}\hat  f^{G}_{*}\iota_{V}(M) }&&&f(V)^{G}_{\Delta,*}\hat  f^{G}_{*}V^{G}_{*}M}\ .\]
 The morphisms marked by $!$ belong to $W_{X^{\prime}}$ and induce equivalences in
 $\hat \bV^{G}_{\bC}(X^{\prime})$.
 By the two-out-of-six property, we conclude that all morphisms in this diagram are sent to equivalences in $\hat \bV^{G}_{\bC}(X^{\prime})$.
 
 This shows that  $\ell_{X'} \hat  f^{G}_{*}W_{X} $  consists of equivalences in $\hat \bV^{G}_{\bC}(X^{\prime})$.
  \end{proof}
  
  Let $f\colon X\to X^{\prime}$ be a morphism {of $G$-bornological coarse spaces}.
 As a consequence of  \cref{eriogwetgwegregwrgregwrg}, \cref{regklwgrregwrgegr} and \cref{wetiwoegreergwrger} we get:
  \begin{kor}\label{qergiuqergreqgregwregwr}
  The morphism $f_{*}$ from \eqref{ergoijergerg} restricts to a morphism 
  \[ f_{*}\colon \bV^{G}_{\bC}(X)\to \bV^{G}_{\bC}(X^{\prime}) \]
 {of left-exact $\infty$-categories}.
  \end{kor}

Let $X$ be  {a $G$-coarse space},
and let $\phi\colon \bC\to \bC^{\prime}$ be a morphism in $\Fun(BG,\CL)$. Then we obtain the following diagram:
 \begin{equation}\label{ergoijergerg1}\xymatrix{\Sh^{G}_{\bC}(X) \ar[d]_{\ell_{X} }\ar[r]^{\hat \phi^{G}_{*}} &\Sh^{G}_{\bC}(X^{\prime}) \ar[d]^{\ell_{X } }\\
 \hat \bV^{G}_{\bC}(X) \ar@{..>}[r]^{\phi_{*} }&\hat \bV^{G}_{\bC^{\prime}}(X  )}\end{equation}
\begin{lem}\label{wetgowtgrgregwrg}
The morphism $\hat \phi^{G}_{*}$ descends essentially uniquely to a morphism $\phi_{*}\colon \hat \bV^{G}_{\bC}(X)\to \hat \bV^{G}_{\bC^{\prime}}(X)$
in $\CLL$
completing the square \eqref{ergoijergerg1}.
\end{lem}
\begin{proof}
We write $W_{\bC,X}$ and $W_{\bC',X}$ for the subcategories {generated by} \eqref{rqwfoijqoiwefdewdqdewdqd} for $\bC$ and $\bC'$, respectively. Then it is obvious that 
  $\hat \phi^{G}_{*}(W_{\bC,X})\subseteq W_{\bC',X}$. This implies the assertion. 
 \end{proof}

If  $X$ is  {a $G$-bornological coarse space},
then \cref{qrkgoqrgergwrgwegwrg} implies:
\begin{kor}\label{giowegrwergwergwreg}
The morphism $\phi_{*}$  from \eqref{ergoijergerg1}  restricts to 
a morphism
\[ \phi_{*}\colon \bV^{G}_{\bC}(X)\to \bV^{G}_{\bC^{\prime}}(X)\ .\]
\end{kor}

 Let $\tilde W_{X} $ denote the subcategory of $\Sh^{G}_{\bC}(X)$ of morphisms which are sent to equivalences by the localisation functor  $\ell_{X}$. By \cref{iofjowefwfewfewfef} and \cref{wetgowtgrgregwrg},
  the functor $(X,\bC )\mapsto \Sh^{ {G}}_{\bC}(X)$
  promotes to a functor 
    \begin{equation}\label{ergregfwrefrfwerfer}
\ell\Sh^{G}\colon G\Coarse\times \Fun(BG,\CL)\to  \mathbf{REL}^\mathrm{Lex}_{\infty,*} \ , \quad (X,\bC)\mapsto ( \Sh^{ {G}}_{\bC}(X),\tilde W_{X})\ ,\end{equation}
where $ \mathbf{REL}^\mathrm{Lex}_{\infty,*}$ is the large version of $\Rel^\mathrm{Lex}_{\infty,*}$ from \cref{rgiojreogiergregregreg}.
The large  left-exact $\infty$-categories defined in \cref{trhwrthwgrgrewgwreg} for all $X$ and $\bC$ now become the values of 
 the functor
   \begin{equation}\label{bojoijgoi3jg3g34fhat}
\hat \bV^{G} \colon  G\Coarse\times  {\Fun(BG,\CL)} \xrightarrow{\ell\Sh^{G} } \mathbf{REL}^\mathrm{Lex}_{\infty,*} \xrightarrow{\Loc}  \CLL\ ,
\end{equation}
 where $\Loc$ is the large version of the functor from \eqref{fqwepofjkqweopfdqewdq}.
 As a consequence of \cref{qergiuqergreqgregwregwr} and \cref{giowegrwergwergwreg}, we further obtain a subfunctor (the restriction along \eqref{gwegljgoiregwrgergregewg} is hidden)
  \begin{equation}\label{bojoijgoi3jg3g34f}
 \bV^{G}\colon G\BC\times \Fun(BG,\CL)   \xrightarrow{\ell\Sh^{G,\eqsm}_{\bC}}  {\Rel^\mathrm{Lex}_{\infty,*}} \xrightarrow{\Loc}  \Cle\end{equation}
 with the values prescribed by \cref{rtheorthertherthetrhe}.

Finally, consider a coarse covering $f\colon X^{\prime}\to X$ (\cref{wefgihjwiegwergrwrg})  and the square
\begin{equation}\label{qewfoijfiofefqewfqewfqewfefqffe}
\xymatrix@C=4em{\Sh_{\bC}^{G}(X )\ar[d]_{\ell_{X}}\ar[r]^{L^{\pi_{0},G}\hat f^{*,G}}&\Sh_{\bC}^{G}(X^{\prime})\ar[d]^{\ell_{X^{\prime}}}\\
\hat \bV_{\bC}^{G}(X)\ar@{..>}[r]^{f^{*}}&\hat \bV_{\bC}^{G}(X^{\prime})
}\ .
\end{equation}
\begin{lem}\label{iqerfjrfqwuef98weufeqwfqfe}\mbox{}\begin{enumerate}
\item \label{hbwrthggwgrew}
The morphism $L^{\pi_{0},G}\hat f^{*}$ descends essentially uniquely to a morphism $f^{*}$ completing the square \eqref{qewfoijfiofefqewfqewfqewfefqffe}.
\item \label{eewgergwergrewgwregwergwfw} If $f$ is a covering of bornological coarse spaces (\cref{ergwergggqrgqregqergq}),
then the  morphism $f^{*}$ (in \eqref{hbwrthggwgrew}) restricts to a morphism
\[ f^{*}\colon \bV_{\bC}^{G}(X)\to \bV_{\bC}^{G}(X^{\prime})\ .\] \end{enumerate}
\end{lem}
\begin{proof}
It follows from \cref{wergiowggwergwergreg} that $L^{\pi_{0},G}\hat f^{*,G}$ sends  $W_{X}$ to $W_{X^{\prime}}$.
This implies \eqref{hbwrthggwgrew}. \eqref{eewgergwergrewgwregwergwfw} is now an immediate consequence of \cref{qergioewgergregwergergwergwreg} and \cref{regklwgrregwrgegr}.
\end{proof}

The full contravariant functoriality of  $\bV_{\bC}^{G} $ will be part of the discussion in \cref{tiowtwrtbtwtbewbtw}.

\section{Properties of \texorpdfstring{$\bV^{G}_{\bC}$}{VGC}}\label{sec:props}

 We now study the behavior of $\bV^G_\bC$ as a functor on $G$-bornological coarse spaces. The results of this section are instrumental in showing that $\bV^G_\bC$ gives rise to a coarse homology theory upon application of a finitary localising invariant, which will be the subject of \cref{roijqreoeffqewfewfqewfefq}.
 
 {The first three subsection record elementary properties of our construction of controlled objects.
 \cref{sec:coarseinv} introduces the notion of coarse invariance, which encodes a kind of homotopy invariance on $G$-bornological coarse spaces, and shows that $\bV^G_\bC$ is coarsely invariant.
 \cref{sec:flasques} describes a special class of $G$-bornological coarse spaces called flasque spaces, and proves that the controlled objects over a flasque space admit an Eilenberg swindle.
 \cref{sec:u-cont} describes $\bV^G_\bC(X)$ as a filtered colimit whose individual stages only depend on coarse structures generated by a single entourage. We call this property $u$-continuity.}
 
 {The main non-trivial property of $\bV^G_\bC$ is a version of excision on $G$-bornological coarse spaces, which we formulate and prove in \cref{sec:excision}. The preceding \cref{sec:subspaces} prepares the proof of excision by showing that subspace inclusions induce fully faithful functors on $\bV^G_\bC$.}
 
 {Finally, \cref{sec:subspaces} introduces the free union of a collection of $G$-bornological coarse spaces and shows that the category of controlled objects over a free union is equivalent to the product of the categories of controlled objects over the individual components of the free union.}

\subsection{Coarse invariance}\label{sec:coarseinv}
Below we consider  the set $\{0,1\}$ with the trivial $G$-action. Recall \cref{qrgioqjrgoqrqfewfeqfqewfe} and the monoidal structure from \cref{rgiogjweoirgjwergwergwrg}.
 {For every $G$-bornological coarse space $X$}
the projection 
\begin{equation}\label{rfqwfqewqeq}
\{0,1\}{_{max,max}}\otimes X\to X
\end{equation} is a morphism  {of $G$-bornological coarse spaces}.

Let $\bM$ be some $\infty$-category and consider a functor  $E\colon G\BC\to \bM$. 
\begin{ddd}\label{rgiherogiergregreg}
$E$ is called coarsely invariant if the projection \eqref{rfqwfqewqeq}  induces an equivalence  
$E(\{0,1\}{_{max,max}} \otimes X)\to E(X)$ for every {$G$-bornological coarse space $X$}.
\end{ddd}

\begin{rem}\label{rewkgowegrerfrewfwr}
 We say that  two morphisms $f,g\colon X\to X^{\prime}$ in $G\BC$ are close to each other if $(f \times g)(\diag(X))$ is a coarse entourage of $X'$.
 {This is equivalent to requiring that $f$ and $g$ define a morphism of $G$-bornological coarse {spaces} $\{0,1\}_{max,{max}} \otimes X \to X'$}.
 Closeness is an equivalence relation on $\Hom_{G\BC}(X,X^{\prime})$ for all  {$G$-bornological coarse spaces} $X,X^{\prime}$ which is compatible with composition. 
 {A morphism of $G$-bornological coarse spaces} is a coarse equivalence if it can be inverted up to closeness.
     
 It is easy to see that the following assertions on $E$ are equivalent:
 \begin{enumerate}
 	\item $E$ is coarsely invariant.
	\item For every pair of close morphisms $f,g$ we have $E(f)\simeq E(g)$. \item $E$ sends coarse equivalences to equivalences. \qedhere
 \end{enumerate}
\end{rem}

\begin{ex}
 The following are examples of coarse equivalences. By \cref{rewkgowegrerfrewfwr}, these morphisms are sent to equivalences by any coarsely invariant functor.
 \begin{enumerate}
  \item For every set $X$, the projection $X_{max,max}\to *$ is a coarse equivalence.
 {The inclusion of} any point of $X$ is a coarse inverse.
  \item Consider the bornological coarse space $\R_d$ given by the set $\R$ with the bornology and coarse structure induced by the Euclidean metric.  Then the inclusion of the bornological coarse subspace $\Z$ into $\R$ is a coarse equivalence: the floor function is one choice of coarse inverse. \qedhere
 \end{enumerate}
\end{ex}
  
 \begin{lem}\label{riogogrgregergerg}
The functor
\[ \bV^{G}_{\bC} \colon  G\BC\to \Cle \]
   is coarsely invariant.
\end{lem}
\begin{proof}
We consider a pair $f,g\colon X\to X^{\prime}$ of morphisms which are close to each other. By \cref{rewkgowegrerfrewfwr} it suffices to show 
  that $\bV^{G}_{\bC}(f)\simeq  \bV^{G}_{\bC}(g)$.
   
     {Let $Y^{\prime}$ be a subset of $X^{\prime}$ and let $V^{\prime}$ be an invariant, symmetric coarse entourage of $X^{\prime}$ containing the diagonal}
      such that  $(f {\times} g)(\diag(X))\subseteq V^{\prime}$. Then we 
  have the following  chain of inclusions of subsets of $X$:
  \[ f^{-1}(Y^{\prime})\supseteq g^{-1}(V^{\prime}(Y^{\prime}))\supseteq f^{-1}(V^{\prime,2}(Y^{\prime}))\supseteq g^{-1}(V^{\prime,3}(Y^{\prime}))\ .\]
    For $M$ in $\Sh^{G,\eqsm}_{\bC}(X)$     we then  get induced  morphisms 
\begin{equation}\label{relkjferojoerfefreferf}
\xymatrix{\hat   f^{G}_{*}M\ar@/^1cm/[rr]^{!} \ar[r]& V^{\prime,G}_{*}\hat   g^{G}_{*}M \ar[r]\ar@/_1cm/[rr]^{!}& V^{\prime,2,G}_{*}\hat    f^{G}_{*}M\ar[r]& V^{\prime,3,G}_{*}\hat   g^{G}_{*}M}\ .
\end{equation}  
The marked morphisms
are sent to  equivalences in $\bV^{G}_{\bC}(X^{\prime})$ since they belong to the set \eqref{bvsfvdvdfvrgherthrtgfew} (for $X'$ in place of $X$) {generating} $W^{\eqsm}_{X^{\prime}}$. By the  {two-out-of-six} property for equivalences,
 all  morphisms in \eqref{relkjferojoerfefreferf} are sent to equivalences, in particular, the first one.
 Consequently, the second map in the zig-zag
\begin{equation*}\label{gregreggre343454}
\xymatrix@C=4em{\hat g^{G}_{*}M\ar[r]^-{ {\iota_V( \hat {g}_*^GM)}} &V^{\prime,G}_{*}\hat  g^{G}_{*}M &\hat  f^{G}_{*}M\ar[l]}
\end{equation*}
 is an equivalence. The morphism $\iota_V( \hat {g}_*^GM)$ is also an equivalence because it belongs to  $W^{\eqsm}_{X^{\prime}}$, too. Since all the arrows above are components of natural transformations between functors evaluated at $M$, we can
conclude that $f_*$ and $g_*$ are naturally equivalent functors.
  \end{proof}

 \subsection{Flasques}\label{sec:flasques}
Let $X$ be  {a $G$-bornological coarse space}.
\begin{ddd}[{\cite[Def. 3.8]{equicoarse}}] \label{rgiojgogregrgregre}
	$X$ is flasque if it admits an endomorphism
$f\colon X\to X$ with the following properties: \begin{enumerate}
\item \label{ergvpoerigrep0gregeg} $f$ is close to $\id_{X}$.
\item\label{ergoiergergregregregr} For every  $U$ in $\cC_{X}$ we have 
$\bigcup_{n\in \nat}  {f^{n}}(U)\in \cC_{X}$.
\item \label{sgrspoguo4tt} For every  $B$  in $\cB_{X}$  there exists   $n$ in $\nat$ such that $f^{n}(X)\cap B=\emptyset$. \qedhere
\end{enumerate}
\end{ddd}

We say that $f$ implements the flasqueness of $X$.
 
 \begin{ex}\label{ex:flasques}
 {The standard example of a flasque space is $[0,\infty)$ equipped with the metric bornological coarse structure induced by the Euclidean metric.
  In this case, flasqueness is implemented by the shift map sending $x$ to $x+1$.
  More generally, every $G$-bornological coarse space of the form $X \otimes [0,\infty)$ is flasque.}
 \end{ex}
 
\begin{ddd}\label{tiowgwtrgwergregwgrg}
$X$ is pre-flasque  if it admits an endomorphism
$f \colon X\to X$  with properties \cref{rgiojgogregrgregre} \eqref{ergoiergergregregregr} and \cref{rgiojgogregrgregre} \eqref{sgrspoguo4tt}.
\end{ddd}
We say that  $f$ implements the  pre-flasqueness of $X$. 

Let $\bM$ be a semi-additive $\infty$-category (\cref{ghiowjgoergwergwergregwg}). A semi-additive category is enriched in commutative monoids, and we use the symbol $+$ in order to denote the sum of morphisms. 

Let $M$ be an object of $\bM$.

\begin{ddd}\label{efiwrog34tt34t34t34t}
  $M $ is flasque  if it admits  an endomorphism   $ S$   such that \begin{equation}\label{qewflihqweiofqwef123r123r213}
 S\simeq \id_{M}+S\ .\qedhere
\end{equation}
\end{ddd}

We again say   that $S $ implements the flasqueness of $M$.
By $\bM^{\fl}$ we denote the smallest full subcategory of $\bM$ which is closed under filtered colimits and contains all flasque objects.
   
Let $E\colon G\BC\to \bM$ be a functor.

\begin{ddd}\label{rt98uerg9rt342trg32432f}
 $E$ preserves flasqueness if it sends flasque objects of $G\BC$   to objects in $\bM^{\fl}$.
\end{ddd}

\begin{rem}\label{erguiheriwgregregwgrwegwg}
If $\bM$ is additive (\cref{fqehwfiuhiqwef}), then a flasque object is a zero object. Consequently, if $M$ is additive, then $\bM^{\fl}$ consists of zero objects. In this case, the following conditions are equivalent:
\begin{enumerate}
\item $E$ preserves flasquenss.
\item $E$ vanishes on flasques, i.e., sends flasque $G$-bornological coarse spaces to zero objects. \qedhere
\end{enumerate}
 \end{rem}

Below in \cref{ergoegergregreg} \eqref{ergoewrpgwergregrgwgregwergneuneu1}
 we will show that the   functor $\bV_{\bC}^{G}$ is  flasqueness preserving. But note that $\bV_{\bC}^{G}$ is not the final object of consideration.  We will also need to show that certain functors derived from  $\bV^{G}$ by auxiliary constructions  {in} \cref{wiofoqrefwefeqfewfqewf} are flasqueness preserving, too. In this case it is important to keep track of the morphisms implementing the flasqueness of the values of the functor. We therefore
 introduce the stronger notions
of a  functorially flasqueness and functorially pre-flasqueness preserving functor, and we verify that $\bV^{G}$ has these properties.

We let $\bEnd(\bM):=\Fun(\Delta^{1}/\partial \Delta^{1},\bM)$ be the  category of endomorphisms of objects of $\bM$.
Its objects are pairs $(M,S)$ of an object $M$ of $\bM$ and an endomorphism $S:M\to M$.
We have  a forgetful functor $\bEnd(\bM)\to \bM$ sending $(M,S)$ to $M$.
 {Note that this functor preserves colimits.}

 \begin{ddd}\label{thgiojgoiggfergrewgerewergewrgreggw}\mbox{}
 We define the $\infty$-category
	  $\widetilde{\Fl}(\bM)$ 
	  by the pullback
	  \[\xymatrix{\widetilde{\Fl}(\bM)\ar[d]\ar[rrrrr]&&&&&\bEnd(\bM)\ar[d]^{\diag}\\\bEnd(\bM)\times_{\bM}\bEnd(\bM)\ar[rrrrr]^-{((M,S),(M,T))\mapsto ((M,S),(M,\id_{M}+T\circ S) )}&&&&&\bEnd(\bM)\times_{\bM}\bEnd(\bM)}\ .\]
	 We furthermore define $\Fl(\bM)$ as the pullback
	  \[\xymatrix@C=4em{
	   \Fl(\bM)\ar[r]\ar[d] & \widetilde{\Fl}(\bM)\ar[d]^{e} \\
	   \bM\ar[r]^-{M \mapsto (M,\id)} & \bEnd(\bM)
	   }\  ,\]
	   where the functor $e$ is given by composing $\widetilde{\Fl}(\bM) \to \bEnd(\bM) \times_{\bM} \bEnd(\bM)$ with the projection onto the second factor.
 \end{ddd}
 The objects of $\widetilde{\Fl}(\bM)$ are tuples $(M,S,T,\phi)$ with an equivalence
 \[ \phi \colon S \xrightarrow{\simeq} \id_{M}+T\circ S\ .\]
 Objects in $\Fl(\bM)$ are tuples $(M,S,T,\phi,\psi)$ with $(M,S,T,\phi)$ in $\widetilde{\Fl}(\bM)$ together with an equivalence $\psi \colon T \simeq \id$.
  In particular, $S$ implements flasqueness of $M$ in this case.
 There are forgetful functors
  \begin{equation}\label{fjqwiefojqweoidjoqwedqewd}
   \widetilde{p} \colon \widetilde{\Fl}(\bM) \to \bM \quad\text{and}\quad p \colon \Fl(\bM) \to \widetilde{\Fl}(\bM) \xrightarrow{\widetilde{p}} \bM
\end{equation}
which project to the underlying object.
By \cite[Lem.~5.4.5.5]{htt}, a diagram in $\widetilde{\Fl}(\bM)$ and $\Fl(\bM)$ admits a colimit if and only if its composition with $\widetilde{p}$ or $p$ admits a colimit,
and both $\widetilde{p}$ and $p$ preserve colimits.

 \begin{ddd}\mbox{}\begin{enumerate} 
 \item Let $\preFl(G\BC)$ be the  {full subcategory} of $\bEnd(G\BC)$ consisting of pairs $(X,f)$ where $f$ implements pre-flasqueness of $X$.
 \item Let $\Fl(G\BC)$ be the full subcategory of $\preFl(G\BC)$ of those pairs $(X,f)$,  where $f$ implements flasqueness of $X$.  \qedhere\end{enumerate}
 \end{ddd}
 
 Let $\bP$ be some auxiliary $\infty$-category  ($\Fun(BG,\CL)$ in our application), and let
 $E \colon G\BC\times \bP\to \bM$ be a functor.

\begin{ddd}\label{goiijgowgergegergregewrg}
 $E $ functorially preserves pre-flasqueness if
 there exist a functor
\[ \Flrm^{\pre}(E) \colon \preFl(G\BC)\times \bP\to \widetilde{\Fl}(\bM) \]
 {and a commutative diagram}
\begin{equation}\label{grelgijeroigergregergwergrewgwreg}
\xymatrix@C=4em{\preFl(G\BC)\times \bP\ar[dr]\ar[r]^-{\Flrm^{\pre}(E)} & \widetilde{\Fl}(\bM) \ar[d]^{e}\\ & \bEnd(\bM)\ ,}
\end{equation}
where the diagonal arrow is the functor sending $((X,f),P)$ to $(E(X,P),E(f,P))$
and $e$ is the forgetful functor from \cref{thgiojgoiggfergrewgerewergewrgreggw}.
\end{ddd}

  \begin{ddd}\label{goiijgowgergegergregewrg-111}
 $E $ functorially preserves flasqueness if
 {there exists a functor
 \[ \Flrm(E) \colon \Fl(G\BC)\times \bP\to \Fl(\bM) \]
 and a commutative diagram analogous to \eqref{grelgijeroigergregergwergrewgwreg} in which we replace $\preFl(G\BC)$, $\Flrm^{\pre}(E)$ and $\widetilde{\Fl}(\bM)$ by $\Fl(G\BC)$, $\Flrm(E)$ and $\Fl(\bM)$, respectively.}
  \end{ddd}
  
 The following is obvious from the definitions.
 \begin{kor}\label{tgiowergwregwregwregeg}
 If $E$  functorially preserves flasqueness, then $E(-,P)$ preserves flasqueness for every object $P$ in $\bP$.
 \end{kor}

 {A priori, showing that a functor is functorially flasqueness preserving requires us to produce more structure than showing that it is pre-flasqueness preserving.
The next lemma asserts that this additional structure comes for free if the functor is coarsely invariant.}

\begin{lem}\label{qwrgfiqojwfgwefwefewfqwefe}
Assume:
\begin{enumerate}
\item $E$ functorially preserves pre-flasqueness.
\item For every $P$ in $\bP$ the functor $E(-,P) \colon G\BC\to\bM$ is coarsely invariant.
\end{enumerate}
Then  $E$  functorially preserves flasqueness.
\end{lem}
\begin{proof}
 {By assumption, the composite
 \[ \Fl(G\BC) \times \bP \to \preFl(G\BC) \times \bP \xrightarrow{\preFl(E)} \widetilde{\Fl}(\bM) \xrightarrow{e} \bM \]
 is equivalent to the functor sending $((X,f),P)$ to $(E(X,P), E(f,P))$.
 The inclusions $\{0\} \to \{0,1\}$ and $\{1\} \to \{0,1\}$ induce a zig-zag of natural equivalences
 \[\xymatrix{
  E(X,P)\ar[d]_{E(f,P)}\ar[r]^-{\simeq} & E(X \otimes \{0,1\}_{max,max},P)\ar[d]^{E(f \sqcup \id_X,P)} & E(X,P)\ar[l]_-{\simeq}\ar[d]^{\id} \\
  E(X,P)\ar[r]^-{\simeq} & E(X \otimes \{0,1\}_{max,max},P) & E(X,P)\ar[l]_-{\simeq}
  }\]
  which induce a lift
  \[ \Flrm(E) \colon \Fl(G\BC) \times \bP \to \Fl(\bM) \]
  of $\Flrm^{\pre}(E)_{|\Fl(G\BC) \times \bP}$.
  It follows that $E$ functorially preserves pre-flasqueness.}
  \end{proof}

Let $\bC $ be in $\Fun(BG,\CL)$. Note that $\Cle$ is semi-additive by \cref{girgjowegfewfw9ef}.
 So \cref{rt98uerg9rt342trg32432f,goiijgowgergegergregewrg,goiijgowgergegergregewrg-111} apply to the functor
 \[ \bV^{G} \colon G\BC\times \Fun(BG,\CL)\to \Cle \]
 from \eqref{bojoijgoi3jg3g34f}.
In this case we use the more natural symbol $\times$ instead of $+$ for the sum of morphisms.

 \begin{lem} \label{ergoegergregreg} \mbox{}
 \begin{enumerate}
 \item \label{ergoewrpgwergregrgwgregwerg}
 The functor $\bV^{G}$ functorially  preserves pre-flasqueness. \item \label{ergoewrpgwergregrgwgregwergneuneu}
 The functor $\bV^{G} $ functorially  preserves flasqueness.
 \item\label{ergoewrpgwergregrgwgregwergneuneu1}
 The functor $\bV^{G}_{\bC}$  preserves flasqueness for every $\bC$ in $\Fun(BG,\CL)$.
 \end{enumerate}
\end{lem}
\begin{proof}
We start with  \eqref{ergoewrpgwergregrgwgregwerg}. 
 If we apply the functor from \eqref{regewgk2p5getgwreg} (for the trivial group) to $G$-coarse spaces and coefficient categories 
 with $G$-action, then we get a functor
 \[ \Sh\colon G\Coarse\times \Fun(BG,\CL)\to \Fun(BG,\CLL)\ . \] 
  We extend it to a functor
  \[ \Flrm^{\pre}(\Sh)\colon \preFl(G\BC)\times \Fun(BG,\CL)\to \widetilde{\Fl}(\Fun(BG,\CLL))\ .\]
   such that
   \begin{equation}\label{vqlkhjklwejflkqwefewqfewfef}
   	\Flrm^{\pre}(\Sh) ((X,f),\bC):=(\Sh_{\bC}(X),\prod_{n\in \nat} \hat f^{n}_{*}, \hat f_{*} ,\hat \phi)\ ,
\end{equation}
where $\hat \phi$ is the canonical identification $\prod_{n\in \nat} \hat f^{n}_{*}\simeq \id_{\Sh_{\bC}(X)}\times \hat f_{*}\circ  \prod_{n\in \nat} \hat f^{n}_{*}$.
 We must argue that the infinite product $\hat  S(X,f):=\prod_{n\in \nat} \hat f^{n}_{*}$ of morphisms in \eqref{vqlkhjklwejflkqwefewqfewfef}  exists.
 For $M$ in $\Sh_{\bC}(X)$, by \cref{rgiojgogregrgregre} \eqref{ergoiergergregregregr} there exists $U$ in $\cC_{X}^{G}$ such that $\hat f^{n}_{*}(M)$ in $\Sh_{\bC}^{U}(X)$ for all $n$ in $\nat$.  Since $\Sh_{\bC}^{U}(X)$ belongs to $\CL$, the product $\prod_{n\in \nat} \hat f^{n}_{*}(M)$ exists in
$ \Sh_{\bC}^{U}(X)$ and hence in $\Sh_{\bC}(X)$.
 
By  composing $\Flrm^{\pre}(\Sh)$ with $\lim_{BG}$  we obtain a functor 
\[ \Flrm^{\pre}(\Sh)^{G}\colon \preFl(G\BC)\times \Fun(BG,\CL)\to \widetilde{\Fl}(\CLL)\ ,\]
\[ ((X,f),\bC)\mapsto (\Sh^{G}_{\bC}(X),\hat   S(X,f)^{G}, \hat f^{G}_{*},\hat \phi^{G})\ .\]
Consider the square
\[ \xymatrix{
	\Sh^{G}_{\bC}(X)\ar[r]^{\hat S(X,f)^{G}}\ar[d]_{\ell_{X}} & \Sh^{G}_{\bC}(X)\ar[d]^{\ell_{X}} \\
	\hat \bV^{G}_{\bC}(X)\ar@{..>}[r]^{S(X,f)} & \hat \bV^{G}_{\bC}(X)}\ .\]
Similarly as in the proof of  \cref{iofjowefwfewfewfef}, using   \cref{rgiojgogregrgregre} \eqref{ergoiergergregregregr} on $f$,   we check that
$\ell_{X} \hat   S(X,f)^{G}$ sends the generators \eqref{rqwfoijqoiwefdewdqdewdqd}  of $W_{X}$ to equivalences in $\hat \bV^{G}_{\bC}(X)$.
It follows that
$\hat S(X,f)^{G}$ descends to the desired functor
\begin{equation*}
S(X,f)\colon\hat \bV^{G}_{\bC}(X) \to \hat \bV^{G}_{\bC}(X)\ . \end{equation*}
Similarly, the functor $ \hat f^{G}_{*}$ descends to a functor $  f_{*}$,  and $\hat \phi^{G}$ descends to an equivalence $\phi^{G} \colon  S(X,f)\stackrel{\simeq}{\to}\id_{\hat \bV^{G}_{\bC}(X)}\times
 {(f_{*} \circ S(X,f))}$.
 
  It remains to show that $S(X,f)$ preserves the full subcategory $\bV^{G}_{\bC}(X)$ of $\hat \bV^{G}_{\bC}(X)$. 
  To this end, we show that $\hat   S(X,f)^{G}$ preserves the subcategory $\Sh_{\bC}^{G,\eqsm}(X)$ of $\Sh_{\bC}^{G}(X)$.
 
 Let $M$ be in $\Sh^{G,\eqsm}_{\bC}(X)$.  Consider an $H$-bounded subset  $Y$ of $X$.
 By \cref{rgiojgogregrgregre}~\eqref{sgrspoguo4tt}, there exists $n_{0}$ in $\nat$ such that   $f^{n_0}(X) \cap Y = \emptyset$. Then we have
 \[ (\hat S(X,f)^G(M))(Y) \simeq \prod_{n \leq n_0} \hat f^{n,G}_*M(Y) \ .\]
 By \cref{eriogwetgwegregwrgregwrg}, $\hat f^{n,G}_*M$ is equivariantly small for all $n$.
 Thus every factor of the product  belongs to $\bC^{H,\omega}$, and therefore the finite product, too.
We conclude that $\hat S(X,f)^G(M)\in \Sh^{G,\eqsm}_{\bC}(X)$.
 This finishes the proof of assertion \eqref{ergoewrpgwergregrgwgregwerg}. 
 
 \eqref{ergoewrpgwergregrgwgregwergneuneu} follows from
 \eqref{ergoewrpgwergregrgwgregwerg} by   \cref{qwrgfiqojwfgwefwefewfqwefe}
 since $\bV^{G}_{\bC}$ is coarsely invariant by   \cref{riogogrgregergerg}. 
 
 \eqref{ergoewrpgwergregrgwgregwergneuneu1} follows from 
  \eqref{ergoewrpgwergregrgwgregwergneuneu}  and  \cref{tgiowergwregwregwregeg}.
\end{proof}

  \subsection{\texorpdfstring{$u$-continuity}{u-continuity}}\label{sec:u-cont}
  Let $X$ be  {a $G$-bornological coarse space}.
  If $U$ is  {an invariant coarse entourage on $X$},
  then $X_{U}$ denotes the $G$-bornological coarse space obtained from $X$ by replacing the original coarse structure of $X$ by the $G$-coarse structure $\cC\langle \{U\}\rangle$ generated by $U$.
  There is a canonical morphism $X_{U}\to X$ given by the identity of the underlying sets.

Let $\bM$ be an $\infty$-category which admits all small filtered colimits.  
 Consider a functor $E \colon G\BC \to \bM$.

\begin{ddd}\label{rgiuerogergergergre}
	$E$ is $u$-continuous if the natural morphism
	\[ \colim_{U\in \cC_{X}^{G}} E(X_{U})\to E(X) \]
	is an equivalence for every {$G$-bornological coarse space} $X$.
\end{ddd}

 Let $\bC$ be in $\Fun(BG,\CL)$ and note that $\CLL$ admits small filtered colimits by (the large version of) \cref{prop:catex finitely complete}.
Recall the functor $\hat \bV^G_\bC$ from \eqref{bojoijgoi3jg3g34fhat}.
\begin{prop}\label{prop:u-continuous}
	The functor $\hat \bV^G_\bC \colon G\BC\to \CLL$ is $u$-continuous.
\end{prop}
\begin{proof}
Let $X$ be  {a $G$-bornological coarse space}.
	Filtered colimits and limits in  $\CLL$ may by computed in  $\CATi$.
	Consider the following chain of functors:
	\begin{align*}
	\colim_{U \in \cC^G_X} \hat \bV^G_\bC(X_U)
	&\simeq \colim_{U \in \cC^G_X} \left((\lim_{BG} \Sh_\bC(X_U))[W^{-1}_{X_U}]\right) \\
	&\xrightarrow{i} \left(\colim_{U \in \cC^G_X} \lim_{BG} \Sh_\bC(X_U)\right)\left[\colim_{U \in \cC^G_X} W^{-1}_{X_U}\right] \\
	&\xrightarrow{ii} \left(\lim_{BG} \colim_{U \in \cC^G_X} \Sh_\bC(X_U)\right)\left[\colim_{U \in \cC^G_X} W_{X_U}^{-1}\right] \\
	&\xrightarrow{iii} \left(\lim_{BG} \Sh_\bC(X)\right)[W^{-1}_{X}] \\
	&\simeq \hat \bV^G_\bC(X)
	\end{align*}
	The morphism marked by $i$ is an equivalence since the functor $\Loc$ in \eqref{rghvfigvsfgfdsg} is a left adjoint and therefore preserves all colimits.
	The morphism marked by $ii$ is an equivalence by \cref{lem:lim-colim-commute} 
	since $BG$ has only a single object,  {the poset of invariant coarse entourages} $\cC_{X}^{G}$ is filtered, and $\Sh(X_U) \to \Sh(X_{U'})$ is fully faithful for all $U, U'$ in $\cC^{G}_{X}$ with $ U \subseteq U'$.  
	The morphism marked by $iii$ is  an equivalence since the map
	\[ \cC^G_X \to \{ (U,V) { \in \cC^{G}_{X}\times \cC^{G}_{X}} \mid U \in \cC^G_X,\ V \in \cC^G_{X_U} \}, \quad U \mapsto (U,U) \]
	is cofinal.  
\end{proof}

\begin{lem}\label{gsnsioghetg}
The functor
$\bV_{\bC}^{G} \colon G\BC\to \Cle$ is $u$-continuous.
\end{lem}
\begin{proof}
\cref{eriogwetgwegregwrgregwrg} applied to the morphisms $X_{U}\to X$, \cref{prop:u-continuous}, and the fact that a filtered colimit of fully faithful functors is fully faithful  implies that we have an inclusion
$\colim_{U\in \cC^{G}_{X}} \bV^{G}_{\bC}(X_{U}) \to \bV^{G}_{\bC}(X)$ of full subcategories of $\hat \bV^{G}_{\bC}(X)$.
 It is also essentially surjective since the notion of equivariant smallness is independent of the coarse structure.
\end{proof}

\subsection{Subspace inclusions}\label{rgiowegwrgrewrgwegwregw}\label{sec:subspaces}
In this section, we derive some technical results which will enter the discussion of excision for $\bV_{\bC}^{G}$ in \cref{gbiwrotwrtbwrgergeg}.

Let $i\colon Z\to X$ be the inclusion of a subspace {of a $G$-bornological coarse space $X$}.
\begin{prop}\label{prop:inclusion-fff}
	The functors $i_* \colon \hat \bV^G_\bC(Z) \to \hat \bV^G_\bC(X)$ and  $i_* \colon \bV^G_\bC(Z) \to \bV^G_\bC(X)$ are  fully faithful.
\end{prop}

The proof requires a little preparation. 
We fix an invariant coarse entourage $U$ of $X$ which contains the diagonal.
Note that for any subset $A$ of $X$ we have $U(A)\cap Z\subseteq U_{Z}(A\cap Z)$,  where $U_Z := U \cap (Z \times Z)$ is considered as an invariant entourage of $Z$. We
consider the sub-poset 
  \begin{equation}\label{43t29t8u349t2t34t234t2t243t}
\cP^{U}_{X}:=\{A\in \cP_{X} \mid  U(A)\cap Z =U_{Z}(A\cap Z) \} 
\end{equation}
of the power set $\cP_{X}$.

\begin{lem}\label{wgioowtreg234wergregwerg}\mbox{}
	\begin{enumerate}
		\item \label{eriogwergergwggwregrgwg} If  $A,A^{\prime}$ are in $\cP^{U}_{X}$, then  we also have $A\cap A^{\prime}\in \cP^{U}_{X}$.
		\item \label{eriogwergergwggwregrgwg1} For every $Y$ in $\cP_{X}$ the  {slice} $(\cP^{U}_{X})_{Y/}$ has a unique minimal  {element $U[Y,Z]$. It satisfies  $Z \cap U[Y,Z] = Z \cap Y$.} 	\end{enumerate}
\end{lem}
\begin{proof}
\eqref{eriogwergergwggwregrgwg} follows from the fact that thinning (see \eqref{V-thin}) preserves intersections.

We now show \eqref{eriogwergergwggwregrgwg1}.
We will show that  the unique minimal element of $(\cP^{U}_{X})_{Y/}$ is given by \[ U[Y,Z] := Y \cup \bigcup_{x \in Y \cap Z, U_Z[x] \subseteq Y \cap Z }  U[x] \setminus Z \ .\]
 {Note that $Z \cap U[Y,Z] = Z \cap Y$ by definition.}
We first check that $U[Y,Z]$ belongs to the set $(\cP_{X}^{U})_{Y/}$.

By definition, we have  $Y\subseteq U[Y,Z]$.
It remains to check that
\[ U_Z(U[Y,Z] \cap Z) = U(U[Y,Z]) \cap Z\ .\]
We must show that
\[ U_Z(U[Y,Z] \cap Z)  \subseteq  U(U[Y,Z]) \cap Z  \ .\]
Let $x$ be in $U_Z(U[Y,Z] \cap Z)$. Then we have the relations  $x\in Y \cap Z$  and $U_Z[x] \subseteq Y \cap Z$. 
This implies  
\[ U[x] = U_Z[x] \cup (U[x] \setminus Z) \subseteq (Y \cap Z) \cup U[Y,Z] = U[Y,Z]\ ,\]
and hence $x\in U(U[Y,Z]) \cap Z$.
	
We now  show that $U[Y,Z]$ is the unique minimal element in $(\cP_{X}^{U})_{Y/}$.
 Suppose that $A$ is any other  element of $(\cP_{X}^{U})_{Y/}$.  
 Assume that $x$ is in $U[Y,Z]$.  We consider two cases:
  \begin{enumerate}
 \item  $x\in Y$:  Then  $x\in A$.
 \item  $x\notin Y$: Then $x\in U[x'] \setminus Z$ for some $x'$ in $Y \cap Z$ satisfying $U_Z[x'] \subseteq Y \cap Z \subseteq A\cap Z$. Since $U_Z(A\cap Z) = U(A) \cap Z$, we also have $U[x'] \subseteq A$.
	Because of  $x\in U[x']$, this implies   $x\in A$. \qedhere\end{enumerate}
\end{proof}

\begin{proof}[Proof of \cref{prop:inclusion-fff}]
 In view of \cref{regklwgrregwrgegr}, it suffices to show that the functor $i_{*} \colon \hat  \bV^G_\bC(Z) \to \hat \bV^G_\bC(X)$ is fully faithful.
 
 Let $M$ and $N$ be in $\Sh_{\bC}^{G}(Z)$. We will use the formula for mapping spaces in the localisation provided by \cref{qriooqergegwgrer}.
 The map of posets $\cC^{G,\Delta}_X \to \cC^{G,\Delta}_Z$ given by  $U \mapsto U_Z:=U\cap (Z\times Z)$ is cofinal.
 Note that $\hat i^{G}_{*}$ is fully faithful by the second assertion of \cref{rqkjgqregergwgwregweg}.
 {Therefore the} map
 \[ \Map_{\hat \bV^G_\bC(Z)}(\ell_{Z}M,\ell_{Z}N) \to \Map_{\hat\bV^G_\bC(X)}(i_* \ell_{X}M,i_*\ell_{X}N) \] induced   by $i$  
 can be identified with the map
 \[ \colim_{U \in \cC^{G,\Delta}_X} \Map_{\Sh^G_\bC( {X})}(\hat i_*^GM,\hat  i_*^G  U^{G}_{Z,*}N) \to \colim_{U \in \cC^{G,\Delta}_X} \Map_{\Sh^G_\bC(X)}(\hat  i_*^GM,  U^{G}_*\hat  i_*^GN) \]
 induced by the canonical map $\hat i_*^G U^{G}_{Z,*} N \to U_*  \hat i_*^GN$.
 We claim that the map
 \[  \Map_{\Sh^G_\bC( {X})}(\hat i_*^GM,\hat  i_*^G  U^{G}_{Z,*}N) \to  \Map_{\Sh^G_\bC(X)}(\hat  i_*^GM,  U^{G}_*\hat  i_*^GN) \]	is an equivalence for every $U$ in $\cC_{X}^{G,\Delta}$.  	
 {Since $\Sh_\bC(X)$ is a full subcategory of $\PSh_\bC(X)$,   {by \cref{tieqorgfgregwegergw}.\eqref{reigowergwrgrgwgrg}}
 it suffices to prove that
 \[ \Map_{\PSh_\bC(X)}(\hat i_*M, \hat i_*U_{Z,*}N ) \to \Map_{\PSh_\bC(X)}( \hat i_*M, U_*\hat i_*N ) \]
 is an equivalence.}
 
 \cref{wgioowtreg234wergregwerg}.{\eqref{eriogwergergwggwregrgwg1}} implies that the inclusion functor
 \[ L \colon \cP^{U,\op}_{X}\to \cP^{\op}_{X} \]
 admits a right adjoint $R$ which sends $Y$ to $U[Y,Z]$. Consequently, we obtain an induced co-Bousfield localisation
 \[ R^* \colon \Fun(\cP^{U,\op}_X,\bC) \rightleftarrows \PSh_\bC(X) \cocolon L^*\ . \]
 The map $R^*L^*\hat i_*M \to \hat i_*M$ is an equivalence since $Z \cap R(Y) = Z \cap Y$ for all $Y$, i.e., ${\hat i_{*}M}$ is a colocal object.
 Hence the map in question becomes identified with the canonical map
 \[ \Map_{\Fun(\cP^{U,\op}_X,\bC)}(L^*\hat i_*M, L^*\hat i_*U_{Z,*}N) \to \Map_{\Fun(\cP^{U,\op}_X,\bC)}(L^*\hat i_*M, L^*U_*\hat i_*N )\ . \]
 By definition of $\cP^U_X$, the map $L^*\hat i_*U_{Z,*}N \to L^*U_*\hat i_*N$ is an equivalence, so we are done.
\end{proof}

\subsection{Excision}\label{gbiwrotwrtbwrgergeg}\label{sec:excision}
 Let $\cY:=(Y_{\ell})_{\ell\in L}$ be a filtered family of invariant subsets of  {a $G$-bornological coarse space $X$}.
 The members $Y_{\ell}$ will be considered as {$G$-bornological coarse spaces}
 with the coarse and bornological structures induced from $X$.
\begin{ddd}\label{trhrthgwregwgwrgwrg}
$\cY$ is a big family if for every  {coarse entourage $U$ of $X$}
and $\ell$ in $L$ there exists $\ell'$ in $L$ such that $U[Y_{\ell}]\subseteq Y_{\ell'}$ ({see \eqref{V-thick} for the thickening construction}).
\end{ddd}
For a functor   $E\colon G\BC\to \bM$  to a cocomplete   target we set
\begin{equation}\label{ewroijiowgjoewgergergegwegw}
E(\cY):=\colim_{\ell\in L} E(Y_{\ell})\ .
\end{equation}
The family of inclusions $(Y_{\ell}\to X)_{\ell\in L}$ induces a canonical morphism 
\begin{equation}\label{ehrgjkwegrregrewgwgg}
E(\cY)\to E(X)\ .
\end{equation}

 \begin{ddd}\label{erpogpergergergegerge} A complementary pair on $X$ is a  pair $(Z,\cY)$ of an invariant subset $Z$ and a  big family $\cY$ such that
  there exists $\ell $ in $L$   with $Z\cup Y_{\ell}=X$.
\end{ddd}

\begin{ex}\label{ex:comppair}
 {Consider the bornological coarse space $\R_d$, where $d$ is the Euclidean metric. Then
 \[ \left( (-\infty,0], ([-n,\infty) )_{ n \in \nat } \right) \]
 is a complementary pair on $\R_d$.}
\end{ex}

 We can form the big family $Z\cap \cY:=(Z\cap Y_{\ell})_{\ell\in L}$ on $Z$. 
 
\begin{ddd}\label{rgoiruegoiregregregreg}
 $E$ is called excisive if for every {$G$-bornological coarse space} $X$
 with a complementary pair $(Z,\cY)$ the commutative square
\[\xymatrix{E(Z\cap \cY)\ar[r]\ar[d] & E(Z)\ar[d] \\
	E(\cY)\ar[r] & E(X)}\]
is a pushout square.
 \end{ddd}

\begin{ex}
  Consider the complementary pair from \cref{ex:comppair}. An excisive functor $E$ then gives rise to a pushout
  \[\xymatrix{
    E(( [-n,0] )_{n \in \nat })\ar[r]\ar[d] & E((-\infty,0])\ar[d] \\
    E(( [-n,\infty) )_{n \in \nat })\ar[r] & E(\R_d)}\]
  If $E$ additionally vanishes on flasque spaces (\cref{rgiojgogregrgregre}), the two outer corners of this square are trivial due to \cref{ex:flasques}. We therefore get an equivalence
  \[ E(\R_d) \simeq \Sigma E(( [-n,0] )_{n \in \nat })\ .\]
  If $E$ is also coarsely invariant (\cref{rewkgowegrerfrewfwr}), the right-hand side is equivalent to $\Sigma E(*)$.
\end{ex}

\begin{rem} We do not expect that
$\bV_{\bC}^{G} $ is excisive in the sense of   \cref{rgoiruegoiregregregreg}.
\cref{rgijrgoirejgoergergreg} below is the appropriate statement using the notion of an excisive square for left-exact categories as introduced in \cref{ugioerguoerug}.
\end{rem}

Let $E\colon G\BC\to \Cle$ be a functor. 
\begin{ddd}\label{rgoiruegoiregregregreg-modified}
$E$ is called $l$-excisive if   for every  {$G$-bornological coarse space} $X$
with a complementary pair $(Z,\cY)$ the commutative square
\[ \xymatrix{E(Z\cap \cY)\ar[r]\ar[d] & E(Z)\ar[d] \\ E(\cY)\ar[r] & E(X)} \]
is  an excisive square in $\Cle$ (\cref{ugioerguoerug}).
 \end{ddd}

Note that the composition of an $l$-excisive functor with a homological functor is excisive.

Let $\bC$ be in $\Fun(BG,\CL)$.
\begin{prop}\label{rgijrgoirejgoergergreg}  
  The functor $\bV_{\bC}^{G}\colon G\BC\to \Cle$ is $l$-excisive.
  \end{prop}

 The proof of \cref{rgijrgoirejgoergergreg} occupies the remainder of this section.

By \cref{eriogwetgwegregwrgregwrg},
 there is a commutative square
 \begin{equation}\label{gvrogpij3ggrgrgrfffo0i3g34g34g34g}
\xymatrix{\Sh^{G,\eqsm}_{\bC}(Z\cap \cY) \ar[r]^-{\hat \iota^{G}_{Z}}\ar[d]_{ {\hat {g}^{G}_{*}}}&\Sh^{G,\eqsm}_{\bC}(Z)  \ar[d]^{ {\hat {i}^{G}_{*}}}\\
\Sh^{G,\eqsm}_{\bC}(\cY) \ar[r]^-{\hat  \iota^{G}_{X} }&\Sh^{G,\eqsm}_{\bC}(X) }
\end{equation}
in $\Cle$, where $ {\hat {g}^{G}_{*}}$ 
is induced by the family of morphisms $(i_{\ell}\colon Z\cap Y_{\ell}\to Y_{\ell})_{\ell\in L}$ using \cref{eriogwetgwegregwrgregwrg},
and $\hat \iota^{G}_{Z}$ and $\hat \iota^{G}_{X}$ are instances of the morphism \eqref{ehrgjkwegrregrewgwgg}.
By   \cref{qergiuqergreqgregwregwr}, the square \eqref{gvrogpij3ggrgrgrfffo0i3g34g34g34g}  induces a square   \begin{equation}\label{frflkrlmgregergreger}
\xymatrix{\bV^{G}_{\bC}(Z\cap \cY) \ar[r]^-{\iota_{Z} }\ar[d]_{ {g_*}}&\bV^{G}_{\bC}(Z) \ar[d]^{i_{*} }\\
\bV^{G}_{\bC}(\cY) \ar[r]^-{\iota_{X} }&\bV^{G}_{\bC}(X) }
\end{equation} 
 in   $\Cle$.
 In order to show \cref{rgijrgoirejgoergergreg}, we must show that 
the     square \eqref{frflkrlmgregergreger} in $\Cle$
 is  excisive.
This amounts to showing that the horizontal functors are fully faithful, and that  the induced morphism between their stable cofibres (see \cref{rgioergieorfrgergergergergerge}) is an equivalence.

Note that
$\iota_{X}$ is the colimit   of the family of functors
${(\bV^{G}_{\bC}(Y_{\ell})\to \bV^{G}_{\bC}(X))_{\ell\in L}}$.
Since the members of this family are fully faithful by \cref{prop:inclusion-fff}, and a filtered colimit   of fully faithful functors is again fully faithful, we conclude that $\iota_{X}$ is fully faithful. Similarly, $\iota_{Z}$ is fully faithful.

The rest of the argument is devoted to the comparison of the stable cofibres.
We define the  {subcategory} $W_{\iota_{X}}$ of $\bV^{G}_{\bC}(X)$ to be
 {the smallest subcategory containing all morphisms with fibres in the essential image of $\iota_{X}$, which is closed under pullbacks and satisfies the two-out-of-three property}.
 {Then \cref{rgioowfwefwefw} shows that} the stable cofibre of $\iota_{X}$  {can be explicitly described as
\[ \Cofib^{s}(\iota_{X}) \simeq \tilde \Sp(\bV_{\bC}^{G}(X)[W_{\iota_{X}}^{-1}])\ .\]
We define the subcategory $W_{\iota_{Z}}$ of $\bV^{G}_{\bC}(Z)$ similarly, and observe that $\Cofib^{s}(\iota_{Z})$ admits an analogous description}.
 
We consider the following diagram
  \begin{equation}\label{gvkjvjejrfreogrjgeiorgergergeg} \xymatrix{
 \Sh^{G,\eqsm}_{\bC}(X) \ar[d]_-{ {\hat {i}^{*,G}} }\ar[r]^-{\ell_{X}}&\bV^{G}_{\bC}(X) \ar[r]^{\ell^{s}_{\iota_{X}}}\ar@{..>}[dr]_(0.35){ {\tilde {i}^{*}} }& \Cofib^{s}(\iota_{X} )\ar@{-->}[d]^-{\bar i^{*} }\\
\Sh^{G,\eqsm}_{\bC}(Z) \ar[r]_-{\ell_{Z}} &\bV^{G}_{\bC}(Z) \ar[r]_-{\ell^{s}_{\iota_{Z}}}&\Cofib^{s}(\iota_{Z} )}
\end{equation}
 in $\Cle$.
At the moment forget the dashed part.
\begin{lem}\label{geroigpergergregreg}
The 
universal property of $\ell_{X}$  provides the dotted arrow $\tilde i^{*}$.
\end{lem}
\begin{proof} 
 It suffices to show   that the composition
\[ \tilde i^{*,G}\colon \Sh^{G,\eqsm}_{\bC}(X)\xrightarrow{\hat {i}^{*,G}} \Sh^{G,\eqsm}_{\bC}(Z)\xrightarrow{\ell_{Z}} \bV^{G}_{\bC}(Z)\xrightarrow{\ell^{s}_{\iota_{Z}}}
 \Cofib^{s}(\iota_{Z} )\]
 sends the morphisms in $W^{\eqsm}_{X}$ (see the text before \cref{rtheorthertherthetrhe}) to equivalences.
 
Let $M$ be in $ \Sh^{G,\eqsm}_{\bC}(X)$ and  {let $V$ be an invariant coarse entourage of $X$ containing the diagonal}.
 Then we consider the generator $\iota_{V}(M) \colon M\to V^{G}_{*}M$ of $W^{\eqsm}_{X}$.
 It suffices to   show that $\ell^s_{\iota_Z}\ell_Z\hat i^{*,G}\iota_V(M)$ is an equivalence.

Let $U$ be  {be an invariant coarse entourage of $X$ containing the diagonal}
such that $M\in \Sh^{U,G}_{\bC}(X)$. Using that the family $\cY$ is big,  we can choose a member
$Y$ of $\cY$ such that $(Z,Y)$ is a $VUV{^{-1}}$-covering family of $X$ (see \eqref{comp-def-ent}, \cref{rgiqjrgioqfweewfqewfqewf} {and \cref{ex-excisive-pair}}).
Since $U\subseteq VUV^{-1}$, the pair  $(Z,Y)$ is also a $U$-covering  family, and we have  $V_{*}^{G}M\in \Sh^{VUV{^{-1}},G}_{\bC}(X)$ by \cref{equi-sheaffff}.

We now use the notation introduced in connection with the Glueing \cref{wgkwkgrewrgrg}.
 By  \cref{wgkwkgrewrgrg}, the morphism $\iota_{V}(M)$  is equivalent in $\Sh^{G}_{\bC}(X)$  to the morphism \begin{equation}\label{verij93g3ggr4}
 {\hat {i}^{G}_{*}\hat {i}^{*,G} M \mathop{\times}\limits_{ \hat {k}^{G}_{*}\hat {k}^{*,G} M} \hat {j}^{G}_{*}\hat {j}^{*,G} M   \longrightarrow \hat {i}^{G}_{*}\hat {i}^{*,G}V^{G}_{*}M \mathop{\times}\limits_{ \hat {k}^{G}_{*}\hat {k}^{*,G}V^{G}_{*}M} \hat {j}^{G}_{*}\hat {j}^{*,G}V^{G}_{*}M} \ .
\end{equation}
 
 The functor $\tilde i^{*,G}$ is left-exact and therefore preserves 
the  fibre products in \eqref{verij93g3ggr4}.
The functor $\ell_{Z}  {\hat {i}^{G,*}}$ sends
objects in the image of $ {\hat {k}^{G}_{*}}$ or $ {\hat {j}^{G}_{*}}$ to objects in the image of $\iota_{Z}$ by the commutativity (justified by \cref{ropgkpwoegrewgwregwgregw} and \cref{lem:restrict-eqsm})
of the diagram 
\begin{equation}\label{wergwreggregw5}
\xymatrix{\Sh^{G,\eqsm}_{\bC}(Z\cap \cY) \ar[r]^(0.55){\hat  \iota^{G}_{Z} } &\Sh^{G,\eqsm}_{\bC}(Z) \\
\Sh^{G,\eqsm}_{\bC}(\cY) \ar[u]^{ { {\hat {g}^{*,G}}}}\ar[r]^{\hat  \iota^{G}_{X} }&\Sh^{G,\eqsm}_{\bC}(X) \ar[u]_{ {\hat {i}^{*,G}} }\\\Sh^{G,\eqsm}_{\bC}(Y)\ar[ur]_-{\hat j^{G}_{*}}\ar[u]&}\ .
\end{equation}
Finally by definition, the functor  $\ell^{s}_{\iota_{Z}}$ 
 sends  objects in the image of $ \iota_{Z}$ to zero objects.
 Using the above observations, in $\Cofib^{s}(\iota_{Z} )$
   we get an equivalence 
   \[ {\ell^s_{\iota_Z}\ell_Z\hat i^{*,G}\iota_V(M)}  \simeq \ell^{s}_{\iota_{Z}} \ell_{Z} {\hat {i}^{*,G}[\hat {i}^{G}_{*}\hat {i}^{*,G}}M\to {\hat {i}^{G}_{*}\hat {i}^{*,G}V^{G}_{*}M}]\ . \]
 Using \cref{egiwetthtrehtrhethetheth}, setting $M^{\prime}:=\hat i^{*,G} \hat h^{G}_{*} \hat h^{*,G}M$ for the inclusion $h\colon V(Z)\to X$ and  
 $V_{Z}:=V\cap (Z\times Z)$, we have an equivalence 
  $ \hat {i}^{*,G}V^{G}_{*}M\simeq   V^{G}_{Z,*}M^{\prime}$.  
 
 Using in addition the equivalence $ {\hat {i}^{*,G} \hat {i}^{G}_{*}}\simeq \id$ (\cref{sub-adj})  we 
 get a factorisation of
 $\ell^s_{\iota_Z}\ell_Z\hat i^{*,G}\iota_V(M)$ as
 \begin{equation}\label{vfkneoirqjvoevffevsvd}
\ell^{s}_{\iota_{Z}} \ell_{Z}\hat {i}^{*,G} M\xrightarrow{ {\ell^s_{\iota_Z}\ell_Z\alpha}}  \ell^{s}_{\iota_{Z}} \ell_{Z} M^{\prime}\xrightarrow{ \ell^{s}_{\iota_{Z}} \ell_{Z}\iota_{V_{Z}}(M^{\prime})}  \ell^{s}_{\iota_{Z}} \ell_{Z}V^{G}_{Z,*}M^{\prime}\ ,
\end{equation}
where  {$\alpha \colon \hat i^{*,G} M \to M'$} is induced by the unit $\id\to \hat h^{G}_{*} \hat h^{*,G}$. 

We first observe that  $ \ell^s_{\iota_Z}\ell_{Z}\iota_{V_{Z}}(M^{\prime})$ is an equivalence since $\iota_{V_{Z}}(M^{\prime})\in W^{\eqsm}_{Z}$.

We then argue that  {$\ell^s_{\iota_Z}\ell_Z\alpha$} is an equivalence, too. To this end we show  that the fibre of $ \ell_Z\alpha \colon \ell_{Z}\hat {i}^{*,G} M\to \ell_{Z} M^{\prime}$ belongs to the essential image of $\iota_{Z}$, which implies that $ \ell^{s}_{\iota_{Z}}$ sends this morphism to an equivalence.

Let $l \colon V(Z) \to Z$ denote the inclusion map. Since $h = i \circ l$ and $\hat i^G_* \hat i^{*,G} \simeq \id$, the morphism
\[ \hat i^{*,G}M \to M' \simeq \hat i^{*,G} \hat h^G_* \hat h^{*,G} M \simeq \hat l^G_* \hat l^{*,G} \hat i^{*,G} M \]
is  equivalent to the one induced by the unit $\id \to \hat l^G_* \hat l^{*,G}$.
Arguing similarly as above using \cref{wgkwkgrewrgrg} again and that $(V(Z),Y \cap Z)$ is a $U_Z$-covering family of $Z$, we have an equivalence
\[ \Fib(\hat i^{*,G}M \to \hat l^G_* \hat l^{*,G} \hat i^{*,G} M) \simeq \Fib( \hat m^G_* \hat m^{*,G} \hat i^{*,G} M \to \hat n^G_* \hat n^{*,G} \hat i^{*,G} M)\ ,\]
where $m\colon {Y} \cap Z\to Z$ and $n \colon Y \cap V(Z) \to Z$ are the inclusion maps.
The latter object obviously lies in the essential image of $\iota_Z$.
 
Since both maps in \eqref{vfkneoirqjvoevffevsvd} are equivalences, we conclude
that  {$\ell^s_{\iota_Z}\ell_Z\hat i^{*,G}\iota_V(M)$}
is an equivalence.
\end{proof}

The square \eqref{frflkrlmgregergreger} induces a morphism between stable cofibres
\[ \bar i_{*} \colon \Cofib^{s}(\iota_{Z} ) \to  \Cofib^{s}(\iota_{X} )\ .\]
We now show that it 
is an equivalence.
Our candidate for the inverse functor is the dashed arrow $\bar i^{*}$ in diagram \eqref{gvkjvjejrfreogrjgeiorgergergeg}.
We will obtain $\bar i^{*}$ using the universal property of the morphism $\ell^{s}_{\iota_{X}}$. 
To this end, we note that its target is stable. 
In order to construct $\bar i^{*}$, it therefore suffices to show the following result.

\begin{lem} \label{rgqoigqregqerg} The morphism 
$\tilde i^{*}$ from  \eqref{gvkjvjejrfreogrjgeiorgergergeg} sends the morphisms 
 in $W_{\iota_{X}}$   to
equivalences. 
\end{lem}
\begin{proof}
  Let $f$ be a morphism in $\bV^{G}_{\bC}(X)$ such that 
$\Fib(f)$ belongs to the essential image of $\iota_{X}$.
Since $\ell_{X}$ is left-exact,  {there exists} by \cref{ih43iugu34g34g3}
a morphism $\tilde f$ in $\Sh^{G,\eqsm}_{\bC}(X)$ such that $\ell_{X}(\tilde f) $ is equivalent to $f$.
In the following, we use the commutativity of the established  part of the diagram \eqref{gvkjvjejrfreogrjgeiorgergergeg}.
We  have an equivalence
 \[ {\tilde i^{*}}(f)\simeq \ell^{s}_{\iota^{G}_{Z}}\ell^{G}_{Z}{\hat {i}^{*,G}}(\tilde f)\ .\]
By the assumption on $f$,   there   exists an object   $P$   of
$\Sh^{G,\eqsm}_{\bC}(\cY)$ such that
\[ \iota_{X}\ell_{\cY}(P)\simeq  \Fib(f) \ .\]
On the one hand, since $\tilde i^{*}$ is left-exact, we have the equivalence  
\[ {\tilde i^{*}}\iota_{X}\ell_{\cY}(P)\simeq   {\tilde i^{*}}  \Fib(f)\simeq  \Fib( {\tilde i^{*}}f) \ .\]
On the other hand, we have an equivalence
\begin{eqnarray*}
 \tilde i^{*}\iota_{X}\ell_{\cY}(P) &\stackrel{\eqref{ergoijergerg}}{\simeq} &{\tilde i^{*}}\ell_{X}\hat \iota^{G}_{X} (P) \\
&\stackrel{{ \eqref{gvkjvjejrfreogrjgeiorgergergeg}}}{\simeq} & \ell^{s}_{\iota_{Z}}\ell_{Z}{\hat {i}^{*,G}}\hat \iota^{G}_{X}(P) \\
&\stackrel{\eqref{wergwreggregw5}}{\simeq}& \ell^{s}_{\iota_{Z}}\ell_{Z}\hat \iota^{G}_{Z}{\hat g^{*,G}}(P) \\
&\stackrel{\eqref{ergoijergerg}}{\simeq}&\ell^{s}_{\iota_{Z}}  \iota_{Z}\ell_{Z}{\hat g^{*,G}}(P) \\
&\stackrel{ }{\simeq} &0\ .
\end{eqnarray*}
 We now use that $\Cofib^{s}(\iota_{Z} )$  is stable in order to conclude that $\tilde i^{*}(f) $ is an equivalence.
 \end{proof}

\begin{proof}[Proof of \cref{rgijrgoirejgoergergreg}]
From \cref{rgqoigqregqerg} we get a further factorisation 
\begin{equation}\label{gergpojoij34t0934jt34t34t}
\bar i^{*} \colon \Cofib^{s}(\iota_{X}) \to \Cofib^{s}(\iota_{Z} )\ ,
\end{equation}
of $\tilde i^{*}$, namely
the dashed arrow in \eqref{gvkjvjejrfreogrjgeiorgergergeg}.
By construction, we have the equivalence
\[ \bar i^{*}\circ \bar i_{*}\simeq \id_{ {\Cofib^{s}(\iota_{Z})} } \ .\]
The transformation
$\id_{\Sh^{G}_{\bC}(X)}\to  {\hat {i}^{G}_{*} \hat {i}^{*,G}}$ furthermore induces a transformation
\[ \id_{\Cofib^{s}(\iota_{Z})}\to  {\bar{i}_{*}  \bar{i}^{*}}\ .\]
It remains to  show that the latter is an equivalence.
Let $M$ be an object of $\Sh_{\bC}^{G }(X)$.
We can choose {an invariant entourage $U$ of $X$ containing the diagonal}
such that {$M$ is a $U$-sheaf on $X$}
and a member $Y$ of $\cY$ such that $(Y,Z)$ is a $U$-covering family of $X$.
 By \cref{wgkwkgrewrgrg}, we have an equivalence 
\[ M\simeq {\hat {i}^{G}_{*}\hat {i}^{*,G}(M)\times_{\hat {k}^{G}_{*}\hat {k}^{*,G}(M)} \hat {j}^{G}_{*}\hat {j}^{*,G}(M)}\ .\]
We must show that $\ell^{s}_{\iota_{X}}  \ell_{X}$ sends the   projection
\[ {\hat {i}^{G}_{*}\hat {i}^{*,G}(M)\times_{\hat {k}^{G}_{*}\hat {k}^{*,G}(M)} \hat {j}^{G}_{*}\hat {j}^{*,G}(M)\to \hat {i}^{G}_{*}\hat {i}^{*,G}(M)} \]
to an equivalence in $ {\Cofib^{s}(\iota_{X})} $. This is clear since $\ell^{s}_{\iota^{G}_{X}}  \ell^{G}_{X}$ preserves  fibre products    and sends the objects
$ {\hat {j}^{G}_{*}\hat {j}^{*,G}(M)}$ and $ {\hat {k}^{G}_{*}\hat {k}^{*,G}(M)}$ (which belong to the essential image of $\hat  \iota^{G}_{X}$)  to zero objects in $ {\Cofib^{s}(\iota_{X})} $.
\end{proof}

\subsection{Strong additivity}\label{oijqrogrgrgrgrewgwergrwegrwg}
Let $X$ be  {a $G$-bornological coarse space},
and let $(Y,Z)$ be a partition of $X$ into invariant subsets. 
 We say that $(Y,Z)$ is a coarsely disjoint decomposition if 
 $Y$ and $Z$ are both unions of collections of coarse components.

Let $E\colon G\BC\to \bM$ be a functor with target an $\infty$-category admitting an initial object $\emptyset$.    

\begin{ddd}\label{rgiowergerwgregwergergwerg}
$E$ is  $\pi_{0}$-excisive if for every {$G$-bornological coarse space} $X$
with a coarsely disjoint decomposition $(Y,Z)$ the square
\[\xymatrix{
 \emptyset\ar[r]\ar[d] & E(Z)\ar[d] \\
 E(Y)\ar[r] & E(X)
}\]
is a pushout square.
\end{ddd}

Note that $\CLL$ and $\Cle$ are pointed  {by the}
category $0:=\Delta^{0}$. Let $\bC$ be in $\Fun(BG,\CLL)$.  
\begin{lem}\label{ergieorogergergergregerge}
The functors $\Sh^{G}_{\bC}$, $\Sh^{G,\eqsm}_{\bC}$,  and $\bV^{G}_{\bC}$ are $\pi_{0}$-excisive.
\end{lem}
\begin{proof} We consider  {a $G$-bornological coarse space} $X$
with a coarsely disjoint decomposition $(Y,Z)$.    

We first consider the case of  $\Sh^{G}_{\bC}$.
We have an isomorphism  of posets
\[ s\colon \cP_{X}\xrightarrow{\cong} \cP_{Y}\times \cP_{Z} \ ,\quad B\mapsto (B\cap Y,B\cap Z)\ .\]
 Since $\bC$ is left-exact, we have a product functor  $\times\colon \bC\times \bC\to \bC$. It induces the functor denoted by the same symbol in
 \begin{equation}\label{regvervlknweklvwevewv} 
 \xymatrix{ \Sh^{G}_{\bC}(Y)\times \Sh^{G}_{\bC}(Z)\ar@{..>}[r]^-{m}\ar[d]^{\times}&\Sh^{G}_{\bC}(X)\ar[d]\\\lim_{BG}\Fun(\cP_{Y}^{\op}\times \cP_{Z}^{\op},\bC)\ar[r]^-{s^{*}}&\lim_{BG}\PSh_{\bC}(X)}\ ,
  \end{equation}
  where  the dotted arrow $m$ is obtained by checking that the {composition along the lower left corner}
  takes values in the subcategory of sheaves. As a consequence of \cref{wgkwkgrewrgrg}, we see 
 that the functor   $m$   is an equivalence with inverse 
   induced by the restrictions of  sheaves along $\cP_{Y}\to \cP_{X}$ and $\cP_{Z}\to \cP_{X}$.
  We now use that $\CLL$  is semi-additive (the $\CLL$-version of \cref{girgjowegfewfw9ef}) in order to get the first equivalence in
 \begin{equation}\label{f3fmfpko4f3f34f3f} \Sh_{\bC}^{G}(Y)\sqcup  \Sh_{\bC}^{G}(Z) \xrightarrow{\simeq}  \Sh_{\bC}^{G}(Y)\times  \Sh_{\bC}^{G}(Z) \xrightarrow{m,\simeq}  \Sh_{\bC}^{G}(X) \ , \end{equation}
  showing that $\Sh^{G}_{\bC}$ is $\pi_{0}$-excisive.

In order to get the result for $\Sh^{G,\eqsm}_{\bC}$ we must check that $m$ preserves equivariantly small sheaves.  
 But this is clear since (by an inspection of  \cref{qreigoqrgwergwrgre}) the  square
   \[\xymatrix{\Sh_{\bC}^{G,\eqsm}(Y)\sqcup  \Sh_{\bC}^{G,\eqsm}(Z) \ar[r]\ar[d]& \Sh_{\bC}^{G,\eqsm}(X)\ar[d]\\\Sh_{\bC}^{G}(Y)\sqcup  \Sh_{\bC}^{G}(Z) \ar[r]^-{m}& \Sh_{\bC}^{G}(X)}\]
   is a pullback.

 The family $(Y)$   consisting of the single member $Y$ is a big family (\cref{trhrthgwregwgwrgwrg}). Then $(Z,(Y))$ is a complementary pair (\cref{erpogpergergergegerge}).
By  \cref{rgijrgoirejgoergergreg} and using that $\bV^{G}_{\bC}(\emptyset)\simeq 0$  and $\bV_{\bC}^{G}(Y)\simeq \bV^{G}_{\bC}((Y))$,
we obtain an excisive square \[\xymatrix{0\ar[r]\ar[d]&\bV_{\bC}^{G}(Z) \ar[d]^{i_{Z,*}}\\\bV^{G}_{\bC}(Y)\ar[r]^{i_{Y,*}}&\bV^{G}_{\bC}(X)} \]
in  $\Cle$, where   $i_{Y}\colon Y\to X$ and $i_{Z}\colon Z\to X$  are the inclusions. 
 It induces a morphism
 \begin{equation}\label{trwhiowtherhwehwerh}
 \bV_{\bC}^{G}(Y)\sqcup \bV_{\bC}^{G}(Z) \to  \bV_{\bC}^{G}(X)  \ .
\end{equation}
We want to show that this morphism is an equivalence in $\Cle$. To this end
we show that this morphism is fully faithful and essentially surjective. 
 
 The map \eqref{trwhiowtherhwehwerh} is the localisation of \eqref{f3fmfpko4f3f34f3f}. 
 This shows that  \eqref{trwhiowtherhwehwerh} is essentially surjective.
  
 We already know from    \cref{prop:inclusion-fff} that the   components  $i_{Y,*} $ and $i_{Z,*}$ of the morphism  \eqref{trwhiowtherhwehwerh} are   fully faithful.
 It remains to show that if $  M $ is in  $  \Sh_{\bC}^{G}(Y)$ and $N$ is in $  \Sh_{\bC}^{G}(Z)$, then
\[ \Map_{\bV_{\bC}^{G}(X)}(i_{Y,*}\ell_{Y}M,i_{Z,*}\ell_{Y}N)\simeq *\ .\]
 This is obvious  from the  formula \eqref{qwefqwfoij11111}, noting that 
 \[ \Map_{\Sh_{\bC}^{G}(X)}(\hat i^{G}_{Y,*}M,V^{G}_{*}\hat i^{G}_{Z,*}N)\simeq * \]
 for all  {invariant coarse entourages $V$ of $X$ which contain the diagonal},
 since the sheaves $\hat i^{G}_{Y,*}M$ and $V^{G}_{*}\hat i^{G}_{Z,*}N$ have disjoint support.
\end{proof}

 Consider a $\pi_{0}$-excisive functor $E\colon G\BC\to \bM$.
 Let $X$ be  {a $G$-bornological coarse space}
 with a coarsely disjoint  {decomposition} $(Y,Z)$.
 Then we can define a projection map as the composition 
\begin{equation}\label{ergioregjewrgrgwregwergwergr}
p_{Y}\colon E(X)\stackrel{\pi_{0}-\text{exc.}}{\simeq} E(Y)\sqcup E(Z) \xrightarrow{q_{Y}} E(Y)\ ,
\end{equation}
 where $q_{Y}$ is classified by the morphisms $\id_{E(Y)}\colon E(Y)\to E(Y)$ and $0\colon E(Z)\to E(Y)$. 

Let $(X_{i})_{i\in I}$ be a family  {of $G$-bornological coarse spaces}.
 
 \begin{ddd}[{\cite[Ex.2.16]{equicoarse}}]  \label{fqiofwefewfeqwfeqwf}
 We define the free union
$\bigsqcup_{i\in I}^{\free}X_{i}$
to be the following $G$-bornological coarse space:
\begin{enumerate}
\item The  underlying $G$-set  is the disjoint union of $G$-sets $\bigsqcup_{i\in I}X_{i}$.
\item  \label{r3fqu9z3fpu8934f13314f1} The coarse structure   is generated by entourages $\bigcup_{i\in I} U_{i}$ for all families $(U_{i})_{i\in I}$, where $U_{i}$ is in $\cC_{X_{i}}$   for every $i$ in $I$.
\item \label{r3fqu9z3fpu8934f13314f} The bornology is generated by the set
$\{B\mid B\in \cB_{X_i}, i\in I\}$ of subsets of $\bigsqcup_{i\in I}X_{i}$. \qedhere
\end{enumerate}
\end{ddd}

\begin{rem}
Note that in general the free union $\bigsqcup_{i\in I}^{\free}X_{i}$ of the family $(X_{i})_{i\in I}$ differs from the coproduct $\coprod_{i}X_{i}$ in $G\BC$. The underlying $G$-sets of the coproduct and the free union coincide. But
the coarse structure of the coproduct
is generated by the set $\{U\mid  U\in \cC_{X_{i}}, i\in I \}$. It is in general smaller than the coarse structure of the free union. Furthermore, the bornology of the coproduct is generated by the sets $\bigcup_{i\in I} B_{i}$ for all families $(B_{i})_{i\in I}$ with $B_{i}$ in $\cB_{X_{i}}$. In general, this bornology is bigger than the one of the free union. The identity of the underlying sets is a morphism
$\coprod_{i\in I}X_{i}\to \bigsqcup_{i\in I}^{\free}X_{i}$ in $G\BC$.
 \end{rem}

If  $E\colon G\BC\to \bM$ is  $\pi_{0}$-excisive, then for
 every $i_{0}$ in $I$ we have the coarsely disjoint decomposition of $\bigsqcup_{i\in I}^{\free}X_{i}$ into the invariant subsets $X_{i_0}$ and 
$\bigsqcup_{i\in I\setminus\{i_{0}\}}X_{i}$,  and therefore a projection (see \eqref{ergioregjewrgrgwregwergwergr}) 
\[ p_{i_{0}}\colon E(\bigsqcup_{i\in I}^{\free}X_{i})\to E( X_{i_{0}})\ .\]
The family of projections $(p_{i})_{i\in I}$ provides a map
\begin{equation}\label{weiowwgwtgrgrw}
E(\bigsqcup_{i\in I}^{\free}X_{i})\xrightarrow{(p_{i})_{i\in I}} \prod_{i\in I} E(X_{i})\ .
\end{equation}

Consider a  pointed  $\infty$-category $\bM$ admitting  all products indexed by sets and a  $\pi_{0}$-excisive functor
 $E\colon G\BC\to \bM$.
\begin{ddd}[{\cite[Ex.~3.12]{equicoarse}}] \label{wergiugewrgergergerwgwregw}
$E$ is strongly additive if the maps \eqref{weiowwgwtgrgrw} are equivalences for all   families $(X_{i})_{i\in I}$  {of $G$-bornological coarse spaces}.
\end{ddd}

Let $\bC$ be in $\Fun(BG,\CL)$. By \cref{ergieorogergergergregerge} and  \cref{prop:catex finitely complete} (and its $\CLL$-version), \cref{wergiugewrgergergerwgwregw} applies to the functors $\Sh_{\bC}^{G}$, $\Sh^{G,\eqsm}_{\bC}$, and $\bV^{G}_{\bC}$.
 \begin{prop}\label{ergiowjoetgergreergwergwreg}
 The functors $\Sh_{\bC}^{G}$, $\Sh^{G,\eqsm}_{\bC}$, and $\bV^{G}_{\bC}$ are strongly additive.
 \end{prop}
\begin{proof}
	Let $(X_{i})_{i\in I}$ be a family  {of $G$-bornological coarse spaces and let}
	$X:=\bigsqcup_{i\in I}^{\free}X_{i}$.

We start with the functor $\Sh^{G}_{\bC}$. We consider the entourage $U:=\bigsqcup_{i\in I}X_{i}\times X_{i}$ (note that $U$  {is not a coarse entourage of $X$}
in general). Then we have an equivalence
\[ \Sh^{G,U}_{\bC}(X)\simeq \prod_{i\in I} \PSh^{G}_{\bC}(X_{i}) \]
given by the family of restrictions along the family of inclusions $(X_{i}\to X)_{i\in I}$.
By an inspection of the definition of $\Sh^{G}_{\bC}$ and \cref{fqiofwefewfeqwfeqwf} \eqref{r3fqu9z3fpu8934f13314f1}
we conclude that this  equivalence  restricts to an equivalence 
\[ \Sh^{G}_{\bC}(X)\simeq \prod_{i\in I} \Sh^{G}_{\bC}(X_{i})\ .\]
Next we consider  $\Sh^{G,\eqsm}_{\bC}$. Using \cref{lem:restrict-eqsm},
we get a commutative square
 \[\xymatrix{\Sh^{G,\eqsm}_{\bC}(X)\ar[r]\ar[d]&\prod_{i\in I} \Sh^{G,\eqsm}_{\bC}(X_{i})\ar[d]\\\Sh^{G}_{\bC}(X)\ar[r]^-{\simeq}&\prod_{i\in I} \Sh^{G}_{\bC}(X_{i})}\ .\]
We must show the the upper horizontal arrow is an equivalence.
Since the vertical functors are also fully faithful, it suffices to show that it is essentially surjective.

  Let $M$ be in $\Sh^{G}_{\bC}(X)$  such that 
  {$M_{|X_{i}}$ is an equivariantly small sheaf on $X_i$ for every $i$ in $I$}.
Then we must show that  {$M$ is also equivariantly small}.
Let $H$ be a subgroup of $G$, and let $Y$ be an $H$-bounded subset of $X$.
By \cref{riugfhiqeurgfqwefqwefewf}, we can find a bounded subset $B$ of $X$ such that $Y=HB$.
In view of \cref{fqiofwefewfeqwfeqwf} \eqref{r3fqu9z3fpu8934f13314f}, the set $I_{B}:=\{i\in I\mid X_{i}\cap B\not=\emptyset\}$ is finite. In view of  \cref{fqiofwefewfeqwfeqwf} \eqref{r3fqu9z3fpu8934f13314f1}, we see that $(Y\cap X_{i})_{i\in I_{B}}$ is a coarsely disjoint family of $H$-bounded subsets such that  $Y=\bigcup_{i\in I_{B}} Y\cap X_{i}$.
We conclude that  \[ M(Y) \simeq \prod_{i \in I_B} M_{|X_i}(Y \cap X_i) \in \bC^{H,\omega}\]
since the product is finite and each factor  {is cocompact in $\bC^H$ by assumption}.

We conclude that
 \begin{equation}\label{ergiojgoiegwregwreg}
\Sh^{G,\eqsm}_{\bC}(X)\simeq \prod_{i\in I} \Sh^{G,\eqsm}_{\bC}(X_{i})\ .
\end{equation}
Finally, we consider $\bV_{\bC}^{G}$.
In order to show that \begin{equation}\label{owierjgweriogwergwergregegwgr} \bV^{G}_{\bC}(X)\to \prod_{i\in I} \bV^{G}_{\bC}(X_{i})\end{equation}
  is an equivalence, we show that this functor is fully faithful and essentially surjective. 
In fact, essential surjectivity immediately follows from the case of $\Sh^{G,\eqsm}_{\bC}$.
In order to prove fully faithfulness, we use the formula for mapping spaces provided by \cref{wetiwoegreergwrger}.
 
Let $M,N$ be  {equivariantly small sheaves on $X$}.
Then we write $M_{i},N_{i}$ for the restrictions of $M,N$ to $X_{i}$.
For an {invariant coarse entourage $V$ of $X$ which contains the diagonal},
we set $V_{i}:=V\cap (X_{i}\times X_{i})$ in $\cC^{G,\Delta}_{X_{i}}$.
Then we have the chain of equivalences
\begin{eqnarray*}
\lefteqn{\Map_{\bV^{G}_{\bC}(X)}(\ell_{X} M,\ell_{X}  N)}&&\\
&\stackrel{ \eqref{qwefqwfoij11111} }{\simeq}& \colim_{V\in \cC^{G,\Delta}_{X}} \Map_{\Sh_{\bC}^{G}(X)}(M,V^{G}_{*}N)\\
&\stackrel{\eqref{ergiojgoiegwregwreg}}{\simeq}&\colim_{V\in \cC^{G,\Delta}_{X}}  
\prod_{i\in I} \Map_{\Sh_{\bC}^{G,\eqsm}(X_{i})}(M_{i},V^{G}_{i,*}N_{i})\\&\stackrel{!}{\simeq}&\prod_{i\in I} \colim_{V_{i}\in \cC^{G,\Delta}_{X_{i}}}  
  \Map_{\Sh_{\bC}^{G,\eqsm}(X_{i})}(M_{i},V^{G}_{i,*}N_{i})\\&\stackrel{ \eqref{qwefqwfoij11111}}{\simeq}&
\prod_{i\in I} \Map_{\bV^{G}_{\bC}(X_{i})}(\ell_{X_{i}} M_{i},   \ell_{X_{i}} N_{i} )\ ,
 \end{eqnarray*}
 where for the marked equivalence we use  the definition of the coarse structure of the free union (\cref{fqiofwefewfeqwfeqwf} \eqref{r3fqu9z3fpu8934f13314f1}) and the fact that filtered colimits distribute over products in spaces  (see \cref{wergewgregwergw} and \cref{rewgoiu0erogwegwergwrgw}).
 This equivalence shows that \eqref{owierjgweriogwergwergregegwgr} is fully faithful.
\end{proof}

\section{ {Constructing coarse homology theories}}\label{wiofoqrefwefeqfewfqewf}

This section describes constructions of equivariant coarse homology theories from the 
functor $\bV_{\bC}^{G}$ studied in the previous section. The first possibility is to compose this functor with a homological functor as in \cref{qrevoiqrjoirqfcwqecq}. The resulting coarse homology  {theory} lacks the additional property of continuity  {which we introduce in \cref{geriogjergerg}}. 
In \cref{gioegwgwrgrgreggwregwrg} we therefore first apply the construction of {\em forcing continuity} in order to define  an improved functor  $\bV^{G,c}_\bC$. Another path to an equivariant coarse homology theory pursued in \cref{sec:orbit-theory}
is to  apply the  continuous version $\bV_{\bC}^{c}$ of the functor   $\bV_\bC$  to $G$-bornological coarse spaces and to apply $\colim_{BG}$ to the resulting left-exact $\infty$-categories with $G$-action.
The resulting functor  will be denoted by   $\bV^{c}_{\bC,G}$.
In \cref{sec:homol} we recall  the notion of an equivariant coarse homology theory and show  that both $\bV^{G,c}_\bC$ and $\bV^c_{\bC,G}$ induce equivariant coarse homology theories upon  composition with a homological functor. Finally, in
\cref{sec:calculations} we determine   the values of $\bV^{G,c}_\bC$ and $\bV^c_{\bC,G}$ on certain $G$-bornological coarse spaces arising from transitive $G$-set.

\subsection{Forcing continuity}\label{gioegwgwrgrgreggwregwrg}
 Let $X$ be  {a $G$-bornological coarse space},
 and let $F$ be a subset of $X$.
 
 \begin{ddd}[{\cite[Def.~5.1]{equicoarse}}]\label{def:locfin}
  $F$ is  locally finite if for every  {bounded subset $B$ of $X$}
  the set $B\cap F$ is finite.
 \end{ddd}
 We let $\cF(X)$ denote the poset of invariant locally finite subsets  of $X$.
 
 {Let $E\colon G\BC\to \bM$ be a functor with a target admitting filtered colimits.
 \begin{ddd}\label{geriogjergerg}
 $E$ is continuous if the canonical morphism
 	\[ \colim_{F \in \cF(X)} E(F) \to E(X) \]
 	is an equivalence for every  {$G$-bornological coarse space} $X$.
 \end{ddd}
 
 Let
 \[ i\colon G\BC^{\mb}\to G\BC \]
 be the inclusion of the full subcategory of $G$-bornological coarse spaces which have the minimal bornology.

 \begin{lem}\label{lem:forcects}
  The left Kan extension $E^c$ of $E \circ i$ along $i$  \begin{equation*}
 \xymatrix{G\BC^{\mb}\ar[r]^-{E \circ i}
 	 \ar[d]_-{i}&\bM \\
 G\BC\ar@{..>}[ru]_-{E^{c}}&}\ ,
 \end{equation*}
  exists and is  a continuous functor.
  The functor $E$ is continuous if and only if the canonical transformation $E^c \to E$ is an equivalence.
 \end{lem}
 \begin{proof}
  The left Kan extension of $E \circ i$ exists if and only if the colimit
  \[ \colim_{Y\in (G\BC^{\mb})_{/X}}E(Y) \]
 exists in $\bM$ for all $G$-bornological coarse spaces $X$.
 It is easy to see that the functor
 \[ \cF(X)\to (G\BC^{\mb})_{/X},\quad F \mapsto (F \to X) \]
 is cofinal. Since $\cF(X)$ is filtered and filtered colimits in $\bM$ exist by assumption, we conclude that $E^{c}$ exists.
 Furthermore, we have an equivalence  \begin{equation}\label{wergwergwegwregwrewg}
E^{c}(X)\simeq \colim_{F \in \cF(X)} E(F) \ .
\end{equation} 
The canonical transformation $E^c \to E$ is given pointwise by the canonical morphism
 \begin{equation*}
E^{c}(X)\stackrel{\eqref{wergwergwegwregwrewg}}{\simeq}  \colim_{F \in \cF(X)} E(F) \to E(X)\ .
\end{equation*} 
  Hence $E$ is continuous if and only if  $E^{c}\to E$ is an equivalence.
 Since $i$ is fully faithful, we have $E^c \circ i \simeq E \circ i$. It follows that $(E^c)^c \simeq E^c$, so $E^c$ is continuous. 
 \end{proof}

 \begin{ddd}\label{gerklgjerlgergergergergerg}
 	We say that the functor $E^{c}$ is obtained from $E$ by forcing continuity.
 \end{ddd}
 
In the following, we show that if a functor $E^{c}$ is obtained from $E$ by forcing continuity, then it inherits various properties from $E$.

Recall   \cref{rgiherogiergregreg} of coarse invariance.
\begin{lem}\label{etgiowetgregrwegergwr}
If $E$ is coarsely invariant, then $E^{c}$ is coarsely invariant.
\end{lem}
\begin{proof}
 Let $X$ be {a $G$-bornological coarse space}.
 We must show that the projection $\{0,1\}_{max,max}\otimes X\to X$ induces an equivalence
 \[ E^{c}(\{0,1\}_{max,max}\otimes X)\to E^{c}(X)\ .\]
The collection of   subsets
$\{0,1\}\times F$ of $\{0,1\}\times X$ for   $F$ in $\cF(X)$ is cofinal  in $\cF(\{0,1\}_{max,max}\otimes X)$. Therefore,  we get the second equivalence in the following chain:
\begin{eqnarray*}
E^{c}(\{0,1\}_{max,max}\otimes X)&\stackrel{\eqref{wergwergwegwregwrewg}}{\simeq}&\colim_{F^{\prime}\in\cF(\{0,1\}_{max,max}\otimes X)  } E( F^{\prime} )\\
&\simeq& \colim_{F\in\cF(X)  }E(\{0,1\}_{max,max}\otimes F)\\
&\stackrel{!}{\simeq}&  \colim_{F\in\cF(X)  } E( F)\\&\stackrel{\eqref{wergwergwegwregwrewg}}{\simeq}& E^{c}(X)\ .
\end{eqnarray*}
 The equivalence marked by $!$   follows from the coarse invariance of $E$.
\end{proof}

Recall \cref{rgiuerogergergergre} of  $u$-continuity.
\begin{lem}\label{rewgoijerg9rwup29igwg}
 If $E$ is $u$-continuous, then $E^{c}$ is $u$-continuous.
\end{lem}
\begin{proof}
Let $X$ be  {a $G$-bornological coarse space}.
We must show that the canonical morphism
\[ \colim_{U\in \cC^{G}_{X}} E^{c}(X_{U})\to E^{c}(X) \]
is an equivalence. 
This follows from the following chain of equivalences: 
\begin{eqnarray*} \colim_{U\in \cC^{G}_{X}} E^{c} (X_{U})&\stackrel{\eqref{wergwergwegwregwrewg}}{\simeq}&
\colim_{U\in \cC^{G}_{X}} \colim_{F\in\cF(X) } E(F_{X_{U}})\\
&\simeq&  \colim_{F\in\cF(X) } \colim_{U\in \cC^{G}_{X}} E(F_{X_{U}})\\
&\stackrel{!}{\simeq}&   \colim_{F\in\cF(X) } \colim_{U\in \cC^{G}_{X}} \colim_{V\in \cC^{G}_{F_{X_{U}}}}E(F_{V})\\&\stackrel{!!}{\simeq}&
 \colim_{F\in\cF(X) } \colim_{V\in \cC^{G}_{F_{X}}}  E(F_{V})\\
&\stackrel{!}{\simeq}&  \colim_{F\in\cF(X) }     E(F_{X} )\\
&\stackrel{\eqref{wergwergwegwregwrewg}}{\simeq}& E^{c}(X)\ .
\end{eqnarray*}
Here $F_{X_{U}}$ denotes the $G$-bornological coarse space $F$ with the structure induced from $X_{U}$. 
For the equivalences marked with $!$ we use that $E$ is $u$-continuous. For $!!$ we use that
\[ \{(U,V)\:|\: U\in \cC_{X}^{G}\ , V\in \cC^{G}_{F_{X_{U}}}\}\to \cC_{F_{X}}^{G}\ , \quad (U,V)\mapsto V\]
 is cofinal.
\end{proof}

 Let $\bM$ be a semi-additive $\infty$-category  which in addition admits all small filtered colimits, and consider a 
 functor  $E\colon G\BC\to \bM$. 
 Recall  \cref{rt98uerg9rt342trg32432f}  of a flasqueness-preserving functor. 

\begin{lem}\label{ergoiuer9gowergergreg}
If $E$ preserves flasqueness, then so does $E^{c}$.
\end{lem}
\begin{proof}
 Assume that $X$ is a flasque {$G$-bornological coarse space}
 with flasqueness implemented by $f\colon X\to X$.
If $F$ is {an invariant locally finite subset of $X$},
then also  {$\tilde F:=\bigcup_{n\in \nat} f^{n}(F)$ is invariant and locally finite}.
Furthermore, $\tilde F$ is flasque with flasqueness implemented by the restriction $f|_{\tilde F}$. 
 The map $\cF(X)\to \cF(X)$, $F\mapsto \tilde F$ is cofinal. This fact provides the second equivalence in
\[ E^{c}(X)\stackrel{\eqref{wergwergwegwregwrewg}}{\simeq}  \colim_{F\in \cF(X)} E(F)\simeq \colim_{F\in \cF(X)} E(\tilde F)\ .\]
Since $E(\tilde F)$ belongs  to $\bM^{\fl}$  for every $F$  in $\cF(X)$,   $\bM^{\fl}$  is closed under filtered colimits, and since the poset $\cF(X)$ of invariant locally finite subsets is filtered,
also $E^{c}(X)$ belongs to $\bM^{\fl}$.
 \end{proof}

  Let $\bP$ be some $\infty$-category and
$E\colon G\BC\times \bP\to  \bM$ be a  functor with $\bM$ as above.
 We let $E^{c}$ denote the functor obtained from $E$ by {forcing continuity in the first variable}.
Recall  \cref{goiijgowgergegergregewrg} and \cref{goiijgowgergegergregewrg-111} of a functorially (pre-)flasqueness preserving functor.
\begin{lem}\label{goijergoirgergregregwerwergw}
\mbox{}
\begin{enumerate}
\item \label{egopergergwergwegwerg}
If $E$ functorially preserves pre-flasqueness, then so does
$E^{c}$. 
\item\label{gfoiqjgoqiwqqfewfqewfq} If $E$ functorially preserves  flasqueness, then so does
$E^{c}$. 
\end{enumerate}
\end{lem}
\begin{proof}
We  give the argument for \eqref{egopergergwergwegwerg}.
 {Define $\Flrm^{\pre}(E)^c$ as the following left Kan extension:
\[\xymatrix@C=9em{
 \preFl(G\BC^{\mb})\times \bP\ar[r]^-{\Flrm^{\pre}(E)\circ (\Flrm^{\pre}(i)\times \id_{\bP})}\ar[d]_-{\Flrm^{\pre}(i)\times \id_{\bP}} & \widetilde{\Fl}(\bM) \\
  \preFl (G\BC)\times \bP\ar[ur]_-{\Flrm^{\pre}(E)^c} &
}\]
We let $\Phi \colon \preFl(G\BC^{\mb})\times \bP \to \bEnd(\bM)$ denote the functor which sends $((X,f),P)$ to $(E(X,P), E(f,P))$.
Since the functor $e \colon \widetilde{\Fl}(\bM) \to \bEnd(\bM)$ from \cref{thgiojgoiggfergrewgerewergewrgreggw} preserves colimits, the pointwise formula for left Kan extensions implies that the composite $\Phi^c := e \circ \Flrm^{\pre}(E)^c$ is a left Kan extension of the composite functor
\[ \preFl(G\BC^{\mb})\times \bP \xrightarrow{\Flrm^{\pre}(E)\circ (\Flrm^{\pre}(i)\times \id_{\bP})} \Fl(\bM) \xrightarrow{e} \End(\bM)\ .\]
This implies that $\Phi^c$ is a left Kan extension of $\Phi \circ (\Flrm^{\pre}(i) \times \id_{\bP})$ because $\Phi \simeq e \circ \Flrm^{\pre}(E)$.
We have to show that $\Phi^c \colon \preFl(G\BC) \times \bP \to \bEnd(\bM)$ is equivalent to the functor $\Psi$ sending $((X,f),P)$ to $(E^c(X,P), E^c(f,P))$.}

We consider the following diagram
\[ {\xymatrix{
 &\preFl(G\BC^{\mb})\times \bP\ar[dd]\ar[dl]_-{\Flrm^{\pre}(i)\times \id_{\bP}\hphantom{xx}} \ar[dr]^-{\hphantom{xxxx}\Phi \circ (\Flrm^{\pre}(i)\times \id_{P})}&\\
 \preFl (G\BC)\times \bP\ar[dd]_{q}\ar@{..>}[rr]^-(.3){\Phi^c}&& \bEnd(\bM)\ar[dd]^-{u} \\
 &G\BC^{\mb}\times \bP\ar[dr]^-{\hphantom{xx}E\circ ( i\times  \id_{P})}\ar[dl]_-{i\times \id_{P}}&\\
 G\BC\times \bP\ar@{..>}[rr]^-{E^{c}}&&\bM}} \]
 where the dotted arrows are defined as left Kan extensions, respectively.
The universal property of the  left Kan extension  $ {F^c}$ and the fact that $ {u}$ preserves colimits provides a natural transformation
\[  {u \circ \Phi^{c}} \to E^{c}\circ q\ . \]
We must check that this transformation is an equivalence. In view of the pointwise formula for the left Kan extensions, this amounts to showing that for every $P$ in $\bP$ and $(X,f)$ in $\preFl(G\BC)$ the morphism
\begin{eqnarray*}
\lefteqn{\colim_{((F,g)\to (X,f))\in \preFl(G\BC^{\mb})_{/(X,f)}} {u(\Phi}((F,g),P))}\hspace{5cm}&&\\&\to& \colim_{(F\to X)\in (G\BC^{\mb})_{/X}} E(F,P)\end{eqnarray*} is an equivalence.  The argument given in the proof of  \cref{ergoiuer9gowergergreg} shows that
the functor $q$ induces a cofinal functor
\[ \Fl^{\pre}(G\BC^{\mb})_{/(X,f)} \to  (G\BC^{\mb})_{/X}\ .  \]
Consequently, we can rewrite the morphism in question in the form \begin{align*}
 \colim_{ ((F,g)\to (X,f))   \in \preFl(G\BC^{\mb})_{/(X,f)}} & {u(\Phi}(E)((F,g),P)) \\
 &\hspace{-2cm}\to \colim_{ ((F,g)\to (X,f)) \in \preFl(G\BC^{\mb})_{/(X,f)}} E(q(F,g),P)\ .
\end{align*}
This is an equivalence since
 \[ {u \circ \Phi} \simeq E\circ (q\times \id_{\bP})\ .\]
 Consequently, we also have $u \circ \Phi^c \simeq E^c \circ q$.
Since $\Phi \circ (\Flrm^{\pre}(i) \times \id_\bP)$ and $\Psi \circ (\Flrm^{\pre}(i) \times \id_\bP)$ are equivalent, the universal property of the left Kan extension also provides a natural transformation $\Phi^c \to \Psi$.
As $u$ is conservative, the preceding argument shows that this transformation is an equivalence.

Assertion \eqref{gfoiqjgoqiwqqfewfqewfq} is shown analogously.
 \end{proof}

Let $\bM$ be a pointed $\infty$-category admitting finite coproducts,  
small filtered colimits and  small  products. 
Recall   \cref{rgiowergerwgregwergergwerg} and  \cref{wergiugewrgergergerwgwregw} of $\pi_{0}$-excisiveness and strong additivity. 

\begin{lem} \label{ifuewfioqewfwefewfewqf}\mbox{}
\begin{enumerate} \item
 \label{ergioewrjgowregregwrgregwreg}
If $E$ is $\pi_{0}$-excisive, then so is $E^{c}$.
 \item   \label{ergioewrjgowregregwrgregwreg1} If $E$   is strongly additive and filtered colimits distribute over products in $\bM$ (see \cref{wergewgregwergw}), then $E^{c}$ is also strongly additive.  
\end{enumerate} 
\end{lem}
\begin{proof}
We first show \eqref{ergioewrjgowregregwrgregwreg}.
Let $X$ be  {a $G$-bornological coarse space}
with a coarsely disjoint  {decomposition} $(Y,Z)$ into invariant subsets.
 For every  {invariant locally finite subset $F$ of $X$},
 we get a partition $(F\cap Y,F\cap Z)$ of $F$ into coarsely disjoint invariant subsets.
 Since $E$ is $\pi_{0}$-excisive, we conclude that
 \[ E(F)\simeq E(F\cap Y)\sqcup E(F\cap Z)\ .\]
 The formula \eqref{wergwergwegwregwrewg} implies  the equivalence
\[ E^{c}(X)\simeq E^c(Y)\sqcup E^c( Z) \]
by taking the colimit over {the poset $\cF(X)$ of invariant locally finite subsets}. Here we use that the projection  $\cF(X)\to \cF(Y)$, $F\mapsto F\cap Y$  is  cofinal in order to get the equivalence
\[ \colim_{F\in \cF(X)} E(F\cap Y)\simeq \colim_{F^{\prime}\in \cF(Y)} E(F^{\prime}) \]
(and similarly for $Z$).
  
We now  show \eqref{ergioewrjgowregregwrgregwreg1}.
Let $(X_{i})_{i\in I}$ be a family  {of $G$-bornological coarse spaces}
and set $X := \bigsqcup_{i\in I}^{\free}X_{i}$  (\cref{fqiofwefewfeqwfeqwf}). 
 Then a subset $F$ of $X$ is locally finite if and only if $F\cap X_{i}$ is a locally finite subset of $X_i$ for every $i$ in $I$.
Hence we have an isomorphism of posets $\cF(X)\cong \prod_{i\in I} \cF(X_{i})$ given by $F\mapsto (F\cap X_{i})_{i\in I}$.
It gives the first equivalence in the following chain
\begin{equation}\label{rfrfwerfrefrrwrefrfrw}\colim_{F\in \cF(X)}\prod_{i\in I}E ( {F\cap X_{i}})\simeq \colim_{(F_{i})_{i}\in \prod_{i\in I}\cF(X_{i})}\prod_{i\in I}E ( F_{i})
\simeq  \prod_{i\in I} \colim_{F_{i}\in \cF(X_{i})} E(F_{i})\ ,
\end{equation} 
while the second follows from the assumption that filtered colimits distribute over products in $\bM$.

We have an isomorphism $F\cong \bigsqcup_{i\in I}^{\free} (F\cap X_{i})$ in $G\BC$. Using the assumption that $E$ is strongly additive,  
we get the marked equivalence in the following commutative diagram
\[\xymatrix{
\colim_{F\in\cF(X)}E(F)\ar[r]_-{ \eqref{wergwergwegwregwrewg}}^-\simeq\ar[d]_{!}^\simeq&E^{c}(X)\ar[dd]\\
\colim_{F\in \cF(X)}\prod_{i\in I} E ( F\cap X_{i} )\ar[d]_{\eqref{rfrfwerfrefrrwrefrfrw}}^\simeq&\\
\prod_{i\in I}\colim_{F_i\in \cF(X_i)}E(F_{i})\ar[r]_-{ \eqref{wergwergwegwregwrewg}}^-{\simeq}&\prod_{i\in I}E^{c}(X_i)
}\ .\]
We conclude that the right vertical morphism is an equivalence as desired. This implies Assertion \eqref{ergioewrjgowregregwrgregwreg1}.
\end{proof}

 We  consider a functor $E\colon G\BC\to\bM$.
 Recall the notion of excisiveness from \cref{rgoiruegoiregregregreg}  and the notion of $l$-excisiveness (for $\bM=\Cle$) from \cref{rgoiruegoiregregregreg-modified}.
  \begin{lem}\label{egiweogergergewrg}
  If $E$ is  excisive or $l$-excisive, then the same is true for $E^{c}$.
  \end{lem}
\begin{proof}
 Let $X$ be  {a $G$-bornological coarse space}
 with a complementary pair $(Z,\cY)$. Then for every  {invariant locally finite subset $F$ of $X$}
 we get a complementary pair 
 $(F \cap Z, F \cap \cY)$ on $F$. Hence
 \begin{equation}\label{rgvlpgbergbgbpblgebegbegrbegb}
\xymatrix{E^{c} (Z\cap \cY)\ar[r]\ar[d]&E^{c}(\cY)\ar[d]\\E^{c}(Z)\ar[r]&E^{c}(X)}
\end{equation}
 is the colimit over {all invariant locally finite subsets $F$ of $X$}
 of the following pushout squares in $\bM$ (or excisive squares in $\Cle$, see \cref{ugioerguoerug}):
 \[\xymatrix{E (F\cap Z\cap \cY)\ar[r]\ar[d]&E (F\cap \cY)\ar[d]\\E (F\cap Z)\ar[r]&E (F_{{X}})}\]     
 Here we use  the assumption that $E$ is excisive or $l$-excisive, respectively. Since a filtered colimit of pushout squares (or excisive squares in $\Cle$)  is again a pushout square  (or excisive square in $\Cle$, see \cref{rfiorjgfoqfwewfefewfqfe}) we conclude 
  that \eqref{rgvlpgbergbgbpblgebegbegrbegb} is a pushout  square (or an excisive square in $\Cle$).
\end{proof}

We now apply the construction of forcing continuity to the functor $\bV_{\bC}^{G}$
introduced in \cref{rtheorthertherthetrhe}. It is the evaluation at $\bC$ of the two-variable version $\bV^{G}$ from \cref{bojoijgoi3jg3g34f}.

\begin{ddd}\label{iowergergwegr}
 {We define
 \[ \bV^{G,c}\colon G\BC\times \Fun(BG,\CL)\to \Cle \]
 as the functor obtained from $\bV^{G}$ by forcing continuity in the first variable.}
\end{ddd}
{We write $\bV_{\bC}^{G,c}$ for the evaluation of this functor at a fixed object $\bC$ in $\Fun(BG,\CL)$.}

\begin{kor}\label{regwergergergrgwgregwregwregwergwreg}
The functor $\bV_{\bC}^{G,c}$ is
\begin{enumerate}
	\item \label{thwoierjoiwegjergg}coarsely invariant  (\cref{rgiherogiergregreg}),
	\item \label{thwoierjoiwegjergg1}$u$-continuous (\cref{rgiuerogergergergre}),
	\item \label{thwoierjoiwegjergg2} $l$-excisive (\cref{rgoiruegoiregregregreg-modified}),
	\item \begin{enumerate} \item flasqueness preserving (\cref{rt98uerg9rt342trg32432f}),
		\item  functorially pre-flasqueness preserving (\cref{goiijgowgergegergregewrg}),
		\item  \label{thwoierjoiwegjergg3} functorially  flasqueness preserving (\cref{goiijgowgergegergregewrg-111}),
	\end{enumerate}
	\item \begin{enumerate}
		\item $\pi_{0}$-excisive (\cref{rgiowergerwgregwergergwerg}),
		\item strongly additive (\cref{wergiugewrgergergerwgwregw}) and
	\end{enumerate}
	\item\label{wgwijretgowergfwefwref} continuous (\cref{geriogjergerg}).
\end{enumerate}
\end{kor}
\begin{proof}
 Coarse invariance holds by \cref{riogogrgregergerg,etgiowetgregrwegergwr}, $u$-continuity by \cref{gsnsioghetg,rewgoijerg9rwup29igwg}, and excision by \cref{rgijrgoirejgoergergreg,egiweogergergewrg}. The claims about the preservation of flasqueness are contained in \cref{ergoegergregreg,ergoiuer9gowergergreg,goijergoirgergregregwerwergw,goijergoirgergregregwerwergw}.
	$\pi_0$-excision and strong additivity follow from \cref{ergiowjoetgergreergwergwreg,ergieorogergergergregerge,ifuewfioqewfwefewfewqf}. Finally, continuity  follows from \cref{lem:forcects}.
\end{proof}

 {For future reference, we also record the following fact.}
Let $X$ be  {a $G$-bornological coarse space},
and let $\bC$ be in $\Fun(BG,\CL)$.

\begin{lem}\label{qerguigwregergwrgeg}
The canonical morphism $\bV_{\bC}^{G,c}(X)\to \bV_{\bC}^{G}(X)$ is fully faithful.
\end{lem}
\begin{proof}
Let $F$ be {an invariant, locally finite subset of $X$}.
By \cref{prop:inclusion-fff}, the inclusion $F\to X$ induces a fully faithful functor
\[ \bV_{\bC}^{G}(F)\to \bV_{\bC}^{G}(X)\ .\]
Since a filtered colimit  of fully faithful functors is fully faithful, we conclude that
\[ \bV_{\bC}^{G,c}(X)\stackrel{\eqref{wergwergwegwregwrewg}}{\simeq}\colim_{F\in \cF(X)} \bV_{\bC}^{G}(F)\to \bV_{\bC}^{G}(X) \]
 is fully faithful.
\end{proof}

\subsection{The orbit theory}\label{sec:orbit-theory}
The second construction of a coarse homology theory from our categories of controlled objects  starts from $\bV^{c}_\bC=\bV^{\{e\},c}_{\bC}$ given by the case of \cref{iowergergwegr} for the trivial group. If we apply this functor 
 to $G$-bornological coarse spaces and {an object $\bC$ of $\Fun(BG,\CL)$}, then by functoriality  its values become left-exact $\infty$-categories with $G$-action.  In \cref{qoirejoqierjgqergrefqerwfewf}, we then define  the functor $\bV^{c}_G$  by taking  $G$-orbits, i.e.~{by} composing with the functor $\colim_{BG}$. In this section we show that $\bV^{c}_G$ inherits most of the properties established for $\bV^{c}_\bC$ in \cref{regwergergergrgwgregwregwregwergwreg}.

Let $M$ be a cocomplete $\infty$-category, $\bP$ some auxiliary $\infty$-category, and consider a functor $E\colon \BC\times \bP\to \bM$. 
  
  \begin{ddd}\label{thklwhwhtthwgregr}
  We define the functor $E_{G}$ as the composition  \begin{align}\label{rthiowrthrthwhwgwergwrgw}
   G\BC&\times \Fun(BG,\bP)\to \Fun(BG,\BC)\times  \Fun(BG,\bP)  \\  &\xrightarrow{E}\Fun(BG\times BG,\bM)\xrightarrow{\diag_{BG}^{*}} \Fun(BG,\bM)\xrightarrow{\colim_{BG}} \bM  \ . \qedhere\end{align}
  \end{ddd}
 
 Let $\ev\colon \Fun(BG,\bM)\to \bM$ denote the evaluation functor (see \eqref{werwregfwreggregwrweggwergwerg}).
For $P$ in  $\Fun(BG,\bP)$
write $E_{P,G}$ for the specialisation of $E_{G}$ at $P$. The functor $E_{P,G}$ inherits various coarse properties from $E(-,\ev(P))$. 

Recall  \cref{rgiherogiergregreg} of the notion of coarse invariance.
\begin{lem}\label{etgiowetgregrwegergwr-colim}
 If $E(-,\ev(P))$ is coarsely invariant, then $E_{P,G}$ is coarsely invariant.
\end{lem}
\begin{proof}
 Let $X$ be  {a $G$-bornological coarse space}.
 We must show that the projection
 \[ \{0,1\}_{max,max}\otimes X\to X \]
 induces an equivalence
 \[ E_{P,G}(\{0,1\}_{max,max}\otimes X)\to E_{P,G}(X)\ .\]
 Since $E(-,\ev(P))$ is coarsely invariant, and equivalences in $\Fun(BG,\bM)$ are detected  by the evaluation functor 
 $\ev\colon \Fun(BG,\bM)\to \bM$, the projection induces an equivalence 
 \[ E(\{0,1\}_{max,max}\otimes X,P)\to E(X,P) \]
 in $\Fun(BG,\bM)$. Applying $\colim_{BG}$, we get the desired equivalence.
 \end{proof}

 Recall \cref{rgiuerogergergergre} of  $u$-continuity.
\begin{lem}\label{rewgoijerg9rwup29igwg-colim}
 If $E(-,\ev(P))$ is $u$-continuous, then $E_{P,G}$ is $u$-continuous.
\end{lem}
\begin{proof}
Let $X$ be  {a $G$-bornological coarse space}.
We must show that the canonical morphism
\[ \colim_{U\in \cC^{G}_{X}} E_{P,G}(X_{U})\to E_{P,G}(X) \]
is an equivalence.
Since {$E(-,\ev(P))$} is $u$-continuous, {the sub-poset of invariant coarse entourages} $\cC^{G}_{X}$ is cofinal in $\cC_{X}$ by \cref{trbertheheht}~\eqref{igwoegwergergwrgrg}, and 
{the forgetful functor $\Fun(BG,\bM) \to \bM$} preserves colimits,
we have an equivalence
\[ \colim_{U\in \cC^{G}_{X}} E(X_{U},P)\xrightarrow{\simeq} E(X,P) \]
in $\Fun(BG,\bM)$.  Applying $\colim_{BG}$, we get the desired equivalence.
 \end{proof}

 Recall \cref{rgoiruegoiregregregreg} of excisiveness and  \cref{rgoiruegoiregregregreg-modified} of $l$-excisiveness (for $\bM=\Cle$).
\begin{lem}\label{egiweogergergewrg-colim}
 If $E(-,\ev(P))$ is excisive or $l$-excisive, then so is $E_{P,G}$.
\end{lem}
\begin{proof}
Let $X$ be {a $G$-bornological coarse space}
with a complementary pair $(Z,\cY)$. 
 	By  the assumption on $E$, the square
  \begin{equation*}\label{krgwejfewjfllwkeffwefefwef-colim}
		 \xymatrix{
		 	 E(  \cY\cap Z,P)\ar[r]\ar[d]&  E(Z,P)\ar[d] \\
		 	  E( \cY,P)\ar[r]&  E(X,P)
 }\end{equation*}
 is  a pushout square (or excisive square in the case $\bM=\Cle$) in $\Fun(BG,\bM)$.
 Applying $\colim_{BG}$ produces the desired   pushout  square (or excisive square in the case $\bM=\Cle$ by \cref{rfiorjgfoqfwewfefewfqfe})
 \begin{equation*}\label{krgwejfewjfllwkeffwefefwef-colim-1}
	 \xymatrix{
	 	 E_{P,G}(  \cY\cap Z)\ar[r]\ar[d]&  E_{P,G}(Z)\ar[d] \\
	 	 E_{P,G}( \cY)\ar[r]&   E_{P,G}(X)
	}\end{equation*}
in $\bM$.
\end{proof}

Recall \cref{rt98uerg9rt342trg32432f} of a  flasqueness preserving functor and  \cref{goiijgowgergegergregewrg-111} of a  functorially flasqueness preserving functor.
\begin{lem}\label{egoerpgrgwegwegwreg}
If $E$ is functorially flasqueness preserving, then $E_{G}$ is  functorially flasqueness preserving.
\end{lem}
\begin{proof}
 By assumption, $E$ has an extension $\Flrm(E) \colon \Fl(\BC)\times \bP\to \Fl(\bM)$.  Let
 	\[ E_{\mathrm{eq}} \colon G\BC \times \Fun(BG,\bP) \to \Fun(BG,\bM) \]
 	be the composition of the first three morphisms in \eqref{rthiowrthrthwhwgwergwrgw} 
	such that $E_G \simeq \colim_{BG} E_{\mathrm{eq}}$.
	 {Then $E_{\mathrm{eq}}$ induces a functor
	\[ \Phi_{\mathrm{eq}} \colon \Fl(G\BC) \times \Fun(BG,\bP) \to \Fun(BG,\bEnd(\bM)) \]
	sending $((X,f),P)$ to $(E_{\mathrm{eq}}(X,P),E_{\mathrm{eq}}(f,P))$.}
	We define $\Flrm(E)_{\mathrm{eq}}$ similarly. We get the following diagram 
\[ {\xymatrix@C=3em{
 	\Fl(G\BC) \times \Fun(BG,\bP)\ar[dr]_{\Phi_{\mathrm{eq}}}\ar[r]^-{\Flrm(E)_{\mathrm{eq}}} & 
 	\Fun(BG, \Fl(\bM))\ar[r]^-{\colim_{BG}}\ar[d] & \Fl(\bM)\ar[d] \\
	&\Fun(BG,\bEnd(\bM))\ar[r]^-{\colim_{BG}}&\bEnd(\bM)
 }}\ ,\]
 where the left part commutes since $\Flrm(E)$ witnesses that $E$ is functorially flasqueness preserving (see \eqref{grelgijeroigergregergwergrewgwreg}),
 and the right part commutes since the two vertical functors preserve colimits.
 The commutativity of the outer part of the diagram shows that $E_G$ is functorially flasqueness preserving.
\end{proof}

 Recall the functor  $i \colon G\Set \to G\BC$ from \eqref{qreoijqoiegjoqirffewfq} sending  $S$ to $S_{min,max}$.
The orbit category $G\Orb$ is the full subcategory of $G\Set$ of transitive $G$-sets. 
 {We call a $G$-bornological coarse space $X$ bounded if $X$ is a bounded subset of itself}.
  Consider a functor $E' \colon G\BC\to \bM$ with a cocomplete target.

\begin{ddd}\label{ergeiguhgiergwregwergwerg}
$E'$ is called hyperexcisive if for every  {$G$-set} $W$
and bounded  {$G$-bornological coarse space} $X$
the morphism
\begin{equation}\label{fvfvqrw3v}
\colim_{(S\to W)\in G\Orb_{/W}}E'(S_{min,max}\otimes X)\to E'(W_{min,max}\otimes X)
\end{equation}
is an equivalence. 
\end{ddd}
\begin{rem}
The equivalence in \eqref{fvfvqrw3v} can be rewritten as 
\[  \coprod_{S\in W/G} E'(S_{min,max}\otimes X)  \simeq E'(W_{min,max}\otimes X)\ ,\]
i.e., a  hyperexcisive functor is  excisive for certain infinite coarsely disjoint decompositions.  This property  is  {non-trivial}
if $W/G$ is infinite, otherwise it follows from $\pi_{0}$-excisiveness (\cref{rgiowergerwgregwergergwerg}). 
 \end{rem} 
 
 Consider again the situation that $E$ is a functor $\BC \times \bP \to \bM$, and that $P$ is in $\Fun(BG,\bP)$.

\begin{lem}\label{gwgwergrwegwerg}
Assume: 
\begin{enumerate}
\item  {$E(-,\ev(P))$} is continuous (\cref{geriogjergerg}).
\item  {$E(-,\ev(P))$} is $\pi_{0}$-excisive (\cref{rgiowergerwgregwergergwerg}).
 \end{enumerate}
Then  {$E_{P,G}$} is hyperexcisive.
 \end{lem}
\begin{proof}
 Let $X$  {be a bounded $G$-bornological coarse space},
 and let $W$ be  {a $G$-set}.
 Let $G\Set^f$ be the full subcategory of $G\Set$ of $G$-finite $G$-sets.
Then we have a commutative diagram \begin{equation}\label{qefojoqiewfqewfqwefe}
	\hspace{-0.7cm}\xymatrix@R=2em{
	 \colim\limits_{(R \to W) \in G\Set^f_{/W}} E(R_{min,max} \otimes X,\ev(P))\ar[r] & E(W_{min,max} \otimes X,\ev(P)) \\
	 \colim\limits_{(R \to W) \in G\Set^f_{/W}} \colim\limits_{F \in \cF({\res^{G}_{\{1\}}}(R_{min,max} \otimes X))} E(F,\ev(P))\ar[r]\ar[u]^-{\simeq} & \colim\limits_{F \in \cF({\res^{G}_{\{1\}}}(W_{min,max} \otimes X) )} E(F,{\ev(P)})\ar[u]_-{\simeq}
	} \end{equation}
({we omitted the symbol $\res^{G}_{\{1\}}	$ in the first argument of $E$ in the upper line)
}
 in which both vertical maps are equivalences by continuity of $E(-,\ev(P))$.  Every locally finite subset of $W_{min,max} \otimes X$ is contained in a subset of the form $R  \times X$ for some  {$G$-finite $G$-set} $R$.
 Hence the lower horizontal arrow is induced by a cofinal functor, and is thus also an equivalence. It follows that the top horizontal arrow is an equivalence.
 We have the sequence of equivalences
	\begin{align*}
	 \colim_{(S \to W) \in G\Orb_{/W}}& E_G(S_{min,max} \otimes X,\ev(P)) \\
	 	&\simeq \colim_{BG} \colim_{(S \to W) \in G\Orb_{/W}} E(S_{min,max} \otimes X,\ev(P)) \\
	 	&\stackrel{{!}}{\simeq}  \colim_{BG} \colim_{(R \to W) \in G\Set^f_{/W}} \colim_{(S \to R) \in G\Orb_{/R}} E(S_{min,max} \otimes X,\ev(P)) \\
	 	&\stackrel{{!!}}{\simeq}  \colim_{BG} \colim_{(R \to W) \in G\Set^f_{/W}} E(R_{min,max} \otimes X,\ev(P)) \\
	 	&\stackrel{{!!!}}{\simeq} E_G(W_{min,max} \otimes X,\ev(P))\ ,
	\end{align*}
where the  equivalence marked by $!$  follows from a cofinality consideration, the   equivalence marked by $!!$ uses $\pi_0$-excisiveness of $E(-,\ev(P))$, and the equivalence
marked by $!!!$ is the upper horizontal equivalence in \eqref{qefojoqiewfqewfqwefe}.
\end{proof}

\begin{ddd}\label{qoirejoqierjgqergrefqerwfewf}
 We define the functor
 \[ \bV^{c}_{G}\colon G\BC\times \Fun(BG,\CL)\to \Cle \]
 by  applying \cref{thklwhwhtthwgregr} to $\bV^{c}$ (the functor from \cref{iowergergwegr} in the case of trivial $G$).
\end{ddd}

{In other words, $\bV^c_G$ sends a $G$-bornological coarse space $X$ and an object $\bC$ in $\Fun(BG,\CL)$ to the left-exact $\infty$-category $\colim_{BG} \bV^c_\bC(X)$, where $G$ acts  {on both $X$ and $\bC$.} }
As in the case of the fixed point theory, we write $\bV^{c}_{\bC,G}$ for the evaluation of this functor at a fixed object $\bC$ in $\Fun(BG,\CL)$.}

\begin{kor}\label{wthgiowergergwergwerglabel}
The functor $\bV_{\bC,G}^{c}$ is
\begin{enumerate}
	\item coarsely invariant, 
	\item $u$-continuous,
	\item $l$-excisive,
	\item flasqueness preserving and
	\item hyperexcisive.
\end{enumerate}
\end{kor}
\begin{proof}
 Coarse invariance holds by \cref{regwergergergrgwgregwregwregwergwreg} \eqref{thwoierjoiwegjergg} and \cref{etgiowetgregrwegergwr-colim}, $u$-continuity  by \cref{regwergergergrgwgregwregwregwergwreg} \eqref{thwoierjoiwegjergg1}   and \cref{rewgoijerg9rwup29igwg-colim}, and excision by \cref{regwergergergrgwgregwregwregwergwreg} \eqref{thwoierjoiwegjergg2} and \cref{egiweogergergewrg-colim}. The functor $\bV_{\bC,G}^c$ is flasqueness preserving by \cref{regwergergergrgwgregwregwregwergwreg} \eqref{thwoierjoiwegjergg3} and \cref{egoerpgrgwegwegwreg}, and hyperexcisive by \cref{ergieorogergergergregerge}, \cref{regwergergergrgwgregwregwregwergwreg} \eqref{wgwijretgowergfwefwref}, and \cref{gwgwergrwegwerg}.
\end{proof}

\subsection{Coarse homology theories from homological functors}\label{sec:homol}

In the preceeding  \cref{gioegwgwrgrgreggwregwrg} and \cref{sec:orbit-theory}
we introduced the functors $\bV_{\bC}^{G,c}$ and $\bV^{c}_{\bC,G}$.
In the present section, we first recall the notion 
of a coarse homology theory in \cref{erguiheriwgregregwgrwegwg1}. We then show that postcomposing 
the functors above with a homological functor (see \cref{qrevoiqrjoirqfcwqecq}) produces equivariant coarse homology theories.

Consider a cocomplete stable $\infty$-category $\bM$, and let $E \colon G\BC\to \bM$ be a functor.
Recall \cref{rgiojgogregrgregre} of a flasque $G$-bornological coarse space.

\begin{ddd}\label{wergiowreggergwegreg}
We say that $E$ vanishes on flasques if $E$ sends flasque $G$-bornological coarse spaces to zero objects.
\end{ddd}

Since $\bM$ is stable, by   \cref{erguiheriwgregregwgrwegwg} this condition on $E$ is actually equivalent to the condition that $E$ is flasqueness preserving.
 
\begin{ddd}[{\cite[Def.~3.10]{equicoarse}}] \label{erguiheriwgregregwgrwegwg1}
 $E$ is called an equivariant coarse homology theory if it is
\begin{enumerate}
\item coarsely invariant (\cref{rgiherogiergregreg}),
\item excisive (\cref{rgoiruegoiregregregreg}),
\item $u$-continuous (\cref{rgiuerogergergergre}), and
\item vanishes on flasques (\cref{wergiowreggergwegreg}).
 \qedhere
\end{enumerate}
\end{ddd}

If $E$ is a coarse homology theory, then it may additionally be
\begin{enumerate}
\item continuous (\cref{geriogjergerg}),
\item strongly additive (\cref{wergiugewrgergergerwgwregw}) (where we must assume that $\bM$ has set-indexed products),
\item strong (\cref{rgiorejgiowergwergwregwergregw}  {below}) or
\item  {hyperexcisive (\cref{ergeiguhgiergwregwergwerg}).}
\end{enumerate}

In the following, we recall the notion of strongness.
In  \cite[{Sec.~4.1}]{equicoarse} we have constructed a universal equivariant coarse homology theory 
\[ \Yo^{s} \colon G\BC\to G\Sp\cX\ .\]
It has the universal property that  precomposition by $\Yo^{s}$ induces an equivalence between the $\infty$-category of $\bM$-valued coarse homology theories (considered as a subcategory of $\Fun(G\BC,\bM)$), and the $\infty$-category $\Fun^{\colim}(G\Sp\cX,\bM)$ of colimit preserving functors from $G\Sp\cX$ to $\bM$ for any cocomplete stable $\infty$-category $\bM$.

Let $X$ be {a $G$-bornological coarse space}.
\begin{ddd}[{\cite[Def.~4.17]{equicoarse}}] \label{hgiofgregt43teg}
We call $X$ weakly flasque if
it admits an endomorphism $f \colon X\to X$ such that
\begin{enumerate} \item  \label{rgoihgiogregregregergre}  $\Yo^{s}(f)\simeq \id_{\Yo^{s}(X)}$.
\item$f$ implements pre-flasqueness of  $X$ (see \cref{tiowgwtrgwergregwgrg}).
  \qedhere
\end{enumerate}
\end{ddd}
We say that $f$ implements weak flasqueness of $X$.
\begin{rem}
	For a flasque space (\cref{rgiojgogregrgregre}), we require $f$ to be close to the identity. Weak flasqueness replaces this by assumption \eqref{rgoihgiogregregregergre}.
	Since $\Yo^{s}$ is coarsely invariant,
	a flasque $G$-bornological coarse space is weakly flasque.
\end{rem}

 Let $E \colon G\BC\to \bM$ be an equivariant coarse homology theory.
 
 \begin{ddd}[{\cite[Def.~ 4.18]{equicoarse}}]\label{rgiorejgiowergwergwregwergregw}
 We call $E$ strong if it annihilates weakly flasque bornological coarse spaces.
 \end{ddd}

\begin{rem}
The condition of $E$ being strong is important if one wants to construct equivariant homology theories from equivariant coarse homology theories by precomposing with the cone functor {\cite[Sec.~9]{equicoarse}}. 
Strongness of $E$ implies homotopy invariance of the composition.
 We refer to \cite[Sec.~11.3]{equicoarse} for more details.
 \end{rem}
 
 The remainder of this section combines the results of preceding sections to show that both $\bV^{G,c}_\bC$ and $\bV^c_{\bC,G}$ induce equivariant coarse homology theories.}
 
 {We consider a functor
\[ \Homol \colon \Cle\to \bM\ .\]
Recall the notion of a homological functor from \cref{qrevoiqrjoirqfcwqecq}.}
\begin{lem}\label{wregiowergewrgregwrgwregwre}
  If  $\Homol$ is homological, then it annihilates flasques.
\end{lem}
\begin{proof}
Let $(\bC,S)$ be in $\Fl(\Cle)$ (see \cref{thgiojgoiggfergrewgerewergewrgreggw}).
Since $\Homol$ is homological and therefore additive, the relation $S\simeq \id_{\bC} +S$ implies $\Homol(S)\simeq \id_{\Homol(\bC)}+\Homol(S)$. Since $\bM$ is stable and therefore additive, this in turn  implies that $\Homol(\bC)\simeq 0$ (\cref{erguiheriwgregregwgrwegwg}).
\end{proof}

We fix some  functor
\[ \bV\colon G\BC\to \Cle \]
and consider the composition
\[ \Homol\bV := \Homol\circ \bV \colon G\BC\to \bM\ .\]}

\begin{lem}\label{rgoiwgwrgrgwrgg}Assume:
\begin{enumerate}
	\item $\bV$  is
\begin{enumerate} 
\item coarsely invariant (\cref{rgiherogiergregreg}),
\item $l$-excisive (\cref{rgoiruegoiregregregreg-modified}), 
\item $u$-continuous (\cref{rgiuerogergergergre}), and
 \item preserves flasques (\cref{rt98uerg9rt342trg32432f}).
\end{enumerate}
\item $\Homol$ is homological (\cref{qrevoiqrjoirqfcwqecq}).
\end{enumerate}
Then the functor $\Homol\bV \colon G\BC\to \bM$ is an equivariant coarse homology theory (\cref{erguiheriwgregregwgrwegwg1}).
\end{lem}
\begin{proof}
First note that the target category $\bM$ is stable and cocomplete  since it is the target of a homological functor.
We  show that the functor $\Homol\bV$ has the properties   listed in  \cref{erguiheriwgregregwgrwegwg1}.
 
 $\Homol\bV$ is coarsely invariant since $\bV$ is so.
 
 $\Homol\bV$ is excisive since $\bV$ (being $l$-excisive) sends  complementary pairs to excisive squares in $\Cle$, and $\Homol$ (being homological) sends these squares to pushout squares. At this point, we also employ that $\Homol$ (being homological) preserves filtered colimits in order to justify the equivalence
$\Homol\bV(\cY)\simeq \Homol(\bV(\cY))$ for every big family $\cY$ on {a $G$-bornological coarse space}
(see \eqref{ewroijiowgjoewgergergegwegw} for notation).

 $\Homol\bV$ is $u$-continuous since $\bV$ is $u$-continuous   by assumption, and $\Homol$ preserves filtered colimits.

 $\Homol\bV$ vanishes on flasques since $\bV$ preserves flasques   by assumption, and $\Homol$ annihilates flasques by  \cref{wregiowergewrgregwrgwregwre}. 
 \end{proof}

We now discuss the additional properties a coarse homology could have.
   
\begin{lem}\label{regoiwepgergwergwergwregw}
We retain the assumptions of  \cref{rgoiwgwrgrgwrgg}.
\begin{enumerate}
\item \label{rlfijqfoewewfefewfewfewfewfq} If $\bV$ is continuous, then so is $\Homol\bV$ (\cref{geriogjergerg}).
\item \label{wegrwergwregregrgw} If $\bV$ is hyperexcisive, then so is $\Homol\bV$ (\cref{ergeiguhgiergwregwergwerg}).
\item  \label{rlfijqfoewewfefewfewfewfewfq1} Assume:
\begin{enumerate}
\item $\bV$ is strongly additive (\cref{wergiugewrgergergerwgwregw}).
\item $\bM$ admits set-indexed products.
\item $\Homol$ preserves set-indexed products.
\end{enumerate}
Then $\Homol\bV$ is strongly additive.
\end{enumerate}
\end{lem}
\begin{proof} 
 \eqref{rlfijqfoewewfefewfewfewfewfq} and \eqref{wegrwergwregregrgw} 
  follow from the fact that $\Homol$ preserves filtered colimits.
  \eqref{rlfijqfoewewfefewfewfewfewfq1}  {is} obvious.
 \end{proof}

We finally discuss the condition of being strong, see \cref{rgiorejgiowergwergwregwergregw}. 
Recall \cref{goiijgowgergegergregewrg} of a functorially pre-flasqueness preserving functor. Since in the following the auxiliary $\infty$-category $\bP$ satisfies $\bP\simeq *$, we just say that $\bV$ is  pre-flasqueness preserving.

\begin{lem}\label{rgioerwgergregwrgwergrewg}
We retain the assumptions of   \cref{rgoiwgwrgrgwrgg}.
If $\bV$ is   pre-flasqueness preserving, then 
$\Homol\bV$ is strong.
\end{lem}
\begin{proof}
We assume that $X$ is  {a $G$-bornological coarse space}
and that $f \colon X\to X$ implements weak flasqueness (\cref{hgiofgregt43teg}).
  Since $\bV$ is pre-flasqueness preserving, we have an endofunctor
$S \colon \bV(X)\to \bV(X)$ such that
\[ \id_{\bV(X)}+\bV(f)\circ S\simeq S\ .\]
We apply $\Homol$ and use \eqref{rfoijqwfoiejfojoiewfqwefqfqewfqewf} in order to conclude that
\begin{equation}\label{greoijoi34fefe}
\id_{\Homol\bV(X)}+\Homol\bV(f)\circ \Homol(S)\simeq \Homol(S)\ .
\end{equation}
 Since by assumption on $(X,f)$ we have an equivalence $\Yo^{s}(f)\simeq \id_{\Yo^{s}(X)}$, and since $\Homol\bV$ is a coarse homology theory, we have $\Homol\bV(f)
 \simeq \id_{\Homol\bV(X)}$. Hence \eqref{greoijoi34fefe} yields an equivalence
 \[ \id_{\Homol\bV(X)}+  \Homol(S)\simeq \Homol(S)\ ,\]
 which implies that  $\Homol\bV(X)\simeq 0$ since $\bM$ is stable and hence additive (see \cref{erguiheriwgregregwgrwegwg}).
 \end{proof}

 Let $\Homol \colon \Cle\to \bM$ be a functor, and  let $\bC$ be in $\Fun(BG,\CL)$.
Recall the functor $\bV^{G,c}_{\bC}$ from \cref{iowergergwegr}.
\begin{kor}\label{weigowgregwrgrggregwgregwreg}
If $\Homol$ is a homological functor, then $\Homol\bV^{G,c}_{\bC}$ is an equivariant coarse homology theory which in addition is
\begin{enumerate}
\item \label{ergioierjogwergreg432r} strong,
\item \label{ergioierjogwergreg432r1} continuous and
\item  \label{ergioierjogwergreg432r3} strongly additive (provided $\bM$ admits set-indexed products and $\Homol$ preserves set-indexed products).  \end{enumerate}
\end{kor}
\begin{proof}
The first assertion follows from \cref{rgoiwgwrgrgwrgg} together with \cref{regwergergergrgwgregwregwregwergwreg}.
For \eqref{ergioierjogwergreg432r} we use \cref{regwergergergrgwgregwregwregwergwreg} and \cref{rgioerwgergregwrgwergrewg}.
For \eqref{ergioierjogwergreg432r1} and \eqref{ergioierjogwergreg432r3} we use \cref{regwergergergrgwgregwregwregwergwreg} and \cref{regoiwepgergwergwergwregw}.
 \end{proof}

Recall the functor $\bV^{c}_{G,\bC}$ from \cref{qoirejoqierjgqergrefqerwfewf}.
 \begin{kor}\label{qerkgioerggwergergregwe}
 If $\Homol$ is a homological functor, then $\Homol\bV^{c}_{G,\bC}$ is an equivariant coarse homology theory which  in addition is
\begin{enumerate}
\item hyperexcisive and
\item strong.
\end{enumerate}
\end{kor}
\begin{proof}
	This follows from \cref{wthgiowergergwergwerglabel} together with \cref{rgoiwgwrgrgwrgg,regoiwepgergwergwergwregw,rgioerwgergregwrgwergrewg}.
\end{proof}

\subsection{Calculations}\label{sec:calculations}
In this section, we calculate the values of the functors $\bV^{c}_{\bC,G}$ and  $\bV^{G,c}_\bC$  (see \cref{iowergergwegr,qoirejoqierjgqergrefqerwfewf}) on certain $G$-bornological coarse spaces.
 {By restriction  from $G$-sets to $G$-orbits, the functor $G\Set \to G\BC$ from \eqref{qreoijqoiegjoqirffewfq} gives rise to the functor
  \begin{equation}\label{vervelkrvnlervrvrevev}
i \colon G\Orb\to G\BC\ , \quad S\mapsto S_{min,max}\ .
\end{equation}
\cref{efioqjwefoweweqfewfqwefqwefe} gives an intrinsic description of the functor $\bV^c_{\bC,G} \circ i$ without reference to $G$-bornological coarse spaces. 
Moreover, we show  in \cref{cor:coeffs-orbits} that $\bV^{G,c}_\bC$ gives rise to the same functor on the orbit category if we apply it to the  tensor product of the}  $G$-orbit with the $G$-bornological coarse space $G_{can,min}$ from \cref{etwgokergpoergegregegwergrg}.
This result is an essential ingredient in the proof of \cref{wefiofewqfwefqefqewf} below.

Since $\Cle$ is cocomplete by \cref{prop:catex finitely complete}, we have an adjunction
\begin{equation}\label{2rfkhriufh3iuf3f3f}
\Ind^{G} \colon \Fun(BG,\Cle)\leftrightarrows \Fun(G\Orb,\Cle) \cocolon \Res^{G}\ ,
\end{equation}
 where $\Res^{G}:=j^{*}$ is the restriction functor along $ {j \colon BG \to G\Orb}$ from \eqref{rewflkjmo34gergwegrge}, and $\Ind^{G}:=j_{!}$ is the left Kan extension functor along $j$.

 \begin{rem}\label{gikowrteglregergregwerg}
For a subgroup $H$ of $G$ we consider the  {transitive $G$-set $G/H$}. 
We then have an equivalence of categories
 $BH \xrightarrow{\simeq} BG_{/(G/H)} $ which sends the unique object $*_{{BH}}$  of $BH$ to the projection $G\to G/H$ and the element $h$ in $H=\Aut_{BH}(*_{BH})$ to the automorphism of $(G\to G/H) $ in $BG_{/(G/H)}$
 given by right-multiplication on $G$ with $h^{-1}$. If $\bC$ is in $\Fun(BG,\Cle)$, then  the pointwise formula for the left Kan extension provides the first equivalence in   \begin{equation}\label{rlijlokvjoiqwfcwefcewceq}
\Ind^{G}(\bC)(G/H)\simeq \colim_{ BG_{/(G/H)}} \bC\simeq \colim_{BH} \Res^{G}_{H}\bC\ ,
\end{equation}
 where the second equivalence uses the equivalence of index categories discussed above. 
\end{rem}

Let $(-)^{\omega} \colon \CL\to \Cle$ be the functor taking the subcategory of cocompact objects (\cref{rgiuehgieugegewgergwergwerg}). By postcomposition, it induces a functor 
\[ \Fun(BG,\CL)\to \Fun(BG,\Cle)\ , \quad \bC\mapsto \bC^{\omega}\ .\]
Let $\bC$ be in $\Fun(BG,\CL)$ and recall the definition of the functor $\bV^{c}_{\bC,G}$ from \cref{qoirejoqierjgqergrefqerwfewf}.

\begin{prop}\label{efioqjwefoweweqfewfqwefqwefe}
	There is an equivalence
	\begin{equation}\label{fvsdfvvvsav}  \Ind^G(\bC^\omega) \simeq \bV^{c}_{\bC,G} \circ i
	\end{equation}
	of functors $G\Orb \to \Cle$.
\end{prop}
\begin{proof}
We consider the functors
\[ \Sh_{\bC,\mathrm{eq}}, \Sh^{\eqsm}_{\bC,\mathrm{eq}}, \bV^{c}_{\bC,\mathrm{eq}} \colon G\Orb\to \Fun(BG,\CLL) \]
(note that the last two actually take values in $\Cle$)
defined  as the compositions (compare with the first three morphisms in  \eqref{rthiowrthrthwhwgwergwrgw})
 \begin{align*}
  G\Orb \cong G\Orb \times * &\xrightarrow{i \times \bC} G\BC \times \Fun(BG,\CL) \\
  &\to \Fun(BG \times BG, \BC \times \CL) \\
  &\xrightarrow{\diag_{BG}^*} \Fun(BG, \BC \times \CL) \\
  &\xrightarrow{\Sh,\Sh^\eqsm,\bV^{c}} \Fun(BG, \CLL)
  \end{align*}
  {sending a transitive $G$-set $S$ to $\Sh_\bC(S_{min,max})$, $\Sh^\eqsm_\bC(S_{min,max})$ and $\bV^c_\bC(S_{min,max})$, respectively, each equipped with the conjugation action of $G$}.
   We first show that there is an equivalence \begin{equation}\label{ewrgiurhgiu34wetg}
\bV^{c}_{\bC,G}\circ i\simeq \colim_{BG} \circ \Sh^{\eqsm}_{\bC,\mathrm{eq}}\ .
\end{equation}
 Since the functor $i \colon G\Orb \to G\BC$ equips the  transitive $G$-sets with the minimal coarse structure, the localisation morphism $\ell \colon \Sh^{\eqsm}_{\bC}\to \bV_{\bC}$  induces an equivalence
 \begin{equation}\label{sfboijoiewtbrgbfgbsb}
 \Sh^\eqsm_\bC {\circ \res^{G}_{\{1\}}}\circ i \simeq \bV_\bC {\circ  \res^{G}_{\{1\}}}\circ  i
\end{equation}
(see \cref{uihwieurghiwuerhgiurwegwregwregg} and recall that we omit to write $\Res^{G}_{\{1\}}$ in front of $\bC$, but we do not omit $\res^{G}_{\{1\}}$ at the space variables).
 
 Let $S$ be  {a transitive $G$-set}.
 Using \eqref{sfboijoiewtbrgbfgbsb}, the left vertical equivalence in the square \eqref{eqrlwkenlweergegw}, and the equivalence \eqref{wergwergwegwregwrewg}, we identify the transformation
 \[ \bV^{c}_\bC(\res^{G}_{\{1\}}S_{min,max}) \to \bV_{\bC}(\res^{G}_{\{1\}}S_{min,max}) \]
 with  the canonical map
 \[ \colim_{F \subseteq S \text{ finite}} \coprod_F \bC^\omega \to \coprod_S \bC^\omega\ .\]
 Since the latter is obviously an equivalence,  we obtain the first equivalence in 
 \begin{equation}\label{qrviuhqriuvhuqievqervqv}
\bV^{c}_\bC \circ \res^{G}_{\{1\}} \circ i \simeq \bV_\bC  \circ   \res^{G}_{\{1\}}\circ i\stackrel{\eqref{sfboijoiewtbrgbfgbsb}}{\simeq}\Sh^\eqsm_\bC {\circ \res^{G}_{\{1\}}}\circ i \ .
\end{equation}
We now take the $G$-actions on $S$ and $\bC$ into account, which by functoriality induce $G$-actions on the two sides of this equivalence. We therefore have an equivalence \begin{equation}\label{evwervervwvwrwervwv}
\bV^{c}_{\bC,\mathrm{eq}}\simeq \Sh_{\bC,\mathrm{eq}}^{\eqsm}\ .
\end{equation}
Applying $\colim_{BG}$ and  \cref{qoirejoqierjgqergrefqerwfewf} of $\bV^{c}_{\bC,G}$ for the first equivalence, we get
 the chain of equivalences 
 \begin{equation}\label{trhrthehertherth}
\bV^{c}_{\bC,G} \circ i \simeq \colim_{BG}\circ  \bV^{c}_{\bC,\mathrm{eq}}  \stackrel{\eqref{evwervervwvwrwervwv}}{\simeq} \colim_{BG} \circ \Sh^{\eqsm}_{\bC,\mathrm{eq}}
\end{equation}  
showing \eqref{ewrgiurhgiu34wetg}.

The functor corresponding to $\Sh^{\eqsm}_{\bC,\mathrm{eq}}$ by the exponential law is given by the composition
 \[ \Sh^{\eqsm}_{\bC,\mathrm{eq'}} \colon G\Orb \times BG \xrightarrow{(\alpha,\bC \circ \pr)} \BC \times \Cle {\xrightarrow{\Sh^\eqsm} \Cle} \ ,\]
 where $\pr \colon G\Orb \times BG \to BG$ is the projection and $\alpha \colon G\Orb \times BG \to \BC$ is the functor sending objects $(S,*_{{BG}})$ to $\res^{G}_{\{1\}}S_{min,max}$ and morphisms $(S\xrightarrow{\phi} S',*_{BG} \xrightarrow{g} *_{BG})$ to
 \[ \res^{G}_{\{1\}}S_{min,max}\xrightarrow{\phi(i)} \res^{G}_{\{1\}}S'_{min,max}\xrightarrow{g} \res^{G}_{\{1\}}S'_{min,max} \ .\]
 Consider the functor $p := (j,\id) \colon BG \to G\Orb \times BG$ and the projection $\pi \colon G\Orb\times BG\to G\Orb$.
We then have the  following diagram (the left triangle commutes, and the two fillers $\tau$ and $\sigma$, which are  not necessarily equivalences, will be explained below):
\begin{equation*} 
 \xymatrix{BG\ar@/^1cm/[rrd]^{\bC^{\omega}}\ar[dd]^{j}\ar[dr]^{p}&\ar@{=>}[d(0.6)]^{\tau}&\\&G\Orb\times BG\ar[dl]^{\pi}\ar[r]^-{\Sh^{\eqsm}_{\bC,\mathrm{eq'}}} \ar@{=>}[d(1.1)]^{\sigma}&\Cle\\G\Orb \ar@/_1cm/[urr]_{\colim_{BG} \circ \Sh^{\eqsm}_{\bC,\mathrm{eq}}}&&}\ .
\end{equation*}
  In the following we use the subscript $!$ in order to denote left Kan extension functors. There is an obvious  canonical natural transformation $\sigma$ exhibiting $\colim_{BG} \circ \Sh^{\eqsm}_{\bC,\mathrm{eq}}$ as a left Kan extension of
 $\Sh^{\eqsm}_{\bC,\mathrm{eq'}}$ along $\pi$, i.e.,  we have an equivalence 
  \begin{equation}\label{wtrhoijqio45joirthrthwrtht}
\colim_{BG} \circ \Sh^{\eqsm}_{\bC,\mathrm{eq}}\simeq \pi_{!}   \Sh_{\bC,\mathrm{eq'}}^{\eqsm} \ .
\end{equation} 
 The functor $\alpha \circ p \colon BG \to \BC$ sends $*_{{BG}}$ to $\res^{G}_{\{1\}}G_{min,max}$ and a  morphism   $g$ in $G=\Aut_{BG}(*_{BG})$ to the conjugation by $g$ on $\res^{G}_{\{1\}}G_{min,max}$. Since conjugation fixes the identity element,
 the inclusion of the identity element $i_e \colon *  \cong \{e\} \to \res^{G}_{\{1\}} G_{min,max}$ induces a natural transformation
 \[ \tau \colon \bC^{\omega} \simeq \Sh^\eqsm_\bC(*) \xrightarrow{\hat i_{e,*}} \Sh^\eqsm_{\bC,\mathrm{eq'}}  \circ     p \]
 of functors $BG \to \Cle$.
 We claim that  it exhibits $ \Sh^\eqsm_{\bC,\mathrm{eq'}} $ as the left Kan extension of
$\bC^{\omega}$ along $p$, i.e.,     
  that $\tau$ induces an equivalence \begin{equation}\label{vweuvheujkvervewrvewvwervwevervewv}
   p_{!}\bC^{\omega} \simeq\Sh_{\bC,\mathrm{eq'}} ^{\eqsm} \ .
\end{equation}
 Since $\pi\circ p=j$ the claim implies the desired  equivalence  \eqref{fvsdfvvvsav} by
 \[ \Ind^{G}(\bC^{\omega})\simeq j_{!} \bC^{\omega}\simeq \pi_{!}p_{!}\bC^{\omega}\stackrel{\eqref{vweuvheujkvervewrvewvwervwevervewv}}{\simeq} \pi_{!}\Sh_{\bC,\mathrm{eq'}} ^{\eqsm}   \stackrel{\eqref{wtrhoijqio45joirthrthwrtht}}{\simeq}  \colim_{BG}\circ \Sh^{\eqsm}_{\bC,\mathrm{eq}} \stackrel{\eqref{ewrgiurhgiu34wetg}}{\simeq} \bV^{c}_{\bC,G}\circ i\ .\]

 It remains to show the claim. We need to check   {for each transitive $G$-set} $S$
 that the induced morphism
 \[ \tau(S) \colon \colim_{p_{/(S,*_{{BG}})}} \bC^\omega \to  \Sh^{\eqsm}_{\bC,\mathrm{eq'}}(S,*_{BG}) \simeq \Sh^\eqsm_\bC(\res^{G}_{\{1\}}S_{min,max}) \]
 is an equivalence,  where   $p_{/(S,*)}$ is a shorthand for $BG \times_{G\Orb \times BG} (G\Orb \times BG)_{/(S,*)}$, with $p$ being implicitly used in the pullback construction. 
 The set of objects of the category
  $p_{/(S,*)}$ can be identified with the set $S\times G$ such that $(s,g)$ corresponds to  the pair $(*_{BG},(G \xrightarrow{e\mapsto s} S, *_{BG} \xrightarrow{g} *_{BG}))$. A morphism $(s,g) \to (s',g')$ exists (and is then unique) if and only if $g'g^{-1}s' = s$.
  Consequently, the functor $S \to p_{/(S,*)}$ sending $s$ in $S$ to the object $({s,e})$ is an inverse to the projection functor $p/(S,*) \to S$.  This explicit inverse induces the left vertical equivalence in  the commutative  diagram
 \[\xymatrix{
  \colim_{p_{/(S,*)}} \bC^\omega\ar[r]^-{\tau(S)} & \Sh^\eqsm_\bC(\res^{G}_{\{1\}}S_{min,max}) \\
  \coprod_S \bC^\omega\ar[u]^-{\simeq}\ar[ur]_-{\hphantom{xxx}\sum_{s \in S} \hat i_{s,*}} &
 }\]
 where $i_s \colon \{s\} \to \res^{G}_{\{1\}}S_{min,max}$ denotes the inclusion maps.
 The diagonal arrow is an equivalence by \cref{lem:eqsm-minmax}, so $\tau(S)$ is an equivalence for all  {transitive $G$-sets} $S$.
 This finishes the proof of the claim.
\end{proof}

 Let $\bC$ be in $\Fun(BG,\CL)$. Recall the $G$-bornological coarse space $G_{can,min}$ from \cref{etwgokergpoergegregegwergrg}. Recall also the idempotent completion functor $\Idem$ from \eqref{ewfbjhbfjhqerfqfewfqefe}.
   
\begin{prop}\label{prop:coeffs-sets}
	There are equivalences
	\begin{eqnarray}
	\Idem \bV^{G,c}_\bC(G_{can,min} \otimes {(-)}_{min,max})  
	&\simeq& \Idem \bV^{G}_\bC(G_{can,min} \otimes {(-)}_{min,max})  \nonumber\\ &\simeq& \Sh^G_\bC({(-)}_{min})^\omega  \label{wetigowegergwgreg}
	\end{eqnarray}
	of functors $G\Set \to \Cle$.
\end{prop} 
\begin{proof}
	 Let $X$ be {a $G$-set}.
	 Denote by  $\pi \colon G_{can}\otimes X_{min}\to X_{min}$   the projection  and by  $\iota \colon G_{min}\otimes X_{min}\to G_{can}\otimes X_{min}$  the morphism  {of $G$-coarse spaces} 
	 induced by the identity of the underlying $G$-sets. Then we get the  {solid} part of the following commutative diagram: 
 \begin{equation}\label{advadvaqefasda}
\xymatrix{&\Sh_{\bC}(\res^{G}_{\{1\}} X_{min})\ar@/^0.4cm/[dr]^-{\coind_{\{1\}}^{G}}& \\
\Sh^{G}_{\bC}(G_{min} \otimes X_{min})\ar@/^0.4cm/[ur]_-{\theta,\eqref{reoijiojoiervevwev}}^{\simeq}\ar[d]_{\ell_{G_{min} \otimes X_{min }}}^{\simeq}\ar[r]^-{\hat \iota^{G}_{*}}& \Sh^{G }_{\bC}(G_{can } \otimes X_{min })\ar[d]^{\ell_{G_{can} \otimes X_{min }}} \ar[r]^-{\hat \pi^{G}_{*}}&\Sh^{G}_{\bC}( X_{min})
\\\hat \bV^{G}_{\bC}(G_{min} \otimes X_{min })\ar[r]^{\iota_*} &\hat \bV^{G}_{\bC}(G_{can} \otimes X_{min})\ar@{..>}@/_.3cm/[ur]_-{\hat p}& }\ .
\end{equation} 
The left square is a case of  \eqref{ergoijergerg}, and the  commutativity of the upper  {square}  can be reduced to the upper right  triangle in \eqref{wregpokopwtglmrbletemblwtwb}. The  arrow $\ell_{G_{min} \otimes X_{min }}$ is an equivalence by \cref{uihwieurghiwuerhgiurwegwregwregg}.

In order to get the dotted arrow $\hat p$ and the corresponding triangle, we use the universal property of $\ell_{G_{can} \otimes X_{min }} $ and the fact that $\hat \pi^{G}_{*}$ sends the morphisms $M \to V^G_*M$  {to identity morphisms}
 for every equivariantly small sheaf $M$ on $G_{can} \otimes X_{min}$
and invariant coarse entourage $V$ {of $G_{can} \otimes X_{min}$ containing the diagonal.} 

The whole diagram \eqref{advadvaqefasda} is natural  {in the $G$-set} $X$. 
	 
	By \cref{egiojwotgwegergwgrwregwrg}, we know that $\Sh^G_\bC(X_{min})$ is an object of $\CL$  and therefore has a well-defined notion of cocompact objects.
	Further note that if $G$ is not finite, then $\pi$ does not induce a morphism of $G$-bornological coarse spaces 
	from $G_{can,min}\otimes X_{min,max}$ to $X_{min,max}$ since this projection is not proper.
	We nevertheless can show that $\hat \pi_{*}^{G}$ restricts to a functor
	\begin{equation}\label{dfsboirjgioewggrqgrg}
	\hat \pi_{*}^{G,\eqsm} \colon \Sh^{G,\eqsm}_{\bC}(G_{{can,min}} \otimes X_{{min,max}})\to  \Sh^G_\bC(X_{min})^\omega\ .
	 \end{equation}
        Let $M$ be  {an equivariantly small sheaf on $G_{{can,min}} \otimes X_{{min,max}}$}.
        By \cref{prop:Gcan-cpt}, {$M$ is a cocompact object in $\PSh^G_\bC(G \times X)$}. 

	Since the functor $ \PSh^G_\bC(G  \times X ) \to \PSh^G_\bC(X )$ induced by $\pi$ is a morphism in $\CL$, it preserves cocompact objects, and we have $\hat\pi^G_*(M)\in \PSh^G_\bC(X )^{\omega}$. 
	Since the inclusion $\Sh^G_\bC(X_{min})\to \PSh^G_\bC(X )$ clearly detects cocompactness,
	 we conclude that  $\hat\pi^G_*(M) \in \Sh^G_\bC(X_{min})^{\omega}$.	
	 
	 We consider the morphism $\iota \colon G_{min,min} \otimes X_{min,max} \to G_{can,min} \otimes X_{min,max}$ in $G\BC$
	 and note that $\hat\iota^G_*$ preserves equivariantly small objects by \cref{eriogwetgwegregwrgregwrg}.
	 By restricting the diagram in \eqref{advadvaqefasda} to equivariantly small objects, using the equivalence in \eqref{ergggrgewrg243g3rverv2rf3f} for the upper corner, and the second assertion of \cref{cor:coindSh},  we get the following commutative  diagram
	 \begin{equation}\label{advadvaqefasda1} \hspace{-1cm}
\xymatrix{&\Sh_{\bC}(\res^{G}_{\{1\}} X_{min})^{\omega}\ar@/^0.4cm/[dr]^-{\coind_{\{1\}}^{G,\omega}}&\\\Sh^{G,\eqsm}_{\bC}(G_{min,min} \otimes X_{min,max })\ar@/^0.4cm/[ur]_{~~\theta^{\eqsm},\eqref{rewgpogoegergergwergewgergewrg}}^{\simeq}\ar[d]_{\simeq}\ar[r]^{\hat \iota^{G}_{*}}& \Sh^{G,\eqsm }_{\bC}(G_{can,min } \otimes X_{min ,max})\ar[d]  \ar[r]^-{\hat \pi^{G,\eqsm}_{*}}&\Sh^{G}_{\bC}( X_{min})^{\omega}
\\  \bV^{G}_{\bC}(G_{min,min } \otimes X_{min,max }) \ar[r]^{\iota_*} & \bV^{G}_{\bC}(G_{can,min} \otimes X_{min,max})\ar[ur]_-{p}&}
\end{equation}
 which is natural in {the $G$-set} $X$. 
In the next step we construct a factorisation of $\iota_{*}$ over the canonical morphism $i$ as indicated by the dotted arrow in the following diagram:
\begin{equation}\label{advadvaqefasda1ergwergwrgrwegewgwergewrg}\xymatrix{ \Sh^{G,\eqsm}_{\bC}(G_{min,min} \otimes X_{min,max })\ar[r]^{\simeq}&\bV^{G}_{\bC}(G_{min,min } \otimes X_{min,max })\ar[d]^{\iota_*} \ar@{..>}@/^3cm/[dd] \\&\bV^{G}_{\bC}(G_{can,min} \otimes X_{min,max})\\\Sh^{G,\eqsm,c}_\bC(G_{min,min} \otimes X_{min,max})\ar@{-->}[r]\ar@{-->}[uu]& \bV^{G,c}_{\bC}(G_{can,min} \otimes X_{min,max})\ar[u]_{i}}\end{equation}
We first observe that the dashed arrows exist and the diagram commutes for obvious reasons. It now suffices to show that the dashed vertical arrow is an equivalence. In order to see this note that
	every invariant locally finite subset of $G_{min,min} \otimes X_{min,max}$ is isomorphic in $G\BC$ to the image of $G_{min,min}\otimes F_{min,max}$ 
	  for some finite subset $F$ of $X$ (equipped with the trivial $G$-action) under the map $(g,f)\mapsto (g,gf)$. By the first assertion of \cref{lem:eqsm-Gminmax}
	we obtain the vertical equivalences in the  commutative diagram
	\[\xymatrix{
		\colim_{F \subseteq X \text{ finite}} \Sh^{G,\eqsm}_\bC(G_{min,min} \otimes F_{min,max})\ar[r]\ar[d]^-{\simeq} & \Sh^{G,\eqsm}_\bC(G_{min,min} \otimes X_{min,max})\ar[d]_-{\simeq} \\
		\colim_{F \subseteq X \text{ finite}} \Sh^{\eqsm}_\bC(\res^{G}_{\{1\}}F_{min,max}) \ar[r] & \Sh^{\eqsm}_\bC(\res^{G}_{\{1\}}X_{min,max}) 
	}\]
	By \cref{lem:eqsm-minmax}, the lower horizontal arrow is evidently an equivalence.  Hence the upper horizontal morphism is the desired equivalence
	\[ \Sh^{G,\eqsm,c}_\bC(G_{min,min} \otimes X_{min,max})\xrightarrow{\simeq}  \Sh^{G,\eqsm}_\bC(G_{min,min} \otimes X_{min,max})\ .\]

 The morphism $i$ in \eqref{advadvaqefasda1ergwergwrgrwegewgwergewrg} is fully faithful by  \cref{qerguigwregergwrgeg}. We  show that $p$ in \eqref{advadvaqefasda1} is fully faithful, too.
 Indeed, for  {equivariantly small sheaves $M,N$ on $G_{{can,min}} \otimes X_{{min,max}}$}
 we have the following chain of equivalences:
	\begin{eqnarray*}
		&&\Map_{\bV^G_\bC(G_{can,min} \otimes X_{min,max})}({\ell} M,{\ell} N) \\
		&\stackrel{\eqref{qwefqwfoij11111}}{\simeq}& \colim_{V \in \cC^{G,\Delta}_{G_{can} \otimes X_{min}}} \Map_{\PSh^G_\bC(G \times X)}(M,V^G_*N) \\
		 &\stackrel{\eqref{v-adj-ggg}}{\simeq}&\colim_{V \in \cC^{G,\Delta}_{G_{can} \otimes X_{min}}} \Map_{\PSh^G_\bC(G \times X)}(V^{*,G}M,N) \\
		 &\stackrel{!}{\simeq}& \Map_{\PSh^G_\bC(G \times X)}(\lim_{V \in \cC^{G,\Delta}_{G_{can} \otimes X_{min}}} V^{*,G}M,N) \\
		 &\stackrel{!!}{\simeq}& \Map_{\PSh^G_\bC(G \times X)}(\hat\pi^{*,G}\hat\pi^G_*M,N) \\
		 &\stackrel{\eqref{ad-hat-f}}{\simeq}&  \Map_{\PSh^G_\bC( X)}( \hat\pi^G_*M,\hat \pi^{G}_{*}N)  \\
		 &\stackrel{!!!}{\simeq}& \Map_{\Sh^G_\bC(X_{min})^\omega}(p(M),p(N))
	\end{eqnarray*}
	The equivalence marked by $!$
	  holds because $N$ is cocompact in $\PSh^G_\bC(G \times X)$ by \cref{prop:Gcan-cpt}
	  {(note that $\PSh^G_\bC(G \times X)\simeq \PSh^G_\bC(G_{can,min} \otimes X_{min,max})$ by convention).}
	  For the equivalence marked by $!!$ we observe that for every {subset $B$ of $G \times X$}
	 the family $(V[B])_{V \in \cC^{G,\Delta}_{G_{can} \otimes X_{min}}}$ is a $U$-covering family of $G \times \pi(B) = \pi^{-1}(\pi(B))$ for every entourage $U$ of $G_{can}\otimes X_{min}$.
	  Using that $M$ is a sheaf, it follows that the natural transformation  $V[-] \to \pi^{-1} \circ \pi(- )$ of endofunctors of  {the power set $\cP_{G\times X}$}  induces an equivalence
	\[ \hat\pi^{*,G}\hat\pi^G_*M \xrightarrow{\simeq} \lim_{V \in \cC^{G,\Delta}_{G_{can} \otimes X_{min}}} V^{*,G}M\ .\]
	Finally, the equivalence marked by $!!!$ follows from the definition of $p$.
	
We now apply the functor $\Idem$ to
 \eqref{advadvaqefasda1}, contract its middle line, and use  \eqref{3roihfeoirhjo3f123feffqewfq} in order to see that the left column consists of idempotent complete categories. Furthermore, we apply $\Idem$ to the right triangle of \eqref{advadvaqefasda1ergwergwrgrwegewgwergewrg}. The combination of the resulting two commutative diagrams yields the commutative diagram
 \begin{equation}\label{advadvaqefasda2}\hspace{-1cm}
\xymatrix{&\Sh_{\bC}(\res^{G}_{\{1\}} X_{min})^{\omega}\ar@/^0.4cm/[dr]^-{\coind_{\{1\}}^{G,\omega}}&\\\Sh^{G,\eqsm}_{\bC}(G_{min,min} \otimes X_{min,max })\ar@/^0.4cm/[ur]_{~~\theta^{\eqsm},\eqref{rewgpogoegergergwergewgergewrg}}^{\simeq}\ar[d]_{\simeq}\ar[rr]^{\hat \pi^{G,\eqsm}_{*}\circ \hat \iota^{G}_{*}}&  &\Sh^{G}_{\bC}( X_{min})^{\omega}
\\  \bV^{G}_{\bC}(G_{min,min } \otimes X_{min,max })\ar[dr]\ar[r]^{\iota_*} &\Idem  \bV^{G}_{\bC}(G_{can,min} \otimes X_{min,max})\ar[ur]_-{ \Idem( p)}&\\&  \Idem\bV^{G,c}_{\bC}(G_{can,min} \otimes X_{min,max})\ar[u]_-{\Idem(i)}& }
\end{equation} 
which is natural in  {the $G$-set} $X$.

By  \cref{iqjoigrgwegerwgwg} the functors $\Idem(p)$ and $\Idem(i)$ are again fully faithful inclusions of idempotent complete subcategories.
By  the second assertion of \cref{cor:coindSh}, the functor $\coind_{\{1\}}^{G,\omega}$ generates $\Sh^{G}_{\bC}( X_{min})^{\omega}$ under finite limits and retracts. 
 By the commutativity of \eqref{advadvaqefasda2} we can then conclude that  the functors  $\Idem(p)$ and $\Idem(p)\circ \Idem(i)$ are essentially surjective. Hence we have the  asserted natural equivalences
	\begin{eqnarray*}
	\Idem \bV^{G,c}_\bC(G_{can,min} \otimes X_{min,max})  
	&\stackrel{\Idem(i)}{\simeq}& \Idem \bV^{G}_\bC(G_{can,min} \otimes X_{min,max}) \\ &\stackrel{\Idem(p)}{\simeq}& \Sh^G_\bC(X_{min})^\omega   \ .
	\end{eqnarray*}
 \end{proof}

\begin{kor}\label{cor:coeffs-orbits}
	There are equivalences
	\begin{eqnarray}
	 \Idem  \bV^G_{\bC}(G_{can,min} \otimes (-)_{min,max})&\simeq&  
	 \Idem \bV^{G,c}_{\bC}(G_{can,min} \otimes (-)_{min,max}) \nonumber\\
	 &\simeq &\bV^{c}_{\bC,G} \circ i(-)\label{euigwrthwergergwegreg}
	\end{eqnarray}
	of functors $G\Orb \to \Cle$.
\end{kor}
\begin{proof}
 Let $\iota \colon BG^\op \to BG$ be the inversion functor.  We use the notation introduced in the proof of \cref{efioqjwefoweweqfewfqwefqwefe}.
 Since $\iota$ is an equivalence, we have the first equivalence in the chain 
	\begin{equation}\label{qewflkmlqkwefqwefqewfefqwefqwefqwefqwef}
\Sh^G_{\bC}((-)_{min})^\omega \simeq (\lim_{BG^\op} \iota^*\Sh_{\bC,\mathrm{eq}}(- ))^\omega\stackrel{\eqref{wqefqoijfqoiwffqefefewfqefqe}}{\simeq} \colim_{BG} \Sh_{\bC,\mathrm{eq}}(-)^\omega\ .  \end{equation} 
	 We furthermore have equivalences \begin{equation}\label{qewfiojioqefqwefqwefqef}
\Sh_{\bC,\mathrm{eq}}(-)^\omega \stackrel{\eqref{ergggrgewrg243g3rverv2rf3f}}{\simeq} \Sh^\eqsm_{\bC,\mathrm{eq}} (-)\stackrel{\eqref{evwervervwvwrwervwv}}{\simeq} \bV^{c}_{\bC,\mathrm{eq}}(-)\ .
\end{equation}	 
	 Applying $\colim_{BG}$, we get  the second equivalence in  \begin{equation}\label{rijoijioqjrfoijergoergwergegerg} \Sh^G_{\bC}((-)_{min})^\omega  \stackrel{\eqref{qewflkmlqkwefqwefqewfefqwefqwefqwefqwef}}{\simeq}
\colim_{BG}  \Sh_{\bC,\mathrm{eq}}(-)^{\omega} \simeq  \bV^{c}_{\bC,G} \circ i(-)\ .
\end{equation}
Combining the equivalence  \eqref{rijoijioqjrfoijergoergwergegerg}  with \eqref{wetigowegergwgreg},
we get \eqref{euigwrthwergergwegreg}.
\end{proof}

\section{Transfers}\label{tiowtwrtbtwtbewbtw}
In this section, we extend the construction of $\bV_{\bC}^{G}$ in order to capture not only its covariant functoriality for morphisms in $G\BC$ but also the contravariant functoriality for coverings, and the compatibility among these operations. Technically this is accomplished by extending the functor to the  $\infty$-category $G\BC_{\tr}$ of bornological coarse spaces with transfers along the inclusion
  \begin{equation*}
\iota\colon G\BC\to G\BC_{\tr}\ .
\end{equation*} 
The $\infty$-category $G\BC_{\tr}$ was introduced in \cite[Def. 2.27]{coarsetrans}
 {as the $\infty$-category of spans
\[\xymatrix@R=1em@C=1em{
 & X'\ar[dl]_{g}\ar[dr]^{f} & \\
 X & & Y}\]
 in which $X$, $X'$ and $Y$ are $G$-bornological coarse spaces, $g$ is a covering (see \cref{ergwergggqrgqregqergq}), and $f$ is a morphism of bornological coarse spaces.}
We will use the equivalent description of this category given in \cref{efbivhjiqrewvwewcewcq} below (see also \cref{rqegrgrfqefewfewfq}).

\subsection{Overview and first applications}

Let  $E\colon G\BC\to \bM$  be a functor.
\begin{ddd}\label{egioeggewrgegervsfdvfdv}
 $E$ admits transfers if there exists a functor
 \[ E_{\tr}\colon G\BC_{\tr}\to \bM \]
 and an equivalence $E_{\tr}\circ \iota\simeq E$.
\end{ddd}

\begin{rem}
Assuming that $\bM$ is cocomplete, one could consider the left Kan extension $\iota_{!}E\colon G\BC_{\tr}\to \bM$ of $E$ along $\iota$. But in general the morphism $E\to \iota^{*}\iota_{!}E$ is not an equivalence since $\iota$ is not fully faithful. So the problem of showing that $E$ admits transfers does not have such a trivial solution.
\end{rem}

 Recall the functor $\bV^{G}$ from \eqref{bojoijgoi3jg3g34f}.
 The following is the main result of this section. It will be shown in \cref{sec:transfers-proofs}.
 
  \begin{theorem}\label{ugiqergqrefwefqfqewfqef}
 There exists a  functor  $\bV^{G}_{\tr}$   such that the following diagram commutes: 
 \[\xymatrix{G\BC\times  \Fun(BG,\CL) \ar[d]_-{\iota\times \id}\ar[rr]^-{\bV^{G}}&&  \Cle\\
 G\BC_{\tr}\times  \Fun(BG,\CL)\ar[urr]_-{\bV_{\tr}^{G}, }&& }\ . \]
 In particular, $\bV^G_\bC$ admits transfers for every $\bC$ in $\Fun(BG, {\CL})$.
 \end{theorem}

 As a consequence of \cref{ugiqergqrefwefqfqewfqef}, we will also derive the following:
\begin{kor}\label{wergioergewrgregregwegregwregrre}
The functor $\bV_{\bC}^{G,c}$ admits transfers,  {and so does the composition $\Homol \circ \bV^{G,c}_\bC$ for any functor $\Homol \colon \Cle \to \bM$.}
\end{kor}

 {The existence of transfers is an ingredient of the proof of \cref{wefiofewqfwefqefqewf}.
 As an aside, let us already record an easier consequence of the existence of transfers: they induce change of groups-functors.}

 More precisely, let $H$ be a subgroup of $G$ and consider the left adjoint
\[ H\Set \to G\Set, \quad T \mapsto G \times_H T \]
of the canonical restriction functor $\res^G_H \colon G\Set \to H\Set$.
As explained in \cite[Sec.~4]{desc}, this functor refines to a functor
\[ \indd_H^G \colon H\BC \to G\BC \]
by sending an $H$-bornological coarse space $X$ to the $G$-bornological coarse space $\Ind_H^G(X)$ whose underlying set is the induced $G$-set and whose bornological coarse structure is given as follows:
\begin{enumerate}
 \item The bornology is generated by the images of sets of the form $\{g\} \times B$ under the multiplication map $G \times X \to G \times_H X$, where $g$ is any element of $G$ and $B$ is a bounded subset of $X$.
 \item The coarse structure is generated by the images of sets of the form $\diag(G) \times U$ under the multiplication map $G \times X \to G \times_H X$, where $U$ is {any} coarse entourage of $X$.
\end{enumerate}
The canonical map $c_{X} \colon \indd_H^G\res_H^G(X) \to X,\  {[g,x] \mapsto gx}$, given by the counit on the level of $G$-sets is not a morphism of $G$-bornological coarse spaces in general.
We do, however, have the following.

\begin{lem}
 The canonical map $c_{X} \colon \indd_H^G\res_H^G(X) \to X$ is a covering.
\end{lem}
\begin{proof}
We  have a {commutative} diagram
\[\xymatrix{
 \indd_H^G\res_H^G(X) \ar[dr]_{c_{X}} \ar[rr]_{\cong}^{([g,x]\mapsto ([g],gx)}&&\ar[dl]^{\pr_{X}}(G/H)_{min,min} \otimes X\\
  &X&} \]
in $G\BC$. By \cref{egowpegrrewfrewfwrefewfref}, the projection $\pr_{X}$ is a covering.
\end{proof}

Hence, if $E \colon G\BC \to \bM$ is a functor which admits transfers, a choice of extension $E_{\tr}$ induces a natural transformation
\[ c \colon E \to E \circ \indd_H^G \circ \res_H^G\]
given by $c=(c_{X}^{*})_{X\in G\BC}$.

\begin{rem}
Classical examples of equivariant coarse homology theories often come as families
$(E^{H})_{H}$ indexed by the subgroups of $G$ together with comparison maps $E^{G}\to E^{H}\circ \res^{G}_{H}$, where $E^{H}$ is an $H$-equivariant coarse homology theory (see e.g.~\cite[Sec.~8.5]{equicoarse} or \cite[Sec.~11]{Bunke:ad}  for equivariant coarse algebraic or coarse topological $K$-theory, respectively).

Transfers 
 for the functor $ E\colon G\BC\to \bM$ can be considered as a generalisation of this structure.
Indeed, we can define functors  $E^{H}:=E\circ \indd_{H}^{G}\colon H\BC\to \bM$ for all subgroups $H$ of $G$.  The choice of the extension $E_{\tr}$ then provides the comparison maps
$ E^{G}\to E^{H}\circ \res^{G}_{H}$.
\end{rem}

\subsection{Existence of transfers}\label{sec:transfers-proofs}
 This section is devoted to the proof of \cref{ugiqergqrefwefqfqewfqef}. It is based on
Barwick's formalism of Burnside categories \cite{Barwick:Mackey} which allows us to give a coherent description how the transfers interact with the covariant functoriality of $\bV^{G,c}_\bC$.

The composition of the  functor  $\Sh^{\pi_0,G}$ from \eqref{wergpok2pgegr5gwegwefv} with the inclusion $\CL\to \CATi$ and $\op \colon \CATi\to \CATi$   gives rise to a cocartesian fibration
\begin{equation}\label{qwefefewfqfewf}
 s^{\pi_0} \colon \hat \Sh^{\pi_{0},G}\to G\Coarse\times \Fun(BG,\CL)
\end{equation}
which classifies the functor
$(X,\bC)\mapsto \Sh_{\bC}^{\pi_{0},G}(X)^{\op}$.
 Taking the opposite category in the target is motivated by the following.
\begin{lem}\label{rioqgegwgregwerg}
 $s^{\pi_0}$ is also a cartesian fibration.
\end{lem}
\begin{proof} The opposite of  $s^{\pi_0}$, i.e., the map
\begin{equation}\label{qwefefewfqfewf111}
	s^{\pi_0,\op} \colon \hat \Sh^{\pi_{0},G,\op}\to G\Coarse^{\op}\times \Fun(BG,\CL)^{\op}
\end{equation}
is a cartesian fibration classifying $\Sh^{\pi_0,G}$
as a contravariant functor, but such that the fibre over $(X,\bC)$ is equivalent to $\Sh^{\pi_0,G}_\bC(X)$.
In this picture, the functors
$\hat f_{*}^{G}$ for morphisms $f$ in $G\Coarse$ and $\hat \phi^{G}_*$ for morphisms $\phi$ in $\Fun(BG,\CL)$  have left adjoints.  For $f$ this follows from \cref{iuqhfiuewvffewfqfewf}, and for $\phi$ we use \cref{ergioegrergwergwregwergwegrwr}. 
 By Lurie \cite[Cor.~5.2.2.5]{htt},  $s^{\pi_0,\op}$
is also a cocartesian fibration.
 In particular, its opposite $s^{\pi_0}$ is also both cartesian and cocartesian. 
\end{proof}

The subfunctor $\Sh^{G}$ (see \eqref{regewgk2p5getgwreg}) of $\Sh^{\pi_{0},G}$ gives rise to a cocartesian subfibration of $s^{\pi_0}$:
\begin{equation}\label{qfqefewfqfewfqew}
 {s \colon} \hat \Sh^{G}\to G\Coarse\times \Fun(BG,\CL)\ .
\end{equation}

For a morphism $f$  {of $G$-coarse spaces}
the induced morphism $\hat f^{G}_{*}$  on sheaves has a left adjoint only under additional conditions (e.g., if $f$ is a coarse covering, see \cref{unex-left}). Similarly, for   a morphism $\phi$ in $\Fun(BG,\CL)$ the induced morphism $\hat\phi^G_{*}$ on sheaves is a morphism in $\CLL$ and not expected to have a left adjoint, except if it is an equivalence.
 In order to capture this situation,  we will employ Barwick's effective Burnside category formalism.  
To this end we recall some terminology from \cite[Sec.~5]{Barwick:Mackey}.

\begin{ddd}\label{def:triple}
	A triple is an $\infty$-category $\bD$ together with two subcategories $\bD_\dag$ and $\bD^\dag$ of $\bD$, both of which contain the maximal Kan complex $i\bD$ of $\bD$.
\end{ddd}

Let $(\bD, \bD_\dag,\bD^\dag)$ be a triple. 
	 
 \begin{enumerate}
  	 \item The morphisms in $\bD_\dag$ are called ingressive  and will be depicted by the symbol $\hookrightarrow$.
	 \item The morphisms in $\bD^\dag$ are called egressive and will be depicted by the symbol $\twoheadrightarrow$.
 \end{enumerate}
	
	Let $(\bD, \bD_\dag,\bD^\dag)$ be a triple. 
	 
\begin{ddd} \label{regioergjwogregregregwr}
  $(\bD, \bD_\dag,\bD^\dag)$ is called adequate if every diagram \[\xymatrix{
		 & X'\ar@{->>}[d] \\
		Y\ar@{^{(}->}[r] & X
 }\]
 in $\bD$ can be completed to a pullback square   \begin{equation}\label{verv2rggf34fe}
 \xymatrix{
	Y'\ar@{^{(}->}[r]\ar@{->>}[d] & X'\ar@{->>}[d] \\
	Y\ar@{^{(}->}[r] & X
 }
 \end{equation}
 in $\bD$.
\end{ddd}

Pullback squares of the form \eqref{verv2rggf34fe} are called  ambigressive  squares.  
We will often say that a square is ambigressive in $\bD$, understanding that this refers to a given triple structure on $\bD$.

Let $G\Coarse^{\dag}$ be the  wide subcategory of $G\Coarse$ 
of coarse coverings (\cref{wefgihjwiegwergrwrg}), and 
set $G\Coarse_{\dag}:=G\Coarse$. 
Then we set 
\begin{align*}
\bD&:=G\Coarse\times \Fun(BG,\CL)\\
\bD_{\dag}&:=G\Coarse_{\dag}\times \Fun(BG,\CL)\\
\bD^{\dag}&:=G\Coarse^{\dag}\times i\Fun(BG,\CL)\ .
\end{align*}
So we have $\bD_{\dag} = \bD$, and a morphism $(f,\phi)$ in $\bD$ belongs to $ \bD^{\dag}$ if and only if $f$ is a coarse covering and $\phi$ is an equivalence.

\begin{lem} The following triples are adequate:
\begin{enumerate}
\item  $(G\Coarse,G\Coarse_{\dag},G\Coarse^{\dag} )$;
 \item $(  \Fun(BG,\CL), \Fun(BG,\CL), i\Fun(BG,\CL))$;
 \item $(\bD, \bD_\dag,\bD^\dag)$.
\end{enumerate}
\end{lem}
\begin{proof} 
We use that $G\Coarse$ and $\Fun(BG,\CL)$ admit pullbacks, and that a pullback of a coarse covering or an equivalence 
is again a coarse covering  (\cite[Lem.~2.11]{coarsetrans}) or an equivalence, respectively. 
\end{proof}

\begin{lem}\label{prop:cart-cocart}
	\
	\begin{enumerate}
		\item\label{it:cart-cocart1} The  projection 
		$\hat{\Sh}^{G} \mathop{\times}\limits_{\bD} \bD_{\dag}\to\bD_{\dag} $
		is a cocartesian fibration.
		\item\label{it:cart-cocart2} The  projection 
		$\hat{\Sh}^{G} \mathop{\times}\limits_{\bD} \bD^{\dag}\to\bD^{\dag} $
		is a cartesian fibration.
	\end{enumerate}
\end{lem}
\begin{proof}
 The map in \eqref{it:cart-cocart1} is equivalent to   the cocartesian fibration $s$.
For \eqref{it:cart-cocart2} we first observe, using  \cref{rioqgegwgregwerg},
that  $\hat{\Sh}^{\pi_{0},G} \mathop{\times}\limits_{\bD} \bD^{\dag}\to\bD^{\dag} $
is  the pullback of a cartesian fibration and hence itself cartesian. We then use that the cartesian lifts of coarse coverings preserve the  category $\hat \Sh^{G}$ by  the existence of left adjoints asserted in \cref{unex-left}, and that we restricted to the maximal Kan complex in the second factor of $\bD$.
\end{proof}

We define $\hat \Sh^{G}_{\dag}$ to be the  subcategory  of cocartesian morphisms in 
\[ \hat \Sh^{G}\times_{\bD}\bD_{\dag} \to\bD_{\dag}\ .\]
Similarly, we define $\hat \Sh^{G,\dag}$ to be the subcategory  of cartesian morphisms in
\[ \hat \Sh^{G}\times_{\bD}\bD^{\dag} \to \bD^{\dag}\ .\]

\begin{prop}\label{riqrjhgoqgqreggwergwrgw}
	The triple $(\hat{\Sh}^{G}, \hat{\Sh}^{G}_\dag, \hat{\Sh}^{G,\dag})$ is adequate, and the map of triples
	\begin{equation}\label{revvewvwewervewrfwfrfrefwerfwf}
	 {(s,s_\dag,s^\dag) \colon} (\hat{\Sh}^{G}, \hat{\Sh}^{G}_\dag, \hat{\Sh}^{G,\dag}) \to (\bD, \bD_\dag,\bD^\dag)
	\end{equation} 
	preserves ambigressive   squares.
\end{prop}

 The proof of the proposition will be prepared by some intermediate results.
 We fix $\bC$ in $\Fun(BG,\CL)$ and consider the specialisation
 \begin{equation}\label{ewrervwervev23ff}
 s_\bC \colon \hat \Sh_{\bC}^{G}\to G\Coarse
\end{equation}
of  $s$
at $\bC$ in the second factor of the target. We will similarly write
\[ {s_{\bC,\dag} \colon} \hat \Sh_{\bC,\dag}^{G} \to G\Coarse \:\:\:\:\mbox{and}\:\:\:\: s_\bC^\dag \colon \hat \Sh_{\bC}^{G,\dag}\to G\Coarse^{\dag}\ . \]
Let
\begin{equation}\label{eq:square}
\xymatrix{
	M'\ar[r]^-{\phi'}\ar[d]_-{\psi} & N'\ar[d]^-{\psi^{\prime}} \\
	M\ar[r]^-{\phi} & N
}
\end{equation}
be a square in $\hat \Sh^{G}_{\bC}$.

\begin{lem}\label{cor:detect-cartesian}
	Assume:
	\begin{enumerate}
	 \item \label{wegiowgwerwgerg} The square \eqref{eq:square} covers  an ambigressive square in $G\Coarse$.
 	 \item \label{wegiowgwerwgerg1} $\phi$ is cocartesian.
	 \item\label{wegiowgwerwgerg2}  $\psi$ is cartesian.
	\end{enumerate}
	Then $\phi'$ is cocartesian if and only if $\psi^{\prime}$ is cartesian.
\end{lem}
\begin{proof} By \eqref{wegiowgwerwgerg}, the square
  covers a pullback square 
 \[\xymatrix{V\ar[r]^{g}\ar[d]^{v}&W\ar[d]^{w}\\X\ar[r]^{f} &Y}\]
 in $G\Coarse$ such that $v$ and $w$ are coarse coverings. 
 By  \eqref{wegiowgwerwgerg1} and  \eqref{wegiowgwerwgerg2},
  we get the following situation
  \begin{equation} \label{ergierjgiogrjoerijgoiqergqfqwfqefqef}\xymatrix{L^{\pi_{0}}\hat v^{*,G}M\ar[d]^{\simeq}\ar[r]^-{\text{cocart}}&\hat g_{*}^{G}L^{\pi_{0}}\hat v^{*,G} M\ar@{..>}[d]\ar[dr]_{\simeq}^{\hphantom{xxx}\cref{ropgkpwoegrewgwregwgregw}}&\\
	M'\ar[r]^-{\phi'}\ar[d]_-{\psi} & N'\ar[d]^-{\psi^{\prime}} \ar@{..>}[r]&L^{\pi_{0}}\hat w^{*,G} \hat f^{G}_{*}M\ar[d]^-{\text{cart}}\\
	M\ar[r]^-{\phi} & N\ar[r]^{\simeq}& \hat f^{G}_{*}M
}\ ,
\end{equation}
where the dotted arrows are obtained from the universal properties of the cartesian or cocartesian maps as indicated. The composition of the two dotted arrows is the comparison morphism  from \cref{ropgkpwoegrewgwregwgregw}, which has the indicated direction since the upper right triangle in 
\eqref{ergierjgiogrjoerijgoiqergqfqwfqefqef} lives in the fibre  $\Sh^{G}_{\bC}(W)^{\op}$  of $s$ over $(W,\bC)$ (recall the definition \eqref{qwefefewfqfewf} of $s^{\pi_{0}}$ which involved  {taking fibrewise opposites)}. 

If  $\phi^{\prime}$ is cocartesian or $\psi^{\prime}$ is cartesian, then one of the dotted arrows is an equivalence, and hence the other is an equivalence, too.
\end{proof}

The following assertion is a general fact about  an inner fibration in $\Cati$, but for concreteness we formulate it for $s_\bC$.
 We consider again a square of the shape \eqref{eq:square} in $\hat \Sh^{G}_{\bC}$.

\begin{lem}\label{lem:detect-pull-back}
Assume:
\begin{enumerate} 
	\item  \label{rgiojeoigwegrgwergr42r23} $\psi$ is cartesian.
	\item \label{rgiojeoigwegrgwerg} $\psi^{\prime}$ is cartesian.
	\item \label{rgiwhiogwerggwergw} The square \eqref{eq:square} covers a pullback square in $G\Coarse$.
\end{enumerate}
	Then \eqref{eq:square} is a pullback square.
\end{lem}
\begin{proof}
 We contract the notation for mapping spaces $\Map_\bE(-,-)$ to $\bE(-,-)$. We further contract $G\Coarse$ to $G\cC$. Then we have the following chain of equivalences of spaces:
	\begin{align*}
	&\hat \Sh^{G}_{\bC}(P,M') \\
	&\stackrel{(1)}{\simeq} \hat \Sh^{G}_{\bC}(P,M) \mathop{\times}\limits_{G\cC(s_\bC(P),s_\bC(M))}G\cC(s_\bC(P),s_\bC(M^{\prime})) \\
	&\stackrel{ (2)}{\simeq} \hat \Sh^{G}_{\bC}(P,M) \mathop{\times}\limits_{G\cC(s_\bC(P),s_\bC(M))} G\cC(s_\bC(P),s_\bC(M)) \mathop{\times}\limits_{G\cC(s_\bC(P),s_\bC(N))} G\cC(s_\bC(P),s_\bC(N^{\prime})) \\
	&\stackrel{ (3)}{\simeq}  \hat \Sh^{G}_{\bC}(P,M) \mathop{\times}\limits_{G\cC(s_\bC(P),s_\bC(N))} G\cC(s_\bC(P),s_\bC(N^{\prime})) \\
	&\stackrel{(4)}{\simeq} 
	\hat \Sh^{G}_{\bC}(P,M) \mathop{\times}\limits_{\hat \Sh^{G}_{\bC}(P,N)}\hat \Sh^{G}_{\bC}(P,N)
	 \mathop{\times}\limits_{G\cC(s_\bC(P),s_\bC(N))} G\cC(s_\bC(P),s_\bC(N^{\prime}))\\
	&\stackrel{(5)}{\simeq} 
	\hat \Sh^{G}_{\bC}(P,M) \mathop{\times}\limits_{\hat \Sh^{G}_{\bC}(P,N)}\hat \Sh^{G}_{\bC}(P,N^{\prime})
	\end{align*}
  The equivalence marked  by (1) expresses the assumption  that $\psi$ is  cartesian.
  The equivalence marked by (2) follows from the isomorphism
  \[ s_\bC(M^{\prime})\cong s_\bC(M) \mathop{\times}\limits_{s_\bC(N)} s_\bC(N^{\prime}) \]
  in $G\cC$ which holds by assumption \eqref{rgiwhiogwerggwergw}.
  For the equivalence marked  by (3) we just cancelled the factor $G\cC(s_\bC(P),s_\bC(M))$, and for  the equivalence  marked  by (4) we introduced the factor $\hat \Sh^{G}_{\bC}(P,N)$.
  The equivalence marked by  $(5)$ expresses the assumption that $\psi^{\prime}$ is cartesian.
  The chain of equivalences   above is natural in $P$ and   shows that $M^{\prime}\simeq M \times_{N}N^{\prime}$, i.e., that \eqref{eq:square} is a pullback square.
\end{proof}

\begin{lem}\label{lem:preservation-pull-back} 
	If \eqref{eq:square} is an ambigressive square, then its image in $G\Coarse$   is an ambigressive square.
\end{lem}
\begin{proof}
  Applying $s_\bC$ to  \eqref{eq:square}, we obtain the outer square as in the following picture:
\begin{equation}\label{qwfqwfweeqwfeqwfqe}
 \xymatrix{
  V\ar@/_0.4cm/[ddddddr]_{s_\bC(\psi)}\ar@/^0.4cm/[rrrrrrd]^{s_\bC(\phi^{\prime})}\ar@/^0.3cm/@{.>}[dr]^-{a}&&&&&& \\
  &V^{\prime}\ar@/^0.2cm/@{-->}[ul]^(.3){s_\bC(\beta)}\ar[ddddd]\ar[rrrrr]&&&&&W\ar[ddddd]^{s_\bC(\psi^{\prime})}\\
  && 	M'\ar@/^0.2cm/@{-->}@{.>}[dr]^{\alpha}\ar@/^0.5cm/[rdr]^-{\phi'}\ar@/_0.5cm/[ddr]_-{\psi\ }&&&&\\
  &&&P \ar@/^0.3cm/@{-->}[ul]^{\beta}\ar@{-->}[r]^{\bar \phi^{\prime}}\ar@{-->}[d]^{\bar \psi}& N'\ar[d]^-{\psi^{\prime}}&& \\
  &&&M\ar[r]^-{\phi } & N&&\\
  &&&&&&\\
  &X \ar[rrrrr]^{s_\bC(\phi)}&&&&& Y
  }\end{equation}
We must show that this square is ambigressive  in $ G\Coarse$. 

Let $V^{\prime}$ be the pullback $X \times_Y W$.
Then we get a uniquely determined map $a$ as indicated. 
Our task is to show that $a$ is an isomorphism.

 We  choose a cocartesian lift $\alpha$ of $a$ as indicated.
Since $\alpha$ is cocartesian, we obtain the maps $\bar\phi' $ and $\bar\psi $.
Since \eqref{eq:square} is cartesian by assumption, we also get a map $\beta$ as indicated.
We conclude that $\beta\circ \alpha\simeq \id_{M^{\prime}}$ and hence $s_\bC(\beta)\circ a =\id_{V}$. 
Since the square with upper left corner $V^{\prime}$ is cartesian, we also have $a\circ s_\bC(\beta)=\id_{V^{\prime}}$. Hence $a$ is an isomorphism.

Since \eqref{eq:square} is ambigressive, we know that $s_{\bC}(\psi^{\prime})$ is a coarse covering. By \cite[Lem.~2.11]{coarsetrans}, this implies that $s_\bC(\psi)$ is also a coarse covering. Hence the outer square in \eqref{qwfqwfweeqwfeqwfqe}
is  ambigressive.
	\end{proof}

\begin{lem}\label{lem:sh-adequate} 
The triple $(\hat{\Sh}^{G}, \hat \Sh^{G}_\dag,  \hat \Sh^{G,\dag})$ is adequate.
\end{lem}
\begin{proof}
 Consider the diagram \begin{equation}\label{fweiwjefkwefjwkofewfwefwff}
\xymatrix{&N^{\prime}\ar@{->>}[d]^{\chi}\\M\ar@{^{(}->}[r]^{\phi}&N}
\end{equation}
in $\hat{\Sh}^{G}$. We must show that it can be extended to an ambigressive square.

The image  of \eqref{fweiwjefkwefjwkofewfwefwff}    in $\bD=G\Coarse \times \Fun(BG,\CL)$ is a diagram of the form
\[\xymatrix{&(Y^{\prime},\bB^{\prime}) \ar@{->>}[d]^{(h,v)}\\(X,\bC)\ar@{^{(}->}[r]^{(f,u)}&(Y,\bB)}\ .\]
Here $X,Y,Y'$ are in $G\Coarse$ and    $h$ is a coarse covering. Furthermore, $\bC,\bB,\bB'$ are  in  $ \Fun(BG,\CL)$ and $v$ is an equivalence.
We decompose the diagram as follows: \begin{equation}\label{vwefwelkfwefwefewfwefwe}
\xymatrix{&&(Y^{\prime},\bB^{\prime}) \ar@{->>}[d]^{(\id,v)}\\ &&(Y^{\prime},\bB)\ar@{->>}[d]^{(h,\id)}\\(X,\bC)\ar@{^{(}->}[r]^{(\id,u)}&(X,\bB)\ar@{^{(}->}[r]^{(f,\id)}&(Y,\bB)}\ ,
\end{equation}
We   can extend 
the diagram
\[\xymatrix{&Y'\ar@{->>}[d]^{h}\\X\ar@{^{(}->}[r]^{f}&Y}\]
to a pullback square in $G\Coarse$
\[\xymatrix{Y^{\prime}\ar@{->>}[d]^{g}\ar@{^{(}->}[r]^{f^{\prime}}&Y'\ar@{->>}[d]^{h}\\X\ar@{^{(}->}[r]^{f}&Y}\ .\]
Then $g$ is also a coarse covering \cite[Lem.~2.11]{coarsetrans}. 
Hence we get the diagram 
\[\xymatrix{&&(Y^{\prime},\bB^{\prime})\ar@{->>}[d]^{(\id,v)}\\ &(X^{\prime},\bB)\ar@{^{(}->}[r]^{(f^{\prime},\id)}\ar@{->>}[d]^{(g,\id)}&(Y^{\prime},\bB)\ar@{->>}[d]^{(h,\id)}\\(X,\bC)\ar@{^{(}->}[r]^{(\id,u)}&(X,\bB)\ar@{^{(}->}[r]^{(f,\id)}&(Y,\bB)}\]
where the   square is ambigressive. 
This gives finally the following four ambigressive squares
\begin{equation}\label{f34f34f334f3f3ee}
\xymatrix@C=3.5em{(X^{\prime},\bC)\ar@{^{(}->>}[r]^{(\id,v^{-1}u)}\ar@{->>}[d]^{(\id,\id)}&(X^{\prime},\bB^{\prime})\ar@{->>}[d]^{(\id,v)}\ar@{^{(}->}[r]^{(f^{\prime},\id)}&(Y^{\prime},\bB^{\prime}) \ar@{->>}[d]^{(\id,v)}\\ (X^{\prime},\bC)\ar@{->>}[d]^{(g,\id)}\ar@{^{(}->}[r]^{(\id,u)}&(X^{\prime},\bB)\ar@{^{(}->}[r]^{(f^{\prime},\id)}\ar@{->>}[d]^{(g,\id)}&(Y^{\prime},\bB)\ar@{->>}[d]^{(h,\id)}\\(X,\bC)\ar@{^{(}->}[r]^{(\id,u)}&(X,\bB)\ar@{^{(}->}[r]^{(f,\id)}&(Y,\bB)}\ ,
\end{equation}
where for the upper left corner we use that $v$ is an equivalence.
The diagram
\eqref{fweiwjefkwefjwkofewfwefwff} can now be extended as 
\[\xymatrix{&&N^{\prime} \ar@{->>}[d]^{ {\omega_N}}\\ && {\overline{N}'}\ar@{->>}[d]^{{\overline{\chi}}}\\M \ar@/_1cm/@{^{(}..>}[rr]_{\phi}\ar@{^{(}->}[r]^{{\omega_M}} &{\overline{M}}\ar@{^{(}->}[r]^{ \overline{\phi}} &N}\]
over \eqref{vwefwelkfwefwefewfwefwe}.
The morphism ${\omega_M}$
is defined as a cocartesian lift of $(\id_{X},u)$ with domain $M$. The morphism  $\overline{\phi}$
is then obtained from the universal property of the cocartesian lift, and it is also cocartesian since $\phi$ is cocartesian.
The morphisms $\omega_N$ and $\overline{\chi}$ are obtained similarly. In particular, $\omega_N$ is an equivalence and $\overline{\chi}$ is cartesian.

We now choose a cartesian lift ${\overline{\gamma}}$ of the morphism $(g,\id)$ to obtain
 \[\xymatrix{&&N^{\prime} \ar@{->>}[d]^{{\omega_N}}\\ &{\overline{M}'}\ar@{^{(}.>}[r]^{{\overline{\phi}'}} \ar@{->>}[d]^{{\overline{\gamma}}}& {\overline{N}'}\ar@{->>}[d]^{{\overline{\chi}}}\\M  \ar@{^{(}->}[r]^{{\omega_M}} &{\overline{M}}\ar@{^{(}->}[r]^{ {\overline{\phi}}} &N}\ .\]
 The morphism ${\overline{\phi}'}$ is obtained from the universal property of ${\overline{\chi}}$ being cartesian.
 By \cref{cor:detect-cartesian}, ${\overline{\phi}'}$ is also cocartesian and therefore ingressive as indicated. By   \cref{lem:detect-pull-back}, the new square is a pullback and hence an ambigressive  square.
The same argument provides  the lower left ambigressive   square in the following diagram over \eqref{f34f34f334f3f3ee}:
 \[\xymatrix{
  {M''}\ar@{^{(}->}[r]^{\phi''}  \ar@{->>}[d]^{{\omega_{M''}}}& M' \ar@{^{(}->}[r]^{\phi'}  \ar@{->>}[d]^{{\omega_{M'}}}&N^{\prime} \ar@{->>}[d]^{{\omega_N}} \\
  {\overline{M}''} \ar@{->>}[d]\ar@{^{(}->}[r]^{{\overline{\phi}''}}&{\overline{M}'}\ar@{^{(}->}[r]^{{\overline{\phi}'}} \ar@{->>}[d]^{{\overline{\gamma}}}& {\overline{N}^{\prime}}\ar@{->>}[d]^{{\overline{\chi}}} \\
  M  \ar@{^{(}->}[r]^{{\omega_M}} & {\overline{M}}\ar@{^{(}->}[r]^{ \overline{\phi}} &N}\ .\]
For the upper part we can argue similarly, but, using the fact that $v$ is an equivalence, we can choose $\omega_{M'}$ and $\omega_{M''}$ as equivalences (which are both cartesian and cocartesian). 
The outer square is the desired extension of diagram \eqref{fweiwjefkwefjwkofewfwefwff} to an ambigressive square.
\end{proof}

\begin{proof}[Proof of \cref{riqrjhgoqgqreggwergwrgw}]
 The triple $(\hat\Sh^G,\hat\Sh^G_\dag,\hat\Sh^{G,\dag})$ is adequate by \cref{lem:sh-adequate}. It remains to show that $(s,s_\dag,s^\dag)$ preserves ambigressive squares.  
 We consider an ambigressive   square \begin{equation}\label{bretlknjmblrbrtbrtb}
\xymatrix{
	M'\ar@{^{(}->}[r]^-{\phi'}\ar@{->>}[d]_-{\psi} & N'\ar\ar@{->>}[d]^-{\psi^{\prime}} \\
	M\ar@{^{(}->}[r]^-{\phi} & N
}
\end{equation}
in $\hat{\Sh}^{G} $. Then  the projection to the first factor $G\Coarse$   of its image in $\bD$ is ambigressive   by  \cref{lem:preservation-pull-back}.
 The image  of diagram \eqref{bretlknjmblrbrtbrtb} in the second factor $\Fun(BG,\CL)$ of  $\bD$ has the form 
\[\xymatrix{\bC'\ar[r]\ar[d]_{\simeq}&\bB'\ar[d]^{\simeq}\\\bC\ar[r]&\bB}\]
and is therefore a pullback square.
Therefore, it is an ambigressive square in $\Fun(BG,\CL) $.
\end{proof}

We consider the functor $\tw \colon \Delta \to \Cat$ which associates to the poset $[n]$  in $\Delta$ the  {twisted arrow} poset
\[ \tw([n]):=\{(i,j)\in \{0,\dots,n\}^{2} \mid  i \leq j\} \]
with the order relation  $(i,j) \leq (i',j')$ if and only if  $i \leq i' \leq j' \leq j$.

Let $(\bD,\bD_\dag,\bD^\dag)$ be an adequate triple.
 Then we can consider the simplicial space
\[ i\Fun(\tw(-),\bD) \colon \Delta^{\op}\to \Spc\ . \]
For a functor $D \colon \tw([n]) \to \bD$ we write $D(i,j)$ for the value of $D$ at $(i,j)$ in $\tw([n])$.
The simplicial space $i\Fun(\tw(-),\bD)$ has a subobject $i\Fun^\times(\tw(-),\bD)$ given in degree $n$ by the subspace of those functors $D$ such that for all $0 \le i\le  k \le l \le j \le n$ the square
\[\xymatrix{
	D(i,j)\ar@{^{(}->}[r]\ar@{->>}[d] & D(k,j)\ar@{->>}[d] \\
	D(i,l)\ar@{^{(}->}[r] & D(k,l)
}\]
is an ambigressive  square in $\bD$. The simplicial space $i\Fun^\times(\tw(-),\bD)$ is a complete Segal space \cite[ {Prop.~5.9}]{Barwick:Mackey}  {(see also \cite[Thm.~2.13]{hhln:spans})}. 
 {For the next definition, we make use of the fact that $\Cati$ is a Bousfield localisation of $\Fun(\Delta^\op,\Spc)$, with $\Cati$ being equivalent to the full subcategory of complete Segal spaces.}

	\begin{ddd}\label{giroggrgergergergrege}
	 The effective Burnside $\infty$-category $\burn(\bD,\bD_\dag,\bD^\dag)$ is defined to be the $\infty$-category underlying
	 the complete Segal space $i\Fun^\times(\tw(-),\bD)$.
\end{ddd}
We refer to  {\cite[Sec.~3]{Barwick:Mackey}}
 for further details.
By \cite[Not.~3.9]{Barwick:Mackey}, there are canonical functors
\begin{equation}\label{frefklnfnrfklwefwefwefwefwef}
\bD_\dag \to \burn(\bD,\bD_\dag,\bD^\dag)\ ,\quad   \quad  (\bD^\dag)^{\op} \to \burn(\bD,\bD_\dag,\bD^\dag)\ .
\end{equation}  

\begin{rem}\label{ergiojwrgioergwergwrgregregwrg}
 Unwinding  \cref{giroggrgergergergrege}, one checks that diagrams of the form
\[\xymatrix{&D(0,1)   \ar@{^{(}->}[dr]   \ar@{->>}[dl]&\\ D(0,0)&&D(1,1)}\]
constitute the $1$-simplices in $\burn(\bD,\bD_\dag,\bD^\dag)$.
 {This also suggests why the functors in \eqref{frefklnfnrfklwefwefwefwefwef} exist}.
 \end{rem}
   
 Applying the effective Burnside category functor to $(s,s_\dag,s^\dag)$ in \eqref{revvewvwewervewrfwfrfrefwerfwf}, 
 we get a map
 \begin{equation}\label{qervqjrhvgbquierwqwvcewcwedcew}
 \burn(s) \colon \burn(\hat{\Sh}^{G}, \hat{\Sh}^{G}_\dag, \hat{\Sh}^{G,\dag}) \to  \burn(\bD,\bD_\dag,\bD^\dag)\ .
\end{equation} 
  
\begin{prop}\label{prop:unfurling}
 $\burn(s)$ is a cocartesian fibration.
\end{prop}
\begin{proof}
 {Our goal is to apply Barwick's criterion as presented in \cite[Thm.~3.1]{hhln:spans} to show the existence of cocartesian lifts.}
	We verify conditions  {(1) and (2) of \cite[Thm.~3.1]{hhln:spans}.}
		
	We start with condition \cite[Thm.~3.1.(2)]{hhln:spans}. Suppose that we are given a commutative square
	\begin{equation}\label{gergoij34optergergergergg}
\xymatrix{
	M'\ar@{^{(}->}[r]^-{\phi'}\ar@{->>}[d]_-{\psi} & N'\ar[d]^-{\psi^{\prime}} \\
	M\ar[r]^-{\phi} & N
	}\end{equation}
in $\hat \Sh^{G}$ whose images in $\bD$ and hence in $ G\Coarse$ are ambigressive squares, 
and such that $\phi'$ is ingressive, $\psi$ is egressive, and $\phi$ is cocartesian. We must show that the square \eqref{gergoij34optergergergergg} is  ambigressive if and only if $\phi'$
is cocartesian.

Note that $\phi'$ is always cocartesian by the definition of $\hat \Sh^{G}_{\dag}$, the morphism $\psi$ is cartesian, and $\phi$ is ingressive.
 It follows from \cref{cor:detect-cartesian} that then $\psi^{\prime}$ is also cartesian and hence egressive.
By \cref{lem:detect-pull-back}, 
square \eqref{gergoij34optergergergergg} is cartesian and hence an ambigressive square.
 This finishes the verification of  condition \cite[Thm.~3.1.(2)]{hhln:spans}.

	If $f \colon X \to Y$ is in $G\Coarse_{\dag}$  and $M$ is a preimage of $X$, there exists a cocartesian lift $\tilde f \colon M \to \hat f^{G}_{*}M$ of $f$. By definition, $\tilde f$ lies in $\hat{\Sh}^{G}_\dag$.
	 {Since restricting the domain of a cocartesian fibration to its subcategory of cocartesian morphisms also gives a cocartesian fibration (even a left fibration, see \cite[Cor.~2.4.2.5]{htt})},
	 $\tilde f$ is also cocartesian with respect to the map $\hat{\Sh}^{G}_\dag \to  G\Coarse_{\dag}$.
	This verifies condition \cite[Thm.~3.1.(1)]{hhln:spans}
	for morphisms in $ G\Coarse_{\dag}$.
	A similar argument applies to morphisms
	in $\Fun(BG,\CL)$. 
	
	\cite[Thm.~3.1]{hhln:spans} 	now asserts that an edge in $ \burn(\hat{\Sh}^{G}, \hat{\Sh}^{G}_\dag, \hat{\Sh}^{G,\dag})$ is cocartesian (with respect to~$\burn(s)$)
	if and only if it is represented by a span
	\[\xymatrix{
		& M\ar@{->>}[ld]_-{w}\ar@{^{(}->}[rd]^{f} & \\
		N & & P
	}\]
	in which $w$ is cartesian with respect to the map $\hat{\Sh}^{G,\dag} \to \bD^\dag$  and $f$ is cocartesian.
	Using    \cref{prop:cart-cocart} and the explicit description of the edges in $ \burn(\hat{\Sh}^{G}, \hat{\Sh}^{G}_\dag, \hat{\Sh}^{G,\dag}) $ given in \cref{ergiojwrgioergwergwrgregregwrg},
	we see that every morphism in  $\burn(\bD,\bD_\dag,\bD^\dag)$ has a cocartesian lift.  This shows that $\burn(s)$
	is a   cocartesian fibration
\end{proof}

We define the category $G\Coarse_{\tr}$ of $G$-coarse  spaces with transfers by
\begin{equation}\label{wergoij2oi54gwtgrgwger}
G\Coarse_{\tr}:=\burn(G\Coarse, G\Coarse_{\dag},G\Coarse^\dag )\ .
\end{equation}
An  instance of the canonical functor \eqref{frefklnfnrfklwefwefwefwefwef} applied to \eqref{wergoij2oi54gwtgrgwger} provides
\begin{equation}\label{qergiuhqegiuqgwefewfqfe}
G\Coarse\to G\Coarse_{\tr}\ .
\end{equation}
Using in addition the canonical inclusion
 {\small
 \[\Fun(BG,\CL) \to  \burn(\Fun(BG,\CL),\Fun(BG,\CL) ,i\Fun(BG,\CL))\]}
 in the second component, we get a functor \begin{equation}\label{wefweijfwejkfwefewfwefwefq}
 G\Coarse_{\tr} \times \Fun(BG,\CL)  \to \burn(\bD,\bD_\dag,\bD^\dag) \ . \end{equation} 

 	We define the functor \begin{equation}\label{ervewvervwevewrvw}\Sh^{G}_{\tr} \colon G\Coarse_{\tr} \times \Fun(BG,\CL)  \to \CAT_{\infty}
 	\end{equation}
	by restricting the cocartesian fibration $\burn(s)$
	along the functor from \eqref{wefweijfwejkfwefewfwefwefq}, applying the straightening functor, and applying $(-)^{\op}$ to the values.
	
	Note that the value  of $\Sh^{G}_{\tr}$ at $(X,\bC)$ is  the object $\Sh^{G}_{\bC}(X)$ of $\CLL$.
	By construction,
	$\Sh^{G}_{\tr}$ sends a morphism $f$ {of $G$-coarse spaces}
	to the morphism $\hat f^{G}_{*}$ in $\CLL$, 
	a coarse covering $w$
	to the morphism $L^{\pi_{0}}\hat w^{G,*}$ in $\CLL$,  and
	a morphism $\phi$ in $ \Fun(BG,\CL) $ to the morphism $\hat \phi_{*}^{G}$ in $\CLL$. This implies that 
	$\Sh^{G}_{\tr}$ actually takes values in the subcategory $\CLL$ of  $\CAT_{\infty}$, i.e., we have a functor
	 \begin{equation}\label{ervewvervwevewrvw1}
	 	\Sh^{G}_{\tr} \colon G\Coarse_{\tr} \times \Fun(BG,\CL)  \to \CLL \ .
 	\end{equation}
 Let $\gbct{G}$ be the category with the same objects as $G\BC$ and all equivariant maps which are controlled and bornological (\cref{rgiojrgofdewqfe}). Consider the wide subcategory $\gbct{G}_\dag$ of $\gbct{G}$ of those morphisms which are in addition
  proper,  and the wide  subcategory $(\gbct{G})^\dag$ of $\gbct{G}$  of coverings (\cref{ergwergggqrgqregqergq}).
Using \cite[Lem.~2.20 and 2.21]{coarsetrans}, one verifies that the triple $(\gbct{G},\gbct{G}_\dag,(\gbct{G})^\dag)$ is admissible.

\begin{ddd}\label{efbivhjiqrewvwewcewcq}
 We define
 \[ G\BC_{\tr}:=\burn(\gbct{G} ,\gbct{G}_{\dag},(\gbct{G})^{\dag})\ .\qedhere \]
 \end{ddd}

   \begin{rem}\label{rqegrgrfqefewfewfq}
One checks using \cref{ergiojwrgioergwergwrgregregwrg}  that $G\BC_{\tr}$ precisely coincides with the $\infty$-category  defined in \cite[Def.~2.27]{coarsetrans}
 and denoted there by the same symbol. 
 \end{rem}

There is an inclusion functor  {(see also \cite[Def.~2.35]{coarsetrans})}
\begin{equation}\label{ewrglkklwergwrgwregwregw}
\iota \colon G\BC \to G\BC_{\tr}
\end{equation}
which is determined by the requirement that it sends a $G$-bornological coarse space to itself and a morphism $f \colon X \to X'$ to the span $X \leftarrow \tilde X \xrightarrow{f} X'$, where $\tilde X$ is obtained from $X$ by replacing the bornology on $X$ by $f^{-1}\cB_{X'}$.

The forgetful functor $\gbct{G} \to G\Coarse$ induces a functor
\begin{equation}\label{fvqeervqvrvwevqw}
G\BC_{\tr}\to G\Coarse_{\tr}\ .
\end{equation}
such that the following diagram commutes:
\[\xymatrix@C=3.5em{
G\BC\ar[d]_{\eqref{gwegljgoiregwrgergregewg}}\ar[r]_{\eqref{ewrglkklwergwrgwregwregw}}^{\iota}&G\BC_{\tr}\ar[d]^{\eqref{fvqeervqvrvwevqw}}\\ 
G\Coarse\ar[r]^{\eqref{qergiuhqegiuqgwefewfqfe}}&G\Coarse_{\tr}} \]

If we precompose  {$\Sh^G_{\tr}$ from \eqref{ervewvervwevewrvw1} with the forgetful map 
 \eqref{fvqeervqvrvwevqw}}, we obtain a functor (we use the same symbol)
\[ \Sh^{G}_{\tr} \colon G\BC_{\tr}\times  \Fun(BG,\CL)  \to  \CLL \]
 such that $\Sh^G_{\tr} \circ (\iota \times \id) \simeq \Sh^G$.
 In particular, its value at $(X,\bC)$ is given by
 \[ \Sh^{G}_{\tr}(X,\bC)\simeq \Sh^{G}_{\bC}(X)\ .\]
 It follows from a combination of \cref{eriogwetgwegregwrgregwrg}, \cref{qrkgoqrgergwrgwegwrg}, and 
 \cref{qergioewgergregwergergwergwreg} that we have a subfunctor
 \[ \Sh^{G,\eqsm}_{\tr} \colon G\BC_{\tr}\times  \Fun(BG,\CL)  \to  \Cle \]
 with values 
 \[ \Sh^{G,\eqsm}_{\tr}(X,\bC)\simeq \Sh^{G,\eqsm}_{\bC}(X)\ .\]
 For every pair $(X,\bC)$   we have a localisation functor 
 $\ell \colon  \Sh^{G,\eqsm}_{\bC}(X)\to \bV^{G}_{\bC}(X)$  (\cref{rtheorthertherthetrhe}), and we let $\tilde W^{\eqsm}_{X,\bC}$ denote the  {subcategory} of 
 $ \Sh^{G,\eqsm}_{\bC}(X)$  {given} by the morphisms which are sent to equivalences by $\ell$.
 If $f$ is a morphism in $G\BC$, then $\hat f^{G}_{*}$ sends $\tilde W^{\eqsm}_{X,\bC}$ to $\tilde W^{\eqsm}_{X^{\prime},\bC}$
 by \cref{qergiuqergreqgregwregwr}. If $\phi$ is a morphism in  $\Fun(BG,\CL)$, then $\hat \phi^{G}_{*}$ sends  
 $  \tilde W^{\eqsm}_{X,\bC}$ to $\tilde W^{\eqsm}_{X ,\bC^{\prime}}$ by  \cref{giowegrwergwergwreg}.
 Finally, if $w \colon X\to X^{\prime}$ is a  morphism in $G\BC^{\dag}$, then $L^{\pi_{0}} \hat w^{G,*}  $ sends
 $\tilde W^{\eqsm}_{X^{\prime} ,\bC}$ to  $\tilde W^{\eqsm}_{X  ,\bC}$ by  \cref{iqerfjrfqwuef98weufeqwfqfe} \eqref{eewgergwergrewgwregwergwfw}.
 This implies that the functor $\Sh^{G,\eqsm}_{\tr}$ refines
 to a functor
 {\[ \ell \Sh^{G,\eqsm}_{\tr}\colon G\BC_{\tr}\times  \Fun(BG,\CL)  \to   \Rel_{\infty,*}^{\mathrm{Lex}} \]
 given on objects by $(X,\bC)\mapsto (\Sh^{G,\eqsm}_{\bC}(X), \tilde W^{\eqsm}_{X,\bC})$, see \cref{gjierogjeoigerg3554656}.
 We now postcompose with the Dwyer--Kan localisation functor $\Loc$ from \eqref{fqwepofjkqweopfdqewdq}.} 
 
  \begin{ddd}\label{egojwogergrewgwrgregwergreg}
 {We define the functor
   \[ \bV^G_{\tr} \colon G\BC_{\tr} \times\Fun(BG,\CL)\xrightarrow{\ell\Sh^{G,\eqsm}_{\tr}} 
 \Rel^\mathrm{Lex}_{\infty,*}\xrightarrow{\Loc} \Cle\ .\qedhere \]}
  \end{ddd}
 
 \begin{proof}[Proof of \cref{ugiqergqrefwefqfqewfqef}]
  By construction, $\bV^G_{\tr}$ satisfies
 	\[ \bV^G_{\tr} \circ (\iota \times \id) \simeq \bV^G\ .\qedhere\]
 \end{proof}
 
 {\cref{wergioergewrgregregwegregwregrre} makes an assertion about the continuous version of $\bV^G$, so we show that we can also force continuity for $\bV^G_{\tr}$}.
Let $E_{\tr} \colon G\BC_{\tr}\to \bM$ be some functor with a target which admits small filtered colimits. Let $\iota$  be the inclusion functor from \eqref{ewrglkklwergwrgwregwregw}.
Recall \cref{geriogjergerg} of a continuous functor.

\begin{ddd}
We call $E_{\tr}$ continuous if $E_{\tr}\circ \iota$ is continuous. 
\end{ddd}

In the following, we show that we can force continuity for $E_{\tr}$ in such a way that we obtain a functor $E_{\tr}^c$ satisfying $E_{\tr}^c \circ \iota \simeq (E_{\tr} \circ \iota)^c$ (see \cref{gerklgjerlgergergergergerg}).

 We consider the inclusion of the full  $\infty$-subcategory
 \[ i_{\tr} \colon G\BC^{\mb}_{\tr}\to G\BC_{\tr} \]
  of  $G$-bornological coarse spaces with the minimal bornology.  We consider the following diagram
\[\xymatrix{G\BC_{\tr}^{\mb}\ar[d]^-{i_{\tr}}\ar[rr]^-{E_{\tr}\circ i_{\tr}}\ar@{=>}[dr(0.4)]&&\bM\\G\BC_{\tr}\ar@{..>}[urr]_-{E_{\tr}^{c}}&}\ ,\]
 where $E^{c}_{\tr}$ is defined by left Kan extension. Similarly as in  \cref{gerklgjerlgergergergergerg},
we call $E_{\tr}^{c}$ the functor obtained from $E_{\tr}$ by forcing continuity. 
 
 We let $E:=E_{\tr}\circ \iota$ and 
 consider the following diagram
 \[\xymatrix{G\BC_{\tr}^{\mb}\ar[dr]^{i_{\tr}} \ar[rr]^-{E_{\tr}\circ i_{\tr}}&\ar@{=>}[d(0.6)]&\bM\ar@/^0.5cm/@{=}[dd]\\&G\BC_{\tr} \ar@{..>}[ur]_-{E_{\tr}^{c}}\\G\BC^{\mb} \ar@/^-0.5cm/[uu]^{\iota^{\mb}}\ar[rr]^{E \circ i}\ar[dr]^{i}&\ar@{=>}[d(0.6)]&M\\&G\BC \ar@/^-0.5cm/[uu]^(0.7){\iota}\ar@{..>}[ur]_-{E^{c}}&}\ ,\] 
 where $\iota^{\mb}$ is defined as the restriction of $\iota$ and
$E^{c}$ is defined by left Kan extension. The universal property of left Kan extensions provides a transformation
\begin{equation}\label{fewfoih4i3uo3rferfe}
E^{c} \to E^{c}_{\tr}\circ \iota\ .
\end{equation}

\begin{lem}\label{rhqewiuhqggqfewfeqwfqef}
Transformation \eqref{fewfoih4i3uo3rferfe} is an equivalence.
\end{lem}
\begin{proof}
Let $X$ be in $G\BC$. We must show that
\[ \colim_{(F\to X)\in (G\BC^{\mb})_{/X}} E_{\tr}(\iota(F))\to \colim_{(F\to X)\in (G\BC_{\tr}^{\mb})_{/X}} E_{\tr}(F) \]
is an equivalence, where the morphism is induced by the functor
\[ \iota_{X} \colon (G\BC^{\mb})_{/X} \to (G\BC_{\tr}^{\mb})_{/X} \]
between index categories induced by $\iota $.

We claim that $\iota_{X}$ is cofinal. 
We first show that the index categories are filtered.
We give the argument for  $(G\BC_{\tr}^{\mb})_{/X}$. The case of 
 $(G\BC^{\mb})_{/X}$ is similar and simpler.
 
  Let
  \[ F \colon K\to (G\BC_{\tr}^{\mb})_{/X} \]
  be a
 functor from a finite poset. We must find an extension to a functor  $F^{\triangleright}\colon K^\triangleright \to (G\BC_{\tr}^{\mb})_{/X}$.
  
  The functor $F$
  corresponds to a functor $F_+ \colon K_+ \to G\BC_{\tr}^{\mb}$, where $K_+$ is the finite poset $K$ with an additional maximal element $+$.
  By \cite[Thm.~2.18]{hhln:spans}, the functor $F_+$ in turn corresponds to a morphism of adequate triples
  \[ \hat F \colon \Tw(K_+) \to (\gbct{G} ,\gbct{G}_{\dag},(\gbct{G})^{\dag})\ , \]
  where a morphism in $\Tw(K_+)$ is ingressive if its image under the target projection $\Tw(K_+) \to K_+^\op$ is an equivalence, and egressive if its image under the source projection $\Tw(K_+) \to K_+$ is an equivalence.
  Note that $\hat F$ is a functor of ordinary categories.
  For each object $k$ of $K$, the morphism $\hat F(k,+) \to \hat F(k,k)$ is a covering.
  Using \cref{ergwergggqrgqregqergq}, it is easy to see that $\hat F(k,+)$ is equipped with the minimal bornology.
  Consider the subset
  \[ \widetilde F := \bigcup_{k \in K} \kappa_k(\hat F(k,+)) \]
  of $X$, where $\kappa_k$ denotes the morphism $\hat F(k,+) \to F(+,+) = X$.
 Since $K$ is finite, $\widetilde F$ is a $G$-invariant locally finite subset of $X$.
 
 We observe that  replacing the value of $\hat F$ at $(+,+)$ by $\widetilde F$ still defines a morphism of 
 adequate triples  $ \hat F' \colon \Tw(K_+) \to (\gbct{G} ,\gbct{G}_{\dag},(\gbct{G})^{\dag})$. This corresponds to a functor $F'\colon K\to  (G\BC_{\tr}^{\mb})_{/\widetilde F}$. We obtain the desired extension $F^{\triangleright}$ as the dashed arrow in the following
diagram
\[\xymatrix{K\ar@/^1cm/[rrrr]^{F}\ar[rr]^{ F'}\ar[dr]&&(G\BC_{\tr}^{\mb})_{/\widetilde F}\ar[rr]^{f}&&(G\BC_{\tr}^{\mb})_{/\widetilde X}\\&K^{\triangleright}\ar@{..>}[ru]\ar@{-->}[urrr]_{F^{\triangleright}}&&&}\ ,\]
where the dotted arrow exists since $(G\BC_{\tr}^{\mb})_{/\widetilde F}$ has a final object, and $f$ is given by composition with the span $\tilde F\xleftarrow{\id} \tilde F\to X$.
     \end{proof}
 
 Let $\bC$ be in $\Fun(BG,\CL)$, and let $\bV^{G}_{\tr, \bC}$ denote the specialisation of the functor  {$\bV^G_{\tr}$ from \cref{egojwogergrewgwrgregwergreg}} at $\bC$.
  We let
  \[ \bV^{G,c}_{\tr,\bC} \colon G\BC_{\tr}\to \Cle \]
  be obtained from $\bV^{G}_{\tr, \bC}$ by forcing continuity  in the sense of \cref{gerklgjerlgergergergergerg}.
 
  	\begin{proof}[Proof of \cref{wergioergewrgregregwegregwregrre}]
 		The equivalence
 		\[ \bV^{G,c}_{\tr} \circ \iota \simeq \bV^{G,c} \]
 		follows immediately from \cref{ugiqergqrefwefqfqewfqef} and \cref{rhqewiuhqggqfewfeqwfqef}.
 	\end{proof}

\section{The K-theoretic Novikov conjecture} \label{roijqreoeffqewfewfqewfefq}
 
 This section combines everything we have done so far to obtain the results promised in the introduction.
 \cref{fgoijegiorgergrgwrgregwergewrg} introduces the notion of a CP-functor and shows that the functor $\Homol  \bD_{G}$  {we will introduce} in \cref{ergoiergergergegerwgerggrerg43252} is  a CP-functor.
 \cref{iu9qweofewfqwefefqe} explains how (non-connective) algebraic $K$-theory provides a non-trivial example of the entire setup.

 {Finally having examples of CP-functors at our disposal, \cref{wreuigheiugerergegwggreg} recalls the main results of \cite{desc} and applies them to the CP-functors we have constructed.
 \cref{gwegljrgkrewgergreg,ex:spacesop}   were given in the introduction in order to stress the point that
   the  theory of the present paper provides a  unified picture   subsuming} the split injectivity results of both \cite{desc} and \cite{Bunke:aa}.
   The details will be discussed in \cref{sec:linearK,sec:A-coeffs}.

\subsection{CP-functors}\label{fgoijegiorgergrgwrgregwergewrg}
We start with recalling the notion of a CP-functor  {from} \cite{desc}.
Let $E \colon G\BC\to \bM$ be a functor. 
For every  {$G$-bornological coarse space} $X$
we can define a new functor
\begin{equation}\label{regglkml5grgwegergergwergwerg}
E_{X} \colon G\BC\to  \bM\ , \quad Y\mapsto E(X\otimes Y)
\end{equation}
called the twist of $E$ by $X$.
If $E$ is a coarse homology theory, then so is $E_{X}$, see e.g.~\cite[Sec.~10.4]{equicoarse}.
 Below, the twist $E_{G_{can,min}}$ is of particular importance (see \cref{etwgokergpoergegregegwergrg} for the definition of $G_{can,min}$).
 
Let $G\Orb$ be the orbit category of $G$.
 By restricting the functor  $i \colon G\Set \to G\BC$ from \eqref{qreoijqoiegjoqirffewfq}, we obtain the functor (compare with \eqref{vervelkrvnlervrvrevev})
\begin{equation}\label{qreoijqoiegjoqirffewfq1}
i \colon G\Orb\to G\BC\ , \quad S\mapsto S_{min,max} \ .
\end{equation}

  Let $M \colon G\Orb\to \bM$ be a functor.
 \begin{ddd}[{\cite[Def.~1.8]{desc}}]\label{bioregrvdfb}
We call $M$ a CP-functor if it satisfies the following conditions:
	\label{def:cpfunctor}
	\begin{enumerate}
		\item $\bM$ is stable, complete, cocomplete, and compactly generated.  
		\item There exists an equivariant coarse homology theory $E \colon G\BC\to \bM$ satisfying:
		\begin{enumerate}
			\item\label{iojfoiewfwefewqfqewf} $M$ is equivalent to $E_{G_{can,min}}\circ i$ (see \eqref{regglkml5grgwegergergwergwerg} and \eqref{qreoijqoiegjoqirffewfq1} for notation).
			\item $E$ is
			\begin{enumerate}
			\item  strongly additive (\cref{wergiugewrgergergerwgwregw}),
			\item  continuous (\cref{geriogjergerg}) and
			\item  admits transfers (\cref{egioeggewrgegervsfdvfdv}). \qedhere
			\end{enumerate}
		\end{enumerate}
	\end{enumerate}
\end{ddd}

 \begin{ex} \label{wrtohpwrtgrgrgregwr}
 Assume that $E \colon G\BC\to \bM$ is an equivariant coarse homology theory with a stable, complete, cocomplete, and compactly generated  target category. In addition, assume   that $E$ admits transfers, is strongly additive and continuous. Then
 the functor
 \[ E_{G_{can,min}}\circ i \colon G\Orb\to \bM \]
 is a CP-functor.
\end{ex}
  
  Let $\Homol$ be a  homological functor, and let $\bC$ be in $\Fun(BG,\CL)$.
 In view of \cref{wrtohpwrtgrgrgregwr}, \cref{weigowgregwrgrggregwgregwreg,wergioergewrgregregwegregwregrre} immediately imply:  
 
\begin{kor}\label{qroigjqioerefewfewfqf}
 Assume that $\Homol$ has a  complete  and compactly generated target category and is product-preserving.
 Then
 \[ (\Homol \bV^{G,c}_{\bC})_{G_{can,min}}\circ i \]
 is a CP-functor.
\end{kor}

 {Let  $\bD$   be a small left-exact $\infty$-category with $G$-action}. 
We define $\Ind^{G}(\bD) \colon G\Orb\to\Cle$ using the induction functor $\Ind^{G}$ from \eqref{2rfkhriufh3iuf3f3f}.
We furthermore consider a functor $\Homol \colon \Cle\to \bM$.
 \begin{ddd}\label{ergoiergergergegerwgerggrerg43252}
 We define the composition  \begin{equation}\label{ewbjnewoifvmlkrevcewvrvw}
  \Homol\bD_{G}:=\Homol  \circ\Ind^{G}(\bD)\colon G\Orb\to \bM\ . \qedhere
\end{equation}
\end{ddd}

\begin{theorem}\label{wefiofewqfwefqefqewf}
Assume: 
\begin{enumerate}
\item\label{giowgwegerggwergregwgwr3453252345} $\bM$ is stable, complete and cocomplete, and compactly generated.
\item\label{qeriogqwefeqwfqewfqefewfqw} $\Homol$ is a finitary localising invariant and preserves products.
 \item\label{qeriogqregqfqewfqef}  {The underlying $\infty$-category of $\bD$ is idempotent complete}.
\end{enumerate}
Then  $\Homol\bD_{G}$ is a CP-functor.
\end{theorem}
\begin{proof}
We set $\bC:=\Pro_{\omega}(\bD)$ in $\Fun(BG,\CL)$ (see \eqref{adewfqfewfqewfq}). Then \eqref{qeriogqregqfqewfqef} implies that $\bD\simeq \bC^{\omega}$.
 Combining \cref{efioqjwefoweweqfewfqwefqwefe} and \cref{cor:coeffs-orbits}, we have an equivalence
\[ \Homol\bD_{G} \simeq \Homol\circ \bV^{c}_{ \bC,G}\circ i\simeq (\Homol\Idem(\bV^{G,c}_{  \bC}))_{G_{can,min}}\circ i\]
(note that the operations of twisting by $G_{can,min}$ and postcomposing with $\Homol$ obviously commute).
Since $\Homol$ is localising by \eqref{qeriogqwefeqwfqewfqefewfqw}, we have an equivalence
\[ (\Homol \Idem (\bV^{G,c}_{\bC}))_{G_{can,min}}\circ i\simeq (\Homol \bV^{G,c}_{ \bC})_{G_{can,min}}\circ i\ .\]
This last functor is a CP-functor by \cref{qroigjqioerefewfewfqf}.
\end{proof}

We now consider the notion of hereditary CP-functors.
If $\phi \colon G\to Q$ is a surjective homomorphism of groups, then we get a functor
\[ \Res_{\phi}\colon Q\Orb\to G\Orb \]
which sends the $Q$-set $S$ to $S$ considered as a $G$-set via $\phi$.
Since $\phi$ is surjective, $S$ is still transitive as a $G$-set. 
    
  Let $M \colon G\Orb\to \bM$ be a functor. 
\begin{ddd}[{\cite[Def.~2.5]{Bunke:aa}}]\label{regiojergewergegergweggegregrwegreg}
We call $M$ a hereditary CP-functor if $M\circ \Res_{\phi}$ is a CP-functor for every  surjective group homomorphism $\phi$.
\end{ddd}
  
  Recall \cref{ergoiergergergegerwgerggrerg43252} of the functor $\Homol\bD_{G}$.
\begin{theorem}\label{rigjoqergergqregergqergreg}
We retain the assumptions of \cref{wefiofewqfwefqefqewf}. Then 
$\Homol\bD_{G}$ is a hereditary CP-functor.
\end{theorem}
\begin{proof}
Let $\phi \colon G\to Q$ be a surjective homomorphism.
Let $j^G \colon BG \to G\Orb$ and $j^Q \colon BQ \to Q\Orb$ denote the fully faithful inclusions (see \eqref{rewflkjmo34gergwegrge}). The homomorphism $\phi$ induces a map $G \to \Res_\phi(Q)$ in $G\Orb$ which defines a natural transformation $j^G \to \Res_\phi \circ j^Q \circ B\phi$ of functors $BG\to G\Orb$. The resulting natural transformation
\[ \bD \xrightarrow{\simeq} j^{G,*}j^G_!\bD \to B\phi^* j^{Q,*}\Res_{\phi}^* j^G_! \bD \]
induces an adjoint transformation
\[ j^Q_!B\phi_!\bD \to \Res_{\phi}^* j^G_!\bD\ .\]
We claim that this transformation is an equivalence. Using the pointwise formulas for left Kan extensions, we must show for every $S$ in $Q\Orb$ that the map
\[ \colim_{  (j^Q\circ B\phi )_{/S}} \bD \to \colim_{  j^G_{/\Res_\phi(S)}} \bD \]
induced by the functor
\[ (  j^Q\circ B\phi )_{/S} \to j^G_{/\Res_\phi(S)}\ , \quad  (Q \to S) \mapsto (G \xrightarrow{\phi} \Res_\phi(Q) \to \Res_\phi(S)) \]
is an equivalence. This is clear since the functor on indexing categories is an equivalence. We conclude that
\[ \Homol \bD_G \circ \Res_\phi \simeq \Homol (B\phi_!\bD)_Q\ .\]
 {Since $\Homol$ is invariant under idempotent completion and  {$\Idem$ is left adjoint to the fully faithful inclusion $\Clep \to \Cle$ and therefore commutes with the left Kan extension $(-)_{Q}$}, we have}
\[ {\Homol (B\phi_!\bD)_Q \simeq   \Homol  \Idem ((B\phi_!\bD)_Q)\simeq   \Homol (\Idem(B\phi_!\bD))_Q\ .}\]
The right-hand side is a CP-functor by \cref{wefiofewqfwefqefqewf}.
\end{proof}

\subsection{Algebraic \texorpdfstring{$K$-theory}{K-theory}}\label{iu9qweofewfqwefefqe}
 In this section we show that algebraic $K$-theory is an example of a finitary localising invariant to which  \cref{rigjoqergergqregergqergreg}  applies.
  Recall the universal finitary and stable localising invariant $\cU_{loc} \colon \stCat \to \cM_{loc}$ from \eqref{v4toi3hfio3f3f34f3f}.

\begin{ddd}\label{regiowergerregeggregw}
 The (non-connective) algebraic $K$-theory functor for stable $\infty$-categories is defined by	
\[ K^{\mathrm{st}} := \map_{\cM_{loc}}(\cU_{loc}(\Sp^{{\cop}}),\cU_{loc}(-)) \colon \stCat \to \Sp\ .\]
The (non-connective) algebraic $K$-theory functor for left-exact $\infty$-categories  is defined by
\begin{equation}\label{ewrpvoj2oip4bjoiwrgbwbb}
K:= K^{\mathrm{st}} \circ \tilde\Sp \colon \Cle \to \Sp\ .\qedhere
\end{equation}
\end{ddd}

\begin{kor}\label{ergkewgergregregergwer}
	\mbox{}\begin{enumerate}
\item \label{wetihgogergwergwegwerg} The algebraic $K$-theory functor for stable $\infty$-categories is a stable finitary localising invariant.
 \item \label{wetihgogergwergwegwerg1} The algebraic $K$-theory functor for left-exact $\infty$-categories  is a finitary localising invariant (\cref{qrevoiqrjoirqfcwqecq}).
\end{enumerate}
\end{kor}
\begin{proof}
	 We note that $\cU_{loc}(\Sp^{ {\cop}})$ is compact in $\cM_{loc}$. This implies that $K^{\mathrm{st}}$ preserves filtered colimits since $\cU_{loc}$ does so. 
	Since $\cM_{loc}$ is stable, any corepresentable functor preserves finite colimits. 
	This implies the statement \eqref{wetihgogergwergwegwerg}. The statement  \eqref{wetihgogergwergwegwerg1}
	 then follows from   \eqref{wetihgogergwergwegwerg} and \cref{lem:loc-vs-stloc} \eqref{eri9gqergrqgrgreqerggqrg}.
\end{proof}

\begin{rem}
 By \cref{lem:loc-vs-stloc}  \eqref{eri9gqergrqgrgreqerggqrg}, we have $K^{\mathrm{st}}(\bC) \simeq K(\bC)$ for every stable $\infty$-category $\bC$. If $K^{\mathrm{BGT}}$ denotes the non-connective algebraic $K$-theory functor from \cite[Def.~9.6]{MR3070515}, \cite[Thm.~9.8]{MR3070515} yields an identification $K(\bC) \simeq K^{\mathrm{BGT}}(\bC^{\op}
 )$.
\end{rem}

\begin{prop}\label{rwqioqewrfeeffqefe}
The algebraic $K$-theory functor for left-exact $\infty$-categories  (\cref{regiowergerregeggregw}) preserves set-indexed products.
\end{prop}
\begin{proof}
This result will be deduced from the corresponding statement about $K^{\mathrm{st}}$.
Let $(\bC_{i})_{i\in I}$ be a family  {of left-exact $\infty$-categories}.
	By \cite[Lem.~2.39]{Bunke:aa}\footnote{The result in the reference is stated for right-exact $\infty$-categories. We obtain the corresponding result for left-exact $\infty$-categories by considering opposites.} the canonical map \begin{equation}\label{wrkvjqweofweffqefewf}
	\cU_{loc}(\tilde \Sp(\prod_{i\in I}\bC_{i}))\to \cU_{loc}(\prod_{i\in I}\tilde \Sp(\bC_{i}))
	\end{equation}
	is an equivalence. Note that it is important here to apply $\cU_{loc}$ since the map between the arguments of $\cU_{loc}$ is not an equivalence.
	Then the canonical comparison map factors as
	\begin{align*}
	 K(\prod_{i \in I} \bC_i)\stackrel{\eqref{ewrpvoj2oip4bjoiwrgbwbb}}{ \simeq} K^{\mathrm{st}}(\tilde\Sp(\prod_{i \in I} \bC_i)) &\stackrel {\eqref{wrkvjqweofweffqefewf}}{\simeq} K^{\mathrm{st}}(\prod_{i \in I} \tilde\Sp(\bC_i)) \\ &
	\stackrel{!}{ \to} \prod_{i \in I} K^{\mathrm{st}}(\tilde\Sp(\bC_i))\stackrel {\eqref{ewrpvoj2oip4bjoiwrgbwbb}}{\simeq} \prod_{i \in I} K(\bC_i)\ ,
	\end{align*}
	where the morphism marked by $!$ is an equivalence by \cite[Theorem 1.3]{kaswin}.
 \end{proof}

 Let {$\bD$   be a small  left-exact $\infty$-category with $G$-action.}
 Applying \cref{ergoiergergergegerwgerggrerg43252}  for $\Homol=K$, we get  the functor $ K\bD_{G}$.
 
\begin{kor}\label{rgiuhreiguhgwergergrwegwergreg}
If $\bD$ is idempotent complete, then the functor $ K\bD_{G} \colon G\Orb\to \Sp$ is a hereditary CP-functor.
\end{kor}
\begin{proof}
We apply  \cref{rigjoqergergqregergqergreg} with $\Homol=K $.
The target $\Sp$ of $K$ has the required properties.
Furthermore, $K$ is homological by \cref{ergkewgergregregergwer} and preserves products by  \cref{rwqioqewrfeeffqefe}.
\end{proof}

 Let $\bC$ be in $\Fun(BG,\CL)$.
 \begin{ddd}\label{ergoierwjgowregregrwegwregwregw}
We define the coarse algebraic $K$-homology functor with coefficients in $\bC$ by
\[ K\bC\cX^{G}:=K\bV^{G,c}_{\bC} \colon G\BC\to \Sp\ .\qedhere \]
\end{ddd}

\cref{weigowgregwrgrggregwgregwreg} implies:

\begin{kor}\label{rgowekgprgrwgwrgwrgwregwerg}
 $K\bC\cX^{G}$ is an equivariant coarse homology theory  which is continuous, strong and strongly additive, and which admits transfers. \end{kor}

 \begin{ddd}\label{gfoiqerjgoergrgegwregrgregwre}
We define the functor
\[ K\bC\cX_{G}:=K \bV^{c}_{\bC,G} \colon G\BC\to \Sp\ . \qedhere \]
\end{ddd}

\cref{qerkgioerggwergergregwe} implies:

\begin{kor} \label{ergijeogergergwergergwregegw}
	$K\bC\cX_{G}$ is a strong and hyperexcisive equivariant coarse homology theory.
\end{kor}

  \subsection{Split injectivity results} \label{wreuigheiugerergegwggreg}
One of the the main goals of the present paper is to produce new examples of functors on the orbit category of $G$ to which the results about split injectivity of assembly maps from  \cite{desc} can be applied. For the sake of completeness, we list these results, all shown in \cite{desc}.
In view of \cref{rgiuhreiguhgwergergrwegwergreg}, all these results apply to the functor $K\bC_{G}$ in place of $M \colon G\Orb\to \bM$ for any $\bC$ in $\Fun(BG,\Clep)$.

Apart from the assumption on the functor $M$ being a CP-functor, the theorems in \cite{desc} contain the geometric assumption  of finite decomposition complexity.  In the following, we recall the relevant definitions from  \cite[Sec.~3.1]{Bunke:ac}.

Let $U$ be an entourage on a set $X$ and consider two subsets
  $Y,Z$. Recall the definition of the thickening from \eqref{V-thick}.
 \begin{ddd}
 $Y$ and $Z$ are $U$-disjoint if $U[Y]\cap Z=\emptyset$ and $Y\cap U[Z]=\emptyset$.
 \end{ddd}

Let $X$ be a $G$-set.
\begin{ddd}
An equivariant family of subsets of $X$ is a family of subsets $(Y_{i})_{i\in I}$ indexed by a $G$-set $I$ such that $gY_{i}=Y_{gi}$ for every $g$ in $G$ and $i$ in $I$.
\end{ddd}

Let now $X$ be  {a $G$-coarse space},
and let $(Y_{i})_{i\in I}$ be an equivariant family of subsets.
\begin{ddd}
We define the $G$-coarse space $\coprod^{\sub}_{i\in I} Y_{i}$ as follows:
\begin{enumerate}
\item The underlying $G$-set of $\coprod^{\sub}_{i\in I} Y_{i}$  is $\bigsqcup_{i\in I} Y_{i}$.
\item \label{regweroiugogregwrgwgwrgwgr}
The coarse structure on $\bigsqcup_{i\in I} Y_{i}$  is the maximal one such that the family of its
subsets $(Y_{i})_{i\in I}$ is coarsely disjoint and the canonical map $\bigsqcup_{i\in I} Y_{i}\to X$  is controlled.
\qedhere
\end{enumerate}
\end{ddd}

By definition, we have a canonical morphism $\coprod^{\sub}_{i\in I} Y_{i} \to X$  of $G$-coarse spaces.
If $(Y_{i})_{i\in I}$ and $(Y^{\prime}_{i})_{i\in I}$ are two equivariant families with $Y_{i}\subseteq Y_{i}^{\prime}$ for every $i$ in $I$, then we have a morphism $\coprod^{\sub}_{i\in I} Y_{i}\to \coprod^{\sub}_{i\in I} Y^{\prime}_{i}$ {of $G$-coarse spaces}.

Recall the notion of a coarse equivalence (\cref{rewkgowegrerfrewfwr}).
\begin{ddd}
We call the family $(Y_{i})_{i\in I}$ nice if for every  $U$ in $\cC_{X}^{G}$   the canonical morphism
\[ \coprod^{\sub}_{i\in I} Y_{i}\to \coprod^{\sub}_{i\in I} U[Y_{i}] \] is a coarse equivalence.
\end{ddd}

In the following, we will consider  classes $\cC$ of $G$-coarse spaces. 

\begin{ex}
 Let $X$ be {a $G$-coarse space}.
 A subset $Y$  {of $X$}
 is bounded if $Y\times Y\in \cC_{X}$.
A $G$-coarse space $X$ is semi-bounded if every coarse component (\cref{wetghiojweorgergwgwgegrwrg423t2}) of $X$ is bounded.
We consider the class $\cS\cB$ of all semi-bounded $G$-coarse spaces.
\end{ex}

Let $\cC$ be a class of $G$-coarse spaces, and let $X$ be  a $G$-coarse space.
\begin{ddd}
We say that $X$ is decomposable  over $\cC$ if for every $U$  in $\cC_{X}^{G}$ there exist nice equivariant families $(Y_{i})_{i\in I}$ and $(Z_{j})_{j\in J}$ of subsets of $X$ such that the following conditions are satisfied:
\begin{enumerate}
\item $(Y_{i})_{i\in I}$ and $(Z_{j})_{j\in J}$ are pairwise $U$-disjoint.
\item $X=\bigcup_{i\in I}Y_{i}\cup  \bigcup_{j\in J}Z_{j}$.
\item   $\coprod^{\sub}_{i\in I} Y_{i}$ and $\coprod^{\sub}_{j\in J} Z_{j}$ belong to the class $\cC$. \qedhere
\end{enumerate}
\end{ddd}

 Let $\cC$ be a class of $G$-coarse spaces.
\begin{ddd}
The class $\cC$ is closed under decomposition if every $G$-coarse space which is decomposable over $\cC$ belongs to $\cC$.
\end{ddd}
 
\begin{ddd}
We let $G\cF\cD\cC$ be the smallest class of $G$-coarse spaces which is closed under decomposition and contains the class $\cS\cB$ of all semi-bounded $G$-coarse spaces. 
\end{ddd}

Let $X$ be a $G$-coarse space.
\begin{ddd}
$X$ has $G$-finite decomposition complexity ($G$-FDC) if it belongs to the class $G\cF\cD\cC$.
\end{ddd}

Let $\cF$ be a set of subgroups of $G$.
\begin{ddd}\label{etghoiwrththrheh}
$\cF$ is called a family of subgroups if it is  closed under
\begin{enumerate}
\item taking subgroups and
\item conjugation in $G$. \qedhere
\end{enumerate}
\end{ddd}

Let $\cF$ be a family of subgroups.
\begin{ddd} \label{ithowthwergreggwgergwrgwreg}
 By $G_{\cF}\Orb$ we denote the full subcategory of $ {G\Orb}$ of transitive $G$-sets with stabilisers in $\cF$. 
 \end{ddd}

Let $X$ be {a $G$-coarse space}.

\begin{ddd}
	$X$ has $G_{\cF}$-FDC if the $G$-coarse space $S_{min}\otimes X$ has $G$-FDC for all  {$G$-sets $S$ with stabilisers in $\cF$}.
\end{ddd}

In the following, for an $\infty$-category $\cC$ we use the standard notation $\PSh(\cC):=\Fun(\cC^{\op},\Spc)$.
 {By Elmendorf's theorem, $\PSh(G\Orb)$ is equivalent to the $\infty$-category of $G$-topological spaces (see \cite[Rem.~1.12]{desc} for further explanations)}.
We have an adjunction
\[ \Ind_{\cF} \colon \PSh(G_{\cF}\Orb)\leftrightarrows \PSh(G\Orb) \cocolon \Res_{\cF}\ ,\]
where $\Res_{\cF}$ is the restriction along the inclusion $G_{\cF}\Orb\to G\Orb$.
We let $*_{\cF}$ denote the final object of $\PSh(G_{\cF}\Orb)$.
\begin{ddd}\label{qefiuhiqewfqwefwqefefewq}
We define the classifying space of the family $\cF$ by
\[ E_{\cF}G:=\Ind_{\cF}(*_{\cF})\ .\qedhere\]
\end{ddd} 

By definition, $E_{\cF}G$ is an object of $\PSh(G\Orb)$. The object $E_{\Fin}G$ is also called the classifying space for proper actions.

\begin{ex} 
	Examples of families of subgroups are:
	\begin{enumerate}
		\item $\{1\}$ - the family consisting of the trivial subgroup.
		\item $\Fin$ - the family of all finite subgroups.
		\item $\Vcyc$ - the family of virtually cyclic subgroups.
		\item\label{giorgjerogiergreg}   $\FDC$ - the family of subgroups $V$ of $G$ such that $V_{can}$ has $V_\Fin$-FDC.
		\item \label{giorgjerogiergreg1} $\cp$ denotes the family of subgroups of $G$ generated by those subgroups $V$ such that $E_\Fin V$ (see \cref{qefiuhiqewfqwefwqefefewq} below) is a compact object of $\PSh(V\Orb)$.
		\item   $\cFDC$ denotes  the intersection of $\FDC$ and $\cp$.\qedhere 
	\end{enumerate}
\end{ex}

\begin{ddd}\label{ergoiewrgeer243tt43terwgf}
  $E_{\cF}G$ admits a finite-dimensional model if there exists a finite-dimensional $G$-$CW$-complex modelling the homotopy type of $E_{\cF}G$ in $G$-topological spaces.
\end{ddd}
 
 Let $\bM$ be a cocomplete $\infty$-category, and let $M \colon G\Orb \to \bM$ be a functor. 
Let $\cF$ and $\cF^{\prime}$ be families of subgroups  such that $\cF^{\prime}\subseteq \cF$.
\begin{ddd}\label{ergoiegererg}
	The {relative assembly map} $\As_{\cF^{\prime},M}^{\cF}$  is the morphism 
	\[\As_{\cF^{\prime},M}^{\cF}\colon \colim_{ G_{\cF^{\prime}}\Orb}M \to \colim_{ G_\cF \Orb}M \]
	in $\bM$ canonically induced by the inclusion $G_{\cF^{\prime}}\Orb\to G_{\cF}\Orb$.
	 \end{ddd}
 
 {The following are the main results of \cite{desc}}. 
 
\begin{theorem}[{\cite[Thm.~1.11]{desc}}] 
	\label{thm:main-injectivity-kor1}Assume:
	\begin{enumerate}
	\item  $M$ is a CP-functor.
	\item  One of the following conditions holds:
		\begin{enumerate}
			\item\label{thm:main-it1} $\cF$ is a subfamily of $\cFDC$ 
			such that $\Fin\subseteq \cF$.
			\item\label{thm:main-it2} $\cF$ is a subfamily of $\FDC$
			such that $\Fin\subseteq \cF$ and $G$ admits a finite-dimensional model for $E_\Fin G$.
		\end{enumerate}\end{enumerate}		
	Then the relative assembly map $\As_{\Fin,M}^\cF$ admits a left inverse.
\end{theorem}

\begin{theorem}[{\cite[Cor.~1.13]{desc}}]
	\label{kor:vcyc-bartels}
	If $M$ is a CP-functor, then the relative assembly map $\As_{\Fin,M}^\Vcyc$ admits a left inverse.
\end{theorem}

\begin{theorem}[{\cite[Cor.~1.14]{desc}}]
	\label{kor:oldversion}
	Assume:
	\begin{enumerate}
		\item $M$ is a CP-functor;
		\item $G$ admits a finite-dimensional model for $E_\Fin G$;
		\item \label{it:gfinfdc} $G_{can}$ has $G_\Fin$-FDC.
	\end{enumerate}
	Then the assembly map $\As^{\All}_{\Fin,M}$ admits a left inverse.
\end{theorem}
 
Note that the finite decomposition complexity assumption is in general not easy to check. 
The following two theorems are consequences of the theorems on groups with finite decomposition complexity.

\begin{theorem}[{\cite[Thm.~2.9]{desc}}]
Assume:
\begin{enumerate}
\item $M$ is a hereditary $CP$-functor.
\item $G$ fits into an exact sequence $1\to S\to G\xrightarrow{\phi} Q\to1$ such that
\begin{enumerate}
\item $S$ is virtually solvable and has finite Hirsch length.
\item $Q$ admits a finite-dimensional model for $E_{\Fin}Q$.
\end{enumerate}\end{enumerate}
Then $\As_{\Fin,M}^{\phi^{-1}(\Fin(Q))}$ admits a left inverse.
\end{theorem}

Here $\Fin(Q)$ denotes the family of finite subgroups of $Q$.

 \begin{theorem}[{\cite[Cor.~2.11]{desc}}]\label{oijfoifjoewfwejfoijoiewfqwefqwefqwefwefewfffqef}
 Assume:
 \begin{enumerate}
 \item $M$ is a hereditary CP-functor.
 \item $G$ admits a finite-dimensional $CW$-model for the classifying space $E_{\Fin}G$.
 \item  \label{rgiogerervweeververvwevrvrevgergwerg} $G$ is a finitely generated subgroup of a linear group over a commutative ring with unit or of a virtually connected Lie group.
  \end{enumerate}
Then $\As_{\Fin,M}^{\All}$ admits a left inverse.
 \end{theorem}

 {\cref{rgigjwoiegerggwregwerg} from the introduction follows by combining \cref{oijfoifjoewfwejfoijoiewfqwefqwefqwefwefewfffqef} and \cref{rgiuhreiguhgwergergrwegwergreg}.}

\subsection{The universal property of the bounded derived category}\label{sec:linearK}
 For assembly maps associated to the algebraic $K$-theory of discrete group rings, it is customary to consider functors on the orbit category which arise from additive categories.
Formally, the construction of these functors is completely analogous to the functors $G\Orb \to \Cle$ we associate to a left-exact $\infty$-category with $G$-action, see \cref{reoifjuoiewccwec}. In the present section we will show that the additive case can be also captured  in a technical sense
as announced in \cref{gwegljrgkrewgergreg}.

Let $\bA$ be a small additive (ordinary) category.
By $\homot\bA$ we denote the localisation of
the category of bounded chain complexes $\Ch^b(\bA)$ by chain homotopy equivalences.
Then $\homot\bA$ is a stable $\infty$-category, equipped with a canonical finite coproduct-preserving functor
\begin{equation}\label{eq:canonical embedding in HA}
z_{\bA} \colon \bA\to\homot\bA
\end{equation}
sending an object of $\bA$ to the corresponding chain complex concentrated in degree zero.
The purpose of this subsection is to prove that the functor $z_{\bA}$ above is the
universal finite coproduct-preserving functor from $\bA$ to a stable $\infty$-category (\cref{thm:universal property of additive bounded derived category}).
 {As an aside, we refine this statement to provide a universal property for the bounded derived $\infty$-category of any exact category in \cref{cor:universal property of the derived category}}.
As a consequence, we show  {in \cref{thm:cocontinuity of the derived category}} that the functor $\homot{-}$ commutes with colimits indexed by small groupoids.
 This allows us to deduce in \cref{eiohjgwergerqgergergdfbsd} that the two candidates for functors on the orbit category associated to an additive category coincide.

Let $\bC$ be an $\infty$-category.
\begin{ddd}\label{fqehwfiuhiqwef}
$\bC$ is  additive  if it is semi-additive (see \cref{ghiowjgoergwergwergregwg}) and its homotopy category $\ho(\bC)$ is additive.
\end{ddd}

Given a locally small additive $\infty$-category $\bC$, we let $\cP_\Sigma(\bC)$ be the $\infty$-category of
additive presheaves on $\bC$, i.e., the $\infty$-category of  finite product-preserving functors from $\bC^\op$ to the $\infty$-category of
spaces $\Spc$. Given two additive $\infty$-categories $\bC$ and $\bD$, we denote by
$\Fun_\Sigma(\bC,\bD)$ the full subcategory of $\Fun(\bC,\bD)$ spanned
by finite coproduct-preserving functors.

Given a right-exact $\infty$-category $\bC$, we denote by
\[ \SW(\bC) := \tilde{\Sp}(\bC^\op)^\op \]
(see \cref{foiwregergegregreg}) the Spanier-Whitehead stabilisation
of $\bC$; see also \cite[Constr.~C.1.1.1]{SAG}.

Let $\bC$ be an $\infty$-category.
\begin{ddd}[{\cite[Def.~C.1.2.1]{SAG}}]
$\bC$ is called prestable if it satisfies the following conditions:
\begin{enumerate}
\item $\bC$ is right-exact.
\item The suspension functor $\Sigma \colon \bC\to \bC$ is fully faithful.
\item $\bC$ is closed under extensions: For every morphism $f \colon Y\to \Sigma Z$ in $\bC$ there exists a pullback square
\[\xymatrix{X\ar[r]^-{f'}\ar[d]&Y\ar[d]^-{f}\\0\ar[r]&\Sigma Z}\ .\]
Furthermore, this square is also a pushout.\qedhere
\end{enumerate}
\end{ddd}

Let $\bA$ be a small additive category.

\begin{prop}\label{chaincomplexesasadditivefunctors}\mbox{}
\begin{enumerate}
\item \label{qiogrfjoregergwerger} $\cP_\Sigma(\bA)$ is prestable. 
\item   $\cP_\Sigma(\bA)$ is additive.
\item  \label{ewiofjqofqwefqewfqewf}The canonical functor $\cP_\Sigma(\bA)\to\SW(\cP_\Sigma(\bA))$
is fully faithful.
\item The essential image of the functor from \eqref{ewiofjqofqwefqewfqewf} is closed under extensions.
\end{enumerate}
\end{prop}
 \begin{proof}
Assertion \eqref{qiogrfjoregergwerger} follows from  \cite[Prop.~C.1.5.7]{SAG}. The remaining assertions are then given by \cite[Ex.~C.1.5.6]{SAG} and \cite[Prop.~C.1.2.2]{SAG}. 
\end{proof}

Let $\cP_{\Sigma,f}(\bA)$ be the smallest prestable subcategory
of $\cP_\Sigma(\bA)$ containing the representable presheaves.
\begin{prop}\label{prop:bounded derived cat as functors} 
There is a canonical equivalence of $\infty$-categories
\[ \homot\bA \xrightarrow{\simeq} \SW(\cP_{\Sigma,f}(\bA)) \]
fitting into the 
 commutative diagram
  \begin{equation}\label{4hgou42oigj34gf234g34f234f}
\xymatrix{\bA\ar[r]^-{z_{\bA}}_-{\eqref{eq:canonical embedding in HA}}\ar[d]_-{\mathrm{Yoneda}}&\homot\bA   \ar[d]^{\simeq}\\\cP_{\Sigma,f}(\bA) \ar[r]&\SW(\cP_{\Sigma,f}(\bA))}
\end{equation}
\end{prop}
\begin{proof}
 Let $\cA := \Fun_\Sigma(\bA^\op,\Ab)$ be the abelian category of additive presheaves on $\bA$ with values in abelian groups. This is an abelian category with a set of compact projective generators provided by the image of the Yoneda embedding $y_\bA \colon \bA \to \cA$. Note that any projective in $\cA$ is a retract of a sum of representables.
 
 Let $\cD^-(\cA) := \Nerve_{\mathrm{dg}}(\Ch^-(\cA_{\mathrm{proj}}))$ be the bounded below derived category of $\cA$, where $\cA_{\mathrm{proj}}$ denotes the full subcategory of $\cA$ spanned by the projective objects. Note that $y_\bA$ induces a fully faithful functor $\Ch^b(\bA) \to \Ch^-(\cA_{\mathrm{proj}})$ of dg-categories. In view of the equivalence $\Ch^b(\bA)_\infty \simeq \Nerve_{\mathrm{dg}}(\Ch^b(\bA))$ (see \cite[Rem.~2.9]{Bunke:2017aa}), we get a fully faithful functor
 \begin{equation}\label{eq:chains-derived-inc}
  \Ch^b(\bA)_\infty \to \cD^-(\cA)\ .
 \end{equation}
 
 By \cite[Lem.~2.60]{Pstragowski}, there is an equivalence $\cD^-(\cA)_{\geq 0} \xrightarrow{\simeq} \cP_\Sigma(\bA)$\footnote{\cite{Pstragowski} formulates the result for the unbounded derived category $\cD(\cA)$, but note that the connective parts of $\cD^-(\cA)$ and $\cD(\cA)$ agree.} which fits into the following commutative diagram:
 \[\xymatrix{
  \cD^-(\cA)_{\geq 0} \ar[r]^-{\simeq} & \cP_\Sigma(\bA) \\
  \cA_{\mathrm{proj}}\ar[u] & \bA\ar[u]_-{\mathrm{Yoneda}}\ar[l]_-{y_\bA} \\
 }\]
 Define $\Ch^b_{\geq 0}(\bA)_\infty$ as the smallest prestable subcategory of $\Ch^b(\bA)_\infty$ containing the essential image of $z_\bA$. Then \eqref{eq:chains-derived-inc} induces a fully faithful functor $\Ch^b_{\geq 0}(\bA)_\infty \to \cP_\Sigma(\bA)$ whose essential image is by definition $\cP_{\Sigma,f}(\bA)$, and which fits into the following commutative diagram:
 \begin{equation}\label{qrviqjvoervqvevwevv}
\xymatrix{
  \Ch^b_{\geq 0}(\bA)_\infty \ar[r]^-{\simeq} & \cP_{\Sigma,f}(\bA) \\
  \bA\ar[u]^-{z_\bA}\ar[ru]_-{~\mathrm{Yoneda}} & \\
 }
\end{equation} 
 Since $\Ch^b(\bA)_\infty \simeq \SW(\Ch^b_{\geq 0}(\bA)_\infty)$, the proposition follows by applying $\SW$ to the upper horizontal equivalence in \eqref{qrviqjvoervqvevwevv}.
\end{proof}

 {Let $\Cadd$ denote the $\infty$-category of additive $\infty$-categories and finite coproduct-preserving functors, so $\Add_\infty$ is a full subcategory of $\Cadd$.
The functor $\homot{-} \colon \Add\to  {\stCat}$ preserves equivalences of additive categories and therefore induces a functor  (by abuse of notation denoted by the same symbol) 
\[  \Ch^{b}(-)_{\infty} \colon \Add_{\infty}\to  {\stCat}\ . \]
\begin{theorem}\label{thm:universal property of additive bounded derived category}
The functor
\[ \SW(\cP_{\Sigma,f}(-)) \colon \Cadd \to \stCat \]
is left adjoint to the forgetful functor $\stCat \to \Cadd$.
Moreover, the restriction of this functor to $\Add_\infty$ is equivalent to $\homot-$.
\end{theorem}
\begin{proof}
Let $\bB$ be an additive $\infty$-category, and let $\bC$ be a stable $\infty$-category.
Denoting by $\Fun'(\cP_{\Sigma}(\bB),\Ind_\omega(\bC))$ the $\infty$-category of small colimit-preserving functors which in addition send representable presheaves to the essential image of the canonical functor $\bC \to \Ind_\omega(\bC)$, we have the following sequence of natural equivalences:
\begin{eqnarray*}
 \Fun_{{\Cex}}(\SW(\cP_{\Sigma,f}(\bB)),\bC)
&\simeq&
 \Fun_{{\Cre}}(\cP_{\Sigma,f}(\bB),\bC)\\
 &\simeq&
  \Fun'(\cP_{\Sigma}(\bB),\Ind_{{\omega}}(\bC))\\
  &\simeq&
  \Fun_{\Sigma}(\bB,\bC)
\end{eqnarray*}
The first equivalence holds since $\bC$ is stable, and the second equivalence is clear, and the third equivalence follows from the universal property of $\cP_\Sigma(\bB)$ \cite[Prop.~5.5.8.15]{htt}.
This proves the assertion about the left adjoint, and the assertion about the restriction to $\Add_\infty$ is precisely \cref{prop:bounded derived cat as functors}.
\end{proof}}

\begin{rem}
One can prove \cref{prop:bounded derived cat as functors}
(hence also \cref{thm:universal property of additive bounded derived category})
by using directly the version of the Dold--Kan correspondence provided by
\cite[Prop.~1.3.2.23]{HA} and by adapting the proof of
\cite[Prop.~1.3.2.22]{HA} in the case where $\bA$ is an idempotent-complete
additive category, and
then by embedding $\bA$ into its idempotent completion to reach the general case.\end{rem}

Although this will not be needed in these notes,
we add the following natural consequences of
\cref{thm:universal property of additive bounded derived category},
since this gives a natural way to see bounded derived
categories which does not seem to be documented in the
published literature.

For a small exact category $\bE$, we recall that the derived category of $\cD^b(\bE)$
is the quotient of $\homot\bE$ by the thick subcategory of bounded acyclic complexes
with coefficients in $\bE$. In the case $\bA$ is a small abelian category,
seen as an exact category for which the admissible short exact sequences are all
short exact sequences, the stable $\infty$-category $\cD^b(\bA)$ simply is
the localisation of the category of bounded chain complexes of $\bA$ by quasi-isomorphisms
(so that the induced $1$-category is the usual bounded derived category).

Let $\bE$ be a small exact category and let $f\colon \homot\bE\to\bC$
be an exact functor with values in a stable $\infty$-category.

\begin{prop}\label{prop:acyclic complexes versus short exact sequences}
We assume that, for any admissible short exact sequence
\[ 0\to x\to y\to z\to 0 \]
in $\bE$, the induced commutative square
\[\xymatrix{
f(x)\ar[r]\ar[d]& f(y)\ar[d]\\
0\ar[r]& f(z)
}\]
is cocartesian. Then, for any acyclic $k$ in $\Ch^{b}(\bE)$ we have 
 $f(k)\simeq 0$.
\end{prop}
\begin{proof}
We first prove that, for any admissible short exact sequence
of bounded chain complexes
\[0\to k'\to k\to k''\to 0\ ,\]
the induced square
\[\xymatrix{
f(k')\ar[r]\ar[d]& f(k)\ar[d]\\
0\ar[r]& f(k'')
}\]
is cocartesian in $\bC$.
We do this by induction on the amplitude $N$ of $k$
(i.e., the biggest integer $N$ so that there are
integers $a\leq b$ with $b-a\geq N$ and the components of $k$
in degrees $a$ and $b$ are non-zero).
If  $N=0$,
this holds by assumption. If  {$N \geq 1$},
using ``troncation b\^ete'',
we see that the admissible short exact sequence above fits
in a homotopy cofibre sequence in $\homot\bE$ (written vertically below)
of admissible short exact sequences
\[\xymatrix{
0\ar[r]&k'_1\ar[r]\ar[d]&k_1\ar[r]\ar[d]&k''_1\ar[r]\ar[d]&0\\
0\ar[r]&k'\ar[r]\ar[d]&k\ar[r]\ar[d]&k''\ar[r]\ar[d]&0\\
0\ar[r]&k'_2\ar[r]&k_2\ar[r]&k''_2\ar[r]&0
}\]
where $k_i$ has amplitude $<N$ for $i=1,2$, and we conclude by induction.

Now, if $k$ is a bounded acyclic complex of amplitude $N$,
then there is an admissible short exact sequence of the form
\[ 0\to k'\to k\to k''\to 0\ ,\]
where both $k'$ and $k''$ are acyclic, $k'$ is of amplitude $1$
(i.e., $k'$ is a mapping cone of an isomorphism of objects of $\bE$),
and $k''$ is of amplitude $<N$; {see
\cite[diagram (1.11.7.6)]{TT}, for instance.} It is clear that $k'\simeq 0$ in $\homot\bE$,
and therefore, that $f(k)\simeq f(k'')$ in $\bC$.
Therefore, by induction on $N$,
we see that $f(k)\simeq 0$.
\end{proof}

Let $\bE$ be a small exact category.
\begin{kor}\label{cor:universal property of the derived category}
 For any stable $\infty$-category $\bC$,
composing with the canonical functor $\bE\to\cD^b(\bE)$ induces an equivalence
of $\infty$-categories from
the $\infty$-category of exact functors $F \colon \cD^b(\bE)\to\bC$ to
the $\infty$-category of finite coproduct-preserving functors
$f \colon \bE\to\bC$ which send each admissible short exact sequence
\[ 0\to x\to y\to z\to 0 \]
in $\bE$ to a cocartesian square
\[\xymatrix{
f(x)\ar[r]\ar[d]& f(y)\ar[d]\\
0\ar[r]& f(z)
}\]
in $\bC$.
\end{kor}
\begin{proof}
We observe that, a functor $f \colon \homot\bE\to\bC$ sends each admissible short exact sequence
\[ 0\to x\to y\to z\to 0 \]
in $\bE$ to a cocartesian square
\[\xymatrix{
f(x)\ar[r]\ar[d]& f(y)\ar[d]\\
0\ar[r]& f(z)
}\]
if and only if it factors through $\cD^b(\bE)$: this follows right
away from the definition of $\cD^b(\bE)$ and
from \cref{prop:acyclic complexes versus short exact sequences}.
This corollary is then a direct consequence of
\cref{thm:universal property of additive bounded derived category}.
\end{proof}

\begin{rem}
\cref{cor:universal property of the derived category} explains how to compare a triangulated category equipped with
a $t$-structure with the bounded derived category of its heart:
if $\bC$ is a 
stable $\infty$-category with a $t$-structure $(\bC_{\geq 0},\bC_{\leq 0})$
whose heart is $\bA=(\bC_{\geq 0}\cap\bC_{\leq 0})$, then there is a unique
exact functor $\cD^b(\bA)\to\bC$ extending the inclusion $\bA\subset\bC$.
\end{rem}

\begin{theorem}\label{thm:cocontinuity of the derived category}
The functor 
$  \Ch^{b}(-)_{\infty}\colon \Add_\infty \to {\stCat}$
preserves colimits indexed by small groupoids.
\end{theorem}
\begin{proof}
  
 By \cref{thm:universal property of additive bounded derived category}, $\homot{-}$ factors as
 \[ \Add_\infty \xrightarrow{\incl} \Cadd \xrightarrow{\SW(\cP_{\Sigma,f}(-))} \stCat, \]
 and the second functor preserves colimits since it is a left adjoint by \cref{thm:universal property of additive bounded derived category}.
 Hence it suffices to show that $ {\incl}$
 preserves colimits indexed by small groupoids. 
 For this, it is enough to consider colimits indexed by discrete categories and groupoids of the form $BG$ for a group $G$ separately. The case of discrete index categories
 is easy and left to the reader.
 In the following, we give the details for the second case.\footnote{In the application \cref{eiohjgwergerqgergergdfbsd} below, we only need this second case.}.
 
  Let $\bA \colon BG \to \Cadd$ be a functor and assume that it takes values in $\Add_\infty$. Since $\incl$ is fully faithful, it suffices to show that $\colim_{BG}\bA$ again belongs to the full subcategory $\Add_{\infty}$.
  
 {Let $\ev \colon \Fun(BG,\Prl) \to \Prl$ be the evaluation functor and denote by $\alpha \colon \bA \to \underline{\colim_{BG} \bA}$ the canonical morphism.
 Consider the commutative square 
 \[\xymatrix{\ev(\bA)\ar[r]^-{\ev(\alpha)}\ar[d] &\colim_{BG} \bA \ar[d] \\ \ev(\cP_\Sigma(\bA)) \ar[r]^-{\ev(\alpha_{!})} &\cP_\Sigma( \colim_{BG} \bA)} \]
 whose vertical arrows are given by the Yoneda embedding.
 Since $\ev(\alpha)$ is essentially surjective, it suffices to show that the essential image of the composition along the bottom left corner is an ordinary category.}
 
 {The canonical morphism $\eta \colon \underline{\lim_{BG} \cP_\Sigma(\bA)} \to \bA$ fits} into a commutative diagram
 \[\xymatrix@C=3em@R=1em{
 \underline{\lim\limits_{BG}\cP_{\Sigma}(\bA)}\ar[dr]^-{\eta} &  \\
  & \cP_{\Sigma}(\bA) \\
 \cP_{\Sigma}(\underline{\colim\limits_{BG}\bA})\ar[uu]^-{\simeq} \ar[ur]_-{ \alpha^{*}}& }\ ,\]
 where the limit is taken in $\Prl$.
 Considering $G$-sets as $\infty$-categories with $G$-action, the equivariant projection map $p \colon G \to *$ induces a functor
  \[ p^{*} \colon \cP_\Sigma(\bA)\simeq \Fun(*,\cP_\Sigma(\bA))\to \Fun(G,\cP_\Sigma(\bA))\ . \]
  We define $p^{*,G}:=\lim_{BG}p^{*}$ and get the commutative  diagram
  \[\xymatrix@C=2.5em{
    \lim\limits_{BG}\cP_\Sigma(\bA)\ar[r]^-{p^{*,G}}\ar[d]_-{\ev(\eta)} & \lim\limits_{BG}\Fun(G,\cP_\Sigma(\bA))\ar[d]^-{\ev(\eta_G)}\ar@/^3cm/[dd]^{\simeq} \\
    \ev(\cP_\Sigma(\bA))\ar@{=}[dr]\ar[r]^-{\ev(p^{*})}&\ev(\Fun(G,\cP_\Sigma(\bA)))\ar[d]^-{i_{e}^{*}}\\&\ev(\cP_\Sigma(\bA))}\]
where $i_{e} \colon * \to G$ is the inclusion of the identity, the left square commutes by the naturality of $\eta$, and the lower triangle commutes since $i_{e}\circ \ev(p)=\id_{*}$.
The right vertical composition is an equivalence (for example by \cref{cor:coind} whose proof works for every $\infty$-category).
  This identifies $\ev(\eta)$, and hence $\ev(\alpha^*)$, with the functor $p^{*,G}$.
  
  {Since $\cP_\Sigma(\bA)$ is cocomplete},  
  the left Kan extension functor $p_!$ along $p$ exists and provides a left adjoint to $p^*$.
  The functor
  \[ p_!^G := \lim_{BG}p_{!} \colon \lim\limits_{BG}\Fun(G,\cP_\Sigma(\bA))\to \lim\limits_{BG}\cP_{\Sigma}(\bA) \]
  is then a left adjoint of $p^{*,G}$, {which is equivalent to $\ev(\alpha_!)$ by the preceding discussion}.
  
  Since $G$ is discrete, applying $p_{!}$ amounts to forming coproducts. {In particular, the essential image of the composite
  \begin{equation}\label{eq:add-colim} \Fun(G,\bA) \to \Fun(G,\cP_\Sigma(\bA)) \xrightarrow{p_!} \cP_\Sigma(\bA) \end{equation}
  is an ordinary category.
  We now apply $\lim\limits_{BG}$ to this composite.
  \cref{tieqorgfgregwegergw}~\eqref{reigowergwrgrgwgrg} implies that the essential image of the resulting functor is given by applying $\lim\limits_{BG}$ to the essential image of \eqref{eq:add-colim}.
  Since limits of ordinary categories are also ordinary categories, it follows that the essential image of the composite
  \[ \ev(\bA) \to \ev(\cP_\Sigma(\bA) \xrightarrow{\ev(\alpha_!)} \cP_\Sigma(\colim_{BG} \bA) \]
  is an ordinary category as well.}
\end{proof}

Recall that  {$j_!$} denotes the left Kan extension functor along the inclusion $ {j} \colon BG\to G\Orb$ (see \eqref{rewflkjmo34gergwegrge}).
Let $\bA$ be in $\Fun(BG,\Add_{\infty})$.

\begin{kor}\label{eiohjgwergerqgergergdfbsd}
The canonical transformation \begin{equation}\label{sadvojoiasdvasvadsv}
 {j_!}(\homot\bA)\to \homot{ {j_!}\bA)}
\end{equation}
is an equivalence.
\end{kor}
\begin{proof}
For a subgroup $H$ of $G$ we have an equivalence $  BH \xrightarrow{\simeq} G\Orb_{/BH} $ which sends $*_{BH}$ to $G\to G/H$ in $G\Orb_{/BH} $ and $h$ in $H\cong \Aut_{BH}(*_{BH})$ to the right multiplication by $h$ on $G$. 
The evaluation of the transformation \eqref{sadvojoiasdvasvadsv} at $G/H$ becomes equivalent to the morphism
$\colim_{BH} \homot\bA\to \homot{\colim_{BH}\bA}$ which is an equivalence by \cref{thm:cocontinuity of the derived category}.
\end{proof}

\begin{kor}\label{eiohjgwergerqgergergdfbsdd}
 {There is an equivalence
  \[ K\Ch^{b}(\bA)_{\infty,G} \simeq K\bA_{G}\ . \]}
\end{kor}
\begin{proof}
 For any additive category $\bB$,
 the inclusion $z_{\bB} \colon \bB  \to \Ch^{b}(\bB)_{\infty}$ from  \eqref{eq:canonical embedding in HA} induces an equivalence $K(\bB) \xrightarrow{\simeq} K(\Ch^{b}(\bB)_{\infty})$ by the Gillet--Waldhausen theorem. 
Hence we have the chain of equivalences
 \[ K\Ch^{b}(\bA)_{\infty,G}\stackrel{\eqref{fwqewwedeqdwedqwdwdedqwdwedwd}}{\simeq} Kj_! \Ch^{b}(\bA)_{\infty}\stackrel{\eqref{sadvojoiasdvasvadsv}}{\simeq}
 K  \Ch^{b}(j_{!}\bA)_{\infty} \simeq K j_{!}\bA_{\infty}\stackrel{\eqref{sdavjoiwqejffvvsdvasdvadsvadsvasdvdsv}}{\simeq}  K\bA_{G} \ .\qedhere\]
\end{proof}

\subsection{\texorpdfstring{$A$-theory}{A-theory} as a \texorpdfstring{$G\Orb$-spectrum}{GOrb-spectrum}}\label{sec:A-coeffs}

In \cref{ex:spacesop}, we claimed that the functor $\bA_P \colon G\Orb \to \Sp$ from \cite{Bunke:aa} admits an equivalent description in terms of the $G\Orb$-spectra considered in this article. This section supplies a proof of this claim, see \cref{cor:A-coeffs} below.

Let us first recall the construction of the functor $\bA_P$ (see also \cite[Sec.~5.1]{Bunke:aa}). Associated to any topological space $Q$, there is the Waldhausen category $\bR(Q)$ of retractive spaces over $Q$, by which we mean the category of CW-complexes relative $Q$ which are equipped with a retraction to $Q$, and all cellular maps over and under $Q$. Its subcategory $h\bR(Q)$ of weak equivalences is given by those morphisms which are homotopy equivalences under $Q$. The assignment $Q \mapsto \bR(Q)$ defines a functor from topological spaces to Waldhausen categories via cobase change: Given a continuous map $f \colon Q \to Q'$, the induced functor $\bR(f)$ sends a retractive space $Q \leftrightarrows X$ to the retractive space $Q' \leftrightarrows X \cup_{Q} Q'$ determined by the pushout along $f$. If we require the base space to be an actual subspace of every retractive space, there is a strictly functorial choice of this construction.

Denote by $\bR_{\mathrm{f}}(Q)$ the full subcategory of finite retractive spaces over $Q$ and by $\bR_{\mathrm{fd}}(Q)$ the full subcategory of finitely dominated retractive spaces. Both of these define full subfunctors of $\bR$.

The localisation $\bR_{\mathrm{f}}(Q)[h^{-1}] := \bR_{\mathrm{f}}(Q)[h\bR_{\mathrm{f}}(Q)^{-1}]$ at the subcategory of weak equivalences defines a functor
\[ \bR_{\mathrm{f}}(-)[h^{-1}] \colon \Top \to \Cre \]
to the category of right-exact $\infty$-categories. In order to see that this functor takes values in $\Cre$   we apply \cite[Prop.~7.5.6]{Cisinski:2017} to the categories $\bR_{\mathrm{f}}(Q)^{\op}$. Here we use that  
 the opposite of a Waldhausen category whose weak equivalences satisfy the two-out-of-three axiom is naturally an $\infty$-category with fibrations and weak equivalences in the sense of \cite[Def.~7.4.12]{Cisinski:2017}. 
The same is true for the analogous localisation $\bR_{\mathrm{fd}}(-)[h^{-1}]:\Top\to \Cre$.
Applying the algebraic $K$-theory functor \eqref{ewrpvoj2oip4bjoiwrgbwbb} 
to the opposite $\infty$-category gives rise to the $A$-theory functor
\begin{equation}\label{qefvjekvefvfevsd}  {\bA \colon \Top \xrightarrow{\bR_{\mathrm{f}}(-)[h^{-1}]} \Cre \xrightarrow{(-)^\op} \Cle \xrightarrow{K} \Sp}  
\end{equation}
considered in \cite{Bunke:aa}. If $P$ is a principal $G$-bundle for some discrete group $G$, define
\[ \bA_P \colon G\Orb \to \Sp,\quad S \mapsto \bA(P \times_G S)\ .\]

 Since $\Cre$ is cocomplete, every right-exact $\infty$-category $\bC$ determines a colimit-preserving functor
\[ - \otimes \bC \colon \Spc \to \Cre\ .\]

\begin{prop}\label{prop:ret-and-tensor}
	 There exists an equivalence
	\[ \bR_{\mathrm{fd}}(-)[h^{-1}] \simeq - \otimes \Spc^\cop_*\ . \]
\end{prop}

Let $\ell \colon \Top \to \Spc$ denote the canonical functor. 

\begin{kor}\label{cor:A-and-tensor}
	There exists an equivalence of functors $\Top\to \Sp$
	\begin{equation}\label{eoiobewvfbsbsbb} \bA(-) \simeq K((\ell(-) \otimes \Spc^\cop_*)^\op)\ .
\end{equation}
\end{kor}
\begin{proof}
 By the cofinality theorem \cite[Thm.~2.30]{Bunke:aa} (see e.g.~\cite[Constr.~4.13]{Bunke:aa} for the ``mapping cylinder argument''),
 we may replace $\bR_{\mathrm{f}}(-)[h^{-1}]$ by the functor $\bR_{\mathrm{fd}}(-)[h^{-1}]$ in the definition \eqref{qefvjekvefvfevsd} of $\bA$.
 Consequently, the corollary is an immediate consequence of \cref{prop:ret-and-tensor}.
\end{proof}

 Let $P$ be a principal $G$-bundle. Recall the functor $K\bC_G \colon G\Orb \to \Sp$ associated to any left-exact $\infty$-category $\bC$ with $G$-action (see \cref{ergoiergergergegerwgerggrerg43252}).
Note that $\ell(P)$ defines an object in $\Fun(BG,\Spc)$.  We set
\begin{equation}\label{trbebertoijboireberb}
 \bC_{P}:= (\ell(P) \otimes \Spc^\cop_*)^\op
\end{equation}  in $\Fun(BG,\Cle)$.
\begin{kor}\label{cor:A-coeffs}
 There is an equivalence
 \[ \bA_P \simeq K\bC_{P,G}\ .\]
\end{kor}
\begin{proof}
 Let $j \colon BG \to G\Orb$ be the inclusion functor (see e.g.~\eqref{rewflkjmo34gergwegrge}), and let
 \[ o_P := \ell(P \times_G -) \colon G\Orb \to \Spc\ .\]
  We claim that \begin{equation}\label{wregreiojgw0reogerfwrff}j_!\ell(P) \simeq o_P\ ,
\end{equation}
 where $j_!$ denotes the left Kan extension functor.
 By the pointwise characterisation of left Kan extensions, it suffices to check that the canonical map 
 \[ \colim_{(G \to S) \in BG_{/S}} \ell(P) \to \ell(P \times_G S) \]
 is an equivalence for {every transitive $G$-set} $S$. 
 Any choice of base point $s$ in $S$ induces an equivalence $B(G_s) \simeq BG_{/S}$
 (compare \cref{gikowrteglregergregwerg}) and an equivalence $\ell(P \times_G S) \simeq \ell(P/G_s)$.
 {The restriction of the diagram $BG_{/S} \to BG \xrightarrow{\ell(P)} \Spc$ along the equivalence $B(G_s) \simeq BG_{/S}$ is equivalent to the composite $B(G_s) \to BG \xrightarrow{\ell(P)} \Spc$. Hence}
 it suffices to check that
 \[ \colim_{B(G_s)} \ell(P) \to \ell(P/G_s) \]
 is an equivalence,
 {which follows from the fact that $P$ is a principal $G$-bundle.}
 
	 Since $- \otimes \Spc_*^\cop$ is colimit-preserving, we have
 \begin{equation}\label{wvervpowekvpowececsf}
  j_!\ell(P) \otimes \Spc^\cop_*\simeq j_! \bC_{P}^{\op} \simeq (j_{!}\bC_{P})^{\op}  \ .
\end{equation}
The assertion of the corollary now follows from the following chain of equivalences:
 \begin{align*}
  \bA_P &\stackrel{\eqref{eoiobewvfbsbsbb}}{\simeq} K \circ (o_P(-) \otimes \Spc^\cop_*)^\op \stackrel{\eqref{wregreiojgw0reogerfwrff}}{\simeq} K \circ (j_!\ell(P) \otimes\Spc^\cop_*)^\op \\
  &\stackrel{\eqref{wvervpowekvpowececsf}}{\simeq} K\circ j_{!}\bC_{P,G}\stackrel{\eqref{ewbjnewoifvmlkrevcewvrvw}}{\simeq} K\bC_{P,G} \ .\qedhere\end{align*}
\end{proof}
	
\begin{proof}[Proof of \cref{prop:ret-and-tensor}]
	Let $(\Top_{/Q})_*$ denote the category of pointed topological spaces over $Q$, equipped with the
	model structure transferred from $\Top$ via the forgetful functor $(\Top_{/Q})_* \to \Top$.
	Denote the class of weak equivalences in $(\Top_{/Q})_*$ by $W$.
	Note that $\bR(Q)$ is canonically a subcategory of the full subcategory $(\Top_{/Q})_*^c$ of cofibrant objects in $(\Top_{/Q})_*$.
	The inclusion functor induces an equivalence 
	$\bR(Q)[h\bR(Q)^{-1}] \xrightarrow{\simeq} (\Top_{/Q})_*^c[W^{-1}]$ 
	since there exists a functorial cofibrant replacement in $(\Top_{/Q})_*$ which takes values in $\bR(Q)$.
	 
	This equivalence restricts to an equivalence 
	\begin{equation}\label{eq:ret-and-top-over-Q}
	 \bR_{\mathrm{fd}}(Q)[h^{-1}] \xrightarrow{\simeq} (\Top_{/Q})_*^c[W^{-1}]^ {\cop}
	\end{equation}
	of right-exact $\infty$-categories: Every finitely dominated retractive space is compact as an object in $(\Top_{/Q})_*^c[W^{-1}]$. Conversely, consider a retractive space $B \leftrightarrows X$ which is compact in $(\Top_{/Q})_*^c[W^{-1}]$. Since every retractive space is the (homotopy) colimit of its finite subcomplexes, compactness implies that the identity on $X$ factors (up to homotopy) through a finite subcomplex of $X$. Hence $B \leftrightarrows X$ is finitely dominated.
		
	Consider $(\sSet_{/\Sing(Q)})_*$ with the projective model structure. Taking singular complexes induces an equivalence
	\begin{equation}\label{eq:top-over-Q-and-sset-over-Q}
	 (\Top_{/Q})_*^c[W^{-1}] \simeq (\sSet_{/\Sing(Q)})_*[W^{-1}]
	\end{equation}
	since all objects in $(\sSet_{/\Sing(Q)})_*$ are cofibrant.
	Note that the equivalences from \eqref{eq:ret-and-top-over-Q} and \eqref{eq:top-over-Q-and-sset-over-Q} are natural under cobase change. Hence we obtain an equivalence
	\begin{equation}\label{ret-and-sset-over-Q}
	 \bR_{\mathrm{fd}}(-)[h^{-1}] \simeq (\sSet_{/\Sing(-)})_*[W^{-1}]^{\cop}
	 \end{equation}
	of functors $\Top \to \Cre$.
	
	Consider now the functor $(\sSet_{/\Sing(-)})_*[W^{-1}]$ as a contravariant functor on topological spaces (via pullback). As in the case of $\Top_{/Q}$, equip $(\sSet_{/\Sing(Q)})_*$ with the model structure transferred from the Quillen model structure on $\sSet$ via the forgetful functor.
	Any fibrant replacement functor on $(\sSet_{/\Sing(Q)})_*$ induces a natural equivalence
	\begin{equation}\label{eq:sset-over-Q-loc-cf}
	 (\sSet_{/\Sing(Q)})_*[W^{-1}] \simeq N((\sSet_{/\Sing(Q)}))_*^{\mathrm{cf}})
	\end{equation}
	by \cite[Prop.~1.3.4.7]{HA}, where $N$ denotes the homotopy coherent nerve. Note that the cofibrant-fibrant objects in $(\sSet_{/\Sing(Q)})_*$ are precisely the Kan fibrations.
	Let $(\sSet^{\mathrm{cart}}_{/\Sing(Q)} )_*$ denote the category $(\sSet_{/\Sing(Q)})_*$ equipped with the cartesian model structure
	(see \cite[Rem.~2.1.4.12]{htt}, where it is called the contravariant model structure). This category is also contravariantly functorial via pullback. 	Since the right fibrations are precisely the cofibrant-fibrant objects in the cartesian model structure (\cite[Cor.~2.2.3.12]{htt}), and since every right fibration over a Kan complex is a Kan fibration (\cite[Lem.~2.1.3.3]{htt}), we have  
	\begin{equation}\label{eq:sset-and-spc-over-Q}
	  N((\sSet_{/\Sing(-)}^{\mathrm{cart}})_*^{\mathrm{cf}}) \simeq N((\sSet_{/\Sing(-)})_*^{\mathrm{cf}}) \simeq (\Spc_{/\ell(-)})_*\ .
	\end{equation}
	By \cite[Cor.~A.32]{GHN:2017}, there exists a functor
	\[ \Cat_\infty^\op \to \Fun(\Delta^1, \CAT) \]
	whose value at $\cC$ is given by the unstraightening equivalence $\Fun(\cC^\op,\Cati) \xrightarrow{\simeq} \Cat^{\mathrm{cart}}_{\infty/\cC}$.
	Note that the unstraightening equivalence restricts to an equivalence of full subcategories
	\[ \Fun(\cC^\op, \Spc) \xrightarrow{\simeq} \Cat^{\mathrm{rfib}}_{\infty/\cC}\ ,\]
	where $\Cat^{\mathrm{rfib}}_{\infty/\cC}$ denotes the full subcategory of cartesian fibrations whose fibres are objects in $\Spc$. Since $\Cat^{\mathrm{rfib}}_{\infty/\cC} \simeq \Spc_{/\cC}$ when $\cC$ is in $\Spc$, there is an induced natural equivalence
	\begin{equation}\label{eq:spc-over-Q-and-fun}
	 \Fun((-)^\op,\Spc_*) \simeq \Fun((-)^\op, \Spc)_* \simeq (\Spc_{/-})_*
	\end{equation}
	of functors $\Spc \to \Prr$. From \eqref{eq:sset-over-Q-loc-cf}, \eqref{eq:sset-and-spc-over-Q} and \eqref{eq:spc-over-Q-and-fun} we obtain an equivalence
	\[ (\sSet_{/\Sing(-)})_*[W^{-1}] \simeq \Fun(\ell(-)^\op, \Spc_*) \]
	of functors $\Top \to \Prr$.\footnote{This was essentially a repetition of the proof of \cite[Prop.~B.1]{ABG} emphasising the naturality of all identifications.}
	We can now pass to left adjoints and restrict to the full subfunctors on compact objects to obtain an equivalence
	\begin{equation}\label{eq:sset-over-Q-and-fun}
	 (\sSet_{/\Sing(-)})_*[W^{-1}]^{\cop} \simeq \ad^{-1}\Fun(\ell(-)^\op, \Spc_*)^{\cop}
	\end{equation}
	of functors $\Top \to \Cre$. Note that the cobase change functors provide a concrete model of the functor $(\sSet_{/\Sing(-)})_*[W^{-1}]^{\cop}$ by \cite[Prop.~5.2.4.6]{htt}. 
	Consequently, \eqref{ret-and-sset-over-Q} and \eqref{eq:sset-over-Q-and-fun} yield an equivalence
	\begin{equation}\label{eq:ret-and-fun}
	 \bR_{\mathrm{fd}}(-)[h^{-1}] \simeq \ad^{-1}\Fun(\ell(-)^\op, \Spc_*)^{\cop}	\end{equation}
	of functors $\Top \to \Cre$.

	 Finally, we note that  $\Fun((-)^\op,\Spc_*) \colon \Spc^\op \to \Prr$ is limit-preserving. Consequently, $\ad^{-1}\Fun((-)^\op,\Spc_*)$ is colimit-preserving as a functor $\Spc \to \Prl$. Therefore,
	\[ \ad^{-1}\Fun((-)^\op,\Spc_*)^{\cop} \colon \Spc \to \Cre \]
	is a colimit-preserving functor that sends the terminal object in $\Spc$ to $\Spc_*^{\cop}$. It follows that
	\begin{equation}\label{eq:fun-and-tensor}
	 \ad^{-1}\Fun((-)^\op,\Spc_*)^{\cop} \simeq - \otimes \Spc_*^{\cop}\ .
	\end{equation}
	The proposition follows by combining \eqref{eq:ret-and-fun} and \eqref{eq:fun-and-tensor}.
\end{proof}

\bibliographystyle{alpha}
\bibliography{unik}

\newcommand{\etalchar}[1]{$^{#1}$}
\begin{thebibliography}{BEKW20b}

\bibitem[ABG18]{ABG}
M.~Ando, A.~J. Blumberg, and D.~Gepner.
\newblock Parametrized spectra, multiplicative {T}hom spectra and the twisted
  {U}mkehr map.
\newblock {\em Geom. Topol.}, 22(7):3761--3825, 2018.

\bibitem[AR94]{AR}
J.~Ad\'amek and J.~Rosick\'y.
\newblock {\em Locally Presentable and Accessible Categories}.
\newblock London Math. Soc. Lecture Notes Series 189, 1994.

\bibitem[Bar03]{bartels:domain}
A.~C. Bartels.
\newblock On the domain of the assembly map in algebraic {{\(K\)}}-theory.
\newblock {\em Algebr. Geom. Topol.}, 3:1037--1050, 2003.

\bibitem[Bar16]{Barwick:Ktheory}
C.~Barwick.
\newblock On the algebraic {$K$}-theory of higher categories.
\newblock {\em J. Topol.}, 9(1):245--347, 2016.

\bibitem[Bar17]{Barwick:Mackey}
C.~Barwick.
\newblock Spectral {M}ackey functors and equivariant algebraic {$K$}-theory
  ({I}).
\newblock {\em Adv. Math.}, 304:646--727, 2017.

\bibitem[BC20]{Bunke:2017aa}
U.~Bunke and D.-C. Cisinski.
\newblock A universal coarse {{\(K\)}}-theory.
\newblock {\em New York J. Math.}, 26:1--27, 2020.

\bibitem[BE]{Bunke:ad}
U.~Bunke and A.~Engel.
\newblock {Topological equivariant coarse $K$-homology}.
\newblock \href{https://arxiv.org/abs/2011.13271}{arXiv:2011.13271}.

\bibitem[BE20a]{ass}
U.~Bunke and A.~Engel.
\newblock Coarse assembly maps.
\newblock {\em J. Noncommut. Geom.}, 14(4):1245--1303, 2020.

\bibitem[BE20b]{buen}
U.~Bunke and A.~Engel.
\newblock {\em Homotopy theory with bornological coarse spaces}, volume 2269 of
  {\em Lect. Notes Math.}
\newblock Cham: Springer, 2020.

\bibitem[BEKW19]{Bunke:ac}
U.~Bunke, A.~Engel, D.~Kasprowski, and C.~Winges.
\newblock Coarse homology theories and finite decomposition complexity.
\newblock {\em Algebr. Geom. Topol.}, 19(6):3033--3074, 2019.

\bibitem[BEKW20a]{equicoarse}
U.~Bunke, A.~Engel, D.~Kasprowski, and C.~Winges.
\newblock Equivariant coarse homotopy theory and coarse algebraic {K}-homology.
\newblock In {\em \(K\)-theory in algebra, analysis and topology. ICM 2018
  satellite school and workshop, La Plata and Buenos Aires, Argentina, July
  16--20 and July 23--27, 2018}, pages 13--104. Providence, RI: American
  Mathematical Society (AMS), 2020.

\bibitem[BEKW20b]{desc}
U.~Bunke, A.~Engel, D.~Kasprowski, and C.~Winges.
\newblock Injectivity results for coarse homology theories.
\newblock {\em Proc. Lond. Math. Soc. (3)}, 121(6):1619--1684, 2020.

\bibitem[BEKW20c]{coarsetrans}
U.~Bunke, A.~Engel, D.~Kasprowski, and C.~Winges.
\newblock Transfers in coarse homology.
\newblock {\em M{\"u}nster J. Math.}, 13(2):353--424, 2020.

\bibitem[BFJR04]{MR2030590}
A.~Bartels, T.~Farrell, L.~Jones, and H.~Reich.
\newblock On the isomorphism conjecture in algebraic {$K$}-theory.
\newblock {\em Topology}, 43(1):157--213, 2004.

\bibitem[BGT13]{MR3070515}
A.~J. Blumberg, D.~Gepner, and G.~Tabuada.
\newblock A universal characterization of higher algebraic {$K$}-theory.
\newblock {\em Geom. Topol.}, 17(2):733--838, 2013.

\bibitem[BHM93]{bhm:trace}
M.~B{\"o}kstedt, W.~C. Hsiang, and I.~Madsen.
\newblock The cyclotomic trace and algebraic {K}-theory of spaces.
\newblock {\em Invent. Math.}, 111(3):465--539, 1993.

\bibitem[BKW]{fvvsdfvsdfvfvsdfv}
U.~Bunke, D.~Kasprowski, and C.~Winges.
\newblock On the {F}arrell-{J}ones conjecture for localising invariants.
\newblock \href{https://arxiv.org/abs/2111.02490}{arXiv:2111.02490}.

\bibitem[BKW21]{Bunke:aa}
U.~Bunke, D.~Kasprowski, and C.~Winges.
\newblock Split injectivity of {{\(A\)}}-theoretic assembly maps.
\newblock {\em Int. Math. Res. Not.}, 2021(2):885--947, 2021.

\bibitem[BL11]{Bartels:2011fk}
A.~Bartels and W.~L{\"u}ck.
\newblock {The Farrell-Hsiang method revisited}.
\newblock {\em Math. Ann.}, 354:209--226, 2011.

\bibitem[BLR08]{blr}
A.~Bartels, W.~L\"{u}ck, and H.~Reich.
\newblock {The $K$-theoretic Farrell--Jones conjecture for hyperbolic groups}.
\newblock {\em Invent. math.}, 172:29--70, 2008.

\bibitem[BR07]{br:asymptoticdim}
A.~Bartels and D.~Rosenthal.
\newblock On the {{\(K\)}}-theory of groups with finite asymptotic dimension.
\newblock {\em J. Reine Angew. Math.}, 612:35--57, 2007.

\bibitem[Bun19]{Bunke:ae}
U.~Bunke.
\newblock Homotopy theory with *categories.
\newblock {\em Theory Appl. Categ.}, 34(27):781--853, 2019.

\bibitem[CDH{\etalchar{+}}]{thenine2}
B.~Calm{\`e}s, E.~Dotto, Y.~Harpaz, F.~Hebestreit, M.~Land, K.~Moi, D.~Nardin,
  T.~Nikolaus, and W.~Steimle.
\newblock Hermitian {K}-theory for stable {{\(\infty\)}}-categories. {II}:
  {Cobordism categories and additivity}.
\newblock \href{https://arxiv.org/abs/2009.07224}{arXiv:2009.07224}.

\bibitem[Cis19]{Cisinski:2017}
D.-C. Cisinski.
\newblock {\em Higher categories and homotopical algebra}, volume 180 of {\em
  Cambridge studies in advanced mathematics}.
\newblock Cambridge University Press, 2019.
\newblock Available online under
  \url{http://www.mathematik.uni-regensburg.de/cisinski/CatLR.pdf}.

\bibitem[CP95]{calped}
G.~Carlsson and E.~K. Pedersen.
\newblock Controlled algebra and the {N}ovikov conjectures for {$K$}- and
  {$L$}-theory.
\newblock {\em Topology}, 34(3):731--758, 1995.

\bibitem[DL98]{davis_lueck}
J.~F. Davis and W.~L\"{u}ck.
\newblock {Spaces over a Category and Assembly Maps in Isomorphism Conjectures
  in $K$- and $L$-Theory}.
\newblock {\em $K$-Theory}, 15:201--252, 1998.

\bibitem[FRR95]{frr:novikov}
S.~C. Ferry, A.~Ranicki, and J.~Rosenberg.
\newblock A history and survey of the {Novikov} conjecture.
\newblock In {\em Novikov conjectures, index theorems and rigidity. Vol. 1.
  Based on a conference of the Mathematisches Forschungsinstitut Oberwolfach
  held in September 1993}, pages 7--66. Cambridge: Cambridge University Press,
  1995.

\bibitem[GHN17]{GHN:2017}
D.~Gepner, R.~Haugseng, and T.~Nikolaus.
\newblock Lax colimits and free fibrations in {$\infty$}-categories.
\newblock {\em Doc. Math.}, 22:1225--1266, 2017.

\bibitem[Hei]{Heiss:aa}
D.~Heiss.
\newblock Generalized bornological coarse spaces and coarse motivic spectra.
\newblock \href{https://arxiv.org/abs/1907.03923}{arXiv:1907.03923}.

\bibitem[HHLN23]{hhln:spans}
R.~Haugseng, F.~Hebestreit, S.~Linskens, and J.~Nuiten.
\newblock Two-variable fibrations, factorisation systems and $\infty
  $-categories of spans.
\newblock {\em Forum of Mathematics, Sigma}, 11, 2023.

\bibitem[HPR96]{higson_pedersen_roe}
N.~Higson, E.~K. Pedersen, and J.~Roe.
\newblock {$C^\ast$-algebras and controlled topology}.
\newblock {\em $K$-Theory}, 11:209--239, 1996.

\bibitem[Kas15]{kasprowski:fdc}
D.~Kasprowski.
\newblock On the {{\(K\)}}-theory of groups with finite decomposition
  complexity.
\newblock {\em Proc. Lond. Math. Soc. (3)}, 110(3):565--592, 2015.

\bibitem[KL05]{kl:novikov}
M.~Kreck and W.~L{\"u}ck.
\newblock {\em The {Novikov} conjecture. {Geometry} and algebra}, volume~33 of
  {\em Oberwolfach Semin.}
\newblock Basel: Birkh{\"a}user, 2005.

\bibitem[KW19]{kaswin}
D.~Kasprowski and C.~Winges.
\newblock Algebraic {K}-theory of stable $\infty$-categories via binary
  complexes.
\newblock {\em Journal of Topology}, 12(2):442--462, 2019.

\bibitem[Lura]{HA}
J.~Lurie.
\newblock Higher {A}lgebra.
\newblock Available at
  \href{https://www.math.ias.edu/~lurie/papers/HA.pdf}{https://www.math.ias.edu/~lurie/papers/HA.pdf}.

\bibitem[Lurb]{SAG}
J.~Lurie.
\newblock Spectral {A}lgebraic {G}eometry.
\newblock Available at
  \href{https://www.math.ias.edu/~lurie/papers/SAG-rootfile.pdf}{https://www.math.ias.edu/~lurie/papers/SAG-rootfile.pdf}.

\bibitem[Lur09]{htt}
J.~Lurie.
\newblock {\em Higher topos theory}, volume 170 of {\em Annals of Mathematics
  Studies}.
\newblock Princeton University Press, Princeton, NJ, 2009.

\bibitem[Mit01]{MR1834777}
P.~D. Mitchener.
\newblock Coarse homology theories.
\newblock {\em Algebr. Geom. Topol.}, 1:271--297, 2001.

\bibitem[Nov71]{novikov:hermK2}
S.~P. Novikov.
\newblock Algebraic construction and properties of hermitian analogs of
  {K}-theory over rings with involution from the viewpoint of hamiltonian
  formalism. {Applications} to differential topology and the theory of
  characteristic classes. {II}.
\newblock {\em Math. USSR, Izv.}, 4:479--505, 1971.

\bibitem[Pst23]{Pstragowski}
P.~Pstr{\k{a}}gowski.
\newblock Synthetic spectra and the cellular motivic category.
\newblock {\em Invent. Math.}, 232(2):553--681, 2023.

\bibitem[PW85]{MR802790}
E.~K. Pedersen and C.~A. Weibel.
\newblock A nonconnective delooping of algebraic {$K$}-theory.
\newblock In {\em Algebraic and geometric topology ({N}ew {B}runswick,
  {N}.{J}., 1983)}, volume 1126 of {\em Lecture Notes in Math.}, pages
  166--181. Springer, Berlin, 1985.

\bibitem[Ran92]{ranicki:algL}
A.~A. Ranicki.
\newblock {\em Algebraic {L}-theory and topological manifolds}, volume 102 of
  {\em Camb. Tracts Math.}
\newblock Cambridge: Cambridge University Press, 1992.

\bibitem[Roe88]{roe_index_1}
J.~Roe.
\newblock {A}n {I}ndex {T}heorem on {O}pen {M}anifolds, {I}.
\newblock {\em J. Differential Geom.}, 27:87--113, 1988.

\bibitem[Roe93]{MR1147350}
J.~Roe.
\newblock Coarse cohomology and index theory on complete {R}iemannian
  manifolds.
\newblock {\em Mem. Amer. Math. Soc.}, 104(497):x+90, 1993.

\bibitem[Roe96]{roe_index_coarse}
J.~Roe.
\newblock {\em Index theory, coarse geometry, and topology of manifolds},
  volume~90 of {\em CBMS Regional Conference Series in Mathematics}.
\newblock Published for the Conference Board of the Mathematical Sciences,
  Washington, DC; by the American Mathematical Society, Providence, RI, 1996.

\bibitem[Roe03]{roe_lectures_coarse_geometry}
J.~Roe.
\newblock {\em Lectures on coarse geometry}, volume~31 of {\em University
  Lecture Series}.
\newblock American Mathematical Society, Providence, RI, 2003.

\bibitem[RTY14]{gty:novikov}
D.~A. Ramras, Romain Tessera, and Guoliang Yu.
\newblock Finite decomposition complexity and the integral {Novikov} conjecture
  for higher algebraic {{\(K\)}}-theory.
\newblock {\em J. Reine Angew. Math.}, 694:129--178, 2014.

\bibitem[Sch04]{MR2079996}
M.~Schlichting.
\newblock Delooping the {$K$}-theory of exact categories.
\newblock {\em Topology}, 43(5):1089--1103, 2004.

\bibitem[TT90]{TT}
R.~W. Thomason and T.~Trobaugh.
\newblock {Higher Algebraic ${K}$-theory of Schemes and of Derived Categories}.
\newblock In {\em The {G}rothendieck {F}estschrift {III}}, pages 247--435.
  {Birkh\"auser}, 1990.

\bibitem[UW19]{Ullmann:2015aa}
M.~Ullmann and C.~Winges.
\newblock On the {F}arrell--{J}ones conjecture for algebraic {$K$}-theory of
  spaces: the {F}arrell--{H}siang method.
\newblock {\em Ann. K-Theory}, 4(1):57--138, 2019.

\bibitem[Wal70]{wall}
C.~T.~C. Wall.
\newblock {\em Surgery on compact manifolds}, volume~1 of {\em Lond. Math. Soc.
  Monogr.}
\newblock Academic Press, London, 1970.

\bibitem[Wei02]{MR1880196}
M.~Weiss.
\newblock Excision and restriction in controlled {$K$}-theory.
\newblock {\em Forum Math.}, 14(1):85--119, 2002.

\end{thebibliography}

\end{document}